\documentclass[10pt, reqno]{amsart}
\usepackage{appendix}
\usepackage[a4paper, top=1.5cm, bottom=1.5cm, left=2cm, right=2cm, heightrounded, marginparwidth=3.5cm, marginparsep=.2cm, centering]{geometry}
\usepackage{xcolor}
\usepackage[utf8]{inputenc}
\usepackage[colorlinks=true,hyperindex,pagebackref,linktocpage=true]{hyperref}
\usepackage{cleveref}
\usepackage{amsfonts}
\usepackage{amssymb}
\usepackage{tikz-cd}

\hypersetup{
	colorlinks,
	linkcolor={violet!60!black},
	citecolor={cyan!60!black},
	urlcolor={orange!60!black}
}
\usepackage{stmaryrd}
\usepackage{mathrsfs}  
\usepackage{varwidth}
\theoremstyle{plain}
\newtheorem{theorem}{Theorem}[section]
\newtheorem{proposition}[theorem]{Proposition}
\newtheorem{lemma}[theorem]{Lemma}
\newtheorem{corollary}[theorem]{Corollary}
\newtheorem{conjecture}[theorem]{Conjecture}

\theoremstyle{definition}
\newtheorem{definition}[theorem]{Definition}
\newtheorem{example}[theorem]{Example}
\newtheorem{remark}[theorem]{Remark}

\newtheorem{assumption}[theorem]{Assumption}

\newcommand{\nc}{\newcommand}
\nc{\on}{\operatorname}

\nc{\Q}{\mathbb{Q}}
\nc{\Z}{\mathbb{Z}}
\nc{\cl}{\mathrm{cl}}

\nc{\fraka}{{\mathfrak a}} \nc{\bba}{{\mathbf a}}
\nc{\frakb}{{\mathfrak b}}
\nc{\frakc}{{\mathfrak c}}
\nc{\frakd}{{\mathfrak d}}
\nc{\frake}{{\mathfrak e}}
\nc{\frakf}{{\mathfrak f}}
\nc{\frakg}{{\mathfrak g}}
\nc{\frakh}{{\mathfrak h}}
\nc{\fraki}{{\mathfrak i}}
\nc{\frakj}{{\mathfrak j}}
\nc{\frakk}{{\mathfrak k}}
\nc{\frakl}{{\mathfrak l}}
\nc{\frakm}{{\mathfrak m}}
\nc{\frakn}{{\mathfrak n}}
\nc{\frako}{{\mathfrak o}}
\nc{\frakp}{{\mathfrak p}}
\nc{\frakq}{{\mathfrak q}}
\nc{\frakr}{{\mathfrak r}}
\nc{\fraks}{{\mathfrak s}}
\nc{\frakt}{{\mathfrak t}}
\nc{\fraku}{{\mathfrak u}}
\nc{\frakv}{{\mathfrak v}}
\nc{\frakw}{{\mathfrak w}}
\nc{\frakx}{{\mathfrak x}}
\nc{\fraky}{{\mathfrak y}}
\nc{\frakz}{{\mathfrak z}}
\nc{\frakA}{{\mathfrak A}}
\nc{\frakB}{{\mathfrak B}}
\nc{\frakC}{{\mathfrak C}}
\nc{\frakD}{{\mathfrak D}}
\nc{\frakE}{{\mathfrak E}}
\nc{\frakF}{{\mathfrak F}}
\nc{\frakG}{{\mathfrak G}}
\nc{\frakH}{{\mathfrak H}}
\nc{\frakI}{{\mathfrak I}}
\nc{\frakJ}{{\mathfrak J}}
\nc{\frakK}{{\mathfrak K}}
\nc{\frakL}{{\mathfrak L}}
\nc{\frakM}{{\mathfrak M}}
\nc{\frakN}{{\mathfrak N}}
\nc{\frakO}{{\mathfrak O}}
\nc{\frakP}{{\mathfrak P}}
\nc{\frakQ}{{\mathfrak Q}}
\nc{\frakR}{{\mathfrak R}}
\nc{\frakS}{{\mathfrak S}}
\nc{\frakT}{{\mathfrak T}}
\nc{\frakU}{{\mathfrak U}}
\nc{\frakV}{{\mathfrak V}}
\nc{\frakW}{{\mathfrak W}}
\nc{\frakX}{{\mathfrak X}}
\nc{\frakY}{{\mathfrak Y}}
\nc{\frakZ}{{\mathfrak Z}}
\nc{\bbA}{{\mathbb A}}
\nc{\bbB}{{\mathbb B}}
\nc{\bbC}{{\mathbb C}}
\nc{\bbD}{{\mathbb D}}
\nc{\bbE}{{\mathbb E}}
\nc{\bbF}{{\mathbb F}} \nc{\bbf}{{\mathbf f}}
\nc{\bbG}{{\mathbb G}}
\nc{\bbH}{{\mathbb H}}
\nc{\bbI}{{\mathbb I}}
\nc{\bbJ}{{\mathbb J}}
\nc{\bbK}{{\mathbb K}}
\nc{\bbL}{{\mathbb L}}
\nc{\bbM}{{\mathbb M}}
\nc{\bbN}{{\mathbb N}}
\nc{\bbO}{{\mathbb O}}
\nc{\bbP}{{\mathbb P}}
\nc{\bbQ}{{\mathbb Q}}
\nc{\bbR}{{\mathbb R}}
\nc{\bbS}{{\mathbb S}}
\nc{\bbT}{{\mathbb T}}
\nc{\bbU}{{\mathbb U}}
\nc{\bbV}{{\mathbb V}}
\nc{\bbW}{{\mathbb W}}
\nc{\bbX}{{\mathbb X}}
\nc{\bbY}{{\mathbb Y}}
\nc{\bbZ}{{\mathbb Z}}
\nc{\calA}{{\mathcal A}}
\nc{\calB}{{\mathcal B}}
\nc{\calC}{{\mathcal C}}
\nc{\calD}{{\mathcal D}}
\nc{\calE}{{\mathcal E}}
\nc{\calF}{{\mathcal F}}
\nc{\calG}{{\mathcal G}}
\nc{\calH}{{\mathcal H}}
\nc{\calI}{{\mathcal I}}
\nc{\calJ}{{\mathcal J}}
\nc{\calK}{{\mathcal K}}
\nc{\calL}{{\mathcal L}}
\nc{\calM}{{\mathcal M}}
\nc{\calN}{{\mathcal N}}
\nc{\calO}{{\mathcal O}}
\nc{\calP}{{\mathcal P}}
\nc{\calQ}{{\mathcal Q}}
\nc{\calR}{{\mathcal R}}
\nc{\calS}{{\mathcal S}}
\nc{\calT}{{\mathcal T}}
\nc{\calU}{{\mathcal U}}
\nc{\calV}{{\mathcal V}}
\nc{\calW}{{\mathcal W}}
\nc{\calX}{{\mathcal X}}
\nc{\calY}{{\mathcal Y}}
\nc{\calZ}{{\mathcal Z}}

\nc{\scrA}{{\mathscr A}}
\nc{\scrB}{{\mathscr B}}
\nc{\scrC}{{\mathscr C}}
\nc{\scrD}{{\mathscr D}}
\nc{\scrE}{{\mathscr E}}
\nc{\scrF}{{\mathscr F}}
\nc{\scrG}{{\mathscr G}}
\nc{\scrH}{{\mathscr H}}
\nc{\scrI}{{\mathscr J}}
\nc{\scrJ}{{\mathscr I}}
\nc{\scrK}{{\mathscr K}}
\nc{\scrL}{{\mathscr L}}
\nc{\scrM}{{\mathscr M}}
\nc{\scrN}{{\mathscr N}}
\nc{\scrO}{{\mathscr O}}
\nc{\scrP}{{\mathscr P}}
\nc{\scrQ}{{\mathscr Q}}
\nc{\scrR}{{\mathscr R}}
\nc{\D}{{\on{D}}}
\nc{\Div}{{\on{Div}}}
\nc{\Perv}{{\on{Perv}}}

\nc{\bnu}{{\bar{ \nu}}}
\nc{\olO}{\bar{\calO}}

\nc{\al}{{\alpha}} 
\nc{\be}{{\beta}}
\nc{\ga}{{\gamma}} \nc{\Ga}{{\Gamma}}
\nc{\hGa}{\hat{\Gamma}}
\nc{\ve}{{\varepsilon}} 
\nc{\la}{{\lambda}} \nc{\La}{{\Lambda}}
\nc{\om}{\omega} \nc{\Om}{\Omega} 
\nc{\sig}{{\sigma}} \nc{\Sig}{{\Sigma}}
\nc{\dR}{{\mathrm{dR}}}
\nc{\Perf}{{\mathrm{Perf}}}
\nc{\Gm}{{\mathbb{G}_m}}
\nc{\colim}{{\on{colim}}}

\makeatletter
\DeclareFontEncoding{LS1}{}{}
\DeclareFontSubstitution{LS1}{stix}{m}{n}
\DeclareMathAlphabet{\rhomalpha}{LS1}{stixscr}{m}{n}
\makeatother
\DeclareMathOperator{\RHom}{{{\it R} \mathscr{H}\mkern-1.8mu o\mkern-0.7mu m\mkern-0.5mu}}

\nc{\Spa}{\on{{Spa}}}
\nc{\Spd}{\on{{Spd}}}
\nc{\tnb}{\psi_{\rm tame}}
\nc{\oM}{\overline{{M}}}
\nc{\op}{{\on{op}}}
\nc{\ad}{{\on{ad}}}
\nc{\alg}{{\on{alg}}}
\nc{\Ad}{{\on{Ad}}}
\nc{\Adm}{{\on{Adm}}} \nc{\aff}{{\on{af}}}
\nc{\Aut}{{\on{Aut}}}
\nc{\Bun}{{\on{Bun}}}
\nc{\cha}{{\on{char}}}
\nc{\der}{{\on{der}}}
\nc{\Der}{{\on{Der}}}
\nc{\diag}{{\on{diag}}}
\nc{\End}{{\on{End}}}
\nc{\Fl}{{\calF\!\ell}}
\nc{\Tr}{{\on{Transp}}}
\nc{\TR}{{\calT\!\calR}}
\nc{\Gal}{{\on{Gal}}}
\nc{\Gr}{{\on{Gr}}}
\nc{\Hk}{{\on{Hk}}}
\nc{\rH}{{\on{H}}}
\nc{\Hom}{{\on{Hom}}}
\nc{\IC}{{\on{IC}}}
\nc{\id}{{\on{id}}}
\nc{\Id}{{\on{Id}}}
\nc{\ind}{{\on{ind}}}
\nc{\Ind}{{\on{Ind}}}
\nc{\Lie}{{\on{Lie}}}
\nc{\Pic}{{\on{Pic}}}
\nc{\pr}{{\on{pr}}}
\nc{\Res}{{\on{Res}}}
\nc{\res}{{\on{res}}} \nc{\Sat}{{\on{Sat}}}
\nc{\spc}{{\on{sc}}}
\nc{\drv}{{\on{der}}}
\nc{\sgn}{{\on{sgn}}}
\nc{\Spec}{{\on{Spec}}}\nc{\Spf}{\on{Spf}} 
\nc{\Sph}{\on{Sph}}
\nc{\St}{{\on{St}}}
\nc{\tr}{{\on{tr}}}
\nc{\Mod}{{\mathrm{-Mod}}}
\nc{\Hilb}{{\on{Hilb}}} 
\nc{\Ext}{{\on{Ext}}} 
\nc{\vs}{{\on{Vec}}}
\nc{\ev}{{\on{ev}}}
\nc{\nO}{{\breve{\calO}}}
\nc{\tS}{{\tilde{S}}}
\nc{\spe}{{\on{sp}}}
\nc{\loc}{{\on{loc}}}
\nc{\pre}{{\on{pre}}}

\let\x\times

\renewcommand{\x}{\times}

\renewcommand{\r}{\rightarrow}
\newcommand{\lr}{\longrightarrow}

\newcommand{\hr}{\hookrightarrow}
\nc{\co}{\colon}
\nc{\dia}{{\diamondsuit}}

\nc{\nscrR}{{\mathscr{R}^{\on{nr}}}}

\nc{\GL}{{\on{GL}}}

\nc{\Gl}{\on{Gl}} 
\nc{\GSp}{{\on{GSp}}}
\nc{\gl}{{\frakg\frakl}}
\nc{\SL}{{\on{SL}}} 
\nc{\SU}{{\on{SU}}} 
\nc{\SO}{{\on{SO}}}
\nc{\PGL}{{\on{PGL}}}

\nc{\Conv}{{\on{Conv}}}
\nc{\Rep}{{\on{Rep}}}
\nc{\Dom}{{\on{Dom}}}
\nc{\red}{{\on{red}}}
\nc{\act}{{\on{act}}}
\nc{\nr}{{\on{nr}}}
\nc{\ctf}{{\on{ctf}}}

\nc{\str}{{\on{-}}} 
\nc{\os}{{\bar{s}}}
\nc{\oeta}{{\bar{\eta}}}

\nc{\et}{\textup{\'et}}

\nc{\hookto}{\hookrightarrow}
\nc{\longto}{\longrightarrow}
\nc{\leftto}{\leftarrow}
\nc{\onto}{\twoheadrightarrow}
\nc{\lonto}{\twoheadleftarrow}

\nc{\pot}[1]{ [\hspace{-0,5mm}[ {#1} ]\hspace{-0,5mm}] }
\nc{\rpot}[1]{ (\hspace{-0,7mm}( {#1} )\hspace{-0,7mm}) }

\numberwithin{equation}{section}

\begin{document}
	
	\title{On the $p$-adic theory of local models}
	
	\author[J. Ansch\"utz, I. Gleason, J. Louren\c{c}o, T. Richarz]{Johannes Ansch\"utz, Ian Gleason, Jo\~ao Louren\c{c}o, Timo Richarz}
	
	
	
	
	
	\begin{abstract}
		We prove the Scholze--Weinstein conjecture on the existence and uniqueness of local models for local Shimura varieties, as well as the test function conjecture of Haines--Kottwitz in this framework.
		To this end, we establish a specialization principle for well-behaved $p$-adic kimberlites, show that these include the v-sheaf local models, determine their special fibers using hyperbolic localization for the étale cohomology of small v-stacks, and analyze the resulting specialization morphism using convolution.	
		\end{abstract}
	\date{\today}
	
	\maketitle
	\tableofcontents

	\section{Introduction}
	The general theory of Shimura varieties has been developed by Deligne \cite{Del71,Del79} in the seventies. 
	It generalizes classical objects such as modular curves, moduli spaces of principally polarized abelian varieties or Hilbert modular varieties. 
	Shimura varieties naturally occur in the search for higher reciprocity laws within the Langlands program \cite{Lan79}. 
	Their arithmetic properties are encoded in the reduction to positive characteristic $p>0$ and have contributed to spectacular developments in number theory and arithmetic geometry in the past decades.
	
	Local models are flat projective schemes over complete discrete valuation rings of characteristic $(0,p)$ that are designed to model the singularities in the reduction modulo $p$ of Shimura varieties with parahoric level structure. 
	Starting from the pioneering works \cite{DR73, Rap90, CN90, dJ93, DP94}, the theory is formalized to some extent in the book of Rapoport--Zink \cite{RZ96} for those Shimura varieties arising as moduli problems of abelian varieties with extra structures. 
	The recent works of Kisin--Pappas \cite{kisin_pappas_integral_models_of_shimura_varieties_with_parahoric_level_structure} and Kisin--Pappas--Zhou \cite{KisinPappasZhou_IntegralModels} construct natural integral models for all Shimura varieties of abelian type with parahoric level structure when $p>2$. 
	The geometric properties of the corresponding local models are studied in a series of works notably by Faltings, G\"ortz, Pappas and Rapoport \cite{Fal97, Fal01, Pap00, Gor01, Gor03, Fal03, PR03, PR05, PR08, PR09}, see the survey article \cite{PRS10} for details and further references. 
	A breakthrough is due to Pappas and Zhu \cite{PZ13, Zhu14}, who gave a purely group-theoretic construction of (flat) local models, later refined by Levin \cite{Lev16} and the third named author \cite{Lou19}. 
	Roughly, the local models in this approach are constructed as flat, closed subschemes in a power series affine Grassmannian, which depends on certain auxiliary choices, see \Cref{subsec_comparison}. 
	A more functorial approach (without ad hoc choices) was initiated by Scholze--Weinstein \cite{SW20}, 
	using $p$-adic geometry. 
	Unfortunately, the approach has the drawback of not producing schemes, at least a priori.
	
	The aim of the present manuscript is to connect the scheme-theoretic local models to Scholze--Weinstein's $p$-adic approach. 	
	More precisely, we prove the Scholze--Weinstein conjecture \cite[Conjecture 21.4.1]{SW20} on the existence and uniqueness of (weakly normal) local models, representing minuscule portions of the parahoric Beilinson--Drinfeld Grassmannian.	
	Our methods allow us to prove, furthermore, the test function conjecture of Haines--Kottwitz \cite[Conjecture 6.1.1]{Hai14} for these local models in full generality, expressing the trace of Frobenius function on the nearby cycles sheaf in terms of spectral data.
	
	These local models are intimately related with moduli spaces of $p$-adic shtukas by \cite[Lecture XXV]{SW20}. 
	Recently, Pappas--Rapoport \cite{PR21} made progress in the study of moduli spaces of $p$-adic shtukas and their relation to Shimura varieties, partially relying on a positive solution of the Scholze--Weinstein conjecture as given here.
	Let us now discuss our results in more detail.
	
	\subsection{Main results}\label{intro-main-results}
	Fix a prime number $p$.
	Denote by $\bbQ_p$ the field of $p$-adic numbers. 
	(In the main body of the text, we allow more general $p$-adic base fields.)
	Let $G$ be a connected reductive $\bbQ_p$-group and $\calG$ a parahoric $\bbZ_p$-model in the sense of Bruhat--Tits \cite{BT84}.
	Let $\mu$ be a conjugacy class of geometric cocharacters in $G$. 
	Denote by $E/\bbQ_p$ its reflex field with ring of integers $O_E$ and finite residue field $k$. 
	
	For a brief explanation on the following terminology the reader is referred to \Cref{v-sheaves-intro}.
	Scholze--Weinstein \cite[Section 20.3]{SW20} introduce the v-sheaf Beilinson--Drinfeld Grassmannian $\Gr_\calG\to \Spd \bbZ_p$. 
	Its generic fiber is the $B_\dR^+$-affine Grassmannian $\Gr_{G}$ and its special fiber is the v-sheaf $\Fl_{\calG}^\dia$ associated to the Witt vector partial affine flag variety defined by Zhu \cite{Zhu17} and further studied by Bhatt--Scholze \cite{BS17}.
	Attached to the pair $(\calG, \mu)$ is the v-sheaf local model  
	\begin{equation}\label{definition-local-model-intro}
	\calM_{\calG,\mu} \;\subset\; \Gr_{\calG}|_{\Spd O_E}
	\end{equation}
	defined as the v-closure (\Cref{sec:closures-definition-v-sheaf-closure}) of the affine Schubert variety $\Gr_{G,\mu}$, in analogy to the local models from \cite{PZ13}.
	If $\mu$ is minuscule (that is, $\langle \mu,a\rangle\in \{0,\pm 1 \}$ for every root $a$ of $G_{\bbC_p}$), then the generic fiber $\calM_{\calG,\mu}|_{\Spd E}=\Gr_{G,\mu}$ is canonically isomorphic to $\calF_{G,\mu}^\dia$, the v-sheaf associated with the homogenous space $\calF_{G,\mu}$ of parabolic subgroups of type $\mu$. 
	Scholze--Weinstein state the following conjecture \cite[Conjecture~21.4.1]{SW20} on the existence of local models: 
	
	\begin{conjecture}[Scholze--Weinstein]\label{SW-conjecture}
	If $\mu$ is minuscule, then there is a flat projective $O_E$-scheme $\calM_{\calG,\mu}^{\on{sch}}$ with reduced special fiber and with a closed immersion
	\begin{equation}\label{equation-SW-conjecture}
		\calM_{\calG,\mu}^{\on{sch},\dia}\hookto \Gr_{\calG}|_{\Spd O_E}
	\end{equation}
	whose generic fiber is $\calF_{G,\mu}^\dia\cong \Gr_{G,\mu}$.
	\end{conjecture}
	
	Note that the schematic local model $\calM_{\calG,\mu}^{\on{sch}}$, whenever it exists, is normal and uniquely characterized by $\calM_{\calG,\mu}^{\on{sch},\dia}\cong \calM_{\calG,\mu}$ under the closed immersion \eqref{equation-SW-conjecture}, see \cite[Remark~21.4.2]{SW20} and the text thereafter. 	
	In the thesis of the third named author, \Cref{SW-conjecture} is separated in a ``representability part'' and a ``geometry part''.
	More precisely, \cite[Conjecture IV.4.18]{Lou20} states the existence of a unique flat, projective and weakly normal (\Cref{sec:two-diff-diam-definition-semi-normal-etc}) $O_E$-scheme $\calM_{\calG,\mu}^{\on{sch}}$ together with a closed immersion as in \eqref{equation-SW-conjecture} whose generic fiber is $\calF_{G,\mu}^\dia\cong \Gr_{G,\mu}$ whereas \cite[Conjecture IV.4.19]{Lou20} states that the special fiber of $\calM_{\calG,\mu}^{\on{sch}}$ is reduced. 
	The main result of the present work proves the representability part in full generality and the geometry part for all $G$ with the exception of two non-split groups for $p=2,3$:
	
	\begin{theorem}[\Cref{thm_local_model_representable}, \Cref{cor_special_fiber_local_model_reduced}]\label{SW_conjecture_intro}
		Let $\mu$ be minuscule. 
		Then, the following hold:
		\begin{enumerate}
			\item\label{SW_conjecture_intro:it1} Representability:
				There is a unique (up to unique isomorphism) flat, projective and weakly normal $O_E$-scheme $\calM_{\calG,\mu}^{\on{sch}}$ with a closed immersion  
					\begin{equation}\label{equation-SW_conjecture_intro}
						\calM_{\calG,\mu}^{\on{sch},\dia}\hookto \Gr_{\calG}|_{\Spd O_E}
					\end{equation}
				whose generic fiber is $\calF_{G,\mu}^\dia\cong \Gr_{G,\mu}$.
				In addition, $\calM_{\calG,\mu}^{\on{sch}}$ admits a unique $\calG_{O_E}$-action such that the isomorphism $\calM_{\calG,\mu}^{\on{sch},\dia}\cong\calM_{\calG,\mu}$ induced from \eqref{equation-SW_conjecture_intro} is $\calG_{O_E}^\dia$-equivariant.
			\item\label{SW_conjecture_intro:it2} Geometry: 
				Assume that the adjoint group $G_\ad$ satisfies the following conditions over the completion of the maximal unramified extension $\breve\bbQ_p/\bbQ_p$:
				\begin{enumerate}
					\item If $p=2$, then $G_\ad$ has no odd unitary $\breve \bbQ_2$-factors.
					\item If $p=3$, then $G_\ad$ has no triality $\breve{\bbQ}_3$-factors.
				\end{enumerate}
				Then, the special fiber of $\calM_{\calG,\mu}^{\on{sch}}$ is reduced and equal to $\calA_{\calG,\mu}^{\on{can}}$, the canonical deperfection of the $\mu$-admissible locus inside $\Fl_\calG$.
				Moreover, it is Cohen-Macaulay, weakly normal and Frobenius split. 
		\end{enumerate}
	\end{theorem}
	
	Our result provides a functorial and group-theoretic framework for the theory of local models that goes beyond the previous results \cite{SW20, HPR20, Lou20} for certain pairs $(G,\mu)$ of abelian type using Hodge embeddings. 
	The scheme $\calM_{\calG,\mu}^{\on{sch}}$ agrees with the local models defined in \cite{PZ13, Lev16, Lou19} whenever $p\nmid |\pi_1(G_\der)|$ and otherwise with their weak normalization
	(weak normalization is necessary, in general, due to the existence of non-normal Schubert varieties \cite{HLR18}). 
	The works \cite{PZ13, Lev16, Lou19} are complemented by the results of Fakhruddin--Haines--L.--R.~\cite{FHLR22} which handles new cases for wildly ramified groups and leads to the conditions in \Cref{SW_conjecture_intro}\eqref{SW_conjecture_intro:it2}.
	Invoking the results from \cite{Lou20}, the conditions can be relaxed with our methods: for $p=2$ one excludes odd unitary groups defined by a ramified, quadratic root-of-unit extension; for $p=3$ one excludes trialities defined by a ramified, cubic non-(root-of-a-prime) extension.
	In particular, our result is complete for all primes $p\geq 5$ and, up to the two non-split examples, also for $p=2,3$.
	We remark that, after the first version of this paper was written, the remaining cases were handled in \cite{GL24}. 
	Hence, \Cref{SW-conjecture} is proven in all cases, but besides reducedness of the special fiber the additional geometric properties stated in \Cref{SW_conjecture_intro}\eqref{SW_conjecture_intro:it2} are not addressed in the odd unitary case in \cite{GL24}. 
	This is done in the work of Cass and the third named author \cite{CL24} which implies \Cref{SW_conjecture_intro}\eqref{SW_conjecture_intro:it2} without restrictions except that Cohen--Macaulayness is not proven for $p=2$ and odd unitary groups defined by ramified, quadratic root-of-unit extensions, see \Cref{remark_cass-lou,last_cases_remark}. 
	In particular, these works settle our \Cref{conjecture_special_fiber} stating that the special fiber of $\calM_{\calG,\mu}^{\on{sch}}$ is equal to $\calA_{\calG,\mu}^{\on{can}}$ in all cases.
	Furthermore, we conjecture that the singularities of $\calM_{\calG,\mu}^{\on{sch}}$ are pseudo-rational (hence, Cohen--Macaulay) in all cases, see \Cref{pseudo_rational_conjecture}. 
	In addition, let us remark that finding useful moduli interpretations of $\calM_{\calG,\mu}^{\on{sch}}$ is an interesting and difficult problem that maybe does not have a uniform solution. There are however many interesting special cases \cite{Pap00, Gor01, Gor03, PR03, PR05, PR09}.
	
	As a cohomological application, we prove the test function conjecture for $\calM_{\calG,\mu}$, all primes $p$ and all pairs $(\calG,\mu)$ 
	as above (for general $\mu$ which are not necessarily minuscule)
	, see \Cref{test_function_section}.
	Namely, fix an auxiliary prime $\ell\neq p$, a square root $\sqrt q$ of the residue cardinality of $E$ and an embedding $E\hr \bar\bbQ_p$. 
	Put $\Lambda=\bbQ_\ell(\sqrt q)$ which we will use as sheaf coefficients.
	Let $\Gamma$ be the absolute Galois group of $E$ with inertia subgroup $I$, and fix a lift $\Phi\in \Gamma$ of geometric Frobenius.
	Let $E_0=W(k)[{1\over p}]$ be the maximal unramified subextension of $E/\bbQ_p$.
	Let $\IC_{\mu}$ be the intersection complex on $\Gr_{G,\mu}$ with $\Lambda$-coefficients constructed in \cite[Chapter VI]{FS21} and normalized to be of weight zero as in \eqref{eq:normalized_IC}.
	(We remark that $\IC_{\mu}$ exists due to the orbit stratification on $\Gr_{G,\mu}$. So, its existence is a feature of the particular geometry at hand, not a general sheaf-theoretic construction for diamonds.)
	The function $\tau_{\calG,\mu}^\Phi\co G(E_0)/\calG(O_{E_0})\to \Lambda$ is defined, up to sign, by the alternating trace of $\Phi$ on the nearby cycle stalks
	\begin{equation}\label{test_function_intro}
	\tau_{\calG,\mu}^\Phi(x)= (-1)^{\langle2\rho, \mu\rangle}\sum_{i\in \bbZ}(-1)^i\on{trace}\big(\Phi\,|\, \on{R}^i\! \Psi_{\calM_{\calG,\mu}}(\IC_{\mu})_{\bar{x}}\big), 
	\end{equation}
	if $x\in \calM_{\calG,\mu}(\Spd k)$ and $0$ else. 
	Here $\calM_{\calG,\mu}(\Spd k)$ is viewed as a subset of $\Fl_\calG(k)=G(E_0)/\calG(O_{E_0})$, and the trace is well-defined by \Cref{ula_intro}, proving that the nearby cycles are constructible in this setting.
	Also, the function $\tau_{\calG,\mu}^\Phi$ depends on the choice of the geometric Frobenius lift $\Phi$, that is, we do not take semi-simplified traces by passing to inertia invariants, compare with \cite[Section~6.2]{HR20}. However, one recovers the semi-simple trace function as in \cite[Appendix]{HR20} by ``averaging over geometric Frobenius lifts $\Phi$''.
	The function $\tau_{\calG,\mu}^\Phi$ is left-$\calG(O_{E_0})$-invariant and supported on finitely many orbits, hence $\tau_{\calG,\mu}^\Phi\in \calH(G(E_0),\calG(O_{E_0}))_{\Lambda}$ naturally lies in the parahoric Hecke algebra of $\Lambda$-valued functions. 
	
	\begin{theorem}[{\Cref{sec:test-funct-conj-reformulation-implies-test-function-conjecture}}]\label{HK_conjecture_intro}
		For all $\mu$ (not necessarily minuscule), the function $\tau_{\calG,\mu}^\Phi$ lies in the center of $\calH(G(E_0),\calG(O_{E_0}))_{\Lambda}$.
		It is characterized as the unique function in the center that acts on any $\calG(\calO_{E_0})$-spherical smooth irreducible $\Lambda$-representation $\pi$ by the scalar
		\begin{equation}
		\on{trace}\Bigl(s^\Phi(\pi)\;\big|\;  V_{\mu}\Bigr),
		\end{equation}
		where $s^\Phi(\pi)\in [\widehat{G}^{I}\rtimes \Phi]_{\on{ss}}/\widehat{G}^{I}$ is the Satake parameter for $\pi$ and $V_\mu$ the representation of the $L$-group $\widehat G\rtimes \Gamma$ of highest weight $\mu$. 
		Moreover, ${(\sqrt q)}^{\langle2\rho,\mu\rangle}\tau_{\calG,\mu}^\Phi$ takes values in $\bbZ$ and is independent of $\ell\neq p$, $\sqrt q$ and $E\hookto \bar\bbQ_p$.  
	\end{theorem}
	
	The theorem is a solution to the test function conjecture of Haines--Kottwitz \cite[Conjecture 6.1.1]{Hai14} for v-sheaf local models.
	This is new when $\mu$ is non-minuscule: then $\calM_{\calG,\mu}$ is not representable by a scheme due to the theory of Banach--Colmez spaces and, hence, not related to their schematic counterparts defined in \cite{PZ13, Lev16, Lou19, FHLR22}.
	If $\mu$ is minuscule, then the analogue of \Cref{HK_conjecture_intro} holds for $\calM_{\calG,\mu}^{\on{sch}}$ as well, using \Cref{SW_conjecture_intro}.
	In this case, we can work purely algebraically and replace $\IC_\mu$ by the constant sheaf $\bbQ_\ell$ on the smooth $E$-scheme $\calF_{G,\mu}$.
	Here, our result is new for the wildly ramified groups that were excluded in previous work \cite{PZ13, HR20, HR21}.
	With a view towards applications, say, to point counting formulas in the context of the Langlands--Kottwitz method, we have an analogous result where $E$ is replaced by a finite unramified extension, and correspondingly $E_0$ by its unramified subextension. 
	Also, \Cref{HK_conjecture_intro} implies the analogous version for the semi-simple traces following \cite[Appendix]{HR20}.
	
	\subsection{Explanation on perfectoid spaces, diamonds and $v$-sheaves}\label{v-sheaves-intro}
	The present manuscript relies on the $p$-adic geometric methods notably developed by Scholze \cite{Sch12, SW20, Sch17}.
	For the reader's convenience we give a brief overview of some basic terminology including references.
	
	Perfectoid spaces are the full subcategory of Huber's adic spaces $X$ which can be covered by affinoid adic spaces $\Spa(R,R^+)$ with $R$ a perfectoid Tate ring, see \cite[Definitions~3.1, 3.19]{Sch17}.
	Let $\Perf_{\bbF_p}$ be the category of perfectoid spaces in characteristic $p$ equipped with the v-topology generated by surjective maps subject to some finiteness conditions analogous to the definition of the fpqc topology for schemes, see \cite[Definition~8.1(iii)]{Sch17} (here, the ``v'' in v-topology stands for valuation).
	So, $\Perf_{\bbF_p}$ is a site (here, we choose some cut-off cardinal) which is subcanonical by \cite[Theorem~8.7]{Sch17}. 
	We have the following inclusions of full subcategories of the category of v-sheaves on $\Perf_{\bbF_p}$:
	\[
	\{\text{perfectoid spaces}\} \subset \{\text{diamonds}\} \subset \{\text{small v-sheaves}\}
	\]
	By definition, a diamond (respectively, small v-sheaf) admits a pro-étale cover (respectively, v-cover) by a perfectoid space, see \cite[Definition~11.1, Proposition~11.9, Definition~12.1]{Sch17}.
	In particular, every small v-sheaf $X$ has an underlying topological space $|X|$ given by the topological space of a perfectoid cover modulo the equivalence relation, see \cite[Definition~12.8]{Sch17}.
	The notion of small v-sheaves can be generalized to stacks leading to the notion of small v-stacks. 
	It comes with a theory of derived categories of étale sheaves that satisfy a $6$-functor formalism, see \cite[Introduction]{Sch17}.
	
	One has a functor $X\mapsto X^\dia$ from pre-adic spaces (see \cite[Appendix to Lecture~3]{SW20}) over $\bbZ_p$ to small v-sheaves on $\Perf_{\bbF_p}$, see \cite[Lemma~18.1.1]{SW20} and \cite[Proposition~1.20]{Gle20}.
	There is a surjective map $|X^\dia|\to |X|$ of topological spaces, see \cite[Proposition~18.2.2]{SW20}. 
	Every pre-adic space $X$ admits a decomposition $X^\text{an}\sqcup X^\text{na}$ into an open locus of analytic points and a closed locus of non-analytic points, see \cite[Proposition~1.19, Proposition~1.21(6)]{Gle20} for precise statements.


	If $X=X^\text{an}$ is analytic, then $X^\dia$ is a (locally spatial) diamond and $|X^\dia|\cong |X|$, see \cite[Lemma~15.6]{Sch17}.  
	Schemes $X$ in characteristic $p$ can be considered as discrete (hence, non-analytic) adic spaces via $\Spec R\mapsto \Spa(R,\widetilde \bbF_p)$ (here, $\widetilde \bbF_p$ is the integral closure of $\bbF_p$ in $R$), and we get an associated small v-sheaf $X^\dia$ characterized as the sheafification of the functor $\Spa(R,R^+)\mapsto X(\Spec R)$ on affinoid perfectoid spaces, see \Cref{sec:two-diff-diam-functors}.
	More generally, one can extend this construction for schemes $X$ over $\bbZ_p$. 
	Moreover, there is a second way of associating small v-sheaves to schemes. 
	The second construction is characterized as the sheafification of the functor $\Spa(R,R^+)\mapsto X(\Spec R^+)$, for proper schemes over $\bbZ_p$ the two constructions agree.
The reader is referred to \Cref{sec:two-diff-diam-functors} for a more detailed discussion.

	For an affinoid adic space $\Spa(R,R^+)$ we denote by $\Spd(R,R^+)$, or simply by $\Spd R$ whenever $R^+$ is understood, the associated small v-sheaf $\Spa(R,R^+)^\dia$.
	A basic example is $\Spd \bbZ_p=\Spa(\bbZ_p,\bbZ_p)^\dia$ which is stratified by $\Spd\bbQ_p=\Spa(\bbQ_p,\bbZ_p)^\dia$ and $\Spd\bbF_p=\Spa(\bbF_p,\bbF_p)^\dia$, corresponding to the decomposition into the analytic and non-analytic locus. 
	The diamond $\Spd \bbQ_p$ is studied in \cite[Section~8.4]{SW20}.
	Note that $\Spd \bbF_p$ is the terminal object in the category of v-sheaves on $\Perf_{\bbF_p}$ (that is, the functor sending any $S$ to a point), which however is not a diamond. 
	
	Our main example is the v-sheaf Beilinson--Drinfeld Grassmannian $\Gr_\calG\to \Spd \bbZ_p$ associated with $\calG$ over $\bbZ_p$ as in \Cref{intro-main-results}, see \cite[Definition~20.3.1]{SW20}.
	The map $\Gr_\calG\to \Spd \bbZ_p$ is relatively representable by an ind-(proper, spatial diamond), that is, for every spatial diamond mapping to $\Spd \bbZ_p$ the base change is an increasing union of proper, spatial diamonds with closed immersions as transition maps, see \Cref{SW theorem on parahoric grassmanians} and the references cited there.
	Whereas its generic fiber $\Gr_G:=\Gr_\calG|_{\Spd\bbQ_p}$ is itself an ind-(proper spatial diamond), its special fiber $\Gr_\calG|_{\Spd\bbF_p}=\Fl_\calG^\dia$ is the v-sheaf associated with an ind-(perfectly proper $\bbF_p$-scheme), see \cite[Lectures~19 and 21]{SW20}.
	The manuscript at hand studies $\Gr_\calG$ with the geometric and cohomological techniques outlined in the present subsection. 
	
	\subsection{Strategy of proof}
	The proofs of \Cref{SW_conjecture_intro} and \Cref{HK_conjecture_intro} are purely group-theoretic and do not rely on the classification of reductive groups over local fields. 
	Our method is inspired by a conjecture of He--Pappas--Rapoport \cite[Conjecture 2.13]{HPR20} of which we prove a variant, see \Cref{thm_specialization_local_models}. 
	Following a suggestion in \cite[Introduction]{Lou20}, we approach $\calM_{\calG,\mu}$ by studying its associated \textit{specialization triple}, that is, its generic fiber, its special fiber and the specialization map between them.  
	The basic idea is that $\calM_{\calG,\mu}$ should be uniquely determined by its specialization triple, making the comparison with the to-be-constructed $\calM_{\calG,\mu}^{\on{sch}}$ possible.
	Once there is enough progress towards \Cref{SW_conjecture_intro}, the proof of \Cref{HK_conjecture_intro} roughly follows the method from \cite{HR21}, and the reader is referred to \Cref{test_function_section} for details.
	Along the way, we address further questions and conjectures in the field, notably Zhu's conjecture \cite[Appendix B, Conjecture III]{Zhu17} on the geometry of deperfections of affine Schubert varieties in the Witt vector affine flag variety.

	The proof of \Cref{SW_conjecture_intro} proceeds in the following three steps outlined in the next subsections: 
	\begin{enumerate}
		\item Establish a theory of specialization triples for suitable v-sheaves including local models, see \Cref{section:specialization-triple-intro}. 
		\item Determine the special fibers of v-sheaf and schematic local models, see Sections~\ref{sec:vsheaf_LM_intro} and \ref{sec:schematic_LM_intro}. 
		\item Study the specialization map for local models, see \Cref{sec:specialization_intro}. 
	\end{enumerate}
	\Cref{sec:conclusion_intro} concludes the outline of the proof with a more detailed version of \Cref{SW_conjecture_intro}, see \Cref{SW_conjecture_detailed_intro}.
		
	\subsubsection{Kimberlites and specialization triples}\label{section:specialization-triple-intro}
	Recall from \eqref{definition-local-model-intro} that the v-sheaf local model $\calM_{\calG,\mu} \subset \Gr_{\calG}|_{\Spd O_E}$ is defined as the v-closure of $\Gr_{G,\mu}$, that is, the smallest closed sub-v-sheaf in $\Gr_{\calG}|_{\Spd O_E}$ containing $\Gr_{G,\mu}$, see \Cref{sec:closures-definition-v-sheaf-closure}.
	A basic difficulty is the determination of the underlying topological space $|\calM_{\calG,\mu}|$, that is, to show the density of $|\Gr_{G,\mu}|$ inside $|\calM_{\calG,\mu}|$.
	(We warn the reader that such density results fail in general, see \Cref{density_example}.)
	The following theorem, in particular, shows that the density still holds true in the case of v-sheaf local models:
	
	
	
	\begin{theorem}[\Cref{prop_fully_faith_triples_kimberlites}, \Cref{LM_kimberlite}]
		\label{specialization_triple_intro}
		For all $\mu$ (not necessarily minuscule) the v-sheaf local model $\calM_{\calG,\mu}$ is a quasi-compact, separated, flat $p$-adic kimberlite over $\Spd O_E$ that satisfies the properties of \Cref{defi category K}.
		In addition, the functor sending such kimberlites over $\Spd O_E$ to their specialization triples 
		\begin{equation}
		X\mapsto \big(X_\eta,X^\red, \on{sp}_{\breve{X}}\big)
		\end{equation} 
		is fully faithful.
	\end{theorem}	
	
	Kimberlites are introduced by the second named author in \cite{Gle20} and are a certain class of v-sheaves (containing the v-sheaves associated to flat, projective $O_E$-schemes) that admit a well-behaved theory of specialization maps, see \Cref{sec:formal-theory-v-sheaves-p-adic-kimberlites} for details. 
	In the above result, $X_\eta=X|_{\Spd E}$ is the generic fiber, $X^\red$ is a perfect $k$-scheme that plays the role of the special fiber (see Definitions~\ref{definition-reduced-subsheaf} and \ref{definition p-adic kimberlite}) and $\on{sp}_{\breve{X}}\colon |X_{\breve{\eta}}|\to |X_{\bar{k}}^\red|$ is a continuous map on the underlying topological spaces after base change to $\Spd O_{\breve{E}}$. 
	For v-sheaf local models, $\calM_{\calG,\mu}^\red$ is a union of Schubert perfect $k$-varieties in $\Fl_\calG$ such that its associated v-sheaf is the special fiber $\calM_{\calG,\mu}|_{\Spd k}$.
	Since the generic fiber of $\calM_{\calG,\mu}$ is $\Gr_{G,\mu}$ by definition, it remains to determine the special fiber and the specialization map which we outline in Sections~\ref{sec:vsheaf_LM_intro} and \ref{sec:specialization_intro}.
	
	Moreover, a similar fully faithfulness result for the functor $X\mapsto X^\dia$ from weakly normal flat projective $O_E$-schemes to v-sheaves over $\Spd O_E$ is shown by the third named author in \cite{Lou20}.
	So, \Cref{specialization_triple_intro} shows that the to-be-constructed schematic local model $\calM_{\calG,\mu}^{\mathrm{sch}}$ is uniquely determined by the specialization triple associated to $\calM_{\calG,\mu}^{{\mathrm{sch}},\dia}$.
	Again, its generic fiber is $\Gr_{G,\mu}$ by definition and it remains to determine the special fiber and the specialization map which we outline in Sections~\ref{sec:schematic_LM_intro} and \ref{sec:specialization_intro}.
	We observe that the argument presented in Section~\ref{sec:vsheaf_LM_intro} holds for arbitrary $\mu$, whereas the reasoning in Sections~\ref{sec:schematic_LM_intro} and \ref{sec:specialization_intro} is valid only when $\mu$ is minuscule.
	
	
	\subsubsection{Special fibers of v-sheaf local models}\label{sec:vsheaf_LM_intro}
	The Witt vector partial affine flag variety $\Fl_\calG$ is the increasing union of perfect projective varieties $\Fl_{\calG, w}$ indexed by the double coset of the Iwahori--Weyl group, see \cite{Zhu17, BS17}.
	The $\mu$-admissible locus $\calA_{\calG,\mu}$ in the sense of Kottwitz--Rapoport is defined as the $k$-descent of the $\bar{\bbF}_p$-union
	\begin{equation}
	\calA_{\calG, \bar\bbF_p,\mu}= \bigcup_\lambda \Fl_{\calG, \bar\bbF_p, \la_I(p)},
	\end{equation}
	where $\la$ runs over all absolute Weyl conjugates of $\mu$ and $\la_I(p)$ denotes the associated translation in the Iwahori--Weyl group of $G_{\breve \bbQ_p}$, see \Cref{admissible_locus_definition}.

	\begin{theorem}[\Cref{theorem_special_fiber_admissible}]
		\label{v_special_fiber_intro}
		The special fiber of $\calM_{\calG,\mu}$ is the v-sheaf $\calA_{\calG,\mu}^\dia$ associated with the $\mu$-admissible locus inside $\Fl_\calG^\dia$.
	\end{theorem}
	
	The difficulty in determining the special fiber lies in the rather abstract definition of $\calM_{\calG,\mu}$ via closure operations inside $\Gr_\calG$. 
	In a different context, \cite[Theorem 6.12]{HR21} determines such special fibers based on cohomological considerations by calculating the support of the nearby cycles $\Psi_{\calM_{\calG,\mu}}\IC_\mu$ appearing in \eqref{test_function_intro}.
	A key input is the study of $\bbG_m$-actions on $\Gr_\calG$ coming from the choice of a cocharacter $\lambda\co \bbG_m\r \calG$.
	Let $\calM$, respectively $\calP$ the closed subgroup scheme of $\calG$ with Lie algebra the subspace of $\Lie\, \calG$ with weights ${\lambda=0}$, respectively ${\lambda\geq 0}$.
	This induces maps $\calM\gets \calP\r \calG$ of $\bbZ_p$-group schemes and, by functoriality, also of the associated Beilinson--Drinfeld Grassmannians.
	The geometric input towards \Cref{v_special_fiber_intro} is the following result which extends \cite[Chapter VI.3]{FS21} beyond the case where $\calG$ is reductive:
	
	\begin{theorem}[\Cref{BD.attractor.groups.comparison.theorem}]
		The $\bbG_{m}^\dia$-action on $\Gr_{\calG}$ via $\la$ induces a commutative diagram of v-sheaves
		\begin{equation}\label{Gm_diagram_intro}
		\begin{tikzpicture}[baseline=(current  bounding  box.center)]
		\matrix(a)[matrix of math nodes, 
		row sep=1.5em, column sep=2em, 
		text height=1.5ex, text depth=0.45ex] 
		{\Gr_\calM & \Gr_{\calP} & \Gr_\calG \\ 
			(\Gr_{\calG})^0& (\Gr_{\calG})^+& \Gr_{\calG}, \\}; 
		\path[->](a-1-2) edge node[above] {}  (a-1-1);
		\path[->](a-1-2) edge node[above] {}  (a-1-3);
		\path[->](a-2-2) edge node[below] {}  (a-2-1);
		\path[->](a-2-2) edge node[below] {} (a-2-3);
		\path[->](a-1-1) edge node[left] {} (a-2-1);
		\path[->](a-1-2) edge node[left] {} (a-2-2);
		\path[->](a-1-3) edge node[left] {$\id$} (a-2-3);
		\end{tikzpicture}
		\end{equation}
		such that the vertical arrows are open and closed immersions that induce isomorphisms over $\Spd \bbQ_p$. 
		If, in addition, $\calG$ is special parahoric (for example, reductive) over $\breve \bbZ_p$, then the vertical arrows are also isomorphisms over $\Spd \bbF_p$.	
	\end{theorem}
	
	Here, $(\Gr_{\calG})^0$ is the fixed points locus for the $\bbG_{m}^\dia$-action and $(\Gr_{\calG})^+$ its attractor locus, which is loosely speaking the union of the strata of points flowing into the fixed points, see also \cite[discussion above Theorem~I.6.5]{FS21}.
	Both v-sheaves are representable by increasing unions of spatial diamonds relative over $\Spd \bbZ_p$ in the sense of \cite[Definition~13.3]{Sch17}, see \Cref{lemma-representability-attractor}. 
	
	If $\lambda$ is regular, then the special fiber of $(\Gr_\calG)^+$ consists of the v-sheaves associated to the semi-infinite orbits $\calS_w$ inside $\Fl_\calG$ indexed by certain cosets of the Iwahori--Weyl group, see \eqref{semi_infinite_stratification}.  
	By comparison, $\Gr_\calP$ corresponds to those semi-infinite orbits attached to translation elements.  
	The following proposition generalizes the closure relation of semi-infinite orbits and the equi-dimensionality of Mirkovi\'c--Vilonen cycles \cite[Theorem 3.2]{MV07} from split groups to twisted groups, questions left open in the context of the ramified Satake equivalence \cite{Zhu15, Ric16}:
	
	\begin{proposition}[\Cref{prop_closure_relations_semi_infinite_orbits}, \Cref{equidimensionality.MV.cycles.lemma}]
		For a regular coweight $\la$ and the induced semi-infinite orbits, there is an equality inside $\Fl_\calG$:
		\begin{equation}
		\overline{\calS_w}= \bigcup_v \calS_v,
		\end{equation}
		where $v$ runs through all elements less than or equal to $w$ in Lusztig's semi-infinite Bruhat order. 
		If $\calG$ is special parahoric (for example, reductive) over $\breve \bbZ_p$, then the non-empty intersections $\calS_{\nu_I} \cap \Fl_{\calG,\mu_I}$ are equi-dimensional of dimension $\langle \rho, \nu+\mu\rangle$.
	\end{proposition}
	
	Here, $v\leq^{\infty\over 2} w$ holds in Lusztig's semi-infinite Bruhat if and only if $v\leq w$ holds in the Bruhat order after translation by a sufficiently dominant cocharacter, see \eqref{equation-semiinifinite-bruhat}.
	
	Let us now explain the cohomological results going into \Cref{v_special_fiber_intro}. 
	As there is no general theory of nearby cycles for v-sheaves, we develop the foundational results for the Hecke stack.
	Fix a prime number $\ell\neq p$ and let $\Lambda$ be one of the rings $\bbZ/\ell^n$, $\bbZ_\ell$ or $\bbQ_\ell$ for some $n\geq 0$. 
	For a small v-stack $X$, let $\D(X,\Lambda)$ be the derived category of \'etale $\Lambda$-sheaves on $X$ as defined in \cite{Sch17}, where for $\Lambda=\bbQ_\ell$ we additionally invert $\ell$ in the adic formalism developed in \cite[Section 26]{Sch17}.
	We consider the inclusion of the geometric generic, respectively geometric special fiber into the integral Hecke stack
	\begin{equation}
	\Hk_{G}|_{\Spd \bbC_p} \xrightarrow{j} \Hk_\calG|_{\Spd O_{\bbC_p}} \xleftarrow{i} \Hk_{\calG}|_{\Spd \bar\bbF_p},
	\end{equation}
	see Section~\ref{section_local_model} for the definition of Hecke stacks. 
	The following result is inspired by work of Hansen--Scholze \cite{HS21} for schemes and proves constructibility of nearby cycles in our context:
	
	\begin{theorem}[\Cref{prop_ula_special_fiber}, \Cref{ULA_nearby_cycles}]
		\label{ula_intro}
		The pullback functor 
		\begin{equation}
		j^*\co \D\big(\Hk_\calG|_{\Spd O_{\bbC_p}},\Lambda\big) \r \D\big(\Hk_{G}|_{\Spd\bbC_p},\Lambda\big)
		\end{equation}
		induces an equivalence on the full subcategories of universally locally acyclic sheaves with bounded support, with inverse given by the derived push forward $Rj_*$.
		Consequently, the nearby cycles functor $\Psi_\calG:=i^*\circ Rj_*$ restricts to a functor
		\begin{equation}
		\D\big(\Hk_{G}|_{\Spd\bbC_p},\Lambda\big)^{\on{bd}, \on{ula}} \r \D\big(\Hk_{\calG}|_{\Spd\bar\bbF_p},\Lambda\big)^{\on{bd}, \on{ula}}.
		\end{equation}
		Furthermore, the target category is equivalent to the derived category on the schematic Witt vector Hecke stack $\D_{\on{cons}}\big(\Hk_{\calG}^{\on{sch}}|_{\Spec\,\bar\bbF_p},\Lambda\big)^{\on{bd}}$ of perfect-constructible sheaves with bounded support (see \Cref{sec:comp-etale-cohom}). 
	\end{theorem}
	
	The intersection complex $\IC_\mu$ descends along $\Gr_G\r \Hk_G$ and defines an object in $\D(\Hk_{G}|_{\Spd \bbC_p},\Lambda)^{\on{bd}, \on{ula}}$. 
	So \Cref{ula_intro} implies constructibility of $\Psi_{\calM_{\calG,\mu}}\IC_\mu$.
	To compute the support, we use that nearby cycles commute with the constant term functor defined by the generic and special fibers of diagram \eqref{Gm_diagram_intro}:
	\begin{equation}\label{constant_term_intro}
	\on{CT}_\calP\circ \Psi_{\calG} \cong \Psi_\calM\circ \on{CT}_P,
	\end{equation}
	see \Cref{prop_commutativity_nearby_cycles_constant_term} for details.
	If $\lambda$ is regular, then $\calM$ is the connected Néron model of a torus so that the right hand side of \eqref{constant_term_intro} can be computed in representation theoretic terms using the geometric Satake equivalence for the $B_\dR^+$-affine Grassmannian \cite[Chapter VI]{FS21}. 
	In down to earth terms, we are able to determine the compactly supported cohomology of $\Psi_{\calM_{\calG,\mu}}\IC_\mu$ on the stratification given by the semi-infinite orbits.
	For carefully chosen $\lambda$ (see \Cref{lemma_attractor_schubert_isolated}), this is enough together with the density of generic fibers (\Cref{section:specialization-triple-intro}) to deduce \Cref{v_special_fiber_intro}, see \Cref{theorem_special_fiber_admissible}.
	
	\subsubsection{Special fibers of schematic local models}\label{sec:schematic_LM_intro}
	In this subsection, we assume that $\mu$ is minuscule. 
	The construction of schematic local models \cite{PZ13, Lev16, Lou20, FHLR22} relies on lifting $\calG$ to a group scheme $\underline \calG$ over $\breve\bbZ_p[\![t]\!]$ equipped with isomorphisms
	\begin{equation}
	\underline \calG \otimes_{\breve\bbZ_p[\![t]\!], t\mapsto p}\breve\bbZ_p \simeq \calG \otimes \breve\bbZ_p, \;\;\;\;\; \underline \calG \otimes \breve\bbQ_p[\![t-p]\!] \simeq G\otimes \breve\bbQ_p[\![t-p]\!].
	\end{equation}
	Let us temporarily denote by $\calN_{\underline\calG,\mu}^{\on{sch}}$ the weak normalization of the closure of $\calF_{G,\mu}|_{\Spec\,\breve E}$ inside the schematic Beilinson--Drinfeld Grassmannian $\Gr_{\underline \calG}|_{\Spec\, O_{\breve E}}$ attached to $\underline \calG$, a weakly normal, flat, projective $O_{\breve E}$-scheme.
	As we explain in \Cref{sec:specialization_intro}, this is, in almost all cases, isomorphic to the base changed local model $\calM_{\calG,\mu}^{{\mathrm{sch}}}|_{\Spec\, O_{\breve E}}$ from \Cref{SW_conjecture_intro}, except in the aforementioned cases where $p=2,3$.
	The most general group lifts $\underline \calG$ are constructed in \cite{FHLR22}, based on \cite{Lou20}, under the following assumption: 
	
	\begin{assumption}\label{hyp_wild_odd_unitary}
		If $p=2$, then $G_{\ad}$ has no odd unitary $\breve{\bbQ}_2$-factors.
	\end{assumption}
	
	The reason for its appearance is the difficult structure of the integral root groups inside $\calG\otimes \breve \bbZ_p$ in the wildly ramified, odd unitary case for $p=2$.
	More precisely, quadratic field extensions of $\breve{\bbQ}_2$ fall into two classes: square roots of uniformizers and of units.  
	The first class is handled in \cite{Lou20}, leading to the milder assumption for $p=2$ as discussed in the text after the statement of \Cref{SW_conjecture_intro}. 
	It is the second class that appears most difficult. 
	
	The determination of the special fiber of $\calN_{\underline\calG,\mu}^{\on{sch}}$ relies on the coherence conjecture of Pappas--Rapoport proved by Zhu \cite{Zhu14}.
	The following result is the version in our context, and moreover, confirms \cite[Appendix B, Conjecture III]{Zhu17} for the Schubert varieties in the $\mu$-admissible locus:	
	
	\begin{theorem}[\Cref{theorem_coherence_allp}]
		\label{deperf_admissible_locus_intro}
		Under \Cref{hyp_wild_odd_unitary}, the canonical deperfection $\calA_{\calG,\mu}^{\on{can}}$ of the $\mu$-admissible locus is Cohen--Macaulay and its components are compatibly Frobenius-split.
		Moreover, for every ample line bundle $\calL$ on $\calA_{\calG,\mu}^{\on{can}}$, there is an equality
		\begin{equation}
		\dim_{k} \on{H}^0(\calA_{\calG,\mu}^{\on{can}}, \calL)=\dim_E \on{H}^0(\calF_{G,\mu}, \calO(c)),
		\end{equation}
		where $c\in \bbZ$ denotes the central charge of $\calL$ and $\calO(c)$ the $c$-th multiple of the ample generator of $\Pic(\calF_{G,\mu})$.
	\end{theorem}
	
	The canonical deperfection of $\calA_{\calG,\mu}$ is induced from the Greenberg realization of the Witt vector loop groups, see \Cref{sec:canon-deperf}. 
	The proof of \Cref{deperf_admissible_locus_intro} proceeds by comparing the $p$-adic admissible loci to their analogues in the equicharacteristic situation, and ultimately relies on the normality of Schubert varieties \cite{Fal03, PR08} where we use \cite{FHLR22} for wildly ramified groups. 
	We first compare the perfect(ed) Demazure resolutions and then apply Bhatt--Scholze's $h$-descent results \cite{BS17} to the ample line bundles on the resolutions.
	A key ingredient is He--Zhou's calculation \cite{HZ20} of the Picard group of $\Fl_{\calG}$ as the free $\bbZ[p^{-1}]$-module dual to the lines stable under a fixed Iwahori dilated from $\calG$. 

	\begin{theorem}[\Cref{lem_comparison_sch_vars_equi_and_mixed}, \cite{FHLR22}]
		\label{alg_special_fiber_intro}
		Under \Cref{hyp_wild_odd_unitary}, there is an isomorphism of $\bar\bbF_p$-schemes
		\begin{equation}
		\calN_{\underline\calG,\mu}^{\on{sch}}|_{\Spec\, \bar \bbF_p} \simeq \calA_{\calG,\mu}^{\on{can}}|_{\Spec\, \bar \bbF_p}.
		\end{equation}
		Hence, $\calN_{\underline\calG,\mu}^{\on{sch}}$ is normal, Cohen--Macaulay and has reduced special fiber.
	\end{theorem}
	
	The theorem holds, more generally, under the milder assumption explained above when $p=2$ by \cite{Lou20}.   
	The reader is referred to \cite{FHLR22} for a finer study of the singularities of the local models.

	\subsubsection{Specialization maps}\label{sec:specialization_intro}
	We continue to assume that $\mu$ is minuscule and focus on the v-sheaf local models $\calM_{\calG,\mu}$.
	The study of specialization maps for $\calM_{\calG,\mu}$ is challenging. 
	A basic problem is that, beyond rare exceptions, the set of $\calG(O_{\bbC_p})$-orbits in $\calF_{G, \mu}(\bbC_p)$ is infinite. 
	To understand this phenomenon, consider the standard $\bbG_m(O_{\bbC_p})$-action on $\bbA^1(\bbC_p)$. 
	Here the orbits are given by the $p$-adic valuation of the coordinate $x\in \bbA^1(\bbC_p)$. 
	We see that in this toy example, one can already find infinitely many orbits specializing to the origin $0\in \bbA^1(\bar{\bbF}_p)$. 
	Similarly, closed points in $\calM_{\calG, \bar \bbF_p,\mu}$ lying outside the open $\calG_{\bar{\bbF}_p}$-orbits will admit infinitely many $\calG(O_{\bbC_p})$-orbits specializing to them. 
	Despite this difficulty, we understand relatively well the reduction of $\Spd O_{\bbC_p}$-valued points lying in a certain cohomologically smooth sub-v-sheaf 
	\begin{equation}
	\calM_{\calG, \mu}^\circ \subset \calM_{\calG, \mu},
	\end{equation}
	defined in \Cref{def_open_semi_homog_fixers}.
	Unfortunately, $\calM_{\calG, \mu}^\circ$ alone does not afford sufficiently many integral points.
	Here we resort to variants of the splitting models of Pappas--Rapoport \cite{PR05} in our situation, that is, we use convolutions to partially desingularize the local models.
	For a sequence $\mu_\bullet=(\mu_1,\ldots,\mu_n)$ of minuscule coweights with pairwise disjoint supports, we consider the sub-v-sheaf
	\begin{equation}
	\calM_{\calG, \mu_\bullet} \subset \Gr_{\calG}\tilde\times\dots\tilde\times \Gr_\calG,
	\end{equation}
	defined as the v-closure of $\calF_{G,\mu_\bullet}^\dia=\calF_{G,\mu_1}^\dia\tilde\times\dots\tilde\times\calF_{G,\mu_n}^\dia$ inside the convolution Beilinson--Drinfeld Grassmannian, see \Cref{sec:convolution}.
	Most of the notions discussed before have their convolution counterparts. 
	In particular, the convolution v-sheaf local models support specialization maps just like the v-sheaf local models do.
	Then, functoriality in $(\calG,\mu_\bullet)$ is enough to control the specialization map: 
	
	\begin{theorem}[{\Cref{thm_specialization_local_models}}]\label{specialization_LM_intro}
		The specialization maps 
		\begin{equation}
		\on{sp}_{\calG,\mu_\bullet}\colon \calF_{G, \mu_\bullet}(\bbC_p) \rightarrow \calA_{\calG,\mu_\bullet}(\bar{\bbF}_p)
		\end{equation}
		for all pairs $(\calG,\mu_\bullet)$ as above are the only functorial collection of continuous maps whose restrictions to the sets $\calM^\circ_{\calG,\mu_\bullet}(\Spd O_{\bbC_p})$ agree with the natural reduction maps.
	\end{theorem}
	
	The theorem is a result of our reflections on the He--Pappas--Rapoport conjecture \cite[Conjecture 2.12]{HPR20}, which states that the local model $\calM_{\calG,\mu}$ should be uniquely recovered from its fibers equipped with the $\calG$-action.	
	The key calculation concerns the case where $G$ is a restriction of scalars along $E/\bbQ_p$ of a split group and where all non-zero components of $\mu_\bullet$ have support lying in disjoint irreducible components of the Dynkin diagram of $G$. 
	It follows from \Cref{def_open_semi_homog_fixers} that for all $i\in \{1,\dots,n\}$
	\[\calM_{\calG, \mu_i}^\circ(\Spd O_E)=\bigcup_{\lambda \in W_0\cdot \mu}\calG(O_E)\cdot [\la]\subseteq \calF_{G, \mu}(E),\]
	where $W_0$ is the relative Weyl group of $G$ (see \S\ref{section_minuscule_representable} for details). 
	Applying the Iwasawa decomposition \cite[Theorem 5.3.3]{KP21}, we get an equality
	$\calM_{\calG, \mu_i}^\circ(\Spd O_E)=\calF_{G, \mu_i}(E)$.
	Then, an inductive argument on the number of non-zero components of $\mu_\bullet$ shows more generally that $\calF_{G, \mu_\bullet}(E)=\calM_{\calG, \mu_\bullet}^\circ(\Spd O_E)$ for all the $\mu_\bullet$ as above. 
	In summary, all rational points of the flag variety extend to integral points of $\calM_{\calG, \mu_\bullet}^\circ$. 
	Once the case in which $G$ is a restriction of scalars is settled, functoriality forces uniqueness for all remaining cases.
	
	\subsubsection{Conclusion}\label{sec:conclusion_intro}
	The formulation of \Cref{specialization_LM_intro} requires functoriality of local models in $(\calG,\mu)$.
	This is clear for $\calM_{\calG,\mu}$ but, a priori, problematic for its schematic version $\calN_{\underline\calG,\mu}^{\on{sch}}$.
	We need to impose the following assumption which relates to functoriality problems with the association $\calG\mapsto \underline \calG$ of group lifts:
	
	\begin{assumption}\label{hyp_wild_triality}
		If $p=3$, then $G_\ad$ has no triality $\breve{\bbQ}_3$-factors.
	\end{assumption}
	
	The following result confirms \Cref{SW-conjecture}, except for the two non-split examples for $p=2,3$:
	
	\begin{theorem}\label{SW_conjecture_detailed_intro}
		There is a unique flat, projective and weakly normal $O_E$-model $\calM_{\calG,\mu}^{\on{sch}}$ of the $E$-scheme $\calF_{G,\mu}$ with an isomorphism of v-sheaves
		\begin{equation}\label{SW_conjecture_detailed_intro:1}
		(\calM_{\calG,\mu}^{\on{sch}})^\diamondsuit\cong \calM_{\calG,\mu},
		\end{equation}
		prolonging $\calF^{\diamondsuit}_{G,\mu} \cong \calM_{\calG,\mu}|_{\Spd\, E}$.
		Under \Cref{hyp_wild_odd_unitary} and \Cref{hyp_wild_triality}, there is a unique isomorphism
		\begin{equation}\label{SW_conjecture_detailed_intro:2}
		\calM_{\calG,\mu}^{\on{sch}}|_{\Spec\, O_{\breve E}} \cong \calN_{\underline\calG,\mu}^{\on{sch}}|_{\Spec\, O_{\breve E}}
		\end{equation}
		inducing the identity on generic fibers. 
		So (\Cref{alg_special_fiber_intro}), $\calM_{\calG,\mu}^{\on{sch}}$ is normal, Cohen--Macaulay and has reduced special fiber equal to $\calA_{\calG,\mu}^{\on{can}}$.
		Furthermore, the isomorphisms \eqref{SW_conjecture_detailed_intro:1} and \eqref{SW_conjecture_detailed_intro:2} are equivariant for $\calG_{O_E}^\dia$ and $\calG_{O_{\breve E}}$ respectively.
	\end{theorem}
	
	The unique scheme $\calM_{\calG,\mu}^{\on{sch}}$ satisfying the Scholze--Weinstein conjecture is the weak normalization of the closure of $\calF_{G,\mu}$ inside the Beilinson--Drinfeld Grassmannian attached to a group lift of $\Res_{O_K/\bbZ_p}\calH$, where $K$ is the splitting field of $G$ and $\calH$ the parahoric $O_K$-model of the split Chevalley form.
	By a careful analysis, we get an isomorphism between the specialization triples associated with $\calM_{\calG,\mu}$ and $\calM_{\calG,\mu}^{\on{sch},\dia}$, which is enough to conclude by \Cref{specialization_triple_intro}. 
	
	\subsection{Acknowledgements} 
	We thank Sebastian Bartling, Ana Caraiani, Laurent Fargues, David Hansen, Arthur-César Le Bras, Georgios Pappas, Michael Rapoport, Peter Scholze, Sug Woo Shin, Yuta Takaya, Zhiyou Wu, Jize Yu, Zhiyu Zhang for helpful conversations and email exchanges.
	Also, we thank the anonymous referees for many comments which improved the quality of the manuscript. 
	
	The idea behind this project began to take form during the trimestre thématique in Lyon 2018, attended by J.A., J.L.~and T.R., and continued with several interruptions while they were at Sorbonne Université, the Universität Bonn, the Imperial College London, and TU Darmstadt. 
	I.G. joined this project in 2021 after a series of correspondences that made us realize the techniques from his UC Berkeley thesis \cite{Gle20} were relevant to tackle the Scholze--Weinstein conjecture.
	We are grateful to all the institutions mentioned above for their continuous support.
	
	This project has received funding (J.L.~via Ana Caraiani) from the European Research Council under the European Union’s Horizon 2020 research and innovation program (grant agreement nº 804176), funding (J.L.) from the Excellence Cluster of the Universität Münster, funding (J.L. via Eva Viehmann) from the ERC Consolidator Grant 770936, funding (J.L.) from the CRC 1442 \textit{Geometry: Deformations and Rigidity}, funding (T.R.) from the European Research Council (ERC) under Horizon Europe (grant agreement nº 101040935), funding (T.R.) from the Deutsche Forschungsgemeinschaft (DFG, German Research Foundation) TRR 326 \textit{Geometry and Arithmetic of Uniformized Structures}, project number 444845124, funding (T.R.) from the LOEWE professorship in Algebra, project number LOEWE/4b//519/05/01.002(0004)/87 and funding (I.G.~via Peter Scholze) from the Leibniz-Preis.

	\section{v-sheaf theory}
	\label{v-sheaf-theory}
	Our main reference for the theory of diamonds, v-sheaves and v-stacks is \cite{Sch17}. 
	Here we gather some complementary results in the geometry of these objects that we will need later on.
	
	\subsection{Closures}
	\label{section_closures_of_v_sheaves}
	
	In this subsection, we discuss closures of small v-stacks. 
	Let $\Perf_{\mathbb{F}_p}$ be the v-site of perfectoid spaces of characteristic $p$. 
	All the small v-stacks in the following will be stacks on $\Perf_{\mathbb{F}_p}$.
	
	Let $X$ be a small v-stack and let $Y\subset X$ be a sub-v-stack, by which we mean a monomorphism of small v-stacks $Y\to X$, that is, a morphism whose diagonal
	\begin{equation}
	Y\xrightarrow{\Delta} Y\times_X Y
	\end{equation} is an isomorphism. 
	Just as in \cite[Definition 10.7]{Sch17}, we say that $Y$ is a locally closed sub-v-stack if, for every totally disconnected perfectoid space $S\rightarrow X$, the pullback $Y\times_X S\to S$ is representable by an immersion of perfectoid spaces, see \cite[Definition 5.6]{Sch17}. It is called a closed, respectively open immersion, if so are the respective pullbacks.
	This admits a simpler description for closed sub-v-stacks.
	For a small v-stack $X$ denote by $|X|$ its underlying topological space, see \cite[Definition~12.8]{Sch17}.
	
	\begin{lemma}
		\label{definition-of-closed-sub-v-sheaf-consistent-with-scholze}
		A morphism $Y\to X$ of small v-stacks is a closed immersion if and only if $Y \rightarrow X$ is a quasi-compact monomorphism and the induced map $|Y|\subset |X|$ is a closed embedding.
	\end{lemma}
	
	\begin{proof}
		Assume $Y\to X$ is a closed immersion and let $f\colon S\to X$ be a surjection from a disjoint union of totally disconnected perfectoid spaces. 
		By assumption $Z:=Y\times_X S$ is representable and $Z\to S$ is a closed immersion, in particular it is quasi-compact. 
		By \cite[Proposition 10.11 (o)]{Sch17}, the map $Y\to X$ is then quasi-compact as well. We may check that $Y\xrightarrow{\Delta} Y\times_X Y$ is an isomorphism after base change to $S$. 
		This amounts to verifying that $Z \xrightarrow{\Delta} Z\times_S Z$ is an isomorphism which follows from the fact that closed immersions are monomorphisms \cite[Definition 5.6]{Sch17}. 
		The inclusion $|Z|\subset |S|$ is a closed subset and equal to $|f|^{-1}(|Y|)$. 
		Indeed, if $s\in |S|\setminus |Z|$, $y\in |Y|$ and we let $\tilde{s}\colon\Spa(C_s,C_s^+)\to S$ and $\tilde{y}\colon\Spa(C_y,C_y^+)\to Y$ represent $s$ and $y$ respectively, then $\tilde{s}\times_X \tilde{y}=\emptyset$.
		It follows from the \cite[Proposition~12.7]{Sch17} and from $\tilde{s}\times_X \tilde{y}=\emptyset$ that $s$ and $y$ map to different points in $|X|$. 
		As $|f|$ is a quotient map \cite[Proposition 12.9]{Sch17}, this implies that $|Y|\subset |X|$ is a closed embedding. 
		
		Conversely, assume that $Y\subset X$ is a quasi-compact monomorphism and induces a closed embedding of underlying topological spaces. Let $f\colon S\to X$ be a map from a totally disconnected space $S$. The base change $Y\times_X S\to S$ is still a quasi-compact monomorphism of v-sheaves. By \cite[Corollary 10.6, Lemma 7.6]{Sch17} the v-sheaf $Y\times_X S$ is representable by a pro-constructible generalizing affinoid subset of $S$, and $|Y\times_X S|$ carries the subspace topology of $|S|$. Arguing as above, the image of $|Y\times_X S|$ in $|S|$ is $|f|^{-1}(|Y|)$ which is closed by assumption. This implies that the morphism $Y\times_X S\to S$ is a closed immersion in the sense of \cite[Definition 5.6]{Sch17}. 
	\end{proof}
	
	We now define the v-sheaf closure, or v-closure, of a sub-v-sheaf.
	
	\begin{definition}
		\label{sec:closures-definition-v-sheaf-closure}
		Let $X$ be a small v-stack and $Y\subset X$ a sub-v-stack $Y\subset X$. We define the v-closure $Y^\cl$ in $X$ as the limit (in the $2$-category of v-stacks) of all closed sub-v-stacks of $X$ containing $Y$. 
	\end{definition}
	
	The sub-v-stack $Y^\cl\subset X$ is a closed sub-v-stack. 
	Indeed, we can verify this after base change by a totally disconnected perfectoid space and use \cite[Proposition 6.4 (o)]{Sch17} to conclude.
	It also follows that $Y^\cl$ is small.
	Indeed, a set-theoretic argument using \cite[Proposition 10.5]{Sch17} quickly shows that if $X$ is small and $Z\subseteq X$ is a subsheaf, then $Z$ is also small.
	
	Next, we discuss the relation between the topological space $|Y^\cl|$ of the v-closure $Y^\cl$ and the topological closure $|Y|^\cl$ of the image of $|Y|$ in $|X|$.

	Recall that, for a topological space $X$ and a subset $S\subseteq X$, we say that $S$ is \emph{generalizing} if it is stable under generization (see \cite[Tag 0063]{StaProj}).
	We need a weaker version of this concept in the context of v-sheaves.
	
	\begin{definition}
		\label{sec:closures-sub-v-definition-weakly-generalizing}
		Let $S\subset|X|$ be a subset.
		\begin{enumerate}
			\item We call $S$ \emph{weakly generalizing} if, for any perfectoid field $C$ with open and bounded valuation subring $C^+\subset C$, and every morphism $\Spa(C,C^+)\to X$, the induced morphism $|\Spa(C,C^+)|\to |X|$ factors over $S$ if and only if the closed point of $|\Spa(C,C^+)|$ maps into $S$.
			\item The \emph{weakly generalizing} closure $S^{\on{wgc}}$ of $S$ is defined as the intersection of all closed, weakly generalizing subsets $S^\prime\subset |X|$ containing $S$.
		\end{enumerate}
	\end{definition}
	
	We note that if $X$ is (the v-sheaf associated to) a perfectoid space, then a subset $S\subset |X|$ is weakly generalizing if and only if it is generalizing. Indeed, for each analytic adic space specializations happen only at the same residue field.
	
	The images of morphisms of small v-stacks are weakly generalizing as the next lemma shows.
	
	\begin{lemma}
		\label{sec:closures-sub-v-lemma-image-of-morphism-of-small-v-sheaves-weakly-generalizing} For every morphism $f\colon X\to X^\prime$ of small v-stacks, the image of $|f|\colon |X|\to |X^\prime|$ is weakly generalizing in $|X^\prime|$.
	\end{lemma}
	\begin{proof}
		Let $C$ be a perfectoid field with open and bounded valuation subring $C^+\subset C$. Assume that the morphism
		\begin{equation}
		\Spa(C,C^+)\to X^\prime
		\end{equation}
		sends the closed point of $\Spa(C,C^+)$ into $f(|X|)$. This means that the above morphism factors through $X$ after possibly enlarging $(C,C^+)$, see \cite[Proposition 12.7]{Sch17}. But then the full image of $|\Spa(C,C^+)|$ in $|X^\prime|$ will factor through $f(|X|)$ and this shows that $f(|X|)$ is weakly generalizing.
	\end{proof}
	
	In particular, the topological space of the v-closure ${Y^\cl}$ of some sub-v-stack $Y\subset X$ is always weakly generalizing. 
	Thus, the topological space of the v-closure does not coincide, in general, with the topological closure. 
	
	\begin{example}\label{density_example}
		As a concrete example, consider the inclusion
		\begin{equation}
		\bbD_C^\diamondsuit\to \mathbb{B}_C^\diamondsuit:=\Spd(C\langle T\rangle, \mathcal{O}_C\langle T\rangle)
		\end{equation}
		of the open unit ball into the closed unit ball over a perfectoid base field $C$.
		Here $(-)^\dia$ denotes the functor from analytic pre-adic spaces to diamonds from \cite[Lecture~10]{SW20}. 
		Then $\lvert \bbD_C^\diamondsuit\rvert^{\mathrm{cl}}$ is the complement of the torus $\lvert\bbT_C^{\diamondsuit}\rvert=|\Spd(C\langle T^{\pm 1}\rangle, \mathcal{O}_C\langle T^{\pm 1}\rangle)|$, hence not weakly generalizing, as it misses the Gaußpoint but contains a rank $2$ specialization thereof. 
		
		The weakly generalizing closure $|\bbD_C^\diamondsuit|^{\mathrm{wgc}}$ is given in turn by the complement of every open unit ball $\bbD_{x,C}^\diamondsuit$ centered around $x \in \bbT_C(C)$. This will give rise to the v-closure of $\bbD_C^\diamondsuit$ inside $\mathbb{B}_C^\diamondsuit$, see \Cref{proposition-closure-of-v-subsheaves}.
	\end{example}
	
	Let us recall that, for every small v-stack $X$, there is the canonical morphism
	\begin{equation}
	X\to \underline{|X|},\ (f\colon T\to X)\mapsto (|f|\colon |T|\to |X|)
	\end{equation}
	of v-stacks where $\underline{|X|}$ denotes the v-sheaf represented by the topological space $|X|$, that is, it is given by $\underline{|X|}(T)=\mathrm{Hom}_{\mathrm{cts}}(|T|,|X|)$ for each $T\in \mathrm{Perf}_{\mathbb{F}_p}$.
	
	\begin{remark}
		We warn the reader that $\underline{|X|}$ is not small whenever $|X|$ fails to satisfy the separation axiom $T1$. Indeed, if $|X|$ is a trait, then $\underline{|X|}(R,R^+)$ is the set of closed subsets of $|\Spa(R,R^+)|$, and for each fixed cut-off cardinal $\kappa$, there is $|\Spa(R,R^+)|$ large enough so that not all closed subsets come from pullback of $\kappa$-small ones. 
	\end{remark}
	
	\begin{lemma}
		\label{sec:closures-sub-v-lemma-pullback-of-weakly-generalizing-subset-big-enough} Let $X$ be a small v-stack, and $S\subset |X|$ be a weakly generalizing closed subset. Then $Y:=\underline{S}\times_{\underline{|X|}}X$ is a small closed sub-v-stack satisfying $|Y|=S$ and, moreover, every closed sub-v-stack is of this form. 
	\end{lemma}
	\begin{proof}
		By \cite[Proposition 10.11]{Sch17}, we may check that $Y\subset X$ is a closed sub-v-stack after pullback along a v-cover $f\colon Z \rightarrow X$ with $Z$ a disjoint union of totally disconnected perfectoid spaces. 
		Then $Y\times_X Z=   \underline{|f|^{-1}(S)}\times_{\underline{|Z|}}Z$ and note that $|f|^{-1}(S) \subset |Z|$ is closed as $|f|$ is continuous. 
		Moreover, $|f|^{-1}(S)$ is weakly generalizing, and thus generalizing because $Z$ is a perfectoid space. 
		Consequently, $|f|^{-1}(S)$ is representable by a perfectoid space by \cite[Lemma 7.6]{Sch17}. 
		Its v-sheaf coincides with $Y\times_X Z$ by \cite[Lemma 12.5]{Sch17}, so $Y \subset X$ is a closed immersion by \Cref{definition-of-closed-sub-v-sheaf-consistent-with-scholze}. 
		Clearly, we also have $|Y|=S$ as $|f|$ is surjective. 
		
		Now assume that $Y\subset X$ is a closed sub-v-stack. 
		It follows from \Cref{sec:closures-sub-v-lemma-image-of-morphism-of-small-v-sheaves-weakly-generalizing} that $|Y|\subseteq |X|$ is weakly generalizing.
Let $Y'=X\times_{\underline{|X|}} \underline{|Y|}$.
		The identity $Y=Y\times_X Y'$ is easy to verify (by base change to totally disconnected $S$ and \cite[Proposition 5.3.(iv)]{Sch17}), so the map $Y\to Y'$ is a closed sub-v-stack with the same underlying topological space. 
		By \cite[Lemma 12.11]{Sch17}, $Y\to Y'$ is a surjetive map of v-stacks and consequently an isomorphism. 
	\end{proof}
	
	\Cref{sec:closures-sub-v-lemma-image-of-morphism-of-small-v-sheaves-weakly-generalizing} and \Cref{sec:closures-sub-v-lemma-pullback-of-weakly-generalizing-subset-big-enough} characterize closed weakly generalizing subsets $S\subset |X|$ as exactly those closed subsets $S\subset |X|$ for which the inclusion
	\begin{equation}
	\big|\underline{S}\times_{\underline{|X|}}X\big|\subset S
	\end{equation}
	is an equality. Note that the v-sheaf $Y:=\underline{S}\times_{\underline{|X|}} X$ may even be empty if $S$ is not weakly generalizing. For example, this happens if $S=\{s\}$ for $s\in \Spa(C,C^+)$ the closed point of a perfectoid field $C$ with $C^+\subsetneq O_C\subset C$ an open and bounded valuation subring of rank $>$ 1.
	
	\begin{proposition}
		\label{proposition-closure-of-v-subsheaves}
		Let $X$ be a small v-stack, and let $Y\subset X$ be a sub-v-stack. Let ${Y^\cl}\subset X$ be the v-closure of $Y$ in $X$. Then \begin{equation}{Y^\cl}= \underline{|Y|^{\on{wgc}}}\times_{\underline{|X|}}X
		\end{equation}
		as sub-v-stacks of $X$.
		Hence, $|{Y^\cl}|\subset |X|$ is the weakly generalizing closure $|Y|^{\on{wgc}}$ of $|Y|$ in $|X|$.
	\end{proposition}
	\begin{proof}
		Set $Y^\prime:=\underline{|Y|^{\on{wgc}}}\times_{\underline{|X|}}X$. Then $Y^\prime$ is a closed sub-v-stack of $X$ containing $Y$ and $|Y^\prime|=|Y|^{\on{wgc}}$ by \Cref{sec:closures-sub-v-lemma-pullback-of-weakly-generalizing-subset-big-enough}. Therefore, the v-closure ${Y^\cl}$ is contained in $Y^\prime$. But conversely, the topological space $|Y^\cl|$ must contain $|Y|^{\on{wgc}}$ by \Cref{sec:closures-sub-v-lemma-image-of-morphism-of-small-v-sheaves-weakly-generalizing}. Since $Y^\cl=\underline{|Y^\cl|}\times_{\underline{|X|}}X $ again by \Cref{sec:closures-sub-v-lemma-pullback-of-weakly-generalizing-subset-big-enough}, we conclude that $Y^\prime\subset Y^\cl$ and thus they coincide as desired. 
	\end{proof}
	
	The next result will turn out to be a useful tool later on when computing v-closures.
	Recall the notion of partially proper maps \cite[Definition 18.4]{Sch17}.
	
	\begin{corollary}
		\label{formation of closures commutes with nice basechange}
		The formation of v-closures commutes with base change by partially proper morphisms that are also open maps.
	\end{corollary}
	
	\begin{proof}
		In the following, we identify open substacks of small v-stacks with open subsets of their topological space, see \cite[Proposition 12.9]{Sch17}.
		By \Cref{proposition-closure-of-v-subsheaves}, we need to verify the corresponding assertion at the topological level. 
		Let $f\colon Z\rightarrow X$ be an open and partially proper morphism between small v-stacks and set $g:=\lvert f \rvert$. 
		Let $S\subset |X|$ be a subset, clearly $g^{-1}(S)^{\on{wgc}}\subset g^{-1}(S^{\on{wgc}})$. Let $T:=g^{-1}(S)^{\on{wgc}} \subset \lvert Z \rvert$. 
		Its complement $V$ is an open subset of $|Z|$, and the map $V\to Z$ is partially proper because $T$ is weakly generalizing. 
		Indeed, to verify the valuative criterion of \cite[Definition 18.4]{Sch17} observe that if $f:\Spa(R,R^+)\to Z$ is a map whose restriction to $\Spa(R,R^\circ)$ factors through $V$, then $f$ must already factor through $V$ since every point of $\Spa(R,R^+)$ generalizes to a point in $\Spa(R,R^\circ)$. 
		Since the map $Z\to X$ is open, the subset $U:=g(V)$ is also open. 
		Since $V\to Z$ is partially proper, the map $U\to X$ is partially proper as well. 
		The complement $F \subset \lvert X \rvert$ of $U$ is closed and $g^{-1}(F) \subset T$. 
		Also, $F$ is weakly generalizing since $U$ is partially proper. 
		This implies $S^{\on{wgc}}\subset F$ and consequently $g^{-1}(S^{\on{wgc}})\subset T$, as we wanted to show.  		
	\end{proof}
	
	\subsection{The two different diamond functors}
	\label{sec:two-diff-diam-functors}
	Let $O$ be a complete discrete valuation ring with perfect residue field $k$ of characteristic $p$, and assume that $O$ is flat over $\Z_p$, that is, $p$-torsion free. Let $\pi\in O$ be a uniformizer.

	If $X$ is a pre-adic space over $O$, we can attach to it a v-sheaf $X^{\diamondsuit}$ over $\Spd O$ as in \cite[Lecture~18]{SW20}. 
	(Whenever $R^+=R$, we simply write $\Spd(R)$ for $\Spd(R,R)$ following \cite{SW20}.)
	Namely, if $S\in \Perf_{\mathbb{F}_p}$, then
	\begin{equation}
	X^\diamondsuit(S)=\{(S^\sharp, \iota, f)\}/\mathrm{isom.}
	\end{equation}
	with $(S^\sharp, \iota)$ an untilt of $S$ and $f\colon S^\sharp\to X$ a morphism of pre-adic spaces (and the obvious notion of isomorphism between these triples). 
	On the other hand, given an algebra $A$ over $O$ there are two different ways to associate a v-sheaf to $\Spec(A)$. 
	\begin{definition}
		\label{sec:two-diff-diam-definitions-formal-and-analytic-v-sheaf-associated-to-a-scheme}
		Let $A$ be an $O$-algebra.
		\begin{enumerate}
			\item We let $\Spec(A)^{\diamondsuit}$ denote the functor \begin{equation}(R,R^+)\mapsto \{(R^\sharp,\iota,f)\}/\mathrm{isom.}\end{equation} where $(R^\sharp,\iota)$ is an untilt over $O$ and $f\colon A\to R^{\sharp}$ is an $O$-algebra homomorphism.
			\item We let $\Spec(A)^{\diamond}$ denote the functor \begin{equation}(R,R^+)\mapsto \{(R^\sharp,\iota,f)\}/\mathrm{isom.}\end{equation} where $(R^\sharp,\iota)$ is an untilt over $O$ and $f\colon A\to R^{\sharp,+}$ is an $O$-algebra homomorphism.
		\end{enumerate}
	\end{definition}
	
	Both of these constructions are compatible with localization and glue to functors from the category of schemes over $O$ to the category of v-sheaves over $\Spd O$. 
	Indeed, given $g\in A$, the open subscheme $\Spec(A[1/g])$ is sent to the open subfunctors of $\Spec(A)^{\diamondsuit}$, respectively $\Spec(A)^{\diamond}$, defined by the conditions $f(g)\in (R^{\sharp})^\times$, respectively $f(g)\in (R^{\sharp,+})^{\times}$, that is, by the open loci $\{|g|\neq 0\}\subset \Spa(R^\sharp, R^{\sharp,+})$, respectively $\{ |g|=1 \}\subset \Spa(R^\sharp,R^{\sharp,+})$. 
	We still denote these functors on $\mathrm{Sch}_O$ by $\diamondsuit$ and $\diamond$.
	
	For schemes locally of finite type over $O$ both functors admit a two step construction following \cite[Section~5.1]{Gle20}.
	Let $O_{\text{disc}}$ denote the ring $O$ equipped with the discrete topology.
	The forgetful functor from the category of discrete $O_{\text{disc}}$-adic spaces $\text{DiscAdSp}_{O_{\text{disc}}}$ towards the category of $O$-schemes $\text{Sch}_{O}$ admits a left and right adjoint:
	\[
	\begin{tikzcd}
	\text{Sch}_{O} \arrow[yshift=2mm]{r}{(\str)^{\text{ad}}} \arrow[yshift=-2mm]{r}[yshift=-6mm]{(\str)^{\text{ad}/O}} & \text{DiscAdSp}_{O_{\text{disc}}}\arrow{l}
	\end{tikzcd}		
	\]
	On affinoids, respectively affines the functors are characterized by the formulas
	\[
	\Spa(A,A^+)\mapsto \Spec(A), \;\;\; \Spec(A)^\ad =\Spa(A,A),\;\;\; \Spec(A)^{\ad/O}=\Spa(A,\widetilde O),
	\]
	where $\widetilde O$ is the integral closure of $O$ in $A$. 
	If $X$ is a scheme locally of finite type over $O$, then the fiber products
	\[
	\widehat X:=X^\ad\times_{\Spa O_{\text{disc}}}\Spa O\;\;\; \text{and} \;\;\; X^{\text{an}}:=X^{\ad/O}\times_{\Spa O_{\text{disc}}}\Spa O
	\]
	along the identity $O_{\text{disc}}\r O$ exist in the category of adic spaces by Huber \cite[Proposition~3.7]{huber_a_generalization_of_formal_schemes_and_rigid_analytic_varieties} and the structure map to $\Spa O$ is adic.
	In fact, if $X=\Spec(A)$, then $\widehat{X}=\Spa(\widehat{A},\widehat{A})$ where $\widehat{A}$ is the $\pi$-adic completion of $A$ whereas $X^{\on{an}}$ is in general not affinoid.
	The following lemma is immediate from the construction and \Cref{sec:two-diff-diam-definitions-formal-and-analytic-v-sheaf-associated-to-a-scheme}:
	
	\begin{lemma}
	The diamond functor $(-)^\dia$ on pre-adic spaces induces isomorphisms of functors $\widehat{X}^\diamondsuit=X^{\diamond}$ and $(X^{\on{an}})^\diamondsuit=X^{\diamondsuit}$ from the category of schemes $X$ locally of finite type over $O$ towards the category of v-sheaves over $\Spd O$.
	\end{lemma}

	For any $O$-scheme $X$ there is an evident natural transformation $X^{\diamond}\to X^{\diamondsuit}$. 
	If $X$ is separated over $O$, this map is a monomorphism of small v-sheaves. 
	Indeed, since $\Spec A\subseteq \Spec A^+$ is schematically dense, this follows from \cite[Tag 01RH]{StaProj}. 
	Moreover,  if $X$ is also locally of finite type over $O$, then $X^\diamond\to X^\diamondsuit$ it is an open immersion. 
	Indeed, the case $X=\bbA_O^n$ can be dealt with by hand, and when $X=\Spec( O[X_1,\dots,X_n]/I)$ we have a Cartesian diagram 
	\begin{center}
	\begin{tikzcd}
	 X^\diamond \arrow{r} \arrow{d}  & (\bbA^n_O)^\diamond \arrow{d} \\
	 X^\diamondsuit \arrow{r} & (\bbA^n_O)^\diamondsuit.
	\end{tikzcd}
	\end{center}

	If $X$ is proper over $O$, then the open immersion $X^{\diamond}\to X^{\diamondsuit}$ is an isomorphism because it is surjective on points by the valuative criterion for properness. 
	In the following, let $X^\diamondsuit$ denote the common value $X^{\diamond}=X^{\diamondsuit}$ whenever $X$ is proper over $O$.
	
	The two diamond functors play different roles in the following sections: the ``analytic'' functor $(\str)^{\diamondsuit}$ is the most natural to study $\bbG_m$-actions and nearby cycles while the ``formal'' functor $(\str)^{\diamond}$ carries a specialization map which we will use to determine specialization triples. 
	
	A natural question is to what extent the associated v-sheaves $X^{\diamond}$ and $X^{\diamondsuit}$ reflect the geometry of $X$. 
	In general, neither of the functors is full or faithful. 
	For example, if $A$ is any $O$-algebra and $\widehat{A}$ its $\pi$-adic completion, then the natural morphism
	\begin{equation}
	\Spec(\widehat{A})^{\diamond}\to \Spec(A)^{\diamond}
	\end{equation}
	is an isomorphism (because $R^{\sharp,+}$ is $\pi$-complete by uniformity of affinoid perfectoid spaces). 
	In particular, if $F$ is the fraction field of $O$, and $A$ an $F$-algebra, then $\Spec(A)^{\diamond}=\emptyset$. If $A=F[t]$, then
	\begin{equation}
	\Spec(F[t])^{\diamondsuit} =(\bbA_F^{1,\mathrm{ad}})^\diamondsuit
	\end{equation}
	is the rigid-analytic affine line over $F$, which has many non-algebraic automorphisms.
	
	When we restrict to schemes over $O$ for which $\pi=0$, the situation is more clear. 
	Both $\diamond$ and $\diamondsuit$ are fully faithful on perfect schemes and if we let $Y$ denote the perfection of $X$, then $X^{\diamond}=Y^{\diamond}$ and $X^{\diamondsuit}=Y^{\diamondsuit}$ \cite[Proposition 18.3.1]{SW20}, \cite[Theorem 2.32]{Gle20}. 
	That is, up to a fully faithful embedding, both functors are the perfection functor. 
	Nevertheless, we stress again that the essential images of the functors $(\str)^{\diamond}, (\str)^{\diamondsuit}$ on (perfect) schemes over $k$ are different.
	
	To prove the Scholze-Weinstein conjecture we need to work with schemes that are proper and flat over $O$, and their associated small v-sheaves. Therefore, we have to relate these two notions. 
	
	The functor $X\mapsto X^{\diamond}(=X^\diamondsuit\textrm{ if $X$ is proper})$ from the category of schemes over $O$ to small v-sheaves over $\Spd O$ factors as the composition of the functor
	\begin{equation}
	\widehat{(\str)}_\pi\colon \mathrm{Sch}_O\to \mathrm{fSch}_O,\ Y\mapsto \widehat{Y}_\pi
	\end{equation}
	of $\pi$-adic completion, followed by the functor sending (locally) a formal scheme $\Spf(A)$ over $\Spf(O)$ (with locally finitely generated ideal of definition) to the (pre-)adic space $\Spa(A)$, and then followed by the functor $(\str)^\diamondsuit$ on pre-adic spaces over $\Spa(O)$. 
	We also denote for a formal scheme $Y$ (admitting locally a finite ideal of definition) by $Y^\diamondsuit$ the v-sheaf for the pre-adic space associated with $Y$.
	
	On the category of schemes which are proper over $O$, the functor $\widehat{(\str)}_\pi$ of $\pi$-adic completion is fully faithful by Grothendieck's existence theorem \cite[Tag 08BF]{StaProj} or \cite[Th\'eor\`eme 5.4.1]{Gro61}. 
	Let us note that $\pi$-adic completion maps schemes, which are flat over $O$, to formal schemes, which are flat over $O$. 
	
	
	Let us recall the following notions. 
	
	\begin{definition}[{\cite[Tag 0EUL]{StaProj}}]
		\label{sec:two-diff-diam-definition-semi-normal-etc}
		A ring $A$ is called semi-normal if for all $a,b\in A$ with $a^3=b^2$ there exists a unique $c\in A$ with $a=c^2$ and $b=c^3$. 
		Similarly, $A$ is called absolutely weakly normal if it is semi-normal and if, for any prime $\ell$ and elements $a,b\in A$ with $\ell^\ell a=b^\ell$, there exists a unique $c\in A$ with $a=c^\ell$ and $b=\ell c$.
		Further, a ring $A$ whose spectrum has finitely many irreducible components is called weakly normal if it is semi-normal and if, for any prime $\ell$ and elements $a,b\in A$ such that $a$ is a nonzerodivisor, $\ell a\mid \ell b$ and $a\mid \ell b$, one has $a\mid b$. 
	\end{definition}
	
	Note that the last two properties are automatic for any prime $\ell$, which is invertible in $A$, and that each semi-normal ring is reduced.
	
	Since ring localizations preserve semi-normality, absolute weak normality or weak normality, they can be generalized to schemes (having locally finitely many irreducible components for the last property), see \cite[Tags 0EUN, 0H3I]{StaProj}. 
	Moreover, given any scheme $X$, there exists a morphism $X^{\text{sn}}\rightarrow X$ (respectively, $X^{\text{awn}}\rightarrow X$) from a semi-normal (respectively, absolutely weakly normal) scheme, which is called the semi-normalization (respectively, absolutely weak normalization) of $X$, and which is also the initial morphism $Y\to X$, which is a universal homeomorphism inducing isomorphisms on each residue field (respectively, universal homeomorphism), see \cite[Tags 0EUS]{StaProj}.
	In comparison, when $X$ has locally finitely many irreducible components, then the normalization $X^{\text{n}}\rightarrow X$ factors uniquely through the weak normalization $X^{\text{wn}}\rightarrow X$, which is also the initial homeomorphism $Y\rightarrow X$ admitting a factorization $X^{\text{n}}\rightarrow Y\rightarrow X$, see \cite[Tags 0H3N]{StaProj}.
	
	If $A$ is an $\bbF_p$-algebra, then $A$ is absolutely weakly normal if and only if $A$ is perfect \cite[0EVV]{StaProj}, and thus the absolute weak normalization agrees with the perfection of schemes over $\bbF_p$. 
	From the universal property of $X^{\mathrm{awn}}$ and the fact that universal homeomorphisms are integral, radicial and surjective \cite[Tag 04DF]{StaProj}, it is clear that normal, integral schemes $X$ with perfect function field are absolutely weakly normal. 
	In comparison, note that any normal scheme is weakly normal, without any condition on its function fields at generic points. 
	
	\begin{lemma}
		\label{sec:two-diff-diam-lemma-v-sheaf-only-depends-on-absolute-weak-normalization}
		Let $Y=\Spf(A)$ be an affine formal scheme over $O$, and let $I\subset A$ be a finitely generated ideal of definition. Let $B:=\widehat{A^{\mathrm{awn}}}_I$ be the $I$-adic completion of the absolute weak normalization of $A$. Then the natural map
		\begin{equation}
		\Spf(B)^\diamondsuit\to \Spf(A)^\diamondsuit
		\end{equation}
		is an isomorphism.
	\end{lemma}
	\begin{proof}
		Let $\Spa(R,R^+)\to \Spa(A)$ be a morphism from some affinoid perfectoid space over $O$. 
		Then $R^+$ is automatically $I$-adically complete. By the universal property of the absolute weak normalization and the fact that $\Spf(B)^\diamondsuit, \Spf(A)^\diamondsuit$ are v-sheaves it suffices to see that every affinoid perfectoid space admits a v-cover by one of the form $\Spa(R,R^+)$ with $R^+$ absolutely weakly normal. 
		We can always choose $(R,R^+)$ so that $R^+$ equals a product of perfectoid valuation rings with algebraically closed fraction fields. 
		In this case, $R^+$ is absolutely weakly normal because the conditions in \Cref{sec:two-diff-diam-definition-semi-normal-etc} can be checked in each factor. 
		Indeed, it is clear that any such factor is a normal, integral domain with algebraically closed, in particular perfect, fraction field.
	\end{proof}
	
	\begin{definition}
		\label{weakly-normal-formal-schemes}
		A formal scheme $X$ topologically of finite type and flat over $O$ is called weakly normal
		if it is locally of the form $\Spf(A)$ where $A$ is a weakly normal, flat and $\pi$-adically complete topological algebra of the form $O\langle T_1,\dots,T_n\rangle/I$ for some ideal $I\subset O\langle T_1,\dots,T_n\rangle$.
	\end{definition}
	
	To justify \Cref{weakly-normal-formal-schemes}, we need to prove that ``weak normality'' glues and localizes for the formal schemes that we work with. 
	This is the content of the next statement.
	\begin{proposition}
		Assume that $A$ is flat and topologically of finite type over $O$. 
		Let $\emptyset\neq U_{f_i}\subset \Spf(A)$ with $i\in \{1\dots,n\}$ be an open cover by distinguished open subsets with $U_i=\Spf(B_i)$. 
		Then $A$ is weakly normal if and only if all of the $B_i$ are weakly normal. 
	\end{proposition}
	\begin{proof}

		Since weak normality is compatible with localization $A$ is weakly normal if and only if all of the $A[f_i^{-1}]$ are weakly normal. 
		Now, $B_i$ is the $\pi$-adic completion of $A[f_i^{-1}]$, in particular flat over it. We claim that $A[f_i^{-1}]\to B_i$ is a regular map \cite[Tag 07BZ]{StaProj} and that $A\to \prod_i B_i$ is regular and faithfully flat. 
		Given these, the statement follows directly from \cite[Proposition III.3]{Man80} since a regular map is a reduced and normal map. 
		Let us prove the claim. 
		Observe that all of the rings are Noetherian and excellent because they are obtained from $O$ by adding variables, taking quotient by ideals, completing or localizing. 
		By \cite[Tag 07C0]{StaProj}, we may check regularity after localizing at a maximal ideal $\mathfrak{m}\subset B_i$. 
		Consider the following maps of rings:
		\begin{equation}
		A_{(\mathfrak{m})}\to (B_i)_{(\mathfrak{m})}\to \widehat{(A_{(\mathfrak{m})})_\pi}\to 	\widehat{(A)_\mathfrak{m}} 	
		\end{equation}
		They are all faithfully flat. 
		Since $A$ is excellent, the map $A_{(\mathfrak{m})}\to \widehat{(A)_\mathfrak{m}}$ is regular. 
		By \cite[Tag 07QI]{StaProj}, we conclude $A_{(\mathfrak{m})}\to (B_i)_{(\mathfrak{m})}$ is so as well.
	\end{proof}

	Variants of the following result appear in \cite[Proposition 18.4.1]{SW20} and \cite[IV, Theorem 4.6]{Lou20}:

	\begin{theorem}\label{univalencia do functor diamante restrito a esquemas formais normais}
		The functor $X \mapsto X^{\diamondsuit}$ from the category of absolute weakly normal formal schemes flat, separated and topologically of finite type over $O$, to v-sheaves over $\Spd O$ is fully faithful.
	\end{theorem}
	\begin{proof}
		We begin by proving the case in which $X$ and $Y$ are affine formal schemes. Confusing a formal scheme with its associated adic space we may assume that $X=\Spa(A), Y=\Spa(B)$ with $A$, $B$ absolutely weakly normal flat and topologically of finite type over $O$.
		To prove faithfulness we wish to show that if we are given two maps $g,f:X\to Y$ with $f^\Diamond=g^\Diamond$, then $f=g$.
		This follows from the fact that $A$ admits an injection (as it is reduced) into a product of perfectoid valuation rings.
		For fullness, let $f\colon X^\diamondsuit\to Y^\diamondsuit$ be a morphism of small v-sheaves. We are seeking a morphism $\psi\colon X\to Y$ such that $\psi^\diamondsuit=f$.  
		Let $K=O[\pi^{-1}]$ be the fraction field of $O$.  
		As $A$, $B$ are $\pi$-adic, the generic fibers $X_\eta, Y_\eta$ are given by $\Spa(A[\pi^{-1}],A^\prime)$, $\Spa(B[\pi^{-1}],B^\prime)$ with $A^\prime$, $B^\prime$ the integral closure of $A$, $B$ in $A[\pi^{-1}]$, $B[\pi^{-1}]$. 
		The localizations $A[\pi^{-1}]$ and $B[\pi^{-1}]$ are absolutely weakly normal by \cite[Tag 0EUM]{StaProj}, and thus semi-normal. By \cite[Proposition 10.2.3]{SW20}, we get a morphism $\psi_{\eta}\colon X_{\eta}\rightarrow Y_{\eta}$, so that $\psi_{\eta}^\diamondsuit=f_\eta$.
		Because $A$, $B$ are topologically of finite type over $O$ and reduced, the rings $A^\prime$, $B^\prime$ are finite over $A$, $B$, and thus in particular the subspace topology coming from $A[\pi^{-1}]$, $B[\pi^{-1}]$ is $\pi$-adic on $A^\prime$, $B^\prime$. 
		In particular, $A^\prime, B^\prime$ are Huber. 
		By definition the map $\psi_\eta\colon X_\eta\to Y_\eta$ induces a morphism $\psi'\colon X':=\Spa(A') \rightarrow Y$ over $O$, so that $\psi'_\eta=\psi_\eta$. 
		Denoting by $A''\subset A'$ the (automatically closed as $A$ is noetherian and $A^\prime$ finite over $A$) image of $B\widehat{\otimes}_{O} A$ in $A[\pi^{-1}]$, we even get $\psi''\colon X'':=\Spa(A'') \rightarrow Y$ such that the morphism $X'' \rightarrow Y\times_{O}X $ is a closed embedding of formal schemes. 
		It is easy to see that $(X^{\prime\prime})^\diamondsuit\to (Y\times_{\Spa(O)}X)^\diamondsuit\cong Y^\diamondsuit\times_{\Spd O} X^\diamondsuit$ is a closed immersion of v-sheaves. 
		Inside $Y^\diamondsuit \times_{\Spd O} X^\diamondsuit$, we then have two closed sub-v-sheaves, namely $X''^{\diamondsuit}$ induced by $\psi''^{\diamondsuit}$ and $X^\diamondsuit \simeq \Gamma_f$ induced by the graph of $f$. 
		In both of these closed sub-v-sheaves, the generic fiber is dense by \Cref{density of generic} below (applied to $A$ and $A^{\prime\prime}$), and they carry the same generic fiber. 
		Therefore, the finite birational morphism $X''\rightarrow X$ induced by the inclusion $A\subset A^{\prime\prime}$ becomes an isomorphism in the category of v-sheaves. 
		Passing to special fibers, this implies that $\Spec(A''/\pi)^{\on{perf}}\to \Spec(A/\pi)^{\on{perf}}$ is an isomorphism \cite[Proposition 18.3.1]{SW20}. 
		As $A^{\prime\prime}[\pi^{-1}]\cong A[\pi^{-1}]$ we can conclude that $\Spec(A^{\prime\prime})\to \Spec(A)$ is a universal homeomorphism. 
		Indeed, $A\to A^{\prime\prime}$ is integral, radicial (as can be checked on each fiber over $\Spec(O)$) and surjective. 
		Since $A$ is absolutely weakly normal, we get $A''=A$ and thus $(\psi'')^\diamondsuit=f$.
		
		We now extend the argument to the general case. To verify faithfulness one can easily argue locally on $X$ and $Y$ because if $X=\cup_{i\in I} X_i$ is an open cover by formal schemes, then $\cup_{i\in I} X_i^\diamondsuit$ is an open cover of $X^\diamondsuit$. 
		Proving fullness is more subtle since one has to justify that for a map $f\colon X^\diamondsuit\to Y^\diamondsuit$ and an open subset $U\subset Y$ with $U=\Spf(B)$, the pullback $f^{-1}(U^\diamondsuit)\subset X^\diamondsuit$ is ``classical''. 
		In other words, 
		\begin{equation}
		\label{check-detail-of-specialization}
		f^{-1}(U^\diamondsuit)=V^\diamondsuit
		\end{equation}
		for some open immersion of formal schemes $V\subset X$. 
		Now, by \cite[Proposition 18.3.1]{SW20} the special fiber map $f\times_{\Spd O} \Spd k$ is induced by a map of perfect schemes $f_{\on{red}}\colon X_{\on{red}}^{\text{perf}}\to Y_{\on{red}}^{\text{perf}}$. 		Identifying $|X|$, $|Y|$ with $|X_{\on{red}}^{\text{perf}}|$ and $|Y_{\on{red}}^{\on{perf}}|$, we can construct $V$ as $f_{\on{red}}^{-1}(U_{\on{red}})$. 
		That the identity in \Cref{check-detail-of-specialization} holds will follow from functoriality of the specialization map considered in \cite{Gle20}. 
		Indeed, $U^\diamondsuit=\on{sp}_{Y^\diamondsuit}^{-1}(U_{\on{red}})$.
	\end{proof}

	We used the following lemma. Here, for a Huber pair $(A,A^+)$ over $O$ the notation $\Spd(A,A^+)$ is a shorthand for $\Spa(A,A^+)^\diamondsuit$.

	\begin{lemma}
		\label{density of generic}
		Suppose that $B$ is a $\pi$-adically complete flat and topologically of finite type $O$-algebra, let $B^\prime$ denote the integral closure of $B$ in $B[\pi^{-1}]$. Then the generic fiber $\Spd(B[\pi^{-1}],B^\prime)$ is a dense open subset of $\Spd(B)$.
	\end{lemma}
	\begin{proof}
		Let $X=\Spa(B)$ with $B$ given the $\pi$-adic topology. Let $Y$ be the punctured open unit ball over $X$. 
		That is, $Y=\{y\in \Spa(B\pot{t})\mid |t|_y\neq 0 \}$, where $B\pot{t}$ is endowed with the $(\pi,t)$-adic topology. 
		The map $Y^\diamondsuit\to X^\diamondsuit$ is a v-cover so it is enough to prove $|Y^\diamondsuit_\eta|$ is dense in $|Y^\diamondsuit|$. 
		Now, $Y$ is the diamond associated to an analytic adic space so $|Y|=|Y^\diamondsuit|$ by \cite [Proposition 10.3.7]{SW20}. 
		Let $\Spa(R,R^+)\subset Y$ be a non-empty affinoid rational subset (with $(R,R^+)$ a complete Huber pair). 
		Since $B\pot{t}$ is noetherian,  flat over $B$, and rational localizations are flat for Huber pairs admitting a noetherian ring of definition, we can conclude that $R$ is flat over $O$. 
		Now $\Spa(R,R^+)$ is a pseudorigid space over $\Spa(O)$ in the sense of \cite{Lou17}, and thus in particular $R$ is a Jacobson ring \cite[Proposition 3.3.(3), 4.6]{Lou17}. 
		By flatness of $R$ over $O$ we get that $\pi$ is not nilpotent in $R$. There is a maximal ideal $\mathfrak{m} \subset R$ with $\pi\notin \mathfrak{m}$ as $R$ is a Jacobson ring. 
		By \cite[Lemma 1.4]{huber_a_generalization_of_formal_schemes_and_rigid_analytic_varieties} there is an element $x\in \Spa(R,R^+)$ whose support ideal is $\mathfrak{m}$. 
		In particular, this point lies in $\Spa(R,R^+)\cap Y_\eta\neq \emptyset$, which finishes the proof.
	\end{proof}
	
	The following consequence is the main statement we need from this chapter. 
	
	\begin{proposition}\label{univalencia do functor diamante restrito a esquemas formais normais 2}
		\begin{enumerate}
			\item Let $X$ be a proper, flat scheme over $O$. Then the absolute weak normalization $X^{\mathrm{awn}}\to \Spec(O)$ is proper and flat, and the canonical morphism
			\begin{equation}
			(X^{\mathrm{awn}})^\diamondsuit\to X^\diamondsuit
			\end{equation}
			is an isomorphism
			\item The functor $X\mapsto X^\diamondsuit$ is fully faithul when restricted to proper, flat and absolutely weakly normal schemes over $O$.
		\end{enumerate}          
	\end{proposition}
	\begin{proof}
		Using \Cref{univalencia do functor diamante restrito a esquemas formais normais} and Grothendieck's existence theorem as explained before there remain two statements to check: firstly that $X^{\mathrm{awn}}\to \Spec(O)$ is locally of finite type, and secondly that $\pi$-adic completion preserves absolute weak normality of $O$-algebras of finite type. The first follows from the fact that $X$ is excellent (implying finiteness of the normalization of the reduction of $X$), and that the absolute weak normalization of an integral domain with field of fraction of characteristic $0$ embeds into its normalization. The second follows by stability of absolute weak normality under regular ring homomorphisms, see \cite[Proposition 5.1]{GT80} and \cite[Proposition III.3]{Man80}.
	\end{proof}
	
	\subsection{$\pi$-adic kimberlites}
	\label{sec:formal-theory-v-sheaves-p-adic-kimberlites}	
	As in \Cref{sec:two-diff-diam-functors}, we let $O$ be a complete discrete valuation ring, which is flat over $\Z_p$, with perfect residue field $k$ (of characteristic $p$) and uniformizer $\pi\in O$. 
	Let $F$ denote its fraction field and $C$ a completed algebraic closure of $F$.
	We denote by $\breve F\subset C$ the maximal unramified complete subextension with ring of integers $\breve O$ and algebraically closed residue field $\bar k/k$.
	
	In \cite{Gle20}, the second named author introduced a set of axioms for a v-sheaf to have a well behaved specialization map to its reduced locus. 
	The v-sheaves satisfying these axioms are called kimberlites \cite[Definition 4.35]{Gle20} and they mimic the behavior of formal schemes. 
	Actually (under the very mild conditions of being separated and locally admitting a finitely generated ideal of definition), the v-sheaves associated to a formal scheme are always kimberlites \cite[Proposition 4.31]{Gle20}\footnote{The reference provided here only shows that the v-sheaves associated to formal schemes are valuative prekimberlites, but the additional axiom that the analytic locus is a spatial diamond is easily verified \cite[Remark 4.37]{Gle20}.} and the specialization map of the kimberlite attached to the formal scheme agrees with the traditional one. 
	
	On the other hand, in \cite{Lou17} the third named author considers the functor from the category of formal schemes $X$ over $O$ to the category $\mathcal{C}$ of specialization triples $(X_\eta, X_k^\textrm{perf}, \mathrm{sp}_{\breve{X}})$ where $X_\eta$ is a rigid analytic space over $F$, $X_k^{\textrm{perf}}$ is the perfection of the special fiber $X_k$ and $\on{sp}_{\breve X}\colon|X_{\breve{\eta}}|\to |X_{\bar k}|$ is a continuous map on the underlying topological spaces. 
	Here, $X_{\breve{\eta}}$ and $X_{\bar{k}}$ denote the base changes to $\breve F$ and $\bar k$ respectively.  
	This functor turns out to be fully faithful when one restricts to $X$ locally formally of finite type (that is, locally of the form $O\pot{T_1,\ldots, T_n}\langle X_1,\ldots, X_m\rangle/I$ for some ideal $I$), normal and flat over $O$, see \cite[18.4.2]{SW20}.
	
	In this section we take this approach to study $\pi$-adic kimberlites. That is, to a $\pi$-adic kimberlite $X$ over $\Spd O$ we attach a specialization triple $(X_\eta, X^{\textrm{red}}, \on{sp}_{\breve{X}})$ where now $X_\eta$ a diamond over $\Spd(F)$, $X^{\textrm{red}}$ a perfect scheme over $\Spec(k)$ and $\on{sp}_{\breve{X}}\colon|X_{\breve\eta}|\to |X^{\textrm{red}}_{\bar k}|$ a continuous map. 
	Again, $X_{\breve{\eta}}$ and $X_{\bar{k}}^{\textrm{red}}$ denote the base changes to $\breve F$ and $\bar k$ respectively.
	More importantly, we discuss some conditions on $X$ that make this functor fully faithful.

	We start by giving a review of specializations for kimberlites.
	Set $	\mathrm{SchPerf}_k$
	as the v-site of perfect schemes over $k$ (subject to the usual set-theoretic constraints of fixing some cut-off cardinal), and
	$
	\widetilde{\mathrm{SchPerf}_k}
	$
	the associated topos.
	
	\begin{definition}[{\cite[Definition 3.12]{Gle20}}]\label{definition-reduced-subsheaf}
		Given a v-sheaf $X$ on $\mathrm{Perf}_{\bbF_p}$ over $\Spd O$, one defines $X^{\on{red}}$ as the functor on $\mathrm{SchPerf}_k$ given by
		$Y \mapsto \Hom(Y^{\diamond},X)$.
	\end{definition}
	Thus, if $Y=\Spec(A)$ is an affine perfect scheme, then $X^{\on{red}}(\Spec(A))=X(\Spd(A))$. 
	By \cite[Proposition 3.7]{Gle20}, $X^{\on{red}}$ is in fact a small v-sheaf on $\mathrm{SchPerf}_k$. 
	The functor $(\str)^{\diamond}\colon\Spec(A)\mapsto \Spd(A)$ extends to small scheme-theoretic v-sheaves and the pair $(\diamond, (\str)^{\on{red}})$ forms an adjunction, see \cite[Definition 3.12]{Gle20}.

	For formal schemes over $O$, the reduction functor is simply the functor that assigns the perfection of the reduced locus \cite[Proposition 3.18]{Gle20}. More precisely, if $(B,B)$ is a formal Huber pair over $O$, that is $B$ is a complete $I$-adic $O$-algebra (with $I$ finitely generated), then $\Spd(B)^{\on{red}}=\Spec(B/I)^{\on{perf}}$.
	
	\begin{definition}
		\begin{enumerate}
			\item A map of v-sheaves $X\to Y$ is said to be formally adic if the following diagram is Cartesian:
			\begin{center}
				\begin{tikzcd}
				(X^{\on{red}})^{\diamond} \arrow{r} \arrow{d}  & \arrow{d}X \\
				(Y^{\on{red}})^{\diamond} \arrow{r} & Y 
				\end{tikzcd}
			\end{center}
			\item A v-sheaf $X$ over $\Spd O$ is $\pi$-adic if the structure morphism $X\to  \Spd O$ is formally adic.
		\end{enumerate}
	\end{definition}

	If $\Spa(A,A^+)$ is an affinoid adic space, we let $\Spd(A,A^+)$ denote the associated v-sheaf given by homomorphisms to untilts, see \cite[Subsection 18.1]{SW20}. If $A=A^+$, we abbreviate this by $\Spd(A)$.
	
	\begin{definition}
		Given a v-sheaf $X$, a map $f\colon\Spa(R,R^+)\to X$ from an affinoid perfectoid space formalizes if it factors through a map $g\colon\Spd(R^+)\to X$. 
		Any such $g$ is called a formalization of $f$. 
		The map $f$ v-formalizes if there is a v-cover $h\colon\Spa(S,S^+)\to \Spa(R,R^+)$ such that $f\circ h$ formalizes.
	\end{definition}
	
	\begin{proposition}
		\label{formalizing is v-locally formal}
		For a small v-sheaf $X$, the following are equivalent:
		\begin{enumerate}
			\item There is a set $I$, a family of formal Huber pairs $(B_i,B_i)_{\{i\in I\}}$ and a v-cover \begin{equation}\coprod_{i\in I} \Spd(B_i)\to X.\end{equation} 
			\item There is a set $J$, a family of perfectoid Huber pairs $(R_j,R^+_j)_{\{j\in J\}}$ and a v-cover \begin{equation}\coprod_{j\in J} \Spd(R^+_j)\to X.\end{equation} 
			\item For any perfectoid Huber pair $(R,R^+)$ all the maps $f\colon\Spa(R,R^+)\to X$ v-formalize.
		\end{enumerate}
	\end{proposition}
	\begin{proof}
		This is \cite[Lemma 4.7]{Gle20}.
	\end{proof}
	
	Any v-sheaf satisfying the conditions in \Cref{formalizing is v-locally formal} is said to be \textit{v-locally formal} or alternatively \textit{v-formalizing}.
	
	\begin{definition}
		\label{definition p-adic prekimberlite}
		A small v-sheaf $X$ over $\Spd O$ is a $\pi$-\textit{adic prekimberlite} if it is v-locally formal, the structure map $X\to \Spd O$ is separated and formally adic, and if $X^{\on{red}}$ is represented on $\mathrm{SchPerf}_k$ by a perfect scheme.
	\end{definition}
	
	The more general definition of a prekimberlite is given in \cite[Definition 4.15]{Gle20}, and we justify below why $\pi$-adic prekimberlites are a special type of prekimberlite. 
	For this reason, in our context, we can take \Cref{definition p-adic kimberlite} as our definition.
	
	\begin{proposition}
		\label{pi adic prekimberlite is prekimberlite and pi adic}
		A small v-sheaf $X$ equipped with a separated morphism $X\to \Spd O$ is a $\pi$-adic prekimberlite if and only if $X$ is a prekimberlite and the map $X\to \Spd O$ is formally adic.	
	\end{proposition}
	\begin{proof}
		Formal adicness implies that $X^{\on{an}}=X_\eta$ and $(X^{\on{red}})^{\diamond}=X\times_{\Spd O} \Spd k$. From this it is clear how one definition translates to the other except that to prove $X$ is a prekimberlite we need to justify why it is formally separated. Now, the argument given in \cite[Proposition 3.29]{Gle20} applies with the role of $\bbZ_p$ exchanged for $O$.  
	\end{proof}

	To any prekimberlite $X$, in particular to any $\pi$-adic prekimberlite, one attaches a topological specialization map 
	$\on{sp}_X:|X|\to |X^\red|$ \cite[Definition 4.12]{Gle20},
	and a v-sheaf theoretic specialization map $\on{SP}_X:X\to (X^{\on{red}})^{\diamond/\circ}$ due to Heuer, see \cite[Section~4.4]{Gle20}. 
	Here $(X^{\on{red}})^{\diamond/\circ}$ is as in \cite[Definition 4.23]{Gle20}.

	\begin{definition}
		\label{definition p-adic kimberlite}
		A $\pi$-adic prekimberlite is a $\pi$-adic kimberlite if $X_\eta$ is a locally spatial diamond, the restriction of $\on{sp}_X$ to $|X_\eta|\subseteq |X|$ is a quasi-compact map and $\on{SP}_X$ is partially proper.
	\end{definition}
	\begin{remark}
	The more general definition of a kimberlite is given in \cite[Definition 4.35]{Gle20}.
	Just as in \Cref{pi adic prekimberlite is prekimberlite and pi adic} and with the same argument one can see that a small v-sheaf $X$ equipped with a separated morphism $X\to \Spd O$ is a $\pi$-adic kimberlite if and only if $X$ is a kimberlite and the map $X\to \Spd O$ is formally adic.
	\end{remark}

	If $f\colon S\to T$ is a map of locally spectral spaces, then we call $f$ spectral if for any quasi-compact open $U\subset S, V\subset T$ with $f(U)\subset V$ the induced map $f\colon U\to V$ of spectral spaces is spectral, that is, quasi-compact.

	In what follows we consider the restriction of the topological specialization map $|X|\to |X^{\on{red}}|$ to $|X_\eta|\subseteq |X|$. By abuse of notation, we still use $\on{sp}_X$ to denote the map $\on{sp}_X:|X_\eta|\to |X^{\on{red}}|$.

	\begin{proposition}
		\label{sec:formal-theory-v-proposition-existence-of-specialization-map}
		The following statements hold:
		\begin{enumerate}
			\item The rule $X \mapsto (X_\eta,X^{\on{red}},\on{sp}_X)$ is functorial when $X$ varies along $\pi$-adic prekimberlites.
			\item If $X$ is a $\pi$-adic prekimberlite, and $X_\eta$ is a locally spatial diamond then the specialization map $\on{sp}_X\colon|X_\eta|\to |X^{\on{red}}|$ is spectral.
			\item If $X$ is a $\pi$-adic kimberlite the map $\on{sp}_X$ is a closed map. 
		\end{enumerate}
	\end{proposition}
	\begin{proof}
		Functoriality is \cite[Proposition 4.14]{Gle20} specialized to the $\pi$-adic case considered here.	
		The same argument as in \cite[Theorem 4.40]{Gle20} shows that the map is spectral.
		The last statement is \cite[Theorem 4.40.(2)]{Gle20}.
	\end{proof}
	
	One of the main features of kimberlites is that, as with formal schemes, they come with a notion of tubular neighborhoods (or completion at a point). 
	\begin{definition}[{\cite[4.18]{Gle20}}]
		Given a $\pi$-adic prekimberlite $X$ and a locally closed subset $S\subset |X^{\on{red}}|$, one defines $\widehat{X}_{/S}$ as the v-sheaf making the following diagram Cartesian:
		\begin{center}
			\begin{tikzcd}
			\widehat{X}_{/S} \arrow{r} \arrow{d}  & X \arrow{d} \\
			\underline{|S|} \arrow{r} & \underline{|X^{\on{red}}|}
			\end{tikzcd}
		\end{center}
		Here $\widehat{X}_{/S}$ is called the formal neighborhood of $X$ around $S$, and $(\widehat{X}_{/S})_\eta$ the tubular neighborhood of $X$ around $S$.
	\end{definition}
	
	Here, the right vertical arrow is the composition of the natural map $X\to \underline{|X|}$ and the map $\underline{|X|}\to \underline{|X^{\on{red}}|}$ mentioned in \Cref{sec:formal-theory-v-proposition-existence-of-specialization-map}. 
	We will mostly use tubular neighborhoods when $S=\{x\}$ is a closed (and constructible) point in $X^{\on{red}}$.

	
	\begin{example}
		For any $\pi$-adic prekimberlite $X$ and any locally closed subset $S\subset |X^{\on{red}}|$, one has inclusions
		\begin{equation}
		\label{containment-of-tubular}
		|S^{\diamond}|\subset |\widehat{X}_{/S}| \subset \on{sp}_X^{-1}(S),
		\end{equation}
		which are strict in general.
		For example, let $X=\Spd(A)$ with $A$ a perfect $k$-algebra and let $S\subset \Spec(A)=X^{\on{red}}$ the Zariski closed subset defined by a finitely generated ideal $I\subset A$ with generators $a_1,\dots,a_n$. 
		Then, $S^{\diamond}$ is the locus in $\Spd(A)$ where $a_1=\dots=a_n=0$, $\widehat{X}_{/S}$ is the (open) locus in $\Spd(A,A)$ where $a_1,\dots,a_n$ are all topologically nilpotent and $\on{sp}_X^{-1}(S)$ is the closed subset of points for which $|a_i|<1$. 
		With this description it is immediate to verify the containment of \eqref{containment-of-tubular}. 

		Now, the complement $\on{sp}_X^{-1}(S)\setminus |\widehat{X}_{/S}|$ consists of those higher rank points $(A,A)\to (C,C^+)$, for which at least one of $a_i^{-1}\in C^\circ \setminus C^+$. 
		Note the associated point $(A,A)\to (C,C^\circ)$ is not in $\on{sp}_X^{-1}(S)$. 
		In particular, $\on{sp}_X^{-1}(S)$ is usually not weakly generalizing and does not define a closed subsheaf. 
	\end{example}

	\begin{proposition}
		If $S\subset |X^{\on{red}}|$ is locally closed and constructible, then $\widehat{X}_{/S}\to X$ is an open immersion.
	\end{proposition}
	\begin{proof}
		This is proved in \cite[Proposition 4.22]{Gle20}. 
	\end{proof}
	
	We now introduce a weak form of flatness over $O$ for $\pi$-adic kimberlites.
	
	\begin{definition}
		A $\pi$-adic kimberlite $X$ over $\Spd O$ is said to be \textit{flat} if there is a set $I$, a family of $F$-perfectoid Huber pairs $\{(R_i^\sharp,R_i^{\sharp+})\}_{i \in I}$ and a v-cover over $\Spd O$ \begin{equation}\coprod_{i\in I} \Spd(R_i^{\sharp+})\to X.\end{equation}	\end{definition}
	
	We now construct our first examples of flat $\pi$-adic kimberlites.
	
	\begin{proposition}
		\label{prop huber pairs strong p adic}
		Let $f\colon A\to B$ be a map of complete $\pi$-adic algebras that are flat over $O$. 
		Suppose that $A$ is integrally closed in $A[\pi^{-1}]$  and that $\Spd(B[\pi^{-1}],B)\to \Spd(A[\pi^{-1}],A)$ is a v-cover. 
		Then $\Spd(B)\to \Spd(A)$ is also a v-cover. 
		In particular, for any such $A$ the v-sheaf $\Spd(A)$ is a flat $\pi$-adic kimberlite.
	\end{proposition}
	
	\begin{proof}
		By \cite[Lemma 2.26]{Gle20}, the map $\Spd(B)\to \Spd(A)$ is quasi-compact, so it is enough to prove $|\Spd(B)|\to |\Spd(A)|$ is surjective by \cite[Lemma 12.11]{Sch17}. 
		Surjectivity on the generic fiber follows from the hypothesis. 
		On the special fiber, we use \cite[Lemma 3.5, Proposition 3.7]{Gle20} to prove instead that the map $\Spa(B/\pi)\to \Spa(A/\pi)$ is surjective. 
		
		Let $x\in \Spa(A/\pi)$ and let $\Spa(k(x),k(x)^+)\to \Spa(A/\pi)$ the affinoid residue field map. 
		Let $\mathfrak{p}_x\in \Spec(A/\pi)$ denote the support ideal of $x$. 
		Since $A$ is integrally closed in $A[\pi^{-1}]$, the pair $(A[\pi^{-1}],A)$ is a complete Tate Huber pair and we have a surjective specialization map $\on{sp}_A\colon\Spa(A[\pi^{-1}],A)\to \Spec(A/\pi)$ by \cite[Proposition 4.2]{Gle20}, \cite[Remark 7.4.12]{Bha17}. 
		Let $y\in \Spa(A[\pi^{-1}],A)$ with $\on{sp}_A(y)=\mathfrak{p}_x$. 
		We obtain a map $\Spa(R[p^{-1}],R)\to \Spa(A[\pi^{-1}],A)$ with $R:=k(y)^+$. The residue field of $R$ is $k(x)$ and we can consider $R^+\subset R$ defined as $R\times_{k(x)}k(x)^+$. 
		This promotes to a map $\Spa(R^+)\to \Spa(A)$. 
		As $\Spd(B[\pi^{-1}],B)\to \Spd(A[\pi^{-1}],A)$ is a v-cover we can find a v-cover of $\Spa(C,C^+)\to \Spa(R[\pi^{-1}], R^+)$ with $(C,C^+)$ a perfectoid field and a commutative diagram 
		\begin{center}
			\begin{tikzcd}
			\Spa(C^+) \arrow{r} \arrow{d}  & \Spa(R^+) \arrow{d} \\
			\Spa(B) \arrow{r} & \Spa(A). 
			\end{tikzcd}
		\end{center}
		The map $\Spa(C^+)\to \Spa(R^+)$ is easily seen to be surjective since it is an extension of valuation rings. 
		So $x$ lies in the image of $\Spa(B/\pi)$ as we needed to show.
		
		That $\Spd(A)$ is a valuative prekimberlite for $A$ as above follows from \cite[Proposition 4.31]{Gle20}. 
		To show it is a kimberlite it suffices to know that $\Spd(A)_\eta$ is a spatial diamond, which follows from \cite[Lemma 15.6]{Sch17}. Indeed, in this case the specialization map is automatically quasi-compact \cite[Remark 4.37]{Gle20}. 
		Now, we may always find a cover by an affinoid perfectoid space $\Spd(P,P^+)\to \Spd(A[p^{-1}],A)$ by \cite[Lemma 15.3]{Sch17}. What we have shown so far implies that $\Spd(P^+)\to \Spd(A)$ is also a v-cover. This finishes the proof.
	\end{proof}
	
	\begin{proposition}
		\label{plano is flat}
		If $A$ is the $\pi$-adic completion of a flat and finite type algebra over $O$, then $\Spd(A)$ is a flat $\pi$-adic kimberlite.
	\end{proposition}
	
	\begin{proof}
		We may assume that $A$ is reduced as passing to the absolute weak normalization does not change $\Spd(A)$ by \Cref{univalencia do functor diamante restrito a esquemas formais normais} and \Cref{univalencia do functor diamante restrito a esquemas formais normais 2}. As $A$ is noetherian and quasi-excellent, the integral closure of $A$ in its total ring of fractions is therefore a finite $A$-module. In particular, the integral closure $A^\prime$ of $A$ in $A[p^{-1}]$ is finite over $A$. Thus, we can conclude that $\Spd(A')$ (with $A^\prime$ given the $\pi$-adic topology) is flat by \Cref{prop huber pairs strong p adic} and the map $\Spd(A')\to \Spd(A)$ is a v-cover since it is isomorphism over $\Spd(F)$ (this uses that the $\pi$-adic topology on $A^\prime$ agrees with the subspace topology on $A[\pi^{-1}]$) and the map $\Spec(A'/\pi)\to \Spec(A/\pi)$ is proper and surjective (here we use again \cite[Lemma 3.5, Proposition 3.7]{Gle20} as in \Cref{prop huber pairs strong p adic}).
	\end{proof}
	
	\begin{remark}
		A careful inspection of the proof of \Cref{prop huber pairs strong p adic} above allows us to conclude that a $\pi$-adic formal Huber pair $(A,A)$ will give rise to a flat $\pi$-adic kimberlite $\Spd(A)$ if and only if the specialization map 
		\begin{equation}
		\on{sp}_A\colon\{x\in \Spa(A)\mid |\pi|_x \neq 0\}\subset \Spd(A)\to \Spec(A/\pi)
		\end{equation} 
		is surjective. 
		The hypothesis taken in \Cref{prop huber pairs strong p adic} are easy to verify assumptions that ensure this happens. 
		Without assuming flatness of $A$, this might not hold since for a discrete and perfect $O$-algebra $A$ in characteristic $p$, the v-sheaf $\Spd(A)$ is a $\pi$-adic kimberlite that is not flat.
	\end{remark}
	
	We can relate flatness for $\pi$-adic kimberlites to surjectivity of the specialization map.
	
	\begin{proposition}
		\label{proper and sp surjective gives strong kimberlite}
		Let $X$ be a $\pi$-adic kimberlite over $\Spd O$.
		\begin{enumerate}
			\item If $X$ is flat, then the specialization map $\on{sp}\colon|X_\eta|\to |X^{\on{red}}|$ is surjective.
			\item Conversely, if $X\to \Spd O$ is proper and $\on{sp}\colon|X_\eta|\to |X^{\on{red}}|$ surjective, then $X$ is flat over $\Spd O$.
		\end{enumerate}
		
	\end{proposition}
	\begin{proof}
		The first statement reduces to the case $X=\Spd(R_i^{\sharp +})$ for $(R_i^\sharp, R_i^{\sharp +})$ a perfectoid Huber pair over $F$, where it follows from \cite[Proposition 4.2]{Gle20}. Let us prove the second.
		It follows from the hypothesis that $X_\eta$ is quasi-compact over $\Spd F$, and thus we may find a v-cover $\Spa(R,R^+)\to X_\eta$ by affinoid perfectoid.
		Refining the cover if necessary we may assume this map factors through a map $\Spd(R^+)\to X$ because $X$ is v-formalizing.
		Since $X$ is quasi-separated over $\Spd O$ and $\Spd(R^+)$ is quasi-compact over $\Spd O$ (see \cite[Lemma 2.26]{Gle20}), we may conclude that $\Spd(R^+)$ is quasi-compact over $X$.
		To prove it is a v-cover, it is therefore enough to prove that the map of topological spaces is surjective.
		On the generic fiber this is clear.
		Using \cite[Lemma 3.5, Proposition 3.7]{Gle20}, we need to show $\Spec((R^+/\pi)^{\on{perf}})\to X^{\on{red}}$ is a scheme-theoretic v-cover, or equivalently that the map of the associated adic spectra induced by the morphism of schemes is surjective.
		
		The proof now follows a similar argument to the one given in \Cref{prop huber pairs strong p adic}.
		Given a point $x\in |(X^{\on{red}})^{\on{ad}}|$ in the adic spectrum of $X$ with affinoid residue field $\Spa(k(x),k(x)^+)$ we consider the point in $\mathfrak{p}_x\in |X^{\on{red}}|$ corresponding to the support of $x$.
		By surjectivity of the specialization map there is a point $y\in |X_\eta|$ with $\on{sp}_X(y)=\mathfrak{p}_x$.
		Represent $y$ by a map $\Spa(C,C^+)\to X_\eta$ with $(C,C^+)$ a perfectoid affinoid field over $F$.
		Replacing $\Spd(C,C^+)$ by a v-cover we may assume this map factors over a map $\Spd(C^+,C^+)\to X$.
		In particular, it promotes to a map $\Spd(C^+)\to X$.
		The closed point of $\Spd(C,C^+)$ specializes to a point with the same support as $x$.
		Let $\kappa(y)$ be the residue field of $C^+$.
		Then $\kappa(y)$ is a field extension of $k(x)$, and we can find a valuation ring $\kappa(y)^+\subset \kappa(y)$ making $\kappa(y)^+/k(x)^+$ an extension of valuation rings.
		By pullback along the surjection $C^+\twoheadrightarrow \kappa(y)$ we may construct from $\kappa(y)^+$ an open and bounded valuation  $C_1^+\subset C^+$.
		Since $X_\eta$ is partially proper we may extend $\Spd(C,C^+)$ to a map $\Spd(C^+,C_1^+)\to X_\eta$.
		After possibly replacing $\Spa(C,C^+_1)$ by a v-cover, we may assume it factors through $\Spa(R,R^+)$.
		Then the map extends to $\Spd(C_1^+)\to \Spd(R^+)\to X$. The map of adic spectra $\Spec((C_1^+/\pi)^{\on{perf}})^{\on{ad}}=\Spa((C_1^+/\pi)^{\on{perf}},(C_1^+/\pi)^{\on{perf}})\to (X^{\on{red}})^{\on{ad}}$ has $x$ in its image as we wanted to show.
	\end{proof}

	We now discuss some ad hoc hypothesis on $\pi$-adic kimberlites that allow us to recover them from their specialization triple.
	
	\begin{definition}
		\label{defi category K}
		Let $\calK$ be the full subcategory of v-sheaves over $\Spd O$ consisting of flat $\pi$-adic kimberlites $X$ that are quasi-compact and separated over $\Spd O$ and satisfy the following properties:
		\begin{enumerate}
			\item The $\Spd C$-valued points of $X$ define a dense subset of $|X_C|$. 
			\item The reduction $X^{\on{red}}$ is a perfect $k$-scheme perfectly of finite type.   
			\item Every section $\Spd C\to X_C$ formalizes to a map $\Spd O_C \to X_{O_C}$.
		\end{enumerate}
	\end{definition}
	
	Our main theorem about the category $\calK$ is the following.
	
	\begin{theorem}\label{prop_fully_faith_triples_kimberlites}
		When restricted to the category $\calK$ of \Cref{defi category K}, the functor sending a $\pi$-adic kimberlite to its generic fiber is faithful and the functor that sends it to its specialization triple \begin{equation}X\mapsto (X_\eta,X^{\mathrm{red}},\mathrm{sp}_{\breve{X}})\end{equation} 
		is fully faithful.
		Here, $\mathrm{sp}_{\breve{X}}$ denotes the specialization map associated with the base change $\breve X:=X\x_{\Spd(O)}\Spd(\breve O)$. 
	\end{theorem}

	\begin{proof}
		Let us prove faithfulness. 
		Let $f,g\colon X\to Y$ be two maps such that $f_\eta=g_\eta$. 
		Since $X$ is flat and quasi-compact we may replace it by a cover of the form $\Spd(R^{+})$. 
		Since $Y$ is separated and $\pi$-adic the map $\Delta\colon Y\to Y\times_{\Spd O} Y$ is formally adic and a closed immersion. 
		The pullback of $\Delta$ by $(f,g)$ is closed and formally adic subsheaf of $\Spd(R^{+})$ with the same generic fiber. We may finish by arguing as in the proof of \cite[Proposition 4.9]{Gle20}.
		
		Let us prove the map is full. Fix a map $f:=(f_\eta,f^{\on{red}})$ of triples \begin{equation}f\colon(X_\eta,X^{\on{red}},\mathrm{sp}_{\breve{X}})\to (Y_\eta,Y^{\on{red}},\mathrm{sp}_{\breve{Y}})\end{equation} and let $W= X \times_{\Spd O} Y $.
		Let $g\colon\Spa(R,R^+)\to X_\eta$ be a formalizable v-cover which extends to a surjection $\Spd(R^+)\to X$ and for which $f\circ g$ is also formalizable (this is possible using \Cref{prop huber pairs strong p adic}).
		Let $(g,f\circ g)\colon\Spd(R^+)\to W$ be the induced map and define $Z$ as the sheaf-theoretic image of $(g,f\circ g)$ in $W$. We have a projection map $Z\to X$ and we wish to prove that it is an isomorphism.
		Observe that the graph morphism $(\on{id},f_\eta)\colon X_\eta\to W_\eta$ already identifies $X_\eta$ with $Z_\eta$. In particular, $Z(C)$ is dense inside $|Z_C|$ by our assumption on $X$.
		
		By construction $Z$ is v-locally formal since $\Spd(R^+,R^+)$ surjects onto it. Moreover, since $Z\subset W$ and $W$ is separated over $O$, we see that $Z$ is also separated over $O$. 
		Let us prove that $Z$ if formally $\pi$-adic and that $Z^{\on{red}}$ is isomorphic to $X^{\on{red}}$. 
		
		We claim that $Z_s \subset W_s=(W^{\on{red}})^{\diamond}$ factors through the graph of $(f^{\on{red}})^{\diamond}$.
		Indeed, since $X_{O_C}$ and $Y_{O_C}$ formalize $C$-sections, for any map $q\colon \Spd C\to Z_\eta$ we obtain maps $q^{\on{red}}_x\colon\mathrm{Spec}(\bar{k})\to X^{\on{red}}$ and $q^{\on{red}}_y\colon\mathrm{Spec}(\bar{k})\to Y^{\on{red}}$ intertwinned under $f^{\on{red}}$. 
		This holds because $|f_{\bar{k}}^{\on{red}}|\circ \mathrm{sp}_{\breve{X}}=\mathrm{sp}_{\breve{Y}}\circ |f_{\breve{\eta}}|$ by assumption and $\bar{k}$-sections of $Y^\red_{\bar{k}}$ are determined by the closed point in $|Y^\red_{\bar{k}}|$ that they induce. 
		This shows that $\on{sp}_{{W}}(q)\in \Gamma(f^{\on{red}})$.		
		In particular, $\on{sp}_{{W}}(|Z_{{\eta}}|)\subset \overline{\on{sp}_{{W}}(Z(C))}\subset |W^{\on{red}}|$ is contained in $\Gamma(f^{\on{red}})$.
		By \cite[Proposition 4.2]{Gle20}, we know that the specialization map $\Spd(R,R^+)\to \Spec((R^+/\pi)^{\on{perf}})$ is surjective.
		Because $(g,f_\eta\circ g)\colon \Spd(R,R^+)\to W$ has image $|Z_\eta|$ (on topological spaces), this implies by naturality of the specialization map that the morphism $g^{\on{red}}\colon\mathrm{Spec}(R^+/\pi)^{\on{perf}}\to W^{\on{red}}$ factors through $\Gamma(f^{\on{red}})$ as well.
		Consequently, $Z_s\to W_s$ factors through $\Gamma(f^{\on{red}})^{\diamond}$. 
		On the other hand, since $\Spd(R^+)\to X$ is surjective the projection map
		\begin{equation}
		(\mathrm{Spec}(R^+/p)^{\on{perf}})^{\diamond}\to (X^{\on{red}})^{\diamond}
		\end{equation}
		is a surjection. 
		This implies that the morphism $\Spec((R^+/\pi)^{\on{perf}})\to W_{\on{red}}$ surjects onto $\Gamma(f^{\on{red}})$, and this in turn implies that $Z_s\to \Gamma(f^{\on{red}})^{\diamond}$ is an isomorphism, as it is a monomorphism and surjective. 
		In particular, we get that $Z_s\cong (Z^{\mathrm{red}})^{\diamond}$, that is, $Z$ is formally $\pi$-adic.
		
		As we have seen the map $Z\to X$ is an isomorphism on the generic fiber and on the special fiber. Since $\Spd(R^+)\to Z$ is surjective $Z$ is quasi-compact over $\Spd O$, which is enough to conclude $Z\to X$ is an isomorphism (by \cite[Lemma 12.5]{Sch17}, note that $Z\to X$ is quasi-compact, as $X$ is qcqs over $\Spd O$).	
	\end{proof}
	
	It is also relevant to relate this to a notion of topological flatness that appears in \cite{PR21}.
	
	\begin{lemma}\label{lem_top_flat_implies_formal_flat}
		Let $X$ be a proper $\pi$-adic prekimberlite over $\Spd O$ satisfying conditions (1)-(3) of \Cref{defi category K} and such that $X_\eta$ is a spatial diamond. Then the following hold.
		\begin{enumerate}
			\item $X$ is a $\pi$-adic kimberlite. 
			\item If $\lvert X_\eta\rvert$ is a dense open\footnote{The converse, however, fails. Indeed, let $O\langle t\rangle \subset V$ be a higher rank valuation ring endowed with its $\pi$-adic topology. Then $\Spd(V)$ is a flat $\pi$-adic kimberlite. As in \cite[Definition 2.1, Proposition 2.19]{Gle20}, one can prove that the locus $N_{t\ll1}$ where $t$ is topologically nilpotent is an open subset of $|\Spd(V)|$ that does not meet the generic fiber.} subset of $\lvert X\rvert$, then $X$ is flat, thus lies in $\calK$.  
		\end{enumerate}
	\end{lemma}
	
	\begin{proof}
		By \cite[Proposition 4.32]{Gle20} the map $\on{SP}_X:X\to (X^{\on{red}})^{\diamond/\circ}$ is partially proper. 
		By \Cref{sec:formal-theory-v-proposition-existence-of-specialization-map}.(2) the specialization map $\on{sp}:|X|\to |X^{\on{red}}|$ is a spectral map, but since both topological spaces are qcqs it is a quasi-compact map. 
		This shows that $X$ is a $\pi$-adic kimberlite.

		Now, by \Cref{sec:formal-theory-v-proposition-existence-of-specialization-map}.(3) the specialization map sends closed subsets to closed subsets. Since $X^{\on{red}}$ is perfectly of finite type, it suffices to prove surjectivity of the map $X(\Spd C) \to X^{\on{red}}(\bar{k})$ induced by $\mathrm{sp}$. 
		
		For this, take the associated formal neighborhood $\widehat{X}_{/x}$ over a closed point $x$ of the reduction, which can be represented by $\Spec \bar{k} \to X^{\mathrm{red}}$ uniquely up to Galois automorphisms. It is a non-empty open by \cite[Proposition 4.22]{Gle20}. 
		Hence, it must have topologically dense generic fiber, which is in particular non-empty. By hypothesis, we can find a $C$-valued point mapping to $x$.
	\end{proof}
	
	The following statement gives a v-sheaf theoretic criterion to determine when a weakly normal scheme is already normal.
	
	\begin{proposition}
\label{prop_connected_tubes_implies_normality}
		Let $A$ be a flat, weakly normal and topologically of finite type $\pi$-adically complete domain over $O$. 
		Suppose that $A[\pi^{-1}]$ is normal and that, for every closed point $x\in \Spec(A/\pi)$, the diamond $(\widehat{\Spd(A)}_{/x})_\eta$ (equivalently, the rigid analytic fiber of the formal affine scheme $\widehat{\Spec(A)}_{/x}$) is connected.
		Then $A$ is normal. 
	\end{proposition}
	\begin{proof}
		First off, one has $\widehat{\Spd(A)}_{/x}=(\widehat{\Spec(A)}_{/x})^\lozenge$ by \cite[Proposition~4.19]{Gle20}.
		Since taking generic fibers commutes with $\lozenge$, we see that the connectedness of $(\widehat{\Spd(A)}_{/x})_\eta$ is equivalently to the connectedness of the rigid analytic fiber of $\widehat{\Spec(A)}_{/x}$.
		
		Now, let $B$ denote the integral closure of $A$ in $A[\pi^{-1}]$. 
		Since $A[\pi^{-1}]$ is normal, $B$ is also normal and $B$ is a finite $A$-algebra. 
		We claim that $f\co \Spd(B)\to \Spd(A)$ is an isomorphism, so that $A=B$ by \Cref{univalencia do functor diamante restrito a esquemas formais normais}. 
		By quasi-compactness, it is enough to check this on the generic and special fibers. 
		The generic case follows from the definition of $B$.
		We need to prove $\Spec(B/\pi)^{\mathrm{perf}}\cong \Spec(A/\pi)^{\mathrm{perf}}$ which amounts to proving that the fibers at closed points consists of singletons.
		Let $x\in \Spec(A/\pi)$ denote a closed point.
		By \cite[Proposition 4.20]{Gle20}, we have an identification 
		\begin{equation}
		(\widehat{\Spd(B)}_{/f^{-1}(x)})_\eta \cong (\widehat{\Spd(A)}_{/x})_\eta	
		\end{equation}
		In turn we also have $\coprod_{y\in f^{-1}(x)}(\widehat{\Spd(B)}_{/y})_\eta \cong (\widehat{\Spd(B)}_{/f^{-1}(x)})_\eta$.
		By \Cref{plano is flat} and \Cref{proper and sp surjective gives strong kimberlite} for all $y\in f^{-1}(x)$ the tubular neighborhood $(\widehat{\Spd(B)}_{/y})_\eta$ is a non-empty open subset of $(\widehat{\Spd(A)}_{/x})_\eta$.
		Since we assumed this to be connected we can conclude $f^{-1}(x)$ contains a unique element.
	\end{proof}

	\section{The affine flag variety}
	\label{sec:affine-flag-variety}
	In this section, we discuss some relevant material on perfect schemes and Witt vector affine flag varieties. 
	Namely, we review the calculation of the Picard group by He--Zhou \cite{HZ20}, the definition of canonical finite type deperfections of Schubert perfect schemes and apply a Stein factorization argument to construct a comparison isomorphism between the $p$-adic canonical deperfections of depth $0$ Schubert perfect schemes with the corresponding weakly normal Schubert schemes in the equicharacteristic situation. 
	In particular, we prove \cite[Conjecture III]{Zhu16} on their singularities in this case.
	
	\subsection{Perfect schemes}
	\label{sec:perfect-schemes}
	
	Here, we present some facts on perfect schemes that we will need later. 
	Let $p$ be a prime number. 
	All our schemes in this subsection will be assumed to lie over $\bbF_p$.
	
	The basic theory of perfect schemes is discussed in \cite[A.]{Zhu17} and \cite[Section 3]{BS17}. 
	In particular, we will use the notions of a perfectly finitely presented map between qcqs perfect schemes \cite[Proposition 3.11]{BS17} and of a perfectly proper morphism \cite[Defintion 3.14]{BS17}, \cite[Appendix A.18]{Zhu17}. 
	If $k$ is a perfect field, then we occasionally call a separated, perfectly finitely presented scheme $X$ over $k$ a perfect $k$-variety \cite[Remark A.14]{Zhu17}.
	
	A morphism $Y\to X$ of perfect schemes is called perfectly smooth if, \'etale locally on $Y$, there exists \'etale morphisms to the perfection of some relative affine space over $X$, see \cite[Definition A.25]{Zhu17}.  
	
	Given any normal finite type $k$-scheme $Y$, its perfection $Y_{\mathrm{perf}}$ is normal as it is a filtered colimit of normal schemes along affine transition maps. 
	Conversely, if $X$ is a qcqs normal perfect scheme perfectly of finite type, then using \cite[Tag 01ZA]{StaProj} (and finiteness for integral closures of schemes of finite type over a field), then we can write $X$ as the filtered colimit of perfections of normal schemes $Y_i, i\in I$, which are of finite type over $k$.
	
	The following result gives a topological criterion for normality of perfect schemes. 
	We stress that perfectness is crucial as one sees, for example, by looking at the normalization morphisms of the cuspidal curve.
	
	\begin{lemma}\label{caracterizacao de variedades normais perfeitas}
		Let $f\colon Y\to X$ be a surjective, perfectly proper morphism between qcqs integral perfect schemes. 
		Assume that $Y$ is normal and $f$ birational. 
		Then $X$ is normal if and only if the geometric fibers of $f$ are connected.
	\end{lemma}
	
	\begin{proof}
		If all geometric fibers of $f$ are connected, then the natural map $\calO_X\to f_\ast\calO_Y$ is an isomorphism, see \cite[Proposition 6.1]{BS17}, \cite[Lemma A.21]{Zhu17}. 
		Thus,
		\[
		\mathcal{O}_X(U)\cong \mathcal{O}_Y(f^{-1}(U)),
		\]
		for any open affine $U\subset X$. 
		As $\mathcal{O}_Y(V)$ is a normal ring for any open subset $V\subset Y$, the claim follows (here we use that $Y$ is integral \cite[Tag 0358]{StaProj}).
		
		Conversely, we can write $f$ as the perfection of a proper, finitely presented morphism $f_0\colon Y_0\to X$ by \cite[Proposition 3.13, Corollary 3.15]{BS17}. Let $g_0\colon Y_0\to Z_0=\underline{\Spec}((f_0)_\ast(\mathcal{O}_{Y_0}))$ be the Stein factorization of $f_0$, see \cite[Tag 03H2]{StaProj}. 
		Perfecting again, we get a factorization $f=h\circ g$ with $g\colon Y\to Z:=(Z_0)_{\mathrm{perf}}$ having connected geometric fibers, and $h\colon Z\to X$ an integral, dominant morphism of integral schemes inducing an isomorphism at generic points (because $f$ is birational). 
		As $X$ is normal we obtain that $X\cong Z$, which implies the claim.
	\end{proof}
	
	We now turn to Picard groups of perfect schemes. Given any qcqs perfect $k$-scheme $X$, we have $\on{Pic}(X)\cong \on{Pic}(X_0)[1/p]$ for any preferred choice of finite type deperfection $X_0$, cf.\ \cite[Lemma 3.5]{BS17}. In particular, the Picard groups of perfect schemes are always uniquely $p$-divisible.
	
	If $X$ is perfectly finitely presented over some perfect field $k$ and $X_0\to \Spec(k)$ a finitely presented model for $X$, then the localized Weil divisor class group $\on{Cl}(X_0)[1/p]$ only depends on $X$ and not on $X_0$, and we set
	\begin{equation}
	\label{eq:2}
	\on{Cl}(X):=\on{Cl}(X_0)[1/p].
	\end{equation}
	If $X$ is normal, then by \cite[0BE8]{StaProj} (and passage to the limit over Frobenius for some normal model) there exists a natural, injective map
	\begin{equation}
	\label{eq:1}
	\mathrm{Pic}(X)\hookrightarrow \on{Cl}(X).  
	\end{equation}
	
	Let us recall that a line bundle on a (qcqs) scheme is semi-ample if some positive power of it is globally generated.
	
	\begin{proposition}\label{prop_line_bundles_stein_fac}
		Let $X$ be a perfectly proper perfect $k$-scheme and $\calL$ be a semi-ample line bundle on $X$. There is a unique perfectly proper surjection $X \to Y$ of perfect $k$-schemes with connected geometric fibers such that all sufficiently divisible powers of $\calL$ descend uniquely to ample line bundles on $Y$.
	\end{proposition}
	
	\begin{proof}
		By semi-ampleness of $\mathcal{L}$, we can take $X_0$ to be a finite type deperfection of $X$ over $k$, and let $\calL_0$ be a base point free line bundle on $X_0$ whose pullback to $X$ is a power of $\calL$. 
		Without losing generality we may and do assume that $\calL$ is already base point free.
		Let $Y_0$ be the Stein factorization of the canonical morphism
		\begin{equation}
		\label{eq:3}
		X_0 \to Z_0 \subset \bbP(\Gamma(X_0,\calL_0)),
		\end{equation}
		where $Z_0$ is the (scheme-theoretic) image of $X_0$.
		Clearly, $\calL_0$ descends by construction to an ample line bundle on $Y_0$, pulling back $\calO(1)$ on the right side of \eqref{eq:3}. 
		After taking perfections, we get $X\to Y$ with the desired properties (see \cite[Proposition 6.1]{BS17} for unique descent of line bundles).
		
		In order to prove uniqueness of $Y$, we roughly follow the \cite[Proof of Theorem 8.3 in page 382]{BS17}. 
		Using Galois descent, we may assume that $k$ is algebraically closed.
		The morphism $X\to Y$ is a v-cover (by properness), hence $Y$ is determined by the closed subscheme $X\times_Y X\subset X\times_{\Spec(k)} X$. To identify this closed (and necessarily reduced since it is perfect) subscheme it suffices to identify the geometric fibers of the map $X\to Y$ in terms of $\calL$, and 
		we only have to argue on $k$-valued points as these are dense inside $X\times_Y X$ (since $k$ is algebraically closed). 
		We claim (and show below) that two $k$-rational points of $X$ lie in the same fiber over $Y$ if and only if both points can be linked by a chain of closed integral perfect $k$-curves $C$, such that the restriction $\mathcal{L}|_{C}$ is torsion in $\Pic(C)$. 

		Let $y\in Y(k)$ and consider the fiber $X_y\subseteq X$. 
		Let $x_1,\,x_2\in X_y(k)$. 
		That $x_1$ can be linked to $x_2$ by a chain of $k$-curves contained in $X_y$ is a general fact that applies to any connected $k$-variety, so it suffices to show that $\mathcal{L}|_C$ is torsion in $\Pic(C)$ for every curve contained in $X_y$. 
		By hypothesis for any sufficiently divisible $N$ the line bundle $\calL^{\otimes N}$ descends to $Y$.
		Since the map $X_y\to Y$ factors through $\Spec k$, we can deduce that $\calL^{\otimes N}|_C$ is trivial for all sufficiently divisible $N$, which already implies that $\calL|_C$ is itself torsion. 
		
		Conversely, let $x_1,x_2\in X(k)$ and suppose that these points can be linked by a chain of integral perfect closed $k$-curves $C_1,\dots,C_n$ on which $\calL|_{C_i}$ is torsion.
		Let $y_i=f(x_i)$ for $i\in \{1,2\}$.
		We claim (and show below) that $C_j\subseteq X_{y_i}$ for $j\in \{1,\dots,n\}$ and $i\in\{1,2\}$. 
		This already implies that $y_1=y_2$.
		By an induction argument, we can reduce to the case in which $x_i\in C_j$.
		We observe that given an integral perfect closed $k$-curve $C \subset X$ whose image in $Y$ is not a point, satisfies that all sufficiently divisible powers of $\calL$ restrict to an ample line bundle on $C$. 
		Indeed, after passing to a finitely presented deperfection of $C$ over $k$ the morphism $C_0 \to Y$ is finite and pullback of ample line bundles along affine morphisms are ample. 
		Since $\calL|_{C_j}$ is torsion, the map $C_j\to Y$ factors through a point, so $C_j\subseteq X_{y_i}$ as we wanted to show.
	\end{proof}
	
	Next, we discuss finite type deperfections. Let $k$ be a perfect field and let $X$ be a qcqs perfect $k$-scheme of perfectly finite presentation. 
	For each of the finitely many generic points $\eta\in X$, fix a subfield $k(\eta_0)\subset k(\eta)$, which is finitely generated over $k$ and has perfection $k(\eta)$. 
	Then there exists a unique (up to unique isomorphism) weakly normal finite type $k$-scheme $X_0$ such that $(X_0)_{\on{perf}}\cong X$ and for each generic point $\eta_0\in |X_0|\cong |X|$ the function field of $X_0$ at $\eta_0$ identifies with $k(\eta_0)$, see \cite[Proposition A.15]{Zhu17}. 
	Note that $X_0$ also reflects normality of $X$, that is, $X$ normal implies $X_0$ normal.
	
	For group actions we can draw the following consequence.
	
	\begin{proposition}\label{prop_can_deperf_equiv}
		Let $G$ be an affine perfect $k$-group of perfectly finite presentation and $X$ a qcqs perfect $k$-scheme of perfectly finite presentation equipped with a $G$-action with finitely many orbits.
		\begin{enumerate}
			\item Any reduced deperfection $G_0$ of $G$ is a smooth affine $k$-group.
			\item For such $G_0$, there are unique weakly normal deperfections $X_0$ with $G_0$-action, whose fixers on the open orbits are also smooth.
		\end{enumerate} 
	\end{proposition}
	\begin{proof}
		The first item is \cite[Lemma A.26]{Zhu17}. For the second item, we notice that $X$ has a dense open subset $U$ consisting of the disjoint union of its maximal orbits, cf.\ \cite[Proposition A.32]{Zhu17}. Having constructed the unique deperfection $U_0$ with the desired properties, it has a unique extension to a deperfection $X_0$ of $X$ by \cite[Lemma A.15]{Zhu17}. Furthermore, the action map $G\times X \to X$ also deperfects, because it does so over a dense open (and $X_0$ is weakly normal).
		
		Therefore, we may and do assume that $X=G/H$ is a single orbit around a certain $k$-valued point $x$. But then taking $H_0 \subset G_0$ to be the unique reduced closed subscheme whose perfection recovers $H\subset G$, we get a $G_0$-orbit $X_0=G_0/H_0$ deperfecting $X$ with smooth fixers. Uniqueness is clear.
	\end{proof}
	
	\Cref{prop_can_deperf_equiv} will be useful for constructing finite type deperfections for Schubert varieties in Witt vector affine Grassmannians, see \Cref{sec:canon-deperf}.
	
	\subsection{Affine flag varieties}
	\label{sec:affine-flag-vari}
	
	We now study the geometry of Witt vector affine flag varieties. 
	Assume that $k$ is a perfect field of characteristic $p>0$ and that $F$ is a complete discretely valued field with residue field $k$ and ring of integers $O$. 	
	Exceptionally, we allow $F\cong k\rpot{\pi}$ to be a Laurent series field, since it is needed in \Cref{sec:canon-deperf}.
	
	We denote by $\mathrm{Alg}^{\on{perf}}_k$ the category of perfect $k$-algebras. 
	For $R\in \mathrm{Alg}^{\mathrm{perf}}_k$, we denote by $W_{O}(R)$ the associated ring of $O$-Witt vectors, see \cite[Section 1.2.1]{fargues_fontaine_courbes_et_fibres_vectoriels_en_theorie_de_hodge_p_adique}: if $O$ is $p$-adic, then $W_O(R)=W(R)\otimes_{W(k)}O$; if $O\cong k\pot{\pi}$, then $W_O(R)= R\widehat\otimes_k O \cong R\pot{\pi}$.
	
	Moreover, we fix a (connected) reductive $F$-group $G$ and a parahoric $O$-model $\calG$ in the sense of Bruhat--Tits.
	We note that, over the completion $\breve{F}$ of the maximal unramified extension of $F$, the group $G_{\breve{F}}$ is automatically quasi-split by Steinberg's theorem, see \cite[Chapitre III.2.3]{serre_cohomologie_galoisienne}.
	We let $\calG_k=\calG\otimes_O k$ be the special fiber of $\calG$.
	
	Recall the definition of the Witt vector affine flag variety associated to $\calG$.
	
	\begin{definition}
		\label{sec:affine-flag-vari-1-definition-loop-groups-affine-flag-variety}
		
		\begin{enumerate}
			\item The loop group of $G$ is the functor
			\begin{equation}
			L_kG\colon \mathrm{Alg}^{\mathrm{perf}}_k\to (\mathrm{Sets}),\ R\mapsto G(W_O(R)\otimes_{O}F).
			\end{equation}
			\item The positive loop group of $\calG$ is the functor
			\begin{equation}
			L^+_k\calG\colon \mathrm{Alg}^{\mathrm{perf}}_k\to (\mathrm{Sets}),\ R\mapsto \calG(W_O(R))
			\end{equation}
			
			\item The affine flag variety for $\calG$ is the quotient (for the \'etale topology)
			\begin{equation}
			\Fl_{\calG}:=L_kG/L^+_k\calG.
			\end{equation}
		\end{enumerate}
		
	\end{definition}
	
	Because any $\calG$-torsor on $W_O(R)$ can be trivialized over $W_O(R^\prime)$ for some with $R\to R^\prime$ \'etale, the affine flag variety $\Fl_{\calG}$ is equivalently the functor on perfect $k$-algebras $R$ that classifies $\calG$-torsors $\mathcal{P}$ on $\Spec(W_O(R))$ together with a trivialization over $\Spec(W_O(R)\otimes_{O}F)$.
	
	We have the following crucial representability result, see \cite[Corollary 9.6]{BS17}. 
	
	\begin{theorem}[Bhatt-Scholze]\label{quando as grassmannianas sao projectivas}
		The functor $\Fl_{\calG}$ is representable by an ind-(perfectly projective) ind-(perfect $k$-scheme).
	\end{theorem}
	
	Representability as an ind-(perfect algebraic space) was previously proved by Zhu, \cite{Zhu17}, but is not sufficient for our purpose.
	
	Fix an auxiliary maximal split $F$-torus $A$, a maximal $\breve{F}$-split $F$-torus $A \subset S\subset G$ whose connected Néron $O$-model $\calS$ is contained in $\calG$, see \cite[Proposition 5.1.10]{BT84}. 
	Let $T\subset G$ be the centralizer of $S$, and let $\calT$ be the connected Néron $O$-model of $T$. 
	This yields the Iwahori--Weyl group
	\begin{equation}
	\tilde{W}:=N_G(T)(\breve{F})/\calT(\breve{O})
	\end{equation}
	associated with $S$, see \cite[Definition 7]{HR08}. 
	By \cite[Lemma 14]{HR08}, there exists a short exact sequence
	\begin{equation}\label{eq:affine_Weyl_group_sequence}
	1\to W_{\mathrm{af}}\to \tilde{W}\to \pi_1(G)_{I}\to 1 
	\end{equation}
	with $W_{\mathrm{af}}\subset \tilde{W}$ the affine Weyl group, $I$ the absolute Galois group of $\breve{F}$, and $\pi_1(G)$ Borovoi's algebraic fundamental group of $G$. The choice of an alcove in the apartment for $S$ yields a splitting $W_{\mathrm{af}}\rtimes \pi_1(G)_I$ of the sequence. By declaring the elements of $\pi_1(G)_I$ to have length $0$ and to be pairwise incomparable, we can further extend the length function and the Bruhat partial order on the Coxeter group $W_{\mathrm{af}}$ to $\tilde{W}$. 
	
	By the Cartan decomposition, we may identify the double coclasses
	\begin{equation}\calG(\breve{O})\backslash G(\breve{F})/\calG(\breve{O})\cong W_\calG \backslash \tilde{W}/W_\calG
	\end{equation}
	where $W_{\calG}:=(N_G(T)(\breve{F})\cap \calG(\breve{O}))/\calT(\breve{O})$ is the Weyl $O$-group of $\calG$ relative to its maximal $O$-torus $\calS$, see also \cite[Proposition 8]{HR08}. 
	This double coset carries a natural action of the Galois group $\Gal(\breve F/F)$.
	
	\begin{definition}
		Given a finite subset $W\subset W_\calG \backslash \tilde{W}/W_\calG$ with reflex field\footnote{Concretely, the residue field defined by the $\Gal(\breve F/F)$-stabilizer of $W$.} $k_W$, one defines the associated Schubert perfect $k_W$-scheme $\Fl_{\calG,W}\subset \Fl_{\calG}$ as the closure of the Schubert perfect orbit $\Fl_{\calG,W}^\circ$, the étale descent to $k_W$ of the union of the $L^+_{\bar{k}}\calG$-orbits of the maximal elements $w\in W$.
	\end{definition}
	
	If $W=\{w\}$, then these are perfect $k_w$-varieties denoted by $\Fl_{\calG,w}$, respectively $\Fl_{\calG,w}^\circ$, which are usually called the Schubert perfect variety, respectively Schubert perfect orbit associated with $w$.  
	More generally, if we fix an Iwahori $\calI$ dilating $\calG$ and containing $\calS$, then its $L^+_{\bar{k}}\calI$-orbits are enumerated by $\tilde{W}/W_\calG$. 
	Given some finite subset $W \subset \tilde{W}/W_\calG$, we can define in the same manner
	\begin{equation}
	\Fl_{(\calI,\calG),W}^\circ \subset \Fl_{(\calI,\calG),W}
	\end{equation}
	the finite disjoint union of the $L^+_{\bar{k}}\calI$-orbits corresponding to $W$ and their closure inside $\Fl_{\calG}$. 
	The latter is called an Iwahori--Schubert perfect scheme. 
	We observe that Schubert perfect schemes are always Iwahori--Schubert (but the converse is false). 
	Indeed, given $w\in W_\calG\backslash \tilde{W}/W_{\calG}$ with lift $\dot{w}\in \tilde{W}/W_{\calG}$ of maximal length, we have
	\begin{equation}
	\Fl_{(\calI,\calG),\dot{w}}=\Fl_{\calG,w}.
	\end{equation}
	Here we recall that the length function and Bruhat partial order on $\tilde W$ induces one on the cosets $W_\calG \backslash \tilde{W}/W_\calG$, respectively $\tilde W/W_\calG$ compatibly with the dimensions and closure relations of Schubert varieties, respectively Iwahori--Schubert varieties, see \cite[Section 1, Proposition 2.8]{Ric13} for details and proofs in the equicharacteristic situation (the arguments translate literally). 
	
	\begin{proposition}\label{variedades de schubert perfeitas sao normais}
		For each $w\in W_\calG\backslash \tilde{W}/W_\calG$, the Schubert perfect variety $\Fl_{\calG, w}$ is normal and $\Fl_{\calG,w}^\circ$ is a perfectly smooth dense open with connected fixers.
	\end{proposition}
	
	\begin{proof}
		Let $\scrB(G,F)$ be the Bruhat-Tits building of $G$, and let $\bbf\subset \scrB(G,F)$ be the facet associated to $\calG$, see \cite{BT84}. 
		Given $w\in \tilde{W}/W_{\calG}$, the stabilizer of $wL^+_{\bar{k}}\calG\in \Fl_\calG$ is $L^+_{\bar{k}}\calG \cap wL^+_{\bar{k}}\calG w^{-1}$,
		which is the positive loop group associated to the parahoric group scheme, which is the connected fixer of $\bbf \cup w(\bbf)$. In particular, this stabilizer is pro-(perfectly smooth and connected). We deduce that $\Fl^\circ_{\calG,w}$ is perfectly smooth.

		Fix an auxiliary Iwahori $\calI$ dilating $\calG$ and containing $\calS$. 
		This yields the subgroup functor $L^+_k\calI \subset L^+_k\calG$ and, as explained before, we know that $\Fl_{\calG,w}=\Fl_{(\calI,\calG),w^\bbf}$ where $w^\bbf$ is the maximal lift of $w$ to $\tilde{W}/W_\calG$.		
		Let $_\bbf w^\bbf$ be the minimal lift to $\tilde{W}$, write it as $w_{\on{af}}\tau$ with $w_{\on{af}}\in W_{\on{af}}, \tau\in \pi_1(G)_I$, and fix some reduced word $\dot{w}$ in simple reflections (along the alcove defined by $\calI$) for $w_{\on{af}}$. 
		
		Now consider the Demazure variety
		\begin{equation}
		\calD_{\calI, \bar{k},\dot{w}}:= L^+_{\bar{k}}\calP_{i_1} \times^{L^+_{\bar{k}}\calI} \dots \times^{L^+_{\bar{k}}\calI} L^+_{\bar{k}}\calP_{i_n}/{L^+_{\bar{k}}\calI},
		\end{equation}
		where $L^+_{\bar{k}}\calI \subset L^+_{\bar{k}}\calP_{i_j}$ are the minimal parahoric overgroups attached to the simple reflections. 
		It follows by induction that the geometric fibers of the birational resolution (induced by multiplication)
		\begin{equation}
		\pi_{\dot{w}}\colon \calD_{\calI,\bar{k},\dot{w}} \rightarrow \Fl_{\calG, w}
		\end{equation}
		are connected (see the \cite[Proof of Proposition 9.7.d)]{PR08}). 
		As $\calD_{\calI,\bar{k},\dot{w}}$ is perfectly smooth over $\bar{k}$, normality becomes a consequence of \Cref{caracterizacao de variedades normais perfeitas}.
	\end{proof}
	
	The Picard group of Schubert perfect schemes over $\bar{k}$ can be explicitly determined, see \cite[Theorem 3.1]{HZ20} for the case when $\calG=\calI$ is Iwahori and $W=\{w\}$.

	\begin{theorem}[He--Zhou]\label{calculo do grupo de picard da grassmanniana dos vectores de witt}
		The homomorphism
		\begin{equation}
		\on{Pic}(\Fl_{(\calI,\calG),\bar{k},W})\rightarrow \on{Pic}(\Fl_{(\calI,\calG),\bar{k},\bbS_W})\cong \bbZ[p^{-1}]^{\lvert \bbS_W\rvert} 
		\end{equation} 
		is a bijection where $\bbS_W$ is the set of all length $1$ elements in $W\subset \tilde{W}/W_\calG$.
		(Note that $\Fl_{(\calI,\calG),\bar{k},\bbS_W}\cong \mathbb{P}^{1,\mathrm{perf}}_{\bar{k}}$ if $\bbS_W$ is a singleton.)
	\end{theorem}

	\begin{proof}
		To reduce the question to Iwahori--Schubert perfect varieties, we contemplate the Mayer--Vietoris sequence
		\begin{equation}
		1 \rightarrow \calO_{\Fl_{(\calI,\calG),\bar{k},W_0}}^\times \rightarrow \calO_{\Fl_{(\calI,\calG),\bar{k},W_1}}^\times \oplus \calO_{\Fl_{(\calI,\calG),\bar{k},W_2}}^\times \rightarrow \calO_{\Fl_{(\calI,\calG),\bar{k},W_3}}^\times \rightarrow 1
		\end{equation}
		where the subsets $W_i$ are closed for the Bruhat order 
		$W_0=W_1\cup W_2$ and $W_3=W_1 \cap W_2$. 
		Since we may and do assume all these Schubert perfect schemes to be contained in a single connected component of $\Fl_{\calG,\bar{k}}$ (which implies $H^0(\mathcal{O}^\times)\cong \bar{k}^\times$ by perfectly properness), we get a natural isomorphism
		\begin{equation}
		\text{Pic}(\Fl_{(\calI,\calG),\bar{k},W_0})\cong \text{Pic}(\Fl_{(\calI,\calG),\bar{k},W_1}) \times_{\text{Pic}(\Fl_{(\calI,\calG),\bar{k},W_3})} \text{Pic}(\Fl_{(\calI,\calG),\bar{k},W_2}).
		\end{equation}
		By definition $S_0=S_1\cup_{S_3} S_2$, which implies that it suffices to show the claim for $X=\Fl_{(\calI,\calG),\bar{k},w}$ an Iwahori--Schubert perfect variety.
		
		Injectivity can be reduced to Demazure varieties, see \cite[Theorem 6.1]{BS17}. The Demazure varieties $\calD_{\calI,\bar{k},\dot{w}}$ are $\bbP_{\bar k}^{1,\on{perf}}$-fibrations and can be handled directly, see \cite[Proposition 3.4]{HZ20}.
		To treat surjectivity, it suffices to descend certain line bundles on $\calD_{\calI, \bar{k},\dot w}$ back to the Iwahori--Schubert varieties. By \cite[Theorem 6.13]{BS17}, it remains to check that restriction of $\calL$ to geometric fibers is trivial. For this, the argument in \cite[Proposition 3.9]{HZ20} applies, see \cite[Lemma~4.8]{FHLR22}.
	\end{proof}
	
	The choice of a $\bbZ[p^{-1}]$-basis in $\mathrm{Pic}(\Fl_{\calG,\bar{k},W})$ seems arbitrary, due to $p$-divisibility. However, using the deperfection $\mathcal{I}\otimes_{O} \bar{k}$ for the quotient $R\in \mathrm{Alg}^{\mathrm{perf}}_{\bar{k}}\mapsto \mathcal{I}(R)$ of $L^+_{\bar{k}}\calI$, the perfect curve $\Fl_{(\calI,\calG),\bar{k},\bbS_W}$ has a canonical equivariant deperfection, see \Cref{prop_can_deperf_equiv}, yielding a natural $\bbZ$-lattice. 
	
	\begin{remark}
		\label{sec:affine-flag-vari-1-constructing-schubert-varieties-via-demazure}
		During the proof, we have also determined the Picard group of the Demazure varieties $\calD_{\calI,\bar{k},\dot{w}}$, or more generally those of the convolutions \begin{equation}\Fl_{\calI,\bar{k},W_1}\tilde{\times} \dots \tilde{\times} \Fl_{\calI,\bar{k},W_{n-1}}\tilde{\times}\Fl_{(\calI,\calG),\bar{k}, W_n} \end{equation} of Iwahori--Schubert perfect schemes, where at most the last one is not at full level.
		
		Together with \Cref{prop_line_bundles_stein_fac}, this tells us how to recover, for instance, the perfect Schubert variety $\Fl_{\calG,\bar{k},w}$ just from its Demazure resolution and the sub-$\Z[p^{-1}]$-module $\Pic(\Fl_{\calG,\bar{k},w})\subset \Pic(\calD_{\calI,\bar{k},\dot{w}})$: take any $\mathcal{L}$ on $\calD_{\calI,\bar{k},\dot{w}}$ which is the pullback of a line bundle on $\Fl_{\calG,\bar{k},w}$ whose restriction to $\Fl_{\calG,\bar{k},s}$ has positive degree for each $s\in \mathbb{S}_W$. 
	\end{remark}

	We now turn to equivariant automorphisms of (connected) Schubert schemes.
	
	\begin{proposition}\label{prop_trivial_equiv_autos}
		The group of $L^+_{\bar{k}}\calG$-equivariant automorphisms of a connected Schubert perfect scheme $\Fl_{\calG,\bar{k},W}$ is trivial. 
		In particular, the stabilizers of $\bar{k}$-valued points in $\Fl_{\calG, \bar{k},W}$ for the $L^+_{\bar{k}}\calG$-action are self-normalizing subgroups of $L^+_{\bar{k}}\calG$, that is, they agree with their normalizers.
	\end{proposition}
	
	\begin{proof}
		We prove the more general statement for Iwahori--Schubert perfect schemes. Consider the disjoint irreducible components in the dense open $\Fl_{(\calI,\calG),\bar{k},W}^\circ\subset \Fl_{(\calI,\calG),\bar{k},W}$.
		These will be permuted under any equivariant automorphism $\sigma$. 
		Moreover, $\sigma$ preserves the $\bar{k}$-valued points of $\Fl_{(\calI,\calG),\bar{k},W}$ fixed under $\calS(\breve{O})$. For the entire flag variety, we claim that the $\calS(\breve{O})$-fixed points in $G(\breve{F})/\calG(\breve{O})$ lie in the image of $N(\breve{F})$. Indeed, let $[g]\in G(\breve{F})/\calG(\breve{O})$ be a fixed point. 
		Then $g\bbf$ is a $\calS(\breve{O})$-stable facet (with $\bbf$ the facet determined by $\calG$), hence contained in $\scrA(G,S,\breve{F})$ by \cite[Proposition 5.1.37]{BT84}. 
		Multiplying on the left by a suitable element of $N(\breve{F})$, we can trivialize $[g]$, that is, $[g]\in N(\breve{F})\calG(\breve{O})/\calG(\breve{O})$.
		
		Now, observe that the $L^+_{\bar{k}}\calI$-fixer of some $w\in \tilde{W}/W_\calG$ equals $L^+_{\bar{k}}\calI \cap wL^+_{\bar k}\calG w^{-1}$, and it suffices to recover $w$ from this subgroup alone. 
		Indeed, then $\sigma$ must preserve $w$, and then $\Fl_{(\calI,\calG),\bar{k},W}$ pointwise by $L^+_{\bar{k}}\calI$-equivariance as $w\in \tilde{W}/W_\calG\cap \Fl_{(\calI,\calG),\bar{k},W}$ was arbitrary. 
		If $\bba$ is the alcove fixed by $\calI(\breve{O})$ and $\tilde{w} \in \tilde{W}$ the minimal lift of $w$, then birationality of the Demazure resolution $\pi_{\dot w}$ implies
		\begin{equation}
			\label{weneedthisequation}
		L^+_{\bar{k}}\calI \cap wL^+_{\bar{k}}\calG w^{-1}=L^+_{\bar{k}}\calI \cap \tilde{w}L^+_{\bar{k}}\calI {\tilde{w}}^{-1}.
		\end{equation} 
		Note that the right side is the Bruhat--Tits group attached to $\bba \cup \tilde{w}(\bba)$ (see \cite[Lemma 3.3]{HR19} for a proof and an alternative formulation of \eqref{weneedthisequation}). 
		We need to recover $\tilde{w}$. Moreover, because $\Fl_{(\calI,\calG),\bar{k},W}$ was assumed to be connected, all $\tilde{w}$ considered here project to the same constant $\tau \in \pi_1(G)$, so it is enough to get  $w_{\on{af}}\in W_{\mathrm{af}}$ if $\tilde{w}=w_{\mathrm{af}}\tau$.
		
		By \cite[Corollaire 5.1.39]{BT84}, the fixed point set of $L^+_{\bar{k}}\calI \cap \tilde{w}L^+_{\bar{k}}\calI {\tilde{w}}^{-1}$ inside $\scrB(G,F)$ equals the closed convex hull of $\bba \cup \tilde{w}(\bba)$. In turn, every alcove inside this closed convex hull lies in some minimal gallery connecting $\bba$ to $\tilde{w}(\bba)$ by \cite[Lemme 2.4.4]{BT72}. Since a minimal gallery describes a unique word of simple reflections necessary to move from one alcove to another, this gives back the affine transformation $w_{\on{af}}$.
	\end{proof}

	Among Schubert schemes, we are especially interested in the $\mu$-admissible locus. 
	Recall that $C$ is a completed algebraic closure of $F$ and that $I$ denotes the inertia group of $F$.
	Moreover, let $B\subset G_{\breve{F}}$ be a Borel containing $T_{\breve{F}}$.
	Recall that the inverse of the Kottwitz morphism \cite[Equation (7.2.1)]{Kot97} induces an isomorphism of the coinvariants
	\begin{equation}\label{eq:coinvariants.Kottwitz.iso}
	X_\ast(T)_I\cong T(\breve{F})/\calT(\breve{O}),\ \nu_I\mapsto \nu_I(\pi),
	\end{equation}
	not depending on the choice of uniformizer $\pi\in O$, under which we may regard the former as the subgroup of $\tilde{W}$ acting by translation on the standard apartment, see also \cite[Proposition 13]{HR08}.
	
	\begin{definition}\label{admissible_locus_definition}
		Let $\mu$ be a geometric conjugacy class of cocharacters with reflex field $E$. 
		The $\mu$-admissible locus is the Schubert perfect $k_E$-scheme 
		\begin{equation}
		\calA_{\calG,\mu}= \Fl_{\calG, \{\la_I(\pi)\}},
		\end{equation} 
		where $\la \in X_*(T)$ runs over all representatives of $\mu$ and $\lambda_I\in X_\ast(T)_I$ denotes the associated coinvariant under $I$.	
	\end{definition}
	
	Note that $\calA_{\calG,\mu}$ is geometrically connected because the finite Weyl group acts trivially on $\pi_1(G)_I$.
	It does not depend on the choice of $T$.
	By a result of Haines \cite[Theorem 4.2]{Hai18}, we know that $\calA_{\calG,\mu}^\circ:=\Fl_{\calG, \{\la_I(\pi)\}}^\circ$, where $\la$ now runs over the rational conjugates of $\mu$, that is, all those which are contained in a closed Weyl chamber attached to $w_0Bw_0^{-1}$ for some $w_0 \in W_0$, the finite Weyl group of $G_{\breve{F}}$ with respect to $S_{\breve{F}}$.
	
	It will turn out that $\calA_{\calG, \bar{k},\mu}$ is functorial in $(\calG,\mu)$, as soon as we develop a theory of local models $\calM_{\calG, \mu}$, see \Cref{defn_local_model}, and calculate their special fibers, confer \Cref{theorem_special_fiber_admissible}. 
	The admissible locus also admits the following representation-theoretic interpretation in terms of representations of the Langlands dual group $ \widehat G$ (here, taken over any algebraically closed field) with dual torus $\widehat T$:
	
	\begin{lemma}\label{representation_theoretic_admissible_locus_lemma}
		Let $\widehat{\lambda}^I\in X^\ast(\widehat{T}^I)\cong X_\ast(T)_I$ be running through the set of restrictions of all weights $\widehat{\lambda}$ for $\widehat{T}$ occurring in a finite dimensional algebraic representation of $\widehat{G}$ with fixed highest weight $\widehat{\mu}=\mu$. 
		Then $\calA_{\calG,\mu}=\Fl_{\calG,\{\widehat{\lambda}^I(\pi)\}}$.
	\end{lemma}
	\begin{proof}
		Being a $\widehat G$-representation, $V$ contains all the weights $\widehat \la$ conjugate to $\widehat \mu$ under the absolute Weyl group with the same non-zero multiplicity.  
		Under $X^\ast(\widehat{T})\cong X_\ast(T)$, these correspond to the conjugates of $\mu$ compatibly with the projection to $X^\ast(\widehat{T}^I)\cong X_\ast(T)_I$.
		Hence, the lemma follows from the definition of the admissible locus. 
	\end{proof}
	
	\begin{example}
		The basic example of the admissible locus occurs for $G=\GL_2$, $\mu=(1,0)$ and $\calG=\calI$ an Iwahori. 
		In this case, $\calA_{\calG,\mu}$ is the union of two copies of $\bbP^{1,\on{perf}}_k$ intersecting transversally at a point. 
		More generally, one can enumerate the Iwahori--Schubert orbits of the translated to the neutral component admissible locus $\calA_{\calI,\mu}$ in terms of alcoves in the standard apartment $\scrA(G,S,F)$. 
		For pictures in the case of unitary groups of split rank $2$, the reader is referred to the introduction of \cite{PR09}. 
		For further examples, see the survey \cite{PRS10}.
	\end{example}
	
	\subsection{Canonical deperfections}
	\label{sec:canon-deperf}
	
	Now, we wish to introduce equivariant deperfections of the Schubert perfect schemes $\Fl_{\calG,W}$ following \Cref{prop_can_deperf_equiv} and discuss their geometry, at least for certain $W$. We are especially interested in admissible loci $\calA_{\calG,\mu}$ for $\mu$ minuscule.
	
	First, recall that the congruence quotient $L^{\leq n}_k\calG$ of $L^+_k\calG$ has a deperfection $\mathrm{Gr}_n\calG$,  given by $(n+1)$-truncated Witt vectors and which is called the Greenberg realization.
	We denote by $L^{>n}_{k}\calG$ the kernel of $L^+_k\calG\to L^{\leq n}_k\calG$.
	
	\begin{definition}\label{canonical_deperfection_definition}
		Let $n$ be the smallest nonnegative integer such that $L^{>n}_{\bar{k}}\calG$ acts trivially on $\Fl_{\calG, \bar{k},W}$ and call it the associated depth. 
		The canonical deperfection\footnote{In the equicharacteristic situation, one recovers the weak normalization of classical Schubert schemes, which turn out to be the classical ones under \Cref{hyp_wild_odd_unitary} and also $p\nmid \lvert \pi_1(G_\der)\rvert$, see \cite[Section~4.1]{FHLR22} and \cite{HLR18}.} $\Fl_{\calG, W}^{\on{can}}$ of the perfect Schubert scheme $\Fl_{\calG, W}$ is the finite type $k_W$-scheme with $\mathrm{Gr}_n\calG$-action determined by \Cref{prop_can_deperf_equiv}.
	\end{definition}
	
	Assume the $L^+_{\bar{k}}\calG$-action on $\Fl_{\calG,\bar{k}, W}$ factors through $L^0_{\bar{k}}\calG=\calG_{\bar{k}}^{\mathrm{perf}}$. 
	For $V \leq W$, we get a deperfection 
	\begin{equation}
	\Fl_{\calG, \bar{k},V}^{\on{can}}\to \Fl_{\calG,\bar{k}, W}^{\on{can}}
	\end{equation}
	of the closed immersion of perfect Schubert schemes, because the image is a finite type deperfection with smaller function fields, as it carries a $\calG_{\bar{k}}$-action.
	
	However, it is not clear that the finite type morphism is a closed immersion.	
	To know more about the geometry of $\Fl_{\calG, W}^{\on{can}}$, we exploit the picture in the equicharacteristic situation.
	
	Assume $G$ is adjoint, and also \Cref{hyp_wild_odd_unitary} for $G$ over $\breve F$, that is, if $p=2$, then $G$ has no odd unitary factors over $\breve{F}$.
	Then, for every parahoric $\breve{O}$-group $\calG$ attached to a facet in $\scrA(G,S,\breve{F})$, we find smooth, affine, fiberwise connected $\breve{O}\pot{t}$-lifts $\underline{\calG}$ in the sense of \cite[Proposition 2.8]{FHLR22}. Note that the $\bar{k}\pot{t}$-reductions $\calG'$ are parahoric models of some adjoint connected reductive $\bar{k}\rpot{t}$-group $G'$ attached to a facet in some appartment $\scrA(G',S',\bar{k}\rpot{t}) \cong \scrA(G,S,F)$, see \cite[Lemma 2.7]{FHLR22}.

	In particular, these come with isomorphisms
	\begin{equation}
	\calG \otimes_{\breve{O}} \bar{k} \cong \calG' \otimes_{\bar{k}\pot{t}} \bar{k},
	\end{equation}
	that are functorial as we vary $\calG$ among parahoric models attached to a facet in $\scrA(G,S,\breve{F})$, and
	which we now exploit to compare their Schubert schemes.
	Let us note that the loop groups $L^+_{\bar{k}}\calG$ and its analogue $L^+_{\bar{k}}(\calG^\prime)$ in the equicharacteristic setting admit natural surjections on $\calG_{\bar{k}}^{\on{perf}}\cong \calG^{\prime,\mathrm{perf}}_{\bar{k}}.$
	Below, we use subscripts $(\str)'$ to denote the perfections of the equicharacteristic loop groups and Schubert varieties for $\calG'$.
	\begin{lemma}\label{lem_comparison_sch_vars_equi_and_mixed}
		Under the above constraints, there are unique equivariant isomorphisms
		\begin{equation}
		\Fl^{\on{can}}_{\calG,\bar{k}, W} \cong \Fl^{\on{can}}_{\calG',\bar{k},W'}
		\end{equation}
		for all connected $\Fl_{\calG,\bar{k},W}$ of depth $0$, that is, whose $L^+_{\bar{k}}\calG$-action factors through $\calG_{\bar{k}}^{\on{perf}}$.
	\end{lemma}
	
	\begin{proof}
		As $\calG_{\bar{k}}\cong \calG_{\bar{k}}^\prime$, it suffices by \Cref{prop_can_deperf_equiv} to produce equivariant isomorphism $\Fl_{\calG,\bar{k}, W} \cong \Fl_{\calG',\bar{k},W'}$ of the perfect Schubert schemes.
		During the proof, we fix an auxiliary Iwahori $\calI$ dilated from $\calG$ and consider the corresponding Iwahori--Schubert perfect scheme $\Fl_{(\calI,\calG),\bar{k},W}$. 
		
		First assume that $W=\{w\}$.
		The perfect variety $\Fl_{(\calI,\calG),\bar{k},w}$ can be resolved via a Demazure variety $\calD_{\calI, \bar{k},\dot{w}}$. 
		If $s$ is the first letter of the word $\dot{w}$ and $\dot{v}$ is the word obtained from deleting the first letter, we get		
		\begin{equation}
		\calD_{\calI,\bar{k},\dot{w}}=\Fl_{\calI,\bar{k},s} \tilde{\times}\calD_{\calI,\bar{k},\dot{v}}
		\end{equation}
		where $L^+_{\bar{k}}\calI \subset L^+_{\bar{k}}\calP$ is the minimal parahoric corresponding to $s$. We claim that the action of $L^+_{\bar{k}}\calI$ on $\calD_{\calI,\bar{k},\dot{v}}$ is trivial when restricted to the normal subgroup $L^{\geq 1}_{\bar k}\calP$.
		Otherwise, let $\alpha$ be the negative simple affine root corresponding to $s$ and observe that $L^+_{\bar{k}}\calU_{\alpha+1}$ acts non-trivially on $\calD_{\calI, \bar{k},\dot{v}}$. 
		But conjugating by $s$ yields that $L^+_{\bar{k}}\calU_{-\alpha+1} \subset L^{\geq 1}_{\bar{k}}\calI$ does not act trivially on $\calD_{\calI, \bar{k},\dot{w}}$.

		Arguing inductively on $\dot{w}$, and exploting the above claim, we reach at an $\calI_{\bar{k}}^{\on{perf}}$-equivariant identification of the Demazure perfect varieties 
		\begin{equation}
		\calD_{\calI, \bar{k},\dot{w}} \cong \calD_{\calI', \bar{k},\dot{w}'}
		\end{equation}
		bounded by $\dot{w}$ resp. $\dot{w}'$ and attached to $\calI$, respectively the $\bar{k}\rpot{t}$-reduction $\calI'$ of the Iwahori $\breve{O}\pot{t}$-lift. In the case $l(\dot{w})=1$, then we get the unique equivariant identification of one-dimensional Iwahori--Schubert perfect varieties, which are perfected projective lines.
		
		As the Picard group of the Demazure varieties have already been determined, see \Cref{calculo do grupo de picard da grassmanniana dos vectores de witt}, respectively \cite[Section 3.2]{HZ20} for the equicharacteristic case, the previous isomorphism descends uniquely by \Cref{prop_line_bundles_stein_fac} to an $\calI_{\bar{k}}^{\on{perf}}$-equivariant identification
		\begin{equation}
		\Fl_{(\calI,\calG),\bar{k}, w} \cong \Fl_{(\calI',\calG'),\bar{k},w'}
		\end{equation} 
		of the perfect Schubert varieties, see \Cref{sec:affine-flag-vari-1-constructing-schubert-varieties-via-demazure}. 
		If the left side is stable under $\calG_{\bar{k}}$, then we need to show the map is not only $\calI_{\bar{k}}^{\on{perf}}$-equivariant, but furthermore $\calG_{\bar{k}}^{\on{perf}}$-equivariant. 
		
		Let $\bar Q\subset \calG_{\bar{k}}$ denote the image of $\calI_{\bar{k}}$. By assumption, the $\calI_{\bar{k}}^{\on{perf}}$-action on both sides factors through the perfection of $\bar{Q}$. Using the convolution product
		\begin{equation}
		\calG_{V}^{\on{perf}}\times^{\bar Q^{\on{perf}}}\Fl_{\calG, \bar{k},w}, 
		\end{equation}
		we get a perfect $\bar{k}$-variety mapping $\calI_{\bar{k}}^{\on{perf}}$-equivariantly to $\Fl_{\calG, \bar{k},w}$ and that can be identified with its equicharacteristic analogue in a $\calG_{\bar{k}}^{\on{perf}}$-equivariant fashion. 
		Since $\calG_{\bar{k}}^{\on{perf}}/{\bar Q^{\on{perf}}} \subset \Fl_{\calI,\bar{k}}$ is a Schubert perfect variety at Iwahori level, we know its Picard group by \Cref{calculo do grupo de picard da grassmanniana dos vectores de witt} and \Cref{sec:affine-flag-vari-1-constructing-schubert-varieties-via-demazure}. 
		Applying again \Cref{prop_line_bundles_stein_fac}, we not only recover the original isomorphism, by \Cref{prop_trivial_equiv_autos}, but also conclude it is $\calG_{\bar{k}}^{\on{perf}}$-equivariant and the unique such map.
		
		For a general $W$ as in the statement, we now can glue the above isomorphism to non-irreducible Schubert perfect schemes
		\begin{equation}
		\Fl_{\calG,\bar{k}, W} \cong \Fl_{\calG',\bar{k},W'}
		\end{equation}
		appealing again to \Cref{prop_trivial_equiv_autos}.
	\end{proof}

	From now on $G$ is no longer assumed to be adjoint. 
	We approach the canonical admissible locus $\calA_{\calG,\mu}^{\on{can}}$ for minuscule $\mu$, that is, the canonical deperfection of the admissible locus, with our comparison result, describing its singularities (thereby confirming \cite[Conjecture III]{Zhu17} for Schubert varieties in the admissible locus) and computing its coherent cohomology.
	
	\begin{theorem}\label{theorem_coherence_allp}
		Let $\mu$ be minuscule and assume \Cref{hyp_wild_odd_unitary}. 
		Then, $\calA_{\calG,\mu}^{\on{can}}$ is Cohen--Macaulay and Frobenius split compatibly with its $\calG_{\bar{k}}$-stable reduced $\bar{k}$-subschemes.
		
		Moreover, for every ample line bundle $\calL$ on $\Fl_{\calG,\bar{k}}$ that descends to $\calA_{\calG,\bar{k},\mu}^{\on{can}}$, there is an equality
		\begin{equation}
		\dim_{\bar{k}} \on{H}^0(\calA_{\calG,\bar{k},\mu}^{\on{can}}, \calL)=\dim_C \on{H}^0(\calF_{G,C,\mu}, \calO(c_\calL)).
		\end{equation}
	\end{theorem}
	
	Here, $\calF_{G,\mu}=G_E/P_\mu^-$ is the classical flag variety attached to $\mu$, the central charge $c_\calL$ is given by Kac--Moody coefficients, see \cite[Section~10]{PR08} and \cite[Section10]{BS17}, and the line bundle $\calO(c_\calL)$ is the corresponding power of the ample generator of $\on{Pic}(\calF_{G,C,\mu})$.

	\begin{proof}
		We want to apply \Cref{lem_comparison_sch_vars_equi_and_mixed}, in order to reduce the statements to the equicharacteristic situation, where we refer to \cite[Theorem~4.1, Theorem~4.25]{FHLR22}. 
		
		In order to do this, we first notice that there is an equivariant isomorphism $\calA_{\calG,\bar{k}, \mu}^{\mathrm{can}} \cong \calA_{\calG_\ad,\bar{k}, \mu_\ad}^{\mathrm{can}}$ via the natural map. Here, $\mu_{\ad}$ denotes the composition of $\mu$ with $G_C\to {G_\ad}_C$.
		Indeed, this can be checked on perfections and then at the level of geometric points, where it follows from the assertion that $\Fl^{\on{perf}}_{\calG,\bar{k}}\to \Fl^{\on{perf}}_{\calG_{\mathrm{ad}},\bar{k}}$ induces isomorphisms on connected components. 
		More precisely, 
		\begin{center}
		\begin{tikzcd}
		 \Fl^{\on{perf}}_{\calG,\bar{k}} \arrow{r} \arrow{d}  & \Fl^{\on{perf}}_{\calG_{\mathrm{ad}},\bar{k}} \arrow{d} \\
		 \pi_0(\Fl^{\on{perf}}_{\calG,\bar{k}}) \arrow{r} & \pi_0(\Fl^{\on{perf}}_{\calG_{\mathrm{ad}},\bar{k}}) 
		\end{tikzcd}
		\end{center}
		is Cartesian.
		
		We still have to show that $\calA_{\calG,\bar{k}, \mu}^{\mathrm{can}}$ has minimal depth, that is, $L^+_{\bar{k}}\calG$ acts via $\calG_{\bar{k}}$. 
		Since $L^{\geq 1}_{\bar{k}}\calG$ is a normal subgroup, it suffices to check that it fixes each of the sections $\la$ defining the admissible locus. 
		By the combinatorial dictionary, see our proof of \Cref{prop_trivial_equiv_autos}, it suffices to show that $\lvert a(\la_I) \rvert \leq 1$, that is, the translation $\la_I$ moves every affine root to a parallel one at distance at most one. 
		By definition, one has 
		\begin{equation}
		a(\la_I)= [K:\breve{F}]^{-1}\sum_{\sig \in \Gal(K/\breve{F})} \sig \tilde{a} (\la),
		\end{equation}
		where $K$ is a finite Galois extension of $\breve{F}$ splitting $G_{\breve{F}}$, and $\tilde{a}$ is an absolute root restricting to $a$, so its absolute value is at most $1$, since $\la$ is minuscule.  
	\end{proof}
	
	\begin{remark}\label{remark_cass-lou}
		After the first version of this paper was written, \Cref{hyp_wild_odd_unitary} was removed from \Cref{theorem_coherence_allp} in the equicharacteristic setting in \cite{Lou23b}. 
		These results are extended to the mixed characteristic setting in \cite[Corollary~1.5]{CL24}.
		We remark that Cohen--Macaulayness would follow from a positive solution of \cite[Conjecture~3.6]{FHLR22}.
		\end{remark}
	
	\section{Affine Grassmannians and local models}
	\label{section_local_model}
	
	In this section, we start by gathering several basic facts on the $B_\dR^+$-affine Grassmannian over $\Spd C$. 
	Most of them were established in \cite[Lecture XIX]{SW20} and \cite[Chapters VI.2, VI.5]{FS21}, but our approach is sufficiently different and relevant to later sections that it merits some elaboration.
	
	Then, we introduce the main objects of study of this article, to wit the local models
	\begin{equation}
	\calM_{\calG, \mu} \subset \Gr_{\calG,O_E}
	\end{equation}
	defined for every $\mu\in X_\ast(T)$ via v-closures of Schubert diamonds in a Beilinson--Drinfeld Grassmannian. We dedicate the rest of the section to showing that $\calM_{\calG, \mu}$ is an $L^+_{O_E}\calG$-stable flat, proper $\pi$-adic kimberlite with good finiteness properties. 
	In particular, its special fiber will be shown to be representable by some connected Schubert perfect scheme $\Fl_{\calG, W}$.

	\subsection{The $B_\dR^+$-affine Grassmannian}\label{subsection_affine_grassmannian}

	In this section, we fix a complete discretely valued field $F/\bbQ_p$ with perfect residue field $k$, ring of integers $O$ and uniformizer $\pi$, a complete algebraic closure $C/F$ with ring of integers $O_C$ and residue field $\bar k= k_C$. 
	We denote by $\breve F\subset C$ the maximal unramified complete subextension with ring of integers $\breve O$ and the same residue field $\bar k=k_{\breve F}$.
	Further, we fix a (connected) reductive $F$-group $G$ and a maximal $\breve{F}$-split $F$-torus $S\subset G$ containing a maximal $F$-split torus, see \cite[Proposition 5.1.10]{BT84}.
	As $G$ is quasi-split over $\breve{F}$ by Steinberg's theorem, the centralizer $T$ of $S$ is a maximal torus. 
	Also, we fix a Borel subgroup $B\subset G_{\breve F}$ containing $T_{\breve F}$.
	
	For any affinoid perfectoid space $\Spa(R,R^+)$ in characteristic $p$ equipped with a map to $\Spd \bbZ_p$, let $B_{\dR}^{+}(R^\sharp)$, respectively $B_{\dR}(R^\sharp)$, be the rings of de Rham periods formed using $O$-Witt vectors.
	For convenience, we set $B^{+}_{\dR}:=B^{+}_{\dR}(C)$ and $B_{\dR}:=B_{\dR}(C)$.
	The $B^+_{\mathrm{dR}}$-loop group of $G$ is the group functor over $\Spd F$ given by
	\begin{equation}
	LG\colon (R,R^+)\mapsto G(B_{\dR}(R^\sharp)),
	\end{equation}
	and the positive loop group is the subgroup functor
	\begin{equation}
	L^+G\colon (R,R^+)\mapsto G(B_{\dR}^+(R^\sharp)). 
	\end{equation}
	Their v-sheaf quotient
	\begin{equation}
	\Gr_{G}:=LG/L^+G
	\end{equation}
	is called the $B^+_\dR$-affine Grassmannian. 
	Similarly to \Cref{sec:affine-flag-vari}, $\Gr_{G}(R,R^+)$ parametrizes $G$-torsors on the spectrum $\Spec(B^+_{\dR}(R^\sharp))$ with a trivialization over $\Spec(B_\dR(R^\sharp))$.
	Here, we are primarily interested in the geometry and work therefore over $\Spd C$.
	The base changes are denoted by $L_CG$, $L_C^+G$ and $\Gr_{G,C}$, for convenience. 
	
	As an auxiliary first step, we study the affine flag variety and then translate the results to the affine Grassmannian.
	For this, the Iwahori group $B^+_\dR$-model $\mathcal{I}$ is given as the dilatation of $G \otimes_F B^+_\dR$ along the subscheme $B_C \subset G_C$ of its special fiber. 
	Define
	\begin{equation}
	L^+_C\mathcal{I}\colon (R,R^+)\mapsto \mathcal{I}(B_{\dR}^+(R^\sharp))
	\end{equation}
	which is a subgroup v-sheaf of $L^+G$. 
	It gives rise to the $B^+_\dR$-affine flag variety
	\begin{equation}
	\Fl_{\mathcal{I},C}:=L_CG/L^+_C\mathcal{I},
	\end{equation}
	viewed as a v-sheaf over $\Spd C$.
	
	We recall that $\Gr_{G,C}\to \Spd C$ is an increasing union of proper, spatial diamonds by \cite[Lecture XIX]{SW20}. 
	The same holds for $\Fl_{\calI,C}\to \Spd C$, as the projection
	\begin{equation}\label{projection.from.affine.flag.equality}
	\Fl_{\calI,C}\to \Gr_{G,C}
	\end{equation}
	is a proper, cohomologically smooth $(G_C/B_C)^\diamondsuit$-fibration.
	The following discussion is parallel to parts of \Cref{sec:affine-flag-vari} but simplified by the fact that we consider $G\otimes_F B^+_\dR$ which is a (split) reductive group over $B^+_\dR$ (and not some parahoric group scheme).
	The geometry of the affine flag variety $\Fl_{\calI,C}$ or, better, the v-stack quotient
	\begin{equation}
	\Hk_{\calI,C}:=L^+_C\mathcal{I}\backslash\Fl_{\calI,C}= L^+_C\mathcal{I}\backslash L_CG/L^+_C\mathcal{I}
	\end{equation}
	is reflected in the Iwahori-Weyl group of $G(B_\dR)$,
	\begin{equation}
	\tilde{W}_\dR:=N_G(T)(B_\dR)/T(B_\dR^+),
	\end{equation}
	where $N_G(T)$ denotes the normalizer of $T$ in $G$.
	There is a canonical map $\underline{\tilde W_\dR} \r \Fl_{\calI,C}$ because $T(B^+_\dR)\subset \mathcal{I}(B^+_\dR)$.
	
	\begin{lemma}
		\label{sec:geom-affine-grassm-topological-space-of-affine-flag-stack}
		The map $\underline{\tilde{W}_\dR}\to \Fl_{\calI,C}$ induces a bijection
		\begin{equation}
		\tilde{W}_\dR\cong |\Hk_{\calI,C}|.
		\end{equation}
	\end{lemma}
	\begin{proof}
		Every point of $|\Hk_{\calI,C}|$ is represented by a map $\Spa(K,K^+)\to L_CG$ with $K$ algebraically closed perfectoid. 
		Two $K$-valued points have the same underlying element in $|\Hk_{\calI,C}|$ if, v-locally, they lie in the same double coset 
		\begin{equation}
		\mathcal{I}(B^+_\dR(K^\sharp))\backslash G(B_\dR(K^\sharp))/ \mathcal{I}(B^+_\dR(K^\sharp)). 
		\end{equation}
		The identification now follows from the Bruhat decomposition which is independent of $K^\sharp$.
	\end{proof}
	
	Let $\bba \subset \scrA(G,T,B_\dR)$ be the alcove defined by $\mathcal{I}$ in the apartment for $T$ of the Bruhat--Tits building of $G(B_\dR)$. 
	Let $\mathbb{S}\subset \tilde{W}_\dR$ be the set of simple reflections along the walls bounding $\bba$. 
	The affine Weyl group $W_{\dR,\aff}\subset \tilde{W}_\dR$ is the subgroup generated by the elements in $\mathbb{S}$. 
	Then, $W_{\aff}$ is a Coxeter group which only depends on the Bruhat--Tits building of $G(B_\dR)$.
	As in \eqref{eq:affine_Weyl_group_sequence} there is a canonical short exact sequence
	\begin{equation}
	1\to W_{\dR,\aff}\to \tilde{W}_\dR\to \pi_1(G)\to 1,
	\end{equation}
	which is naturally split by taking the stabilizer $\Omega_{\bba}\subset \tilde{W}_\dR$ of the alcove $\bba$.
	Thus, we can write each $w\in \tilde{W}_\dR$ uniquely as $w=w_{\aff}\tau$ with $\tau\in \Omega_\bba$ and $w_\aff\in W_{\dR,\aff}$.
	
	\begin{lemma}
		\label{sec:geom-affine-grassm-translation-by-element-in-the-normalizer}
		Equip $\pi_1(G)$ with the discrete topology. 
		The morphism
		\begin{equation}
		|\Hk_{\calI,C}|\to \pi_1(G)
		\end{equation}
		is locally constant, thus underlies a morphism $\Hk_{\calI,C}\to \underline{\pi_1(G)}$ of small v-stacks. 
	\end{lemma}
	\begin{proof}
		Here, we follow the argument behind the proof of \cite[Theorem 5.1]{PR08}. 
		If $G=G_\spc$ is simply connected, then we see that, for every algebraically closed perfectoid field $K/C$, the group $L_CG(K)$ is generated by its affine root subgroups $L_CU_a(K)$.
		
		But, since $L_CU_a$ is connected (choose a pinning), we conclude that $L_CG$, hence also $\Fl_{\calI,C}$ and $\Hk_{\calI,C}$ are connected. 
		If $G=T$ is a torus, then we see easily that $\Gr_{T,C}$ equals the v-sheaf $\underline{X_\ast(T)}$ compatibly with the map above. 
		
		Now, suppose that $G_\der=G_\spc$. Then, $\pi_1(G)$ identifies with the fundamental group of the abelian quotient $G/G_\der$, so the claim is clear. 
		Finally, for a general group $G$, we can find an exact sequence, 
		\begin{equation}
		1 \to T \to \tilde{G} \to G \to 1,
		\end{equation}
		such that $\tilde{G}_\der=G_\spc$ and with $T$ an induced torus (such sequence exists by \cite[Proposition 3.1]{deligne1982conjugates}).
		We obtain a commutative diagram
		\begin{center}
		\begin{tikzcd}
			\mid \Hk_{\tilde{G},C} \mid \arrow{r} \arrow{d}  & \mid \Hk_{{G},C} \mid  \arrow{d} \\
			\pi_1(\tilde{G}) \arrow{r} & \pi_1({G}) 
		\end{tikzcd}
		\end{center}
where we can see that the two horizontal arrows are surjective by using the fact that $T_\spc$ is an induced torus.
Moreover, since	$|\Hk_{\tilde{G},C}| \to |\Hk_{{G},C}|$ is a quotient map, the claim follows.
	\end{proof}
	
	For $\tau\in\pi_1(G)$, we denote by $\Fl_{\calI,C}^\tau$ the fiber over $\tau$ of the morphism $\Fl_{\calI,C}\to \underline{\pi_1(G)}$.
	We note that right translation by a representative of $\tau$ in $L_CG$ induces an isomorphism
	\begin{equation}
	\Fl_{\calI,C}^1\overset \cong  \longto \Fl_{\calI,C}^\tau.
	\end{equation}
	Moreover, $\Fl_{\calI,C}^1$ is canonically isomorphic to the affine flag variety $\Fl_{\calI_\spc,C}$ of the simply connected cover $G_{\spc}$. 
	Namely, the transition morphism $\Fl_{\calI_\spc,C}\r \Fl_{\calI,C}^1$ is bijective by checking on geometric points (\cite[Lemma 12.5]{Sch17}) and using the Bruhat decomposition, hence must be an isomorphism as both $\Fl_{\calI_\spc,C}, \Fl_{\calI,C}^1$ are ind-proper over $\Spd C$.

	\begin{definition}
		\label{Schubert.variety.definition}
		Let $w\in \tilde W_\dR$.
		The Schubert cell $\Fl_{\calI,C,w}^\circ\subset \Fl_{\calI,C}$ is the v-sheaf-theoretic image of the orbit map 
		\begin{equation} 
		L^+_C\calI \to \Fl_{\calI,C},\ i \mapsto iw.
		\end{equation}
		The Schubert variety is the v-closure $\Fl_{\calI, C, w}:=\Fl_{\calI, C, w}^{\circ, \cl}$ in the sense of \Cref{section_closures_of_v_sheaves}.
	\end{definition}
	
	By \Cref{proposition-closure-of-v-subsheaves}, we know that the underlying topological space of $\Fl_{I,w}$ is the weakly generalizing closure of $|\Fl_{\calI,C,w}^\circ|$ inside $|\Fl_{\calI,C}|$. 
	But, $\Fl_{\calI,C,w}^\circ$ is possibly ill-behaved because $L^+_C\calI$ is not quasicompact. 
	As we show in \Cref{sec:geom-affine-grassm-proposition-stratification-of-schubert-varieties} and \Cref{Schubert.variety.corollary} for the affine Grassmannian, our definition is equivalent to the pointwise definition in \cite[Definition VI.2.2]{FS21}.
	We start with the case of simple reflections:
	
	\begin{lemma}
		\label{sec:geom-affine-grassm-proposition-closed-schubert-variety-for-simple-reflection}
		Let $s\in \mathbb{S}$ be a simple reflection. 
		Then there is an isomorphism $\Fl_{\calI, C, s}\simeq (\bbP^1_C)^\dia$ that restricts to $\Fl_{\calI, C, s}^{\circ}\simeq (\bbA^1_C)^\dia$. 
		In particular, $\Fl_{\calI, C, s}^{\circ}$  is a topologically dense open subset of $\Fl_{\calI, C, s}$.
	\end{lemma}
	
	\begin{proof}
		Let $\calP_s$ be the parahoric group scheme over $B_\dR^+$ associated to the wall of $\bba$ defining $s$.
		The maximal reductive quotient $H$ of its special fiber over $C$ has semisimple rank $1$.
		Using \cite[Théorème 4.6.33]{BT84}, we see that $L^+_C\calI$ is the preimage of $Q^\dia$ under $L^+\calP_s\r H^\dia$ for some Borel subgroup $Q\subset H$. 
		Thus, there are isomorphisms $L^+_C\mathcal{P}_s/L^+_C\mathcal{I}\simeq  (H/Q)^\dia\simeq (\mathbb{P}^1_C)^\diamondsuit$ which can be made explicit via the choice of a pinning.
		This implies that the monomorphism 
		\begin{equation}
		L^+_C\mathcal{P}_s/L^+_C\mathcal{I}\subset \Fl_{\calI,C}
		\end{equation}
		is a closed embedding, as $\Fl_{\calI,C}$ is separated and $(\mathbb{P}^1_C)^\diamondsuit$ is proper over $\Spd C$.
		The isomorphism $\Fl_{\calI, C, s}^{\circ}\simeq (\bbA^1_C)^\dia$ is now clear, since this is the only non-trivial $Q^\dia$-orbit in $(\mathbb{P}^1_C)^\diamondsuit$.
	\end{proof}
	
	In  order to treat more general $w=w_{\mathrm{af}}\tau\in \tilde W_\dR\cong W_{\dR, \mathrm{af}}\rtimes \Omega_{\mathfrak{a}}$, we invoke Demazure resolutions as follows.
	Let $\dot{w}=s_1 \dots s_n$ be a reduced word for $w_\aff=w\tau^{-1}$ with $s_i\in \mathbb{S}$ and consider the Demazure variety
	\begin{equation}
	\calD_{C,\dot{w}}:= L^+_C\calP_{1} \times^{L^+_C\calI} \dots \times^{L^+_C\calI} L^+_C\calP_n/L^+_C\calI
	\end{equation}
	which will also be denoted by $\Fl_{\calI, C, s_1} \tilde{\times} \dots \tilde{\times} \Fl_{\calI, C,s_n}$. 
	It is connected and cohomologically smooth over $\Spd C$ (being an iterated $\mathbb{P}^1_C$-fibration), and the twisted product 
	\begin{equation}
	\calD_{C,\dot{w}}^\circ=\Fl^\circ_{\calI, C, s_1} \tilde{\times} \dots \tilde{\times} \Fl^\circ_{\calI, C, s_n}
	\end{equation} 
	of the open cells is topologically dense by induction on $n$, starting with \Cref{sec:geom-affine-grassm-proposition-closed-schubert-variety-for-simple-reflection} and using that $L^+_C\calI$ is pro-(cohomologically smooth) over $\Spd C$.
	It carries, moreover, a natural morphism (induced by multiplication)
	\begin{equation}
	\pi_{\dot{w}}\colon \calD_{C,\dot{w}}\rightarrow \Fl_{\calI,C}
	\end{equation}
	which necessarily maps onto $\Fl_{\calI, C, w_\aff}$, by properness, $L^+_C\mathcal{I}$-equivariance and the fact that $\dot{w}$ maps to $w$. 
	After translation by $\tau$, we may regard this as a resolution of $\Fl_{\calI, C, w} $, which is thus in particular connected.
	
	For the next result, we note that $\tilde W_\dR$, in analogy to the discussion following \eqref{eq:affine_Weyl_group_sequence}, is equipped with a length function and Bruhat partial order induced from the quasi-Coxeter structure on $\tilde W_\dR\cong W_{\dR, \mathrm{af}}\rtimes \Omega_{\mathfrak{a}}$.

	\begin{proposition}
		\label{sec:geom-affine-grassm-proposition-stratification-of-schubert-varieties} 
		Let $w\in \tilde{W}_\dR$. 
		Then $\Fl_{\calI, C, w}^{\circ}$, respectively $\Fl_{\calI, C, w}$, agrees with the subfunctor of all maps $S\r \Fl_{\calI,C}$ such that for all geometric points $S'=\Spa(K,K^+)\r S$, the induced point $S'\r \Fl_{\calI,C}\r \Hk_{\calI,C}$ is given by $w$, respectively by $v$ for some $v\leq w$.
		In particular, $\Fl_{\calI, C, w}^{\circ}\subset \Fl_{\calI, C, w}$ is a topologically dense open.
	\end{proposition}
	\begin{proof}
		Observe that the first assertion cannot be verified at geometric points because $L^+_C\calI$ is not quasicompact. 
		However, we see from \Cref{sec:geom-affine-grassm-proposition-closed-schubert-variety-for-simple-reflection} that the result holds for simple reflections. Indeed, we even have by Bruhat--Tits combinatorics		
		\begin{equation}
		L^+_C\calU_{\alpha_s} \cdot  s =\Fl^\circ_{\calI, C, s},
		\end{equation}
		where $\alpha_s$ denotes the positive simple affine root associated to the simple reflection $s$, and $\calU_{\alpha_s}$ is the corresponding $B_\dR^+$-model of the affine root group. 
		Pulling across the reflections $s_i$ appearing in the convolution product of $\calD_{C,\dot{w}}$, we see that $\Fl_{\calI, C, w}^\circ$ surjects to the v-sheaf image of $\calD_{C,\dot{w}}^\circ$ along $\pi_{\dot w}$. 
		This v-sheaf image identifies with the pointwise description of $\Fl_{\calI, C, w}^{\circ}$ by quasicompactness of $\pi_{\dot w}$ and bijectivity at geometric points.
		
		Similarly, the v-sheaf image of $\calD_{C,\dot{w}}$ along $\pi_{\dot w}$ is a proper closed sub-v-sheaf of $\Fl_{\calI, C}$. By generalities of Tits system, see \cite[1.2.6]{BT72}, this v-sheaf image coincides with the desired pointwise description of $\Fl_{\calI, C, w}$. Pulling back again via the quotient map $\pi_{\dot w}$, we see that $\Fl_{\calI, C, w}^{\circ} \subset \Fl_{\calI, C, w}$ is a topologically dense open of the closed v-sheaf image of $\pi_{\dot{w}}$.
	\end{proof}
	
	As a corollary, we get that the bijection $\lvert \Hk_{\calI, C}\rvert\cong \tilde{W}_\dR$ from \Cref{sec:geom-affine-grassm-topological-space-of-affine-flag-stack} is a homeomorphism where $\tilde{W}_\dR$ is endowed with order topology via its Bruhat order, and also that $\pi_0(\Fl_{\calI, C})=\pi_0(\Gr_{G,C})=\pi_1(G)$ via \Cref{sec:geom-affine-grassm-translation-by-element-in-the-normalizer}. 
	
	Now, we apply our results to the affine Grassmannian $\Gr_G$.
	Note that there is the group isomorphism
	\begin{equation}
	X_\ast(T)\cong T(B_\dR)/T(B_\dR^+),\ \chi\mapsto \chi(\xi)
	\end{equation}
	which is independent of the choice of uniformizer $\xi\in B_\dR^+$.
	Then the Cartan decomposition induces a bijection
	\begin{equation}
	\lvert \Hk_{G,C} \rvert \simeq X_*(T)_+,
	\end{equation}
	where $\Hk_{G,C}=L^+_CG\backslash L_CG / L^+_CG$ denotes the Hecke stack. 
	Therefore, we get a Schubert cell $\Gr_{G,C,\mu}^\circ \subset \Gr_{G,C}$ defined as the v-sheaf-theoretic image of the orbit map and the Schubert cell $\Gr_{G,C,\mu}$ defined as its closure, for each $\mu \in X_*(T)_+$, compare with \Cref{Schubert.variety.definition}.
	
	\begin{corollary}
		\label{Schubert.variety.corollary}
		Let $\mu\in X_*(T)_+$. 
		Then $\Gr_{G,C,\mu}^{\circ}$, respectively $\Gr_{G,C,\mu}$ agrees with the subfunctor of all maps $S\r \Gr_{G,C,\mu}$ such that for all geometric points $S'=\Spa(K,K^+)\r S$, the induced point $S'\r \Gr_{G,C}\r \Hk_{G,C}$ is given by $\mu$, respectively by some $\la\leq \mu$ in the dominance order on $X_*(T)_+$.
		In particular, $\Gr_{G,C,\mu}^{\circ}\subset \Gr_{G,C,\mu}$ is a topologically dense open.
	\end{corollary}
	
	\begin{proof}
		This formally follows from \Cref{sec:geom-affine-grassm-proposition-stratification-of-schubert-varieties} by using the projection $\Fl_{\calI,C}\r \Gr_{G,C}$ from \eqref{projection.from.affine.flag.equality} and noting that the dominance order on $X_*(T)_+$ is induced by the Bruhat order, see \cite[Corollaries 1.8, 2.10]{Ric13} for similar arguments. 
		We leave the details to the reader.
	\end{proof}
	
	We also have the following fact which says that the $\Spd C$-valued points are dense in $\Gr_{G,C,\mu}$ even for the constructible topology.

	Recall from {\cite[Definition 4.50]{Gle20}}
	that for $X$ a locally spatial diamond over $\Spa(C,O_C)$ we say that it has \textit{enough facets} over $C$ (or enough $C$-facets) if it admits a v-cover of the form $\coprod_{i\in I} \Spd(A_i,A_i^\circ)\to X$ where each $A_i$ is an algebra topologically of finite type over $C$.
	
	\begin{corollary}
		\label{enough facets}
		Let $\mu\in X_*(T)_+$.
		The spatial diamond $\Gr_{G,C,\mu}$ has enough $C$-facets. 	
	\end{corollary}
	\begin{proof}
		Taking the preimage under the projection $\Fl_{\calI,C}\r \Gr_{G,C}$ from \eqref{projection.from.affine.flag.equality}, this reduces to the analogous assertion for $\Fl_{\calI,C,w}$ for some $w\in \tilde W_\dR$.
		Since the Demazure resolution is a v-cover, it is enough to prove that $\calD_{C,\dot{w}}$ has enough $C$-facets. 
		This in turn can be proved inductively on the length of $\dot{w}$. 
		If $\dot{w}=s\cdot\dot{v}$ then $\calD_{C,\dot{w}}$ is a pro-\'etale $(\bbP^1_C)^\diamondsuit$-bundle over $\calD_{C,\dot{v}}$. 
		We may find a pro-\'etale cover 
		\begin{equation} 
		X\to \calD_{C,\dot{v}} 
		\end{equation} 
		with $X\times_{\calD_{C,\dot v}}\calD_{C,\dot{w}}=X\times_{\Spd C} (\bbP^1_C)^\diamondsuit$. 
		Following the arguments given in \cite[Lemma 5.16, Proposition 5.21]{Gle20}, we may even assume that $X$ has enough facets over $\Spd C$. 
		By \cite[Proposition 4.51.(2)]{Gle20}, $\calD_{C,\dot{w}}$ also has enough facets.
	\end{proof}

	We conclude with some motivation for our later discussion of representability.
	
	\begin{proposition}\label{prop_non_representability_schubert_diamond}
		Let $\mu\in X_*(T)_+$.
		The v-sheaf $\Gr_{G,C,\mu}$ is representable by a projective $C$-scheme $\calF_{G,C,\mu}$ if and only if $\mu$ is minuscule.
	\end{proposition}

	\begin{proof}
		If $\mu$ is minuscule, then the $L^+_CG$-action factors through $G_C^\diamondsuit$ and the Bia{\l}lynicki-Birula map gives an isomorphism $\Gr_{G,C,\mu} \simeq (G_C/P_\mu)^\diamondsuit$, see \cite[Proposition 19.4.2]{SW20}.
		
		Now suppose that $\mu$ is not minuscule, so that $\langle \mu, \theta \rangle\geq 2$ for the highest root $\theta$ of $G_C$. 
		Then, we are going to show that the $L^+_C\calI$-orbit of the point $\mu$ is not representable by a rigid space. First, we notice that this orbit is isomorphic to $L^+_C\calI/H_\mu$ for the subgroup v-sheaf $H_\mu:=L^+_C\calI\cap \xi^\mu L^+_C\calI \xi^{-\mu}$. Note that the positive loop group $L_C^+B^-$ of the negative Borel $B^- \subset G_{\breve F}$ is contained in the stabilizer $H_\mu$, so we deduce that the $L^+_C\calI$-orbit is even a $L^+_C\calU$-orbit, where $\calU\subset \calI$ denotes the flat closure over $B_{\mathrm{dR}}^+$ of the unipotent radical $U \subset G_{\breve F}$ of the fixed positive Borel $B\subset G_{\breve F}$.
		We are going to filter this space further by using the structure of root groups. 
		Fix an ``ordre grignotant'' on the set $\Phi^+$ of positive roots in $G_C$ in the sense of \cite[Definition 3.1.2]{BT84}, that is, a descending sequence $\Psi_i \subset \Phi^+$ for $0\leq i< m=\mathrm{dim}(U)$, of subsets closed under summable roots, such that $\Psi_0=\Phi^+$, $\Psi_{m-1}=\{\theta\}$ and $\Psi_i$ is obtained from $\Psi_{i-1}$ by deleting one of the smallest roots $\alpha_i$. Then, \cite[Proposition 3.3.2]{BT84} yields associated root group models $\calU_i:=\calU_{\Psi_i}$ such that $\calU_{i} \subset \calU_{i-1}$ is a normal subgroup with quotient isomorphic to $\calU_{\alpha_i}$. We now set $X_i:=L^+_C\calU_i/H_\mu\cap L^+_C\calU_i$ for the corresponding orbits and we can realize $X_{i-1}$ as a pro-étale fibration over the Banach--Colmez space $\calB\calC(B_{\mathrm{dR}}^+/\xi^{\langle \mu,\alpha_i\rangle})$ with fiber given by $X_i$. Considering the filtration by powers of $\xi$ on the $L^+_C\calU_{\alpha_i}$, we can even refine this further to a filtration by orbits $Y_{ij}$ for $0\leq j <\langle \mu,\alpha_i \rangle$ such that $Y_{ij}$ is the fiber of a pro-étale fibration $Y_{i,j-1}\to \mathbb{G}_a^\diamondsuit$. 
		If $\calF_{G,C,\mu}$ were representable by a rigid space, it would follow by descending induction that all the $Y_{ij}$ are representable by rigid spaces.
		In particular, representability of $\calF_{G,C,\mu}$ implies representability of the Banach-Colmez space $\calB\calC(B^+_\dR/\xi^2)\cong Y_{m,\langle \mu,\theta\rangle -2}$ where we invoke the inequality $\langle \mu,\theta\rangle \geq 2$. 
		However, this is a non-split self-extension of $\bbG_a^\diamondsuit$ which does not even split étale locally, so cannot be representable (\cite[Example 15.2.9.5.]{SW20}). 
		Indeed, if the extension were split \'etale locally, it would actually split on the nose, as $H^1(\bbA^1_{C},\calO)$ is trivial. 
		However, if
		\begin{equation}\label{prop_non_representability_schubert_diamond:eq1}
		X:=\Spa(C\langle T^{\pm 1}\rangle, O_C\langle T^{\pm 1}\rangle)\subset \bbA^1_C
		\end{equation}
		is the affinoid torus,  the element $T\in \bbA^1_C(X)$ does not admit a lift to $\calB\calC(B^+_\dR/\xi^2)(X)$ as we now show. 
		Let $X^\prime=\Spa(C\langle T^{\pm 1/p^\infty}\rangle, O_C\langle T^{\pm 1/p^\infty}\rangle)$ be the usual perfectoid $\Z_p(1)$-cover of $X$. 
		Elements in $\calB\calC(B^+_\dR/\xi^2)(X^\prime)$ can be represented by $[a]+[b]\xi$ with $a,b\in (C\langle T^{\pm 1/p^\infty}\rangle)^\flat$. 
		Assume now that
		\[
		x:=[a]+[b]\xi
		\] maps to $T$, that means, $a^\sharp=T$ in $\bbA^1_C(X')$ (cf.\ \cite[Section 6.2]{SW20} for the map $(-)^\sharp$), and assume that $x$ is invariant under $\Z_p(1)$. Let $g\in \Z_p(1)$. 
		Then $g$ acts on $[a]$ via multiplication with $[g]$ if we identify $\Z_p(1)\subset C^\flat$. Using the $(\str)^\sharp$-map, we get 
		\begin{equation}\label{prop_non_representability_schubert_diamond:eq2}
		a^\sharp\left(\frac{[g]-1}{\xi}\right)^\sharp=b^\sharp-g(b^\sharp).
		\end{equation}
		But, if $g\in \Z_p(1)$ is a generator, then $a^\sharp(\frac{[g]-1}{\xi})^\sharp=cT$ with $c\in C$ non-zero. 
		Writing $b^\sharp$ as a power series in the $g$-eigenvectors $T^n$ with $n\in \Z[1/p]$, then shows that \eqref{prop_non_representability_schubert_diamond:eq2} cannot hold because $g\in \Z_p(1)$ fixes $T$. 
		This finishes the argument.    
	\end{proof}

	\subsection{Local models}
	We continue with the notation of \Cref{subsection_affine_grassmannian}, and additionally let $\calG$ be a parahoric $O$-model of $G$.
	

	We work with the moduli space $\Gr_\calG$ of $\calG$-torsors over $\Spec(B^+_\dR)$ trivialized over $\Spec(B_\dR)$, see \cite[Definition 20.3.1]{SW20}, which is the Beilinson--Drinfeld Grassmannian over $\Spd O$. 
	A crucial result of Scholze--Weinstein concerns its ind-properness, see \cite[Theorems 19.3.4, 20.3.6, 21.2.1]{SW20}.
	
	\begin{theorem}[Scholze--Weinstein]\label{SW theorem on parahoric grassmanians}
		The structure morphism of the Beilinson--Drinfeld Grassmannian 
		\begin{equation}
		\Gr_\calG \rightarrow \Spd O
		\end{equation} 
		is ind-proper and ind-representable in spatial diamonds.
	\end{theorem}
	
	In \cite[Section 12]{Ans18} and then later in \cite[Section 5]{Gle20}, the first named and second named author respectively constructed and studied the specialization map for $\Gr_\calG$, see also \cite[Theorem 5.1]{Gle20}. 
	
	Again, we have natural loop groups at hand, namely
	\begin{equation}
	L_O\calG \colon (R,R^+) \mapsto \calG(B_\dR(R^\sharp))
	\end{equation}
	and
	\begin{equation}
	L^+_O\calG \colon (R,R^+) \mapsto \calG(B_\dR^+(R^\sharp)),
	\end{equation} 
	where $(R^\sharp, R^{\sharp+})$ denotes an untilt of $(R,R^+)$ over $(O,O)$ and $B_\dR^{(+)}(R^\sharp)$ the ring of de Rham periods formed using $O$-Witt vectors. 
	These define v-sheaves over $\Spd O$ and the base changes to $\Spd F$, respectively $\Spd k$ recover the loop groups $L^+G\subset LG$ from \Cref{subsection_affine_grassmannian}, respectively the v-sheaves $(L^+_k\calG)^\dia\subset (L_kG)^\dia$ associated with the loop groups from \Cref{sec:affine-flag-vari}.  
	Their base changes to $O_C$ are denoted $L_{O_C}\calG$, $L^+_{O_C}\calG$ and $\Gr_{\calG,O_C}$.
	
	\begin{lemma}\label{elementary_BD_lemma}
		There is a natural isomorphism
		\begin{equation}\label{eq1:elementary_BD_lemma}
		L_O\calG/L^+_O\calG \cong \Gr_\calG,
		\end{equation}
		where the left side is a quotient for the étale topology. 
		In particular, on geometric fibers 
		\begin{equation}\label{eq2:elementary_BD_lemma}
		\Gr_{G}\cong \Gr_\calG\times_{\Spd O} \Spd F, \;\;\;\;\;\; \Fl_\calG^\dia\cong \Gr_\calG\times_{\Spd O} \Spd k,
		\end{equation}
		where $\Gr_{G}$ is the affine Grassmannian from \Cref{subsection_affine_grassmannian} and $\Fl_\calG^\dia$ the v-sheaf attached to the Witt vector partial affine flag variety from \Cref{sec:affine-flag-vari}.
	\end{lemma}
	\begin{proof}
		For the uniformization \eqref{eq1:elementary_BD_lemma}, see \cite[Proposition 20.3.2]{SW20}. 
		The isomorphisms \eqref{eq2:elementary_BD_lemma} are given by base change from \eqref{eq1:elementary_BD_lemma} by unwinding the definitions.
	\end{proof}
	
	Let $\mu$ be a conjugacy class of cocharacters of $G_C$, with field of definition $E\subset C$.	
	We denote by $O_E$ its ring of integers with residue field $k_E$.
	We wish to construct a closed sub-v-sheaf
	\begin{equation}
	\calM_{\calG, \mu} \subset \Gr_{\calG}|_{\Spd O_E}
	\end{equation}
	prolonging the Schubert diamonds $\Gr_{G, \mu}$ which are the descent to $\Spd E$ of the ones we studied in the previous subsection.

	\begin{definition}\label{defn_local_model}
		Let $\mu$ be a conjugacy class of cocharacters in $G_C$. 
		The local model $\calM_{\calG, \mu}$ is the v-closure of $\Gr_{G, \mu}$ inside $\Gr_{\calG}|_{\Spd O_E}$.
	\end{definition}
	
	A priori $\calM_{\calG, \mu}$ does not admit a moduli problem description for general parahoric $\calG$, so its structure could be harder to parse. Let us give some examples where the local model $\calM_{\calG,\mu}$ is relatively well understood.
	
	\begin{example}
		If $\calG$ is reductive, then $\calM_{\calG, \mu}$ is the integral Schubert variety over $\Spd O_E$ and generalizes the objects introduced in \Cref{subsection_affine_grassmannian}, see \cite[Proposition 20.3.6]{SW20} and \cite[VI.1]{FS21}. 
		If $G=T$ is a torus, then the explicit description of $\Gr_{\calT}$ furnishes an identity $\calM_{\calT, \mu}= \Spd O_E$, see \cite[Proposition 21.3.1]{SW20}.
	\end{example}

	We need to show permanence of the local model under the $L^+_O\calG$-action.

	\begin{proposition}
		\label{group acts on local models}
		The natural action map $L^+_O\calG \times \Gr_{\calG} \rightarrow \Gr_{\calG}$ restricts, after base change to $\Spd O_E$, to a group action on the closed sub-v-sheaf $\calM_{\calG, \mu}$. 
		Moreover, the generic fiber of $\calM_{\calG, \mu} $ is topologically dense.
	\end{proposition}
	\begin{proof}	
		By embedding $\calG$ in $\on{GL}_n$, we may always find a quasi-compact closed subsheaf $X\subset \Gr_{\calG, O_E}$ with $\calM_{\calG, \mu} \subset X$, stable under $L^+_{O_E}\calG$ whose action factors through a congruence quotient $L^{\leq N}_{O_E}\calG$, where $N$ is a sufficiently large positive integer (necessarily at least $\langle 2\rho, \mu \rangle$ as one sees by restricting to $\Gr_{G, \mu}$).
		
		The structure map $L^{\leq N}_{O_E}\calG\to \Spd O_E$ is partially proper, cohomologically smooth and consequently universally open, see \cite[Proposition 23.11]{Sch17}. 
		By \Cref{formation of closures commutes with nice basechange},
		\begin{equation}
		L^{\leq N}_{O_E}\calG\times_{\Spd O_E}\calM_{\calG,\mu}=(L^{\leq N}_{E}G\times_{\Spd E} \Gr_{G,\mu})^{\on{cl}}	
		\end{equation}
		as closed sub-v-sheaves of $L^{\leq N}_{O_E}\calG\times_{\Spd O_E} X$. 
		Now, we have seen that $L^+_EG$ respects $\Gr_{G, \mu}$, so the multiplication map $L^{\leq N}_EG\times_{\Spd E} \Gr_{G,\mu} \to X$ factors through $ \Gr_{G,\mu}$. 
		This also implies that the integral action map 
		\begin{equation}
		L^{\leq N}_{O_E}\calG\times_{\Spd O_E}\calM_{\calG,\mu}\to X
		\end{equation} 
		factors through the closure $\calM_{\calG,\mu}$, as we desired. 
		
		For the last claim, we consider the restricted variant of the Hecke v-stack  
		\begin{equation}
		\Hk^{\leq N}_{\calG, \mu}:=\big[L^{\leq N}_{O_E}\calG \backslash \calM_{\calG,\mu}\big].
		\end{equation} 
		Its underlying topological space has the extra special property that every subset is weakly generalizing, since for every perfectoid affinoid field $(K,K^+)$ the Bruhat decomposition over $B_\dR(K)$ is insensitive to variation of $K^{+}$. 
		Now, the projection map 
		\begin{equation}
		\on{pr}\colon \calM_{\calG,\mu}\to \Hk^{\leq N}_{\calG, \mu}
		\end{equation} 
		is cohomologically smooth and consequently open. 
		Therefore, by the same argument of \Cref{formation of closures commutes with nice basechange}, we see that both the usual topological and the weakly generalizing closure commute with pullback along $\lvert \on{pr} \rvert$. 
		It results that $\lvert \Gr_{G,\mu} \rvert $ is a dense open of $\lvert \calM_{\calG,\mu}\rvert$.
	\end{proof}

	We can now prove the following structural properties of $\calM_{\calG,\mu}$.
	
	\begin{proposition}
		\label{LM_kimberlite}
		With notation as in \Cref{defi category K}, $\calM_{\calG, \mu}\in \calK$. 
		More specifically, the local model $\calM_{\calG, \mu}$ is a flat $\pi$-adic kimberlite over $O_E$ with enough facets over $C$ and $O_C$-formalizable $C$-sections. 
		Moreover, the special fiber 
		\begin{equation}
		\calM_{\calG, k_E,\mu} := \calM_{\calG,\mu}|_{\Spd(k_E)}
		\end{equation}
		is of the form $\Fl_{\calG, W}^\diamondsuit$ for a connected perfect Schubert scheme $\Fl_{\calG, W}$. 
		In particular, $\calM_{\calG, \mu}^{\on{red}}=\Fl_{\calG, W}$ is perfectly proper and perfectly finitely presented perfect $k$-scheme.
	\end{proposition}
	\begin{proof}
		Choosing a closed embedding $\calG \hookrightarrow \on{GL}_n$, we may find a cocharacter $\nu$ of $\on{GL}_n$ giving rise to a closed immersion 
		\begin{equation}
		\calM_{\calG,\mu} \hookrightarrow \calM_{\on{GL}_{n},\nu},
		\end{equation} so that local models are proper v-sheaves and, in particular, quasi-compact.
		By \cite[Proposition 4.41.(3), Proposition 2.2.5]{Gle20} to prove $\calM_{\calG,\mu}$ is a kimberlite, it suffices to prove it is $\pi$-adic. 
		By \cite[Proposition 3.32]{Gle20}, we may reduce to proving that the special fiber of $\calM_{\calG,\mu}$ is represented by a v-sheaf of the form $X^{\diamond}$ for $X$ a perfect scheme.
		
		By \Cref{group acts on local models}, $\calM_{\calG, \bar k,\mu}$, is of the form $\cup_{i\in I} \Fl_{\calG,\bar{k},W_i}^\diamond$ with finite subsets $W_i\subset \tilde W$, where we have used the fact that $\Fl_{\calG,W}$ is perfectly proper, in order to deduce $\Fl_{\calG,W}^\diamondsuit=\Fl_{\calG,W}^\diamond$. 
		By quasi-compactness, we get $\calM_{\calG, k_E, \mu}=\Fl_{\calG,W}^\diamond$ for some finite subset $W\subset \tilde W$, which finishes the proof that $\calM_{\calG, \mu}$ is a $\pi$-adic kimberlite.
		
		That $\calM_{\calG,\mu}$ has $O_C$-formalizable $C$-sections follows from the main theorem of \cite{Ans18}. 
		Indeed, as in \cite[Section 12]{Ans18} any $C$-point of $\Gr_\calG$ produces canonically a $\calG$-torsor $\calP$ over $\Spec (A_{\on{inf}})$ together with a trivialization $\alpha:\calP\dashrightarrow \calG$ over $\Spec (A_{\on{inf}})\setminus V(\xi)$.
		Any map $\Spa(R,R^+)\to \Spd O_C$ induces a map $f:\Spec (B_{\on{dR}}(R^\sharp))\to \Spec (A_{\on{inf}})$ and the pullback $f^*\alpha\in \Gr_\calG(R,R^+)$ provides a map functors $\Spd O_C\to \Gr_\calG$ extending the original $C$-point. 
		Since $\Spd C\subseteq \Spd O_C$ is dense and $\calM_{\calG,\mu}\subseteq \Gr_\calG$ is closed, if the original $C$-point factors through $\calM_{\calG,\mu}$ then the corresponding $O_C$-point also does. 

		We explained in \Cref{enough facets} that $\calM_{\calG,\mu}$ has enough $C$-facets. Together with \Cref{group acts on local models} and \Cref{lem_top_flat_implies_formal_flat}, flatness follows.
		By \Cref{proper and sp surjective gives strong kimberlite}, $\lvert \Fl_{\calG,W}\rvert=\on{sp}(\lvert \Gr_{G,\mu}\rvert)$ and since $\Gr_{G,\mu}$ is connected $\Fl_{\calG,W}$ is also connected. 
	\end{proof}
	
	\begin{remark}\label{remark_local_model_base_change}
		It follows that the base change $\calM_{\calG, \mu}|_{\Spd O_C}$ is still topologically flat, and hence agrees with the v-closure $\calM_{\calG, O_C, \mu}$ of $\Gr_{G, C,\mu}$ inside $\Gr_{\calG,O_C}$.
		Indeed, repeating the argument of \Cref{group acts on local models} over $\Spd O_C$, we see that the special fiber of $\calM_{\calG, O_C, \mu}$ is represented by a Schubert perfect scheme. 
		But a $\Spd k$-valued point of $\calM_{\calG, \mu}$ is a specialization of some $\Spd C$-valued point by \Cref{LM_kimberlite}, hence equality of both closures is clear.
	\end{remark}
	
	Next, we analyse some functoriality behavior of $ \calM_{\calG, \mu} $ in the pair $(\calG,\mu)$. 
	Here, by definition, a map $(\calG_1,\mu_1)\to (\calG_2,\mu_2)$ is a morphism of $O$-group schemes $\calG_{1}\to \calG_{2}$ such that the image of $\mu_1$ in $G_{2,C}$ lies in the same conjugacy class as $\mu_2$.
	
	\begin{proposition}\label{prop_LM_basic_facts}
		The association $(\calG, \mu)\mapsto \calM_{\calG,O_C, \mu}$, see \Cref{remark_local_model_base_change}, is functorial, preserves closed embeddings and direct products, and induces isomorphisms $\calM_{\calG,O_C, \mu} \cong \calM_{\calG_\ad,O_C, \mu_\ad}$, where $\mu_{\mathrm{ad}}$ denotes the composite of $\mu$ with $G_C\to G_{\mathrm{ad},C}$.
	\end{proposition}
	
	\begin{proof}
		Functoriality follows from that of $\Gr_{G, C,\mu}$ and the definition using v-closures, see \Cref{remark_local_model_base_change}.
		For the claim regarding closed embeddings and central extensions, we refer to \cite[IV, Proposition 4.16, Corollary 4.17]{Lou20}: one checks injectivity at geometric points, using \Cref{connected.components.DB.lemma}. 
		As for direct products, it suffices to check equality at the level of $\Spd k$-valued points by \Cref{LM_kimberlite}, and this is easy because the generic fiber was already a product.
	\end{proof}

	\begin{lemma}\label{connected.components.DB.lemma}
		The specialization map induces a bijection 
		\begin{equation}
		\pi_0(\Gr_{\calG,\breve{O}})\overset\cong\lr \pi_0(\Fl_{\calG,\bar{k}})\cong \pi_1(G)_I,
		\end{equation} 
		where $I$ is the absolute Galois group of $\breve F$.
	\end{lemma}
	\begin{proof}
		For the final bijection, see \cite[Proposition 1.21]{Zhu17}.
		The first is a consequence of proper base change \cite[Theorem 19.2, Remark 19.3]{Sch17} applied to $f\co \Gr_{\calG,\breve O}\r \Spd \breve{O}$ and the base change $i\co \Spd \bar k \to \Spd \breve O$, using that the $0$-th cohomology group computes connected components:
		for some coefficient ring, say, $\Lambda=\bbZ/\ell$ with $\ell\neq p$, we apply the proper base change $i^*R^0f_* \Lambda_X\cong R^0(f')_* (i')^*\Lambda_X$ where $\Lambda_X$ is the constant sheaf supported on increasing closed $\breve O$-proper sub-v-sheaves $X\subset \Gr_{\calG,\breve O}$. 
		Passing to global sections, the second computes $\Lambda^{\pi_0(X_{\bar k})}$ by definition whereas the first computes $\Lambda^{\pi_0(X)}$ by using the v-cover $\Spd O_C\to \Spd \breve O$. 
		Finally, we use \cite[Section 27]{Sch17} to pass between $\Fl_{\calG,\bar{k}}^\dia$ and $\Fl_{\calG,\bar{k}}$.
	\end{proof}
	
	\section{Geometry of multiplicative group actions}\label{section_geometry_strata}
	Our approach to the Scholze--Weinstein conjecture requires determining the special fiber of local models in terms of admissible loci. 
	We follow the general strategy of Haines and the fourth named author \cite{HR21} of calculating the support of nearby cycles using hyperbolic localization. 
	This requires translating the results from \cite[Section 5]{HR21} to the v-sheaf Beilinson-Drinfeld Grassmannian.
	For basic facts pertaining to $\bbG_m^\dia$-actions on small v-stacks, the reader is referred to \cite[Chapter IV.6]{FS21}.
	
	\subsection{Over $O$}
	As in \Cref{subsection_affine_grassmannian}, we continue to fix a complete discretely valued field $F/\bbQ_p$ with ring of integers $O$ and perfect residue field $k$, a complete algebraic closure $C/F$ and a connected reductive $F$-group $G$ with parahoric model $\calG$ over $O$ containing the connected N\'eron model $\calS$ of the maximally $F$-split maximal $\breve F$-split $F$-torus $S$. 
	
	Fix a cocharacter $\la\co \bbG_m\r \calS\subset \calG$ defined over $O$.
	After base change to $F$, this induces a Levi $M=M_\la$ with Lie algebra $\Lie\,M=(\Lie\,G)_{\la =0}$, a parabolic subgroup $P=P_\la^+$ with Lie algebra $\Lie\,P=(\Lie\,G)_{\la \geq 0}$ and an unipotent subgroup $U=U^+_\la$ with Lie algebra $\Lie\,U=(\Lie\,G)_{\la > 0}$ fitting in a semi-direct product decomposition $P=M\ltimes U$.
	Since $\la$ is defined over $O$, the decomposition $P=M\ltimes U$ extends to $O$-models $\calP=\calM\ltimes \calU$, admitting analogous descriptions for their Lie algebras and being equipped with homomorphisms
	\begin{equation}\label{attractor.groups.eq}
	\calM \longleftarrow \calP \lr \calG.
	\end{equation}
	The $O$-group schemes $\calP$, $\calM$, $\calU$ are smooth affine with connected fibers, and $\calM$ is a parahoric $O$-model of the Levi subgroup $M$, see \cite[Lemma 4.5]{HR21} and also \cite[Section 2.1]{CGP15}, \cite[Section 6.2]{KP21} for proofs of these claims.
	
	By functoriality, \eqref{attractor.groups.eq} induces maps of ind-(spatial $\Spd O$-diamonds)
	\begin{equation}\label{BD.groups.eq}
	\Gr_\calM \longleftarrow \Gr_\calP \lr \Gr_\calG,
	\end{equation}
	where $\Gr_\calM\r \Spd O$ and $\Gr_\calG\r \Spd O$ are ind-proper by \Cref{SW theorem on parahoric grassmanians}. 
	On the other hand, the cocharacter $\la$ induces a cocharacter
	\begin{equation}
	\bbG_m^\dia \overset{[\str]}{\lr} L^+_O\bbG_m \overset{L^+_O\la}{\lr} L^+_O\calG,
	\end{equation}
	where $[\str]$ denotes the Teichm\"uller lift. 
	Thus, we obtain a left action of $\bbG_m^\dia$ on $\Gr_\calG$. 
	
	\begin{lemma}\label{lemma-representability-attractor}
		The $\bbG_m^\dia$-action on $\Gr_\calG$ satisfies \cite[Hypothesis IV.6.1]{FS21}.
	\end{lemma}
	\begin{proof}
		Choosing a closed immersion $\calG\hr \GL_{n,O}$ of group schemes, we reduce to the case $\calG=\GL_{n,O}$, using \Cref{SW theorem on parahoric grassmanians} to see that the induced map $\Gr_\calG\hr \Gr_{\GL_{n,O}}$ is a closed immersion.
		Then the lemma is a special case of \cite[Proposition VI.3.1]{FS21}.
	\end{proof}
	
	Consequently, we obtain a $\bbG_m^\dia$-equivariant diagram
	\begin{equation}\label{BD.attractor.eq}
	(\Gr_\calG)^0 \longleftarrow (\Gr_\calG)^+ \lr \Gr_\calG,
	\end{equation}
	where $(\Gr_\calG)^0=(\Gr_\calG)^{\bbG_m^\dia}$ denotes the fixed points and $(\Gr_\calG)^+$ the attractor classifying $\bbG_m^\dia$-equivariant maps $(\bbA^1)^{\dia, +}\r \Gr_\calG$ over $\Spd O$, see also \eqref{section_nearby_cycles:eq1} below. 
	The map $(\Gr_\calG)^+\r (\Gr_\calG)^0$ is the Bia{\l}ynicki-Birula map given by evaluating at the zero section.
	Our aim is to understand the relation of \eqref{BD.groups.eq} with \eqref{BD.attractor.eq}.
	The following result is the analogue of \cite[Theorem 5.6, Theorem 5.19]{HR21} in the context of ind-schemes: 
	
	\begin{theorem}\label{BD.attractor.groups.comparison.theorem}
		The maps \eqref{BD.groups.eq} and \eqref{BD.attractor.eq} fit into a commutative diagram of ind-(spatial $\Spd O$-diamonds)
		\begin{equation}\label{BD.attractor.groups.comparison.eq}
		\begin{tikzpicture}[baseline=(current  bounding  box.center)]
		\matrix(a)[matrix of math nodes, 
		row sep=1.5em, column sep=2em, 
		text height=1.5ex, text depth=0.45ex] 
		{\Gr_\calM & \Gr_{\calP} & \Gr_\calG \\ 
			(\Gr_{\calG})^0& (\Gr_{\calG})^+& \Gr_{\calG}, \\}; 
		\path[->](a-1-2) edge node[above] {}  (a-1-1);
		\path[->](a-1-2) edge node[above] {}  (a-1-3);
		\path[->](a-2-2) edge node[below] {}  (a-2-1);
		\path[->](a-2-2) edge node[below] {} (a-2-3);
		\path[->](a-1-1) edge node[left] {$\iota^0$} (a-2-1);
		\path[->](a-1-2) edge node[left] {$\iota^+$} (a-2-2);
		\path[->](a-1-3) edge node[left] {$\id$} (a-2-3);
		\end{tikzpicture}
		\end{equation}
		with the following properties.
		\begin{enumerate}
			\item \label{BD.attractor.groups.comparison.theorem.1}
			The maps $\iota^0$, $\iota^+$ are open and closed immersions.
			\item \label{BD.attractor.groups.comparison.theorem.2}
			Their base changes $\iota^0_{F}$, $\iota^+_{F}$ are isomorphisms. 
			\item \label{BD.attractor.groups.comparison.theorem.3}
			If $\calG_{\breve{O}}$ is very special parahoric (for example, reductive), then $\iota^0$, $\iota^+$ are isomorphisms.
		\end{enumerate}
		In particular, the complements $(\Gr_\calG)^0\setminus \iota^0(\Gr_\calM)$, $(\Gr_\calG)^+\setminus \iota^+(\Gr_\calP)$ are concentrated over $\Spd k$, that is, their base change to $\Spd F$ is empty.
	\end{theorem}
	
	If $\calG$ is reductive, then \Cref{BD.attractor.groups.comparison.theorem} is proved in \cite[Proposition VI.3.1]{FS21}. 
	Indeed, in this case $\iota^0$, $\iota^+$ are isomorphisms.
	In general, we follow the strategy of \cite{HR21}:
	
	\begin{proof}[Proof of \Cref{BD.attractor.groups.comparison.theorem}]
		First, we construct the maps $\iota^0, \iota^+$.
		The closed subgroup $\calM\hr \calG$ induces a closed immersion $\Gr_\calM\hr \Gr_\calG$ (using \Cref{SW theorem on parahoric grassmanians}) that is $\bbG_m^\dia$-equivariant for the trivial action on the source. 
		Thus, the map factors though the fixed points, defining the necessarily closed immersion $\iota^0\co \Gr_\calM\hr (\Gr_\calG)^0$.
		The construction of $\iota^+$ is more delicate and proceeds as follows. 
		Pick a closed immersion $\calG\hr \GL_{n,O}$ of $O$-group schemes such that the fppf quotient $\GL_{n,O}/\calG$ is quasi-affine, see \cite[Proposition 1.3]{PR08}.
		The cocharacter $\la'\co \bbG_m\overset{\la}{\r} \calG \hr  \GL_{n,O}$ induces a parabolic subgroup $\calP'$ with Lie algebra $\Lie\, \calP'=(\Lie\,\GL_{n,O})_{\la'\geq 0}$.
		The induced map on fppf quotients $\calP'/\calP\r \GL_{n,O}/\calG$ is a monomorphism between finite type $O$-schemes, thus quasi-affine by Zariski's main theorem.
		By functoriality of the construction $\calG\mapsto \Gr_\calG$, we obtain a commutative diagram of ind-(spatial $\Spd O$-diamonds):
		
		\begin{equation}\label{commutative.square.BD.equation}
		\begin{tikzpicture}[baseline=(current  bounding  box.center)]
		\matrix(a)[matrix of math nodes, 
		row sep=1.2em, column sep=2em, 
		text height=1.5ex, text depth=0.45ex] 
		{\Gr_{\calP} & & \\
			&(\Gr_{\calG})^+& \Gr_\calG  \\ 
			\Gr_{\calP'}&(\Gr_{\GL_{n,O}})^+& \Gr_{\GL_{n,O}} \\}; 
		\path[->](a-1-1) edge[dotted] node[below] {$\iota^+$}  (a-2-2);
		\path[->](a-1-1) edge[bend left=15] node[above] {}  (a-2-3);
		\path[->](a-1-1) edge node[above] {}  (a-3-1);
		\path[->](a-3-1) edge node[above] {$\cong$}  (a-3-2);
		\path[->](a-2-2) edge node[above] {}  (a-2-3);
		\path[->](a-2-2) edge node[left] {}  (a-3-2);
		\path[->](a-3-2) edge node[below] {}  (a-3-3);
		\path[->](a-2-3) edge node[right] {}  (a-3-3);
		\end{tikzpicture}
		\end{equation}
		Using \Cref{SW theorem on parahoric grassmanians}, the map $\Gr_\calG\r \Gr_{\GL_{n,O}}$ is a closed immersion so that the square is Cartesian. 
		This proves the existence of $\iota^+$.
		Furthermore, the displayed map $\Gr_{\calP^\prime}\to (\Gr_{\GL_{n,O}})^+$ is an isomorphism by \cite[Proposition VI.3.1]{FS21}. 
		As $\calP'/\calP$ is quasi-affine, so $\Gr_\calP\r \Gr_{\calP'}$ is a locally closed immersion (compare with the proof of \cite[Proposition 1.2.6]{Zhu16} and \cite[Lemma 19.1.5]{SW20}), the map $\iota^+$ is necessarily a locally closed immersion as well.		
		
		Now, part \eqref{BD.attractor.groups.comparison.theorem.2} is immediate from our construction and \cite[Proposition VI.3.1]{FS21} applied over $\Spd F$. 
		For part \eqref{BD.attractor.groups.comparison.theorem.3}, we observe that $\iota^0, \iota^+$ are bijective on geometric points if $\calG$ is very special parahoric: by \eqref{BD.attractor.groups.comparison.theorem.2} for geometric points lying over $\Spd F$ and by the Iwasawa decomposition \cite[Section 3.3]{KP21} for geometric points lying over $\Spd k$.
		As $\iota^0$, $\iota^+$ are locally closed immersions, they must be isomorphisms, so \eqref{BD.attractor.groups.comparison.theorem.3} follows.
		
		For \eqref{BD.attractor.groups.comparison.theorem.1}, it remains to prove that $\iota^0$, $\iota^+$ are open immersions for general parahoric group schemes $\calG$.
		For this, we may and do assume that $k$ is algebraically closed. 
		There are bijections of connected components
		\begin{equation}
		\label{connected.components.bijection.equation}
		\pi_0((\Gr_\calG)^+)\overset{\cong}{\lr} \pi_0((\Gr_\calG)^0)\overset{\cong}{\lr} \pi_0((\Fl_\calG)^0),
		\end{equation}
		where the first holds by general properties of Bia{\l}ynicki-Birula maps (see the proof of \cite[Corollary 1.12]{Ric19a}) and the second by proper base change as in the proof of \Cref{connected.components.DB.lemma}.  
		The fixed points $(\Fl_\calG)^0$ in the Witt vector partial affine flag variety can be analyzed in analogy to \cite[Section 4]{HR21}: 
		concretely, if $\calP_\spc=\calM_\spc\ltimes \calU$ for $\calM_\spc$ being the corresponding parahoric model of $M_\spc$, then there is a disjoint union (on points) into connected locally closed sub-ind-schemes
		\begin{equation}\label{semi-infinite.orbits.equation}	
		\Fl_\calG=\bigcup_{[w]} \calS_w,\ \calS_w=L_k\calP_\spc\cdot w
		\end{equation}
		where $[w]$ runs through the double coset $W_{M,\aff}\backslash \tilde W / W_\calG$ and $w$ denotes the image of a representative under the embedding $\tilde W/W_\calG\hr \Fl_\calG$.
		The image of $\Fl_\calP\hr \Fl_\calG$ consists of those $\calS_w$ for $[w]$ lying in $W_{M,\aff}\backslash \tilde W_M / W_\calM$.
		Passing to fixed points, the image of $\Fl_\calM\hr (\Fl_\calG)^0$ is the union of the $L_k\calM_\spc$-orbits for these $[w]$. 
		So the map $\pi_0(\Fl_\calM)\r \pi_0((\Fl_\calG)^0)$ identifies with the injection
		\begin{equation}\label{Iwahori-Weyl.injection.equation}		
		W_{M,\aff}\backslash \tilde W_M / W_\calM \hr W_{M,\aff}\backslash \tilde W / W_\calG.
		\end{equation}
		We let $\calC^0_\calG$, respectively $\calC^+_\calG$, be the open and closed sub-v-sheaf of $(\Gr_\calG)^0$, respectively of $(\Gr_\calG)^+$, consisting of those components belonging to $\on{im}(\pi_0(\Fl_\calM)\hr \pi_0(\Fl_\calG))$ under \eqref{connected.components.bijection.equation}.
		Then the maps $\iota^0$, $\iota^+$ factor through $\calC^0_\calG$, respectively $\calC^+_\calG$ inducing locally closed immersions
		\begin{equation}
		\Gr_\calM\hr \calC^0_\calG,\ \Gr_\calP\hr \calC^+_\calG,
		\end{equation}
		that are bijective on geometric points, hence isomorphisms. 
		The theorem is thus proven.
	\end{proof}
	
	\subsection{Semi-infinite orbits}
	We end this section with a study of the stratification \eqref{semi-infinite.orbits.equation}. 	
	Throughout, we assume that $k=\bar{k}$ is algebraically closed so that $F=\breve F$ and $O=\breve O$.
	Note that the torus $S$ is then maximal $F$-split.
	The following lemma simplifies some arguments of \cite[Theorem 6.12]{HR21} and is used in the proof of \Cref{theorem_special_fiber_admissible} given in \Cref{nearby.cycles.section}. 
	
	\begin{lemma}\label{lemma_attractor_schubert_isolated}
		For every $w \in \tilde W/W_{\calG}$, there is an $O$-cocharacter $\bbG_m\r \calS\subset \calG$ such that for the induced strata $\calS_w \cap \Fl_{\calG,w}=\{ w\}$.
	\end{lemma}
	
	\begin{proof}
		Up to changing the Iwahori $L_k^+\calI \subset L_k^+\calG$, we may and do assume that the Iwahori--Schubert variety $\Fl_{(\calI,\calG), w}^\circ$ is a dense open of $\Fl_{\calG, w}$. Notice that the closed complement of that dense open is stable under the $\bbG_m^{\on{perf}}$-action, so it follows that the connected component of the fixed point $w$ in the attractor $\Fl_{\calG}^+$ cuts $\Fl_{\calG, w}$ inside $\Fl_{(\calI,\calG), w}^\circ$. 
		In other words, $S_w\cap \Fl_{\calG, v}$ is empty when $v<w$.

		The reduced word $\dot{w}=s_1\dots s_n$ determines a minimal gallery $\Gamma=(\bba_0, \bba_1, \dots, \bba_n)$, where $\bba_i=s_1\dots s_i(\bba)$, going from the alcove $\bba$ fixed by $\calI(O)$ to its $\dot{w}$-conjugate. 
		Let $\alpha_i$ be the unique positive affine root such that $\partial \alpha_i$ is the wall separating $\bba_{i-1}$ and $\bba_i$. We claim that
		\begin{equation}\label{eqn_Demazure_expanded_root_groups}
		\Fl_{(\calI,\calG),w}^\circ= L_k^+\calU_{\alpha_1} \cdot \ldots \cdot L_k^+\calU_{\alpha_n}w.
		\end{equation}
		This follows by expanding the Demazure twisted product, pulling across the simple reflections to the right, compare with \Cref{sec:geom-affine-grassm-proposition-stratification-of-schubert-varieties}. 
		Indeed, $s_{i-1}\dots s_1(\alpha_i) $ is by construction the positive simple affine root attached to $s_i$.		
		We need to construct an $O$-cocharacter $\chi: \bbG_{m} \to \calS$ such that the inverse of the induced $\bbG_m^{\on{perf}}$-action contracts every affine root group $L^+\calU_{\alpha_i}$ to zero, compare \eqref{eqn_Demazure_expanded_root_groups}.
		This is equivalent to the condition that $\chi$ pairs negatively with $a_i=\nabla \alpha_i$ for $i\in \{1,\dots, n\}$, which are the gradients of the prescribed affine roots $\alpha_i $ attached to $\Gamma$.
		This follows now by \cite[Corollary 5.6]{HN02} but we give below a quick proof for the reader's convenience.
		
		Consider the subset $\Phi_{G,w} \subset \Phi_G$ of all euclidean roots $a_i=\nabla\alpha_i$. 		
		If $b$ denotes the barycenter of $\bba$, then by definition $a(\dot{w}b-b)<0$ is stricly negative for $a \in \Phi_{G,w}$. 
		In particular, $\Phi_{G,w}$ consists entirely of negative $B$-roots where $S \subset B \subset G$ is a Borel subgroup whose closed Weyl chamber contains the vector $\dot{w}b-b$. 
		So we may take any $B$-dominant regular coweight $\bbG_m\to S$ which uniquely extends to the desired $ \bbG_m\to \calS$ because $S$ is $F$-split.
	\end{proof}
	
	Now assume that $\la$ is regular.
	Then $M_\la=T$ is a maximal torus, $P_\la=B$ a Borel subgroup with unipotent radical $U$ defined over $F$.
	The stratification \eqref{semi-infinite.orbits.equation} becomes
	\begin{equation}\label{semi_infinite_stratification}
	\Fl_{\calG}=\bigcup_{w \in \tilde{W}/W_\calG}\calS_w,\ \calS_{w}=L_kU\cdot w
	\end{equation}
	and the strata $\calS_w$ are called semi-infinite orbits, compare with \cite[Proposition VI.3.1]{FS21}. 
	Recall that there is a semi-infinite Bruhat order $\leq^{\frac{\infty}{2}}$ on $\tilde{W}/W_\calG$ defined by:
	\begin{equation}\label{equation-semiinifinite-bruhat}
	v \leq^{\frac{\infty}{2}} w \;\;\iff\;\; \forall\; \nu_I\gg 0\colon\; {\nu_I}(\pi)\cdot v \leq {{\nu_I}(\pi)}\cdot w
	\end{equation}
	where $X_\ast(T)_I\subset \tilde W$, $\nu_I\mapsto \nu_I(\pi)$ is viewed as a subgroup using the Kottwitz morphism, see \eqref{eq:coinvariants.Kottwitz.iso}, and where $\nu_I\gg 0$ means that $\nu_I$ is sufficiently $B$-dominant.
	This order was first introduced by Lusztig in \cite{Lus80} and depends on the $F$-Borel subgroup $B$ attracted by $\la$.

	\begin{proposition}\label{prop_closure_relations_semi_infinite_orbits}
		The ind-closure of $\calS_w$ inside $\Fl_{\calG}$ is given by the perfect sub-ind-scheme whose geometric points factor through some $\calS_v$ with $v \leq^{\frac{\infty}{2}} w$.
	\end{proposition}
	
	\begin{proof}
		Let $\calI \to \calG$ be an auxiliary Iwahori model, fixed for the remainder of the proof. Suppose there is a curve $\calC \subset \Fl_{\calG}$ containing $v$ and whose complement $\calC^\circ:=\calC\setminus \{v\}$ is contained in $\calS_w=L_kU \cdot w$. 	
		Now, notice that, for sufficiently dominant $\nu_I$, one gets the inclusion
		\begin{equation}
		\nu_I(\pi)\cdot \calC^\circ \subset L^+_k\calI \cdot {\nu_I}(\pi)\cdot w,
		\end{equation}
		because conjugation by $\nu_I(\pi)$ moves any given perfect $k$-subscheme of $L_kU$ inside the Iwahori loop group $L^+_k\calI \subset L^+_k\calG$. This implies the inequality ${{\nu_I}(\pi)} v \leq {{\nu_I}(\pi)} w$, and therefore $v \leq^{\frac{\infty}{2}} w$.
		
		Conversely, assume that the inequality $v \leq^{\frac{\infty}{2}} w$ holds. 
		By definition, for all sufficiently dominant translations $\nu_I\in X_\ast(T)_I$, the inequality $\nu_I(\pi)\cdot w\leq \nu_I(\pi)\cdot v$ holds in the Bruhat order. 
		By enlarging $\nu_I$ if necessary, we may assume the $\Fl^\circ_{(\calI,\calG), \nu_I(\pi)\cdot w}$ is of the form $\prod_{\alpha \in \Gamma} L_k^+\calU_{\alpha} \cdot \nu_I(\pi)\cdot w$ where all of the $\alpha\in \Gamma$ have positive gradient. 
		There is a curve $\calC$ in $\Fl_{(\calI,\calG),\nu_I(\pi)\cdot w}$ joining $\nu_I(\pi)\cdot w$ to $\nu_I(\pi)\cdot v$ since $\nu_I(\pi)\cdot v\leq \nu_I(\pi)\cdot w$. 
		By our assumption on $\nu_I(\pi)$, we have $\calC^\circ\subset \calS_{\nu_I(\pi)\cdot w}$ for $\calC^\circ:=\calC\cap \Fl_{(\calI,\calG),\nu_I(\pi)\cdot w}^\circ$. 
		Now, the map $t_{\nu_I}\co \Fl_{(\calI,\calG), w}\to \Fl_{(\calI,\calG), \nu_I(\pi)\cdot w}$ induced by left translation with $\nu_I(\pi)$ is an isomorphism and hence induces $\calS_w\cong \calS_{\nu_I(\pi)\cdot w}$.
		Then the curve $t_{\nu_I}^{-1}(\calC)$ joins $w$ to $v$, and $t_{\nu_I}^{-1}(\calC^\circ)\subset \calS_w$.
	\end{proof}
	
	Next, we extend the equi-dimensionality of Mirkovi\'c--Vilonen cycles \cite[Theorem 3.2]{MV07} (see also \cite[Corollary 2.8]{Zhu17} and \cite[Corollary VI.3.8]{FS21}) from split groups to twisted groups as follows.
	We continue to assume that $\la$ is regular and, additionally, that $\calG$ is special parahoric.
	Then $X_\ast(T)_I=\tilde W/W_\calG$ and \eqref{semi_infinite_stratification} becomes
	\begin{equation}
	\Fl_\calG=\bigcup_{\nu_I \in X_*(T)_I} \calS_{\nu_I},\ \calS_{\nu_I}=L_kU\cdot {\nu_I}(\pi).
	\end{equation}
	If $\calG$ is reductive, then $\Fl_\calG$ is the Witt vector affine Grassmannian studied in \cite{Zhu17}.
	So, in general, $\Fl_\calG$ can be regarded as a twisted version when $\calG$ is special parahoric.
	Indeed, $\Fl_\calG=\on{colim} \Fl_{\calG, \mu_I}$ where $\mu_I$ runs through $X_*(T)_{I,+}$, the $B$-dominant elements in $X_*(T)_I$ equipped with the dominance order and $\Fl_{\calG, \mu_I}:=\Fl_{\calG, \mu_I(\pi)}$, as is usual notation for (twisted) affine Grassmannians.
	We note that the projection $X_*(T)\r X_*(T)_I$ restricts to a (not necessarily surjective) map $X_*(T)_+\r X_*(T)_{I,+}$ compatibly with the dominance orders, see \cite[Lemma~2.6]{ALRR}.
	Also, the semi-infinite Bruhat order on $X_\ast(T)_I$ specializes to the dominance relation, that is, $\nu_I'\leq^{\frac{\infty}{2}} \nu_I$ if and only if $\nu_I-\nu_I'$ is a sum of coinvariants of positive coroots with non-negative integer coefficients.
	
	\begin{lemma}\label{equidimensionality.MV.cycles.lemma}
		For any $\nu_I\in X_*(T)_I$, the intersection $\calS_{\nu_I}\cap \Fl_{\calG, \mu_I}$ is non-empty if and only if $\nu_I$ lies in the $W_\calG$-orbit of some $\mu_I'\in X_*(T)_{I,+}$ with $\mu'_I\leq \mu_I$.
		In this case, it is affine and equidimensional of dimension $\langle\rho_G, \nu + \mu\rangle$.
	\end{lemma}
	
	Here $\rho_G\in X^*(T)$ denotes the half sum of the $B$-positive roots.
	We note that the pairing $\langle\rho_G, \nu + \mu\rangle$ is well-defined independently of the choice of lifts $\nu,\mu \in X_*(T)$ of $\nu_I, \mu_I$ because $\rho_G$ is $I$-invariant.
	
	\begin{proof}[Proof of \Cref{equidimensionality.MV.cycles.lemma}]
		The map $(\Fl_{\calG,\mu_I})^+\r \Fl_{\calG,\mu_I}$ induces an isomorphism (see \Cref{BD.attractor.groups.comparison.theorem})
		\begin{equation}\label{attractor.disjoint.sum.equation}
		(\Fl_{\calG,\mu_I})^+ \overset{\cong}{\lr} \bigsqcup_{\nu_I \in X_*(T)_I} \calS_{\nu_I}\cap \Fl_{\calG, \mu_I}.
		\end{equation}
		Under $(\Fl_{\calG,\mu_I})^+\r (\Fl_{\calG,\mu_I})^0$, the component $\calS_{\nu_I}\cap \Fl_{\calG, \mu_I}$ contracts to $\underline{\{\nu_I\}}=\Spec(k)$. 
		Thus, it is affine because Bia{\l}ynicki-Birula maps for schemes are affine by \cite[Corollary 1.12]{Ric19a}.
		Also, $(\Fl_{\calG,\mu_I})^0$ identifies with the constant scheme associated with the subset of $\nu_I\in X_*(T)_I$ lying in the $W_\calG$-orbit of some $\mu_I'\in X_*(T)_{I,+}$ with $\mu'_I\leq \mu_I$. 
		So only such $\nu_I$ contribute to \eqref{attractor.disjoint.sum.equation}, and as the Bia{\l}ynicki-Birula map has a section, the non-emptiness criterion for $\calS_{\nu_I}\cap \Fl_{\calG, \mu_I}$ holds true.
		Furthermore, the union over all $\calS_{\nu'_I}\cap \Fl_{\calG,\mu_I}$ with $\nu_I'\leq^{\frac{\infty}{2}} \nu_I$ is a closed perfect subscheme by \Cref{prop_closure_relations_semi_infinite_orbits}.
		
		As noted in \cite[Corollary VI.3.8]{FS21}, the affineness implies the dimension formula once we show that $\calS_{\mu_I^{\on{anti}}}\cap \Fl_{\calG, \mu_I}$ is a point where $\mu_I^{\on{anti}}$ is the antidominant element in the $W_\calG$-orbit of $\mu_I$.
		This follows from the proof of \Cref{lemma_attractor_schubert_isolated}.
	\end{proof}

	
	\section{Nearby cycles of étale sheaves}\label{section_nearby_cycles}
	
	\subsection{Recollections}
	In \cite{Sch17}, Scholze constructs a category of étale sheaves 
	\begin{equation}
	\D(X,\Lambda):=\D_\et(X,\Lambda) 
	\end{equation}
	for all small v-stacks $X$. 
	As coefficients $\Lambda$, we allow prime-to-$p$ torsion rings or, by the adic formalism of \cite[Section 26]{Sch17}, an $\ell$-torsion free, complete $\ell$-adic ring for $\ell\neq p$, or a ring of the form $\Lambda=\Lambda_0[\ell^{-1}]$ where $\Lambda_0$ is as in the previous case.
	In the final case, as this is not covered in \cite{Sch17}, we define the triangulated category
	\begin{equation}\label{eq:localization_categories}
	\D(X,\Lambda):=\D(X,\Lambda_0)\otimes_{\Lambda_0}\Lambda,
	\end{equation}
	in analogy to the classical definition for schemes, for example, see \cite[Appendix A]{KW01}. 
	The adic formalism of \cite[Section 26]{Sch17} carries over to the categories \eqref{eq:localization_categories}.
	Finally, we also allow $\Lambda$ to be a filtered colimit of the aforementioned rings, with the obvious definition for the categories.
	This includes algebraic field extensions $L/\bbQ_\ell$ and their rings of integers $O_L$.
	
	The categories of \'etale sheaves are equipped with the usual six functors formalism: the endofunctors $\otimes^\bbL$, $\RHom$ and functors $Rf_*$, $f^*$ for a morphism $f\colon X\rightarrow Y$ of small v-stacks. 
	If $f$ is compactifiable and representable in locally spatial diamonds with $\on{dim.\!tr}f<\infty$, we have the functors $Rf_!$, $Rf^!$, completing the six functor formalism.
	
	In general, the categories $\D(X,\Lambda)$ and the six functors are rather inexplicit, constructed through v-descent using Lurie's $\infty$-categorical machinery. 
	Nevertheless, whenever $f\colon X\rightarrow Y$ is a morphism between locally spatial diamonds, then $X$ and $Y$ admit a well-defined étale site and Scholze's operations are very closely related to the operations that one can construct site-theoretically, see \cite[Proposition 14.15, Section 17]{Sch17}.
	
	When $X$ and $Y$ are locally spatial diamonds we say that an object $A \in \D(X,\Lambda)$ is ULA (=universally locally acyclic) with respect to $f$ if, for all locally spatial diamonds $Y'\to Y$, the pullback $A' \in \D(X',\Lambda)$ is overconvergent along the fibers of $f': X'=X\times_Y Y' \to Y'$ and $R(f'\circ j')_{!}j'{}^{*}A$ is perfect-constructible for all separated étale neighborhoods $j'\colon U'\rightarrow X'$ for which $f'\circ j'$ is quasi-compact, see \cite[Definition IV.2.1]{FS21}. 
	If $\Lambda$ is $\ell$-adic as above, then a complex $A\in \D(X,\Lambda)$ is called perfect-constructible if $A\otimes^\bbL_\Lambda \Lambda/\ell$ is étale locally perfect-constant after passing to a constructible stratification, equivalently $A\otimes^\bbL_\Lambda \Lambda/\ell^n$ are so for all $n\geq 1$.  
	Finally, if $\Lambda=\Lambda_0[\ell^	{-1}]$ is as in \eqref{eq:localization_categories}, then an object in $\D(X,\Lambda)$ is called perfect-constructible if it admits a $\Lambda_0$-lattice which is so.
	For $X$ and $Y$ more general v-stacks (and $f\colon X\to Y$ representable in locally spatial diamonds), we call $A$ ULA if it is ULA after any base change $S\to Y$ with $S$ a locally spatial diamond.
	
	Suppose $X$ is a small v-stack proper and representable in spatial diamonds over a base $S$, and that $X$ is equipped with an action by $\bbG_{m,S}^\diamondsuit$ satisfying the conditions \cite[Hypothesis IV.6.1.]{FS21}. 
	One can consider the v-stacks
	\begin{equation}\label{section_nearby_cycles:eq1}
	X^\pm=\Hom_{\bbG_{m}^\diamondsuit}\big((\bbA^1)^{\pm,\diamondsuit},X\big)
	\end{equation}
	which (by hypothesis) are represented by a finite partition of $X$ into locally closed subsets. 
	This also induces a partition of the fixed-point v-stack $X^0=X^{\bbG_m^\diamondsuit}$ into closed and open subsets. 
	We have inclusion maps $q^\pm\co X^\pm\to X$ and projection maps $p^\pm\co X^\pm\to X^0$, from that we obtain the hyperbolic localization functor
	\begin{equation}
	L_{X/S}\colon \D(X/\bbG_{m,S}^\diamondsuit,\Lambda) \rightarrow \D(X^0,\Lambda),
	\end{equation}
	which can be expressed as $R(p^+)_!(q^+)^*$ or equivalently as $ R(p^-)_*R(q^-)^!$ by \cite[Theorem IV.6.5]{FS21}. 
	This functor enjoys many compatibilities, in analogy to \cite{Ric19a}, which we will exploit to compute nearby cycles, see \cite[Propositions IV.6.12, IV.6.13, IV.6.14]{FS21}.
	
	\subsection{Over $C$}

	We continue with the notation and denote by $F/\bbQ_p$ a complete discretely valued field with ring of integers $O$ and perfect residue field $k$ of characteristic $p>0$. 
	Also, we fix a complete algebraic closure $C/F$, and a connected reductive $F$-group $G$. 
	
	In this section, we recall the structure of the categories of monodromic sheaves with bounded support $\D(\Hk_{G,C},\Lambda)^{\on{bd}}$ and $\D(\Gr_{G,C}, \Lambda)^{\on{mon},\on{bd}}$ studied in \cite[Section VI]{FS21}. 
	As in \Cref{section_geometry_strata}, for any cocharacter $\lambda\colon \bbG_m\to G_C$, we have the induced $\bbG_{m}^\diamondsuit$-action on $\Gr_{G,C}$, whose attractors only depend on the attracting parabolic $P \subset G_C$.

	In particular, hyperbolic localization gives the constant term functor
	\begin{equation}
	\on{CT}_P\colon \D(\Hk_{G,C},\Lambda)^{\on{bd}}\rightarrow \D(\Gr_{G,C},\Lambda)^{\on{mon},\on{bd}} \xrightarrow{L_{\Gr_G/_C}} \D(\Gr_{M,C},\Lambda)
	\end{equation} 
	providing the main tool to effectively study the category of derived étale sheaves on $\Hk_{G,C}$ as in \cite[Corollary VI.3.5]{FS21}.
	One of the crucial techniques is the following conservativity lemma \cite[Proposition VI.4.2]{FS21} whose proof we sketch for convenience.

	\begin{lemma}
		\label{generic fiber conservativity result}
		Let $T \subset B\subset G_C$ be an arbitrary maximal torus and a Borel containing it. 
		Then $A \in \D(\Hk_{G,C},\Lambda)^{\on{bd}}$ vanishes if and only if $\on{CT}_B(A) \in \D(\Gr_{T,C},\Lambda)$ does.
	\end{lemma}
	
	\begin{proof}
		The proof is done by considering a maximal strata where $A$ is concentrated. 
		This strata is of the form $[\Spd C/(L^+_CG)_{\mu}]$ for the stabilizer of $\mu\in X_*(T)_+$. 
		The attractor of $\Gr_{G, C, \mu}$ at the anti-dominant coweight $-\mu$ with respect to $B$ is an isolated point. 
		Using the $R(p^+)_!(q^+)^*$-version of hyperbolic localization, we see that the fiber of $\on{CT}_B$ over $\mu\in \Gr_T$ agrees with pullback to this point. 
	\end{proof}
	
	This allows us to localize several properties of derived objects in $\D(\Hk_{G,C},\Lambda)^{\on{bd}}$. 
	For instance, $A$ is ULA if and only if $\on{CT}_B(A)$ is, which, in turn, is equivalent to $[\mu]^*A$ being a perfect object for all maps $[\mu]\colon \Spd C \to \Hk_{G}$ with $\mu \in X_*(T)$, see \cite[Propositions VI.6.4, VI.6.5]{FS21}.
	
	Next, we move to the natural perverse t-structure on $\D(\Hk_{G,C}, \Lambda)^{\on{bd}}$, see \cite[Definition/Proposition VI.7.1]{FS21}.
	This is given in terms of the following subcategory
	\begin{equation}
		^p\D^{\leq 0}(\Hk_{G,C},\Lambda)^{\on{bd}} =\{ A \in \D(\Hk_{G,C},\Lambda)^{\on{bd}}\colon j_\mu^*A \in \D^{\leq -\langle 2\rho, \mu \rangle} \text{ for all } \mu \} ,
	\end{equation}
	which determines $^p\D^{\geq 0}(\Hk_{G,C},\Lambda)$ uniquely. 
	Intersecting these two, we get the category of perverse sheaves $\Perv(\Hk_{G,C},\Lambda)$. 
	
	Thanks to \cite[Proposition VI.7.4]{FS21}, the t-structure is preserved under and detected by $\on{CT}_B[\deg_B]$. 
	Here, for any $C$-parabolic $B\subset P$ with Levi quotient $M$, the degree is the locally constant function on $\Gr_{P,C}$ induced by
	\begin{equation}
	\deg_P(\lambda)=\langle 2\rho_G-2\rho_M, \lambda \rangle,
	\end{equation}
	where $\lambda \in X_*(T)$ is a coweight and $\rho_M$ is the half-sum of all $B$-positive $M$-roots. 
	The main geometric fact used in the proof of \cite[Proposition VI.7.4]{FS21} is the equidimensionality of semi-infinite orbits.
	
	When working with torsion coefficients, it is convenient to single out flat perverse sheaves, which are those objects $A$ such that for every $\Lambda$-module $M$ the complex $A\otimes_\Lambda^\mathbb{L}M$ is perverse.
	
	\begin{definition}
		The Satake category $\Sat(\Hk_{G,C},\Lambda)$ is the full subcategory of flat ULA objects in $\Perv(\Hk_{G,C},\Lambda)$.
	\end{definition}
	
	The Satake category is endowed with a monoidal product 
	\begin{equation}
	\star\colon \Sat(\Hk_{G,C},\Lambda)\times \Sat(\Hk_{G,C},\Lambda) \rightarrow \Sat(\Hk_{G,C},\Lambda)
	\end{equation}
	arising from the convolution Hecke stack $\Hk_{G,C} \widetilde{\times}\Hk_{G,C}$, see \cite[Proposition VI.8.1]{FS21}. 
	Due to the fusion interpretation \cite[Definition/Proposition VI.9.4]{FS21}, the monoidal structure is naturally symmetric monoidal.
	
	Taking cohomology of the affine Grassmannian furnishes a fiber functor 
	\begin{equation}
	F\colon \Sat(\Hk_{G,C},\Lambda) \rightarrow \Rep( \Lambda)
	\end{equation} 
	to the category of $\Lambda$-finite locally free modules and the Tannakian formalism gives us an interpretation of these categories in terms of a group of automorphisms.
	
	\begin{theorem}[Fargues--Scholze]
		Fix a compatible system of primitive prime-to-$p$-th roots of unity in $C$.
		Then, the automorphism group of the fiber functor $F$ is naturally isomorphic to the Langlands dual group $\widehat{G}_\Lambda$ formed over $\Lambda$.
	\end{theorem}
	
	One may regard the dual group $\widehat{G}_\Lambda$ as combinatorially defined in terms of root data. 
	The cyclotomic twist in \cite[Theorem VI.11.1]{FS21} is trivialized using the compatible system of roots of unity in $C$.
	
	\subsection{Over $\bar{k}$}
	Let $\calG$ be a parahoric $O$-model of $G$.
	In this section, we look at what happens with the geometric special fiber Hecke v-stack $\Hk_{\calG,\bar{k}}=L^+_{\bar{k}}\calG^\dia\backslash L_{\bar{k}}G^\dia/L^+_{\bar{k}}\calG^\dia$.
	We note that the categories of \'etale sheaves compare well to their scheme-theoretic companions, see \Cref{sec:comp-etale-cohom-1-derived-categories-of-hecke-stacks-are-isomorphic}.

	We continue with the notation and, in addition, fix a maximal $\breve{F}$-split $F$-torus $ S \subset G$ containing a maximal $F$-split torus (see \cite[Proposition 5.1.10]{BT84}) with centralizer $T$, which is a maximal $F$-torus inside $G$, such that their connected N\'eron $O$-models $\calS\subset \calT$ embed into $\calG$. 
	Each parabolic subgroup $P\subset G_{\breve F}$ with Levi $M$ containing $S_{\breve F}$ extends to a diagram of $O$-group schemes $\calM \leftarrow \calP \rightarrow \calG_{\breve{O}}$ by taking flat closures. 
	Again, choosing a cocharacter $\la\co \bbG_m\to \calS_{\breve O}\subset \calG_{\breve O}$ with $M=M_\la$ and $P=P^+_\la$, the formalism of \Cref{section_geometry_strata} applies to define $\bbG_{m}^\diamondsuit$-actions on the pro-smooth cover $\Gr_{\calG,\bar{k}}=\Fl_{\calG,\bar{k}}^\dia$. 
	This gives rise to a constant term functor
	\begin{equation}
	\on{CT}_\calP\colon \D(\Hk_{\calG,\bar{k}}, \Lambda)^{\on{bd}} \rightarrow  \D(\Fl_{\calG,\bar{k}}^{0,\diamondsuit}, \Lambda)^{\on{bd}},
	\end{equation}
	not depending on the choice of $\la$ such that $M=M_\la$ and $P=P^+_\la$.
	Here, the fixed points $\Fl_{\calG,\bar{k}}^{0}$ contain $\Fl_{\calM,\bar{k}}$ as an open and closed sub-ind-scheme by \Cref{BD.attractor.groups.comparison.theorem}, but are strictly bigger unless $\calG_{\breve{O}}$ is special parahoric. 
	
	As in the previous section, we analyse the key properties ULA, flatness and perversity. 
	The crucial step is the following conservativity result.
	
	\begin{proposition}\label{prop_conservative_constant_term}
		An object $A \in \D(\Hk_{\calG,\bar{k}},\Lambda)^{\on{bd}}$ vanishes if and only if $\on{CT}_\calB(A)$ does for every $\breve{F}$-Borel $S_{\breve{F}} \subset B\subset G_{\breve{F}}$.
	\end{proposition}
	
	\begin{proof}
		Just like in \Cref{generic fiber conservativity result}, we argue on a maximal strata of $\Hk_{\calG,\bar{k}}$ where $A$ does not vanish, say one indexed by some $w$. 
		In this case, by \Cref{lemma_attractor_schubert_isolated} there is a choice of $\breve{F}$-Borel $B$ for which the associated attractor intersects $\Fl_{\calG,\bar{k},w}$ in an isolated point. 
		In this case, $\on{CT}_\calB$ agrees with pullback to this point.
	\end{proof}

	\begin{definition}
		An object $A\in \D(\Hk_{\calG,\bar{k}},\Lambda)^{\on{bd}}$ is ULA whenever its pullback to $\Fl_{\calG,\bar{k},w}^\diamondsuit$ is ULA over $\Spd \bar{k}$. 
	\end{definition}
	
	A priori, this notion depends on the choice of left or right trivialization, but it follows a posteriori from \Cref{prop_ula_special_fiber} that it does not, see \cite[Proposition VI.6.2]{FS21}.
	The ULA property interacts very well with the constant term functor:
	\begin{proposition}
		\label{sec:over-k-check-ula-via-constant-terms}
		If $A \in \D(\Hk_{\calG,\bar{k}},\Lambda)^{\on{bd}}$ is ULA, then so is $\on{CT}_\calP(A)$. 
		Conversely, if $\on{CT}_\calB(A)$ is ULA for all Borel subgroups $S_{\breve{F}} \subset B\subset G_{\breve{F}}$, then so is $A$.
	\end{proposition}
	\begin{proof}
		Abusing notation, we also call $A$ the pullback of this object to $\Fl^\diamondsuit_{\calG,\bar{k}}$. 
		By \cite[Theorem IV.2.23]{FS21}, to prove that for $B=A$ or $B=\on{CT}_\calP(A)$, the object $B$ is ULA it is enough to show that
		\begin{equation}
		\label{ula equation}
		p_1^*\bbD(B) \otimes p_2^*B \rightarrow \RHom(p_1^*B,Rp_2^!B)
		\end{equation}
		is an isomorphism. 
		Now, for any pair of flat closures of parabolics $\calP_1$ and $\calP_2$ a direct computation (using properties of hyperbolic localization, cf.\ the proof of \cite[Proposition VI.6.4]{FS21}) shows that 
		\begin{equation}
		\on{CT}_{\calP_1\times \calP_2}(p_1^*\bbD(A) \otimes p_2^*A)=p_1^*\bbD(\on{CT}_{{\calP^-_1}}(A)) \otimes p_2^*\on{CT}_{\calP_2}(A)
		\end{equation}
		and that
		
		\begin{equation}
		\on{CT}_{\calP_1\times \calP_2}(\RHom(p_1^*A,Rp_2^!A))=\RHom(p_1^*(\on{CT}_{\calP^-_1}(A),Rp_2^!\on{CT}_{\calP_2}(A)))
		\end{equation}
		where $\calP_1^-$ is opposite to $\calP_1$. In the forward direction, it is enough to use this for $\calP_1=\calP^-$ and $\calP_2=\calP$. 
		For the converse, we let $K$ denote the cone of \Cref{ula equation}. 
		By the conservativity of \Cref{prop_conservative_constant_term}, it is enough to prove $\on{CT}_{\calB_1,\calB_2}(K)=0$ for all $\calB_1$, $\calB_2$ since these exhaust the Borel subgroups of of $G_{\breve{F}}\times G_{\breve{F}}$. 
		But this follows from the computation above, \cite[Proposition IV.2.19]{FS21} and the hypothesis that $\on{CT}_{\calB_1}(A)$ is ULA. 
	\end{proof}
	
	We prove that ULA sheaves admit an easy description in terms of restrictions to Schubert strata:
	
	\begin{proposition}\label{prop_ula_special_fiber}
		The following are equivalent for an object $A \in \D(\Hk_{\calG,\bar{k}},\Lambda)^{\on{bd}}$:
		\begin{enumerate}
			\item $A$ is ULA.
			\item For all strata of $A$ pullback along
			$	[w]\colon \Spd \bar{k} \rightarrow \Hk_{\calG,\bar{k}}$
			is a perfect complex\footnote{If $\Lambda=\Lambda_0[\ell^{-1}]$ as in \eqref{eq:localization_categories}, then we require the pullback to arise as the $\ell$ localization of a perfect complex over $\Lambda_0$.} in 
			\begin{equation}
			\D(\Spd \bar{k},\Lambda)=\D(\Lambda).
			\end{equation}
			\item The pullback to $\Fl_{\calG,\bar{k}}^\diamondsuit$ lies in  
			\begin{equation}
			\D_{\on{cons}}(\Fl_{\calG,\bar{k}},\Lambda)^{\on{bd}} \subset \D(\Fl_{\calG,\bar{k}}^\diamondsuit,\Lambda)^{\on{ula},\on{bd}},
			\end{equation} 
			where $\D_{\on{cons}}(\Fl_{\calG,\bar{k}},\Lambda)^{\on{bd}}$ is the category of perfect-constructible $\Lambda$-sheaves with bounded support on the ind-scheme $\Fl_{\calG,\bar{k}}$ and the inclusion is the one constructed in \cite[Section 27]{Sch17}.
		\end{enumerate}
	\end{proposition}
	
	\begin{proof}
		The equivalence of (2) and (3) follows from \Cref{sec:comp-etale-cohom-1-derived-categories-of-hecke-stacks-are-isomorphic}, and the fact that the equivalence in \Cref{sec:comp-etale-cohom-1-derived-categories-of-hecke-stacks-are-isomorphic} is compatible with pullback to $\Fl_{\calG}$, respectively to $\Fl_{\calG}^\diamondsuit$ or to strata.
		
		Let $j_w$ denote the inclusion of the strata in $\Hk_{\calG,\bar{k}}$ corresponding to $w$. 
		{For proving that (2) implies (1) one can,} as in \cite[Proposition VI.6.5]{FS21}, reduce to showing that $R(j_{w})_!\Lambda$ is ULA {(for each $w$)}. 
		But their pullback to $\Fl_{\calG,\bar{k}}$ are clearly algebraic and by \Cref{algebraic-is-ula} they are also ULA, see also \cite[Proposition IV.2.30]{FS21}.
		Alternatively one can use the Demazure resolution, compare with \cite[Proposition VI.5.7]{FS21}. 
		
		For the converse implication, we induct on the number of strata where $A$ does not vanish and consider the cone of $R(j_w)_!j^\ast_wA\to A$ for a maximal strata $w$. 
		Indeed, pullback by open immersion preserves being ULA by \cite[Proposition VI.2.13.(i)]{FS21}. 
		Then we apply \cite[Proposition VI.4.1]{FS21} to see that $j^\ast_w A\in \D(\mathrm{Hk}_{\calG,\bar{k},\mu},\Lambda)$ has a perfect stalk, and use that $R(j_w)_!(j_w^\ast A)$ is ULA, by the proven (2) implies (1). 
		Then we can conclude by induction.
		
	\end{proof}

	Arguing as in \cite[Definition/Proposition VI.7.1]{FS21}, we can define a perverse t-structure.
	
	\begin{definition}
		The perverse t-structure on $\Hk_{\calG,\bar{k}}$ is the only such that
		\begin{equation} 
			^p\D^{\leq 0}(\Hk_{\calG,\bar{k}},\Lambda) =\{ A \in \D(\Hk_{\calG,\bar{k}},\Lambda)\colon j_w^*A \in \D^{\leq -l(w)}(\Hk_{\calG,\bar{k}},\Lambda) \text{ for all } w \},
		\end{equation} 
		respectively
		\begin{equation}
		^p\D^{\geq 0}(\Hk_{\calG,\bar{k}},\Lambda) =\{ A \in \D(\Hk_{\calG,\bar{k}},\Lambda)\colon Rj_w^!A \in \D^{\geq -l(w)}(\Hk_{\calG,\bar{k}},\Lambda) \text{ for all } w  \}.
		\end{equation} 
		Perverse sheaves 
		\begin{equation}
		\Perv(\Hk_{\calG,\bar{k}},\Lambda)=\, ^p\D^{\leq 0}(\Hk_{\calG,\bar{k}},\Lambda) \cap\, ^p\D^{\geq 0}(\Hk_{\calG,\bar{k}},\Lambda)
		\end{equation}
		are the heart of the t-structure. 
		Such an $A$ is flat perverse if in addition $A\otimes^{\mathbb{L}}_{\Lambda} M$ is in $\Perv(\Hk_{\calG,\bar{k}},\Lambda)$ for all $\Lambda$-modules $M$. 
	\end{definition}
	
	We note that, in general, there cannot be any degree shifts such that $\on{CT}_\calP[\deg_\calP]$ preserves the perverse t-structure, due to lack of parity. 
	But, we define
	\begin{equation}
	\deg_\calP(\la_I)=\langle 2\rho_G-2\rho_M, \la \rangle
	\end{equation}
	for translation elements $\la_I\in X_*(T)_I$. 
	This is useful for the following result:
	
	\begin{proposition}
		\label{sec:over-k-test-perversity-for-special-parahorics-via-constant-terms}
		Assume that $\calG_{\breve{O}}$ is special parahoric, and let $A \in \D(\Hk_{\calG,\bar{k}},\Lambda)$. 
		Then $A$ is perverse if and only if $\on{CT}_\calB(A)[\deg_\calB] \in \Perv(\Fl_{\calT,\bar{k}}^\dia,\Lambda)$ for all Borel subgroups $S_{\breve{F}} \subset B \subset G_{\breve{F}}$. 
		The same applies to the flat objects.
	\end{proposition}
	
	\begin{proof}
		It suffices to follow the proof of \cite[Proposition VI.7.4]{FS21}. 
		For preserving the t-structure, we use the fact that the non-empty intersections $\calS_{\bar{k},\la_I} \cap \Fl_{\calG, \bar{k},\nu_I}$
		are equidimensional of dimension $\langle 2\rho_G, \lambda+\nu \rangle$, see \Cref{equidimensionality.MV.cycles.lemma}. 
		The converse then follows from \Cref{prop_conservative_constant_term}.
	\end{proof}
	
	In particular, it is now permitted to introduce the Satake category at special level. 
	Notice that there is no hope of such a well-behaved class of objects to exist at arbitrary level, because the quotient $\tilde W/W_\calG$ carries no natural abelian structure.
	
	\begin{definition}
		Let $\calG_{\breve{O}}$ be special parahoric. 
		Then the Satake category $\Sat(\Hk_{\calG,\bar{k}},\Lambda)$ is the full subcategory of $\Perv(\Hk_{\calG,\bar{k}},\Lambda)$ comprised of flat ULA objects.
	\end{definition}
	
	This category lies within the category of perverse sheaves $\Perv(\Hk_{\calG,\bar{k}}^{\on{sch}},\Lambda)$ on the schematic Hecke stack $\Hk_{\calG,\bar{k}}^{\on{sch}}=L^+_{\bar{k}}\calG\backslash L_{\bar{k}}G/L_{\bar{k}}^+\calG$ by \Cref{prop_ula_special_fiber} and \Cref{sec:comp-etale-cohom} for the comparison with sheaves on schematic v-stacks. 
	It carries moreover a monoidal structure given by convolution $\star$.
	
	\subsection{Over $O_C$}\label{over.oc.section}
	Let $f\colon X\to \Spec(O_C)$ be a scheme of finite presentation over $O_C$ and denote by $j$ the inclusion $X_\eta\hookrightarrow X$ of the generic fiber. 
	In \cite[Theorem 1.7]{HS21}, Hansen and Scholze prove that the pullback functor
	\begin{equation}
	j^*\colon \D(X,\Lambda)\to \D(X_{\eta},\Lambda)	
	\end{equation}
	restricts to an equivalence between $f$-ULA and $f_\eta$-ULA objects.
	In the setup of diamonds, the argument for full faithfulness is the same as was explained to us by Scholze, and it consists of proving the adjunction map $A\to Rj_*j^*A$ is an isomorphism.
	
	\begin{lemma}
		Let $X$ be a small v-sheaf over $\Spd(O_C)$ representable in locally spatial diamonds, compactifiable and of finite trascendence degree. 
		Let $A\in \D(X,\Lambda)$ be ULA for the structure map to $\Spd(O_C)$. 
		Then $A\to Rj'_\ast j'^\ast A$ is a isomorphism, where $j':X_\eta\to X$ and $j:\Spd(C)\to \Spd(O_C)$ denote the inclusion of generic fibers.
	\end{lemma}
	\begin{proof}
		By hypothesis $j'^\ast A$ is ULA with respect to $\Spd(C)$.
		In particular, by \cite[Proposition IV.2.19]{FS21} the map 
		\begin{equation}
		j'^\ast A\otimes_\Lambda^{\mathbb{L}} \Lambda\cong \RHom(\mathbb{D}_{X_\eta/\Spd(C)}(j'^\ast A),Rf_\eta^!\Lambda) 
		\end{equation}
		is an isomorphism. Since $j'$ is an open immersion $j'^\ast=Rj'^!$ and $\mathbb{D}_{X_\eta/\Spd(C)}(j'^\ast A)=j'^\ast\mathbb{D}_{X/\Spd(O_C)}(A)$ as follows from \cite[Theorem 1.8.(v)]{Sch17}. We get
		
		\begin{equation}                  
		\begin{aligned}
		Rj'_\ast j'^\ast A & \cong & Rj'_\ast \RHom(j'^\ast\mathbb{D}_{X/\Spd(O_C)}(A),Rf_\eta^!\Lambda) \\
		& \cong & \RHom(\mathbb{D}_{X/\Spd(O_C)}(A),Rj'_\ast Rf_\eta^!\Lambda) \\
		& \cong & \RHom(\mathbb{D}_{X/\Spd(O_C)}(A),Rf^!Rj_\ast\Lambda) \\
		\end{aligned}
		\end{equation}
		the result now follows from the identity $\Lambda=Rj_\ast\Lambda$ and double duality for ULA sheaves.   
	\end{proof}

	In particular, a $f_\eta$-ULA object $A$ comes from a $f$-ULA object if and only if $Rj_*j^\ast A$ is $f$-ULA.

	Below, we prove essential surjectivity for $\Hk_{\calG, O_C}$, the Hecke stack over $\Spd O_C$. 
	For hyperspecial parahoric $\calG$, this is \cite[Corollary VI.6.7]{FS21}. Before doing this, recall that hyperbolic localization allows us to define again a constant term functor
	\begin{equation}
	\on{CT}_\calP\colon \D(\Hk_{\calG,O_C}, \Lambda)^{\on{bd}} \rightarrow  \D(\Gr_{\calG,O_C}^{0}, \Lambda)^{\on{bd}}.
	\end{equation}
	By \cite[Proposition IV.6.12]{FS21}, there is a natural equivalence
	\begin{equation}
	\on{CT}_\calP \circ Rj_{\calG,*} \cong Rj_{\calM,*}\circ \on{CT}_P,
	\end{equation}
	with $j_\calG, j_\calM$ denoting the inclusion of the respective generic fibers.
	Now, we can probe integral ULA objects.
	
	\begin{proposition}\label{ULA_nearby_cycles}
		Consider the inclusion of Hecke stacks $j\co \Hk_{G,C}\to \Hk_{\calG,O_C}$. 
		There is an equivalence
		\begin{equation}
		j^*\co \D(\Hk_{\calG,O_C},\Lambda)^{\on{bd},\on{ula}}\to \D(\Hk_{G,C},\Lambda)^{\on{bd},\on{ula}},
		\end{equation}
		whose inverse functor is $Rj_*$.
	\end{proposition}
	\begin{proof}
		Suppose $A\in \D(\Hk_{G,C},\Lambda)^{\on{bd},\on{ula}}$, it suffices to prove $Rj_*A\in \D(\Hk_{\calG,O_C},\Lambda)^{\on{bd},\on{ula}}$. 
		Let $B$ denote the pullback of $A$ to $\Gr_{\calG,O_C}$, which by definition is ULA. 
		By smooth base change, $Rj_*A$ pulls back to $Rj_*B$ (here we implicitly use \cite[Proposition VI.4.1]{FS21}). 
		By \cite[Theorem IV.2.23]{FS21}, we must show 
		\begin{equation}
		p_1^*\bbD(Rj_*B) \otimes p_2^*Rj_*B \rightarrow \RHom(p_1^*Rj_*B,Rp_2^!Rj_*B)
		\end{equation}
		is an isomorphism. 
		Let $K$ denote the cone of this map. By assumption, and since $j^*=Rj^!$, this map is an isomorphism on the generic fiber. 
		Consequently, $K=i_*L$ for some $L\in D(\Hk_{\calG,\bar{k}},\Lambda)^{\on{bd}}$ and the inclusion $i\co \Hk_{\calG,\bar{k}}\to \Hk_{\calG,O_C}$. 
		We may use the conservativity result \Cref{prop_conservative_constant_term} to prove $L=0$. 
		This reduces us to proving that $\on{CT}_\calB(Rj_*A)=Rj_*\on{CT}_\calB(A)$ is ULA for all Borel subgroups $S_{\breve{F}} \subset B\subset G_{\breve{F}}$. 
		Now, the fixed-point locus $\Gr_{\calG,O_C}^0$ of the action induced by $\calB$ is ind-representable by a locally finite type scheme over $O_C$ of relative dimension $0$. 
		We call this scheme $X$ and let $h\co X_\eta\to X$ the inclusion of generic fibers. 
		By inspection, $\D(\Gr_{\calG,O_C}^0,\Lambda)\cong \D(X,\Lambda)$ and $\D(\Gr_{\calG,C}^0,\Lambda)\cong \D(X_\eta,\Lambda)$. 
		In particular $Rj_*\cong c_X^*Rh_*Rc_{X_\eta,*}$ with notation as in \cite[Section 27]{Sch17} (or as in \eqref{c-functor}). 
		By \cite[Theorem 1.7]{HS21}, $Rh_*$ preserves ULA objects, which allows us to conclude the same holds for $Rj_*$. 
	\end{proof}
	
	\subsection{Nearby cycles}\label{nearby.cycles.section}
	
	We can now look at the nearby cycles functor 
	\begin{equation}
	\Psi_{\calG}:=i^*Rj_*\colon \D(\Hk_{G,C}, \Lambda)^{\on{bd}} \rightarrow \D(\Hk_{\calG,\bar{k}},\Lambda)^{\on{bd}},
	\end{equation}	
	Arising from the diagram
	\begin{equation}
	\Hk_{G,C} \xrightarrow{j} \Hk_{\calG,O_C} \xleftarrow{i} \Hk_{\calG,\bar{k}}
	\end{equation}
	of geometric fibers inclusions of the integral Hecke stack.
	
	\begin{proposition}\label{prop_commutativity_nearby_cycles_constant_term}
		The functor of nearby cycles lies in a natural equivalence
		\begin{equation}
		\on{CT}_\calP[\deg_\calP] \circ \Psi_{\calG} \cong \Psi_{\calM}\circ \on{CT}_P[\deg_P],
		\end{equation} 
		that is, it commutes with the shifted constant term functor.
	\end{proposition}
	
	\begin{proof}
		Without the shift, this is a direct consequence of \cite[Proposition IV.6.12]{FS21}. 
		Using \Cref{BD.attractor.groups.comparison.theorem}, this also shows that $\on{CT}_\calP\circ \Psi_{\calG}$ is supported on the open and closed sub-v-sheaf $\Fl_{\calM,\bar{k}}\subset \Fl_{\calG,\bar{k}}^0$.
		So the shifts agree by definition. 
	\end{proof}
	
	Surprisingly, this commutativity property delivers us a lot of control on the values assumed by $\Psi_\calG$ on the Satake category.
	
	\begin{corollary}\label{ULAness_nearby_cycles}
		Nearby cycles $\Psi_\calG$ restrict to a functor 
		\begin{equation}
		\D(\Hk_{G,C},\Lambda)^{\on{bd},\on{ula}}\rightarrow \D(\Hk_{\calG,\bar{k}},\Lambda)^{\on{bd},\on{ula}} 
		\end{equation}
		and, if $\calG_{\breve{O}}$ is furthermore special parahoric, then it even restricts to
		\begin{equation}
		\Sat(\Hk_{G,C},\Lambda)\rightarrow \Sat(\Hk_{\calG,\bar{k}},\Lambda)
		\end{equation}
		between the Satake categories.
	\end{corollary}
	\begin{proof}
		This follows from \Cref{sec:over-k-check-ula-via-constant-terms} and \Cref{sec:over-k-test-perversity-for-special-parahorics-via-constant-terms}.
		Indeed, those statements allow us to reduce everything to the torus case $\calG=\calT$ which can be handled by hand.  
	\end{proof}
	
	Let us examine the nearby cycles $\Psi_\calG(\Sat(V))$ applied to a Satake object $\Sat(V) \in \Sat(\Hk_{G,C}, \Lambda)$ corresponding to a $\widehat{G}_\Lambda$-representation $V$ with $\mu$ as its highest weight. 
	Given an $F$-Borel $B \subset G$, the commutativity of \Cref{prop_commutativity_nearby_cycles_constant_term} yields
	\begin{equation}
	\on{CT}_\calB[\deg_\calB]\big(\Psi_\calG(\Sat(V))\big)=\bigoplus_{\lambda_I}V(\lambda_I) \cdot \la_I,
	\end{equation}
	where now the $\widehat{G}_\Lambda$-representation is regarded as a $\widehat{T}^I_\Lambda$-representation by restriction. {Here, we use that (by construction) the constant term functor corresponds via geometric Satake to restriction of representations, see \cite[Section VI.11]{FS21}.}
	In particular, we get:
	
	\begin{corollary}\label{corollary_vanishing_cohomology}
		For a $\widehat{G}_\Lambda$-representation $V$ with highest weight $\mu$, the compactly supported cohomology groups 
		\begin{equation}
		\on{H}^l_{c}\!\big(\calS_{\bar{k},w}, \Psi_\calG(\Sat(V))\big) 
		\end{equation}
		vanish for all $l \in \bbZ$ unless $\Fl_{\calG, \bar{k},w} \subset \calA_{\calG,\bar{k},\mu}$.
	\end{corollary}
	\begin{proof}
		This follows from \Cref{representation_theoretic_admissible_locus_lemma}.
	\end{proof}
	We are now finally ready to compute the special fiber of the local model.
	
	\begin{theorem}\label{theorem_special_fiber_admissible}
		There is an equality $\calA_{\calG,\mu}^\diamondsuit=\calM_{\calG, k_E, \mu}$ as sub-v-sheaves of $\Fl_{\calG,k_E}^\dia$.
	\end{theorem} 
	
	\begin{proof}
		By specializing the orbit of $\mu$ under the finite Weyl group, it is easy to see that $\calA_{\calG,\mu}^\diamondsuit$ is contained in the special fiber of $\calM_{\calG,\mu}$.
		By \Cref{corollary_vanishing_cohomology}, it is thus enough to prove that for a maximal stratum in $\calM_{\calG, \bar{k},\mu}$ enumerated by $w$, we have $\on{H}^l_{c}(\calS_{\bar{k},w} , A)\neq 0$ for some $l\in \bbZ$ and for $A:=\Psi_\calG(\Sat(V))$ for some $\widehat G_\Lambda$-representation $V$ with highest weight $\mu$. 
		Using induction on $\mu$, we may and do assume that $w$ lies in the open complement of the closed union of $\calM_{\calG, O_C, \la }$ for all $\la < \mu$. Indeed, if $\la<\mu$, then $\calA_{\calG,\la}^\diamondsuit\subset \calA_{\calG,\mu}^\diamondsuit$. By \Cref{lemma_attractor_schubert_isolated}, our Borel subgroup $S_{\breve{F}} \subset B \subset G_{\breve{F}}$ can always be chosen such that $w$ is an isolated point of the attractor $\Fl_{\calG,\bar{k},w}^{+}$. Since $w$ enumerates a maximal stratum, we also see that $w$ is an isolated point of $\calM_{\calG, \bar{k},\mu}^{+}$, so that $\on{H}^*_{c}(\calS_{\bar{k},w} ,A)=\on{H}^*(\{w\},A)=:A_w$ is the stalk of $A$ at $w$.

		Consider $X=\calM_{\calG, O_C, \mu}\times_{\Spd O_C} U$ where $U$ denotes the analytic locus of the open unit ball $\mathbb{D}_{O_C}^\diamondsuit$.
		Let $g\colon X_C\hookto X$ be the inclusion of the generic fiber. 
		Let $K$ denote a completed algebraic closure of $k\rpot{t}$. 
		We may choose a $\Spd(K)$-valued point $\overline{w}$ of $X$ that lies over $w$. 
		It suffices to prove $A_{\overline{w}}$ is not identically $0$. 
		Since $U$ is smooth over $\Spd O_C$, the smooth base-change theorem and our inductive assumption on $w$ allows us to compute 
		\begin{equation}
		A_{\overline{w}}=(Rg_*\Lambda[\langle 2\rho,\mu\rangle])_{\overline{w}},
		\end{equation}
		where $V$ is chosen to have weight multiplicity $1$ at $\mu$.
		Since $X_C$, $X$ and $K^\dia$ are locally spatial diamonds, we may compute the right-side term site-theoretically. 
		Letting $l:=-\langle 2\rho,\mu\rangle$, we have 
		\begin{equation}
		\on{H}^l(A_{\overline w})=\on{lim}_{W} \on{H}^0(W_C,\Lambda)
		\end{equation}
		where $W$ ranges over étale neighborhoods of $\overline{w}$ in $X$. 
		By \Cref{remark_local_model_base_change} and openness of $X\to \calM_{\calG,O_C,\mu}$, the generic fiber $X_C$ is dense in $X$ which proves that the above expression does not vanish.
	\end{proof}
	
	\subsection{Centrality of nearby cycles}
	In the classical theory, say, over function fields, it is known that nearby cycles on Hecke stacks give central perverse sheaves on partial affine flag varieties, see \cite{Gai01}.
	Centrality holds true in our context as well:
	
	\begin{proposition}\label{nearby_cycles_central}
		For every $A\in \D(\Hk_{G,C},\Lambda)^{\on{bd}, \on{ula}}$ and $B \in \D(\Hk_{\calG,\bar{k}},\Lambda)^{\on{bd}, \on{ula}}$, there is a canonical isomorphism
		\begin{equation}\label{commutativity_iso}
		\Psi_\calG(A) \star B\cong B \star \Psi_\calG(A)
		\end{equation} 
		in $\D(\Hk_{\calG,\bar{k}},\Lambda)$. 
	\end{proposition}
	\begin{proof}
		We can repeat the proof of \cite[Proposition 7.4]{Zhu14} in our context:
		
		Similar to \cite[Definition/Proposition VI.9.4]{FS21}, we work with the convolution integral Hecke stack 
		\begin{equation}
		\Hk_{\calG}^{I; I_1,\ldots, I_k}\to  (\Spd O)^I, 
		\end{equation}
		where $I=I_1\sqcup\ldots\sqcup I_k$ is a finite partitioned index set.
		It parametrizes $\calG$-bundles $\calE_0,\ldots,\calE_k$ over $B^+_\dR$ together with isomorphisms of $\calE_{j-1}$ and $\calE_j$ outside the union of the divisors $\xi_i$ for all $i\in I_j$. 
		We fix $I:=\{1,2\}$ and drop it from the notation. 
		There are three ordered partitions $\{1\}\sqcup\{2\}$, $\{2\}\sqcup \{1\}$ and $\{1,2\}$, leading to the diagram of v-sheaves over $(\Spd O_C)^2$:
		
		\begin{equation}\label{eq:convolution_Hecke_snacks}
		\begin{tikzpicture}[baseline=(current  bounding  box.center)]
		\matrix(a)[matrix of math nodes, 
		row sep=1.5em, column sep=1.5em, 
		text height=1.5ex, text depth=0.45ex] 
		{\Hk_{\calG}^{\{1\}, \{2\}} & \Hk_{\calG}^{\{1, 2\}} & \Hk_{\calG}^{\{2\}, \{1\}} \\ 
							& \Hk_{\calG, O_C}\times \Hk_{\calG,O_C}&  \\}; 
		\path[->](a-1-1) edge node[above] {$m$}  (a-1-2);
		\path[->](a-1-3) edge node[above] {$n$}  (a-1-2);
		\path[->](a-1-1) edge node[below] {$p$\;\;\;\;\;\;\;\;}  (a-2-2);
		\path[->](a-1-3) edge node[below] {\;\;$q$} (a-2-2);
		\end{tikzpicture}
		\end{equation}

		The maps $m, n$ are the natural projections given by remembering $\calE_0$ and $\calE_2$, and are ind-proper, as one sees by pulling back to the
		convolution affine Grassmannian, combine \Cref{SW theorem on parahoric grassmanians} with the proof of \cite[Proposition 20.4.1]{SW20}.
		The maps $p, q$ are given by sending $(\calE_0,\calE_1,\calE_2)$ to the ordered pair $((\calE_0,\calE_1), (\calE_1,\calE_2))$, respectively $((\calE_1,\calE_2), (\calE_0,\calE_1))$, and are pro-(cohomologically smooth) because $L^+_O\calG\to \Spd O$ is so. 
		More precisely, one passes to a bounded part and factors the action of $L^+_O\calG$ through a congruence quotient.

		We may pullback the diagram from \eqref{eq:convolution_Hecke_snacks} along the closed immersion $\Spd O_C\times \Spd k\to (\Spd O_C)^2$ induced by the divisor $\pi=0$ in the second coordinate to obtain a diagram over $\Spd O_C$:
		\begin{equation}\label{eq:convolution_Hecke_snacks_k}
		\begin{tikzpicture}[baseline=(current  bounding  box.center)]
		\matrix(a)[matrix of math nodes, 
		row sep=1.5em, column sep=1.5em, 
		text height=1.5ex, text depth=0.45ex] 
		{\Hk_{\calG}^{\{1\}, \{2\}}|_{\Spd O_C} & \Hk_{\calG}^{\{1, 2\}}|_{\Spd O_C} & \Hk_{\calG}^{\{2\}, \{1\}}|_{\Spd O_C} \\ 
							& \Hk_{\calG, O_C}\times \Hk_{\calG,k}&  \\}; 
		\path[->](a-1-1) edge node[above] {$m_k$}  (a-1-2);
		\path[->](a-1-3) edge node[above] {$n_k$}  (a-1-2);
		\path[->](a-1-1) edge node[below] {$p_k$\;\;\;\;\;\;\;\;}  (a-2-2);
		\path[->](a-1-3) edge node[below] {\;\;$q_k$} (a-2-2);
		\end{tikzpicture}
		\end{equation}

		We note that the maps $m_k,n_k$ are convolution maps over $\Spd k\subseteq \Spd O_C$ and induce isomorphisms over $\Spd C\subseteq \Spd O_C$ that make \eqref{eq:convolution_Hecke_snacks_k} commutative.
		The commutativity yields a canonical isomorphism by using adjunctions
		\begin{equation}\label{eq:canonical_iso_commutativity}
			R m_{k,\eta,*} p^*_{k,\eta}(A\boxtimes B)\cong R n_{k,\eta,*}q_{k,\eta}^*(A\boxtimes B),
		\end{equation}
		which will induce the desired isomorphism \eqref{commutativity_iso} upon applying the nearby cycles for the family $\Hk_\calG^{\{1,2\}}|_{\Spd O_C}$.
		Indeed, since $(Rj_*A)\boxtimes B$ is ULA by \Cref{ULA_nearby_cycles} and \cite[Corollary IV.2.25]{FS21} for outer tensor products, it is still ULA after cohomologically smooth pullback along $p_k, q_k$ and proper pushforward along $m_k,n_k$ (here, we use that the support of $A,B$ is bounded).
		Thus, \eqref{eq:canonical_iso_commutativity} canonically extends integrally to an isomorphism
		\begin{equation}
			R m_{k,*} p^*_{k}(Rj_*A\boxtimes B)\cong R n_{k,*}q_{k}^*(Rj_*A\boxtimes B),
		\end{equation}
		yielding \eqref{commutativity_iso} after restriction to the special fiber.	
	\end{proof}
	
	The following would be a natural reinforcement of the previous proposition to also preserving perversity. 
	For schemes, nearby cycles always preserve perversity \cite[Lemma 6.3]{HS21}.
	In our setting, this is not immediate and would transport Gaitsgory's central functor \cite{Gai01} to the $p$-adic context.
	
	\begin{conjecture}\label{conjecture_perversity_nearby_cycles}
		For every $A\in \D(\Hk_{G,C},\Lambda)$, the $\Lambda$-flat central sheaf $\Psi_\calG(A) \in \D(\Hk_{\calG,\bar{k}},\Lambda)^{\on{bd},\on{ula}}$ is perverse.
	\end{conjecture}
	
	By a combination of \Cref{ULAness_nearby_cycles} and \Cref{nearby_cycles_central}, we know that \Cref{conjecture_perversity_nearby_cycles} holds true whenever $\calG_{\breve{O}}$ is special parahoric. 
	Also, using the representability in \Cref{SW_conjecture_intro} and a comparison with schematic nearby cycles (see \cite[Proposition 27.6.]{Sch17}), the conjecture holds true whenever $\mu$ is minuscule.
	In general, we lack tools to verify \Cref{conjecture_perversity_nearby_cycles} -- the shifted constant term functor appears to be insufficient -- but one still expects some form of Artin vanishing to hold in this very particular context of the Hecke stack.
	
	\begin{remark}
		After the first version of this paper was written, \Cref{conjecture_perversity_nearby_cycles} was proven in \cite[Theorem 4.17]{ALWY23}.
		We remark that \Cref{conjecture_perversity_nearby_cycles} plays a role in the proof of unibranchness of local models in \cite[Theorem~1.3]{GL24}, and thereby in the proof of \Cref{SW-conjecture} for $p=2,3$ as explained in the text following \Cref{SW_conjecture_intro}.
	\end{remark}

	\section{Minuscule implies representable}\label{section_minuscule_representable}

	Our goal in this section is to prove the Scholze--Weinstein conjecture on minuscule local models \cite[Conjecture 21.4.1]{SW20} as stated in \Cref{SW_conjecture_intro}.
	This is sharp, see \Cref{prop_non_representability_schubert_diamond}.
	We verify the representability part without any assumption on the prime $p$ or the pair $(\calG,\mu)$, thereby showing the existence of weakly normal projective $O_E$-schemes $\calM_{\calG, \mu}^{\on{sch}}$ with natural, equivariant isomorphisms $(\calM_{\calG, \mu}^{\on{sch}})^\diamondsuit\cong \calM_{\calG,\mu}$ in all cases.
	
	As for their geometry, we show under \Cref{hyp_wild_odd_unitary} and \Cref{hyp_wild_triality} that the special fiber is given by $\calA_{\calG, \mu}^{\on{can}}$, in particular reduced and even weakly normal.	
	This implies the geometry part of the Scholze--Weinstein conjecture under those assumptions, see the discussion after \Cref{SW-conjecture}.
	
	Recall that our strategy for representability involves specializations triples. 
	Since explicitly calculating the specialization map seems very hard, we need to consider convolutions of local models, so as to partially resolve $\calM_{\calG,\mu}$ and understand their integral sections better.

	\subsection{Convolution}
	\label{sec:convolution}
	
	We continue to denote by $F/\bbQ_p$ a complete discretely valued field with ring of integers $O$ and perfect residue field $k$ of characteristic $p>0$. 
	Fix a completed algebraic closure $C/F$, and a connected reductive $F$-group $G$ with parahoric $O$-model $\calG$. 
	Also, we fix an auxiliary maximal $\breve{F}$-split $F$-torus $S\subset G$ whose connected N\'eron model $\calS$ embeds in $\calG$, see \cite[Proposition 5.10]{BT84}, and denote by $T$ its centralizer with connected N\'eron model $\calT\subset \calG$. Additionally, we fix an auxiliary $\breve{F}$-Borel $ T_{\breve{F}} \subset B \subset G_{\breve{F}}$.
	
	When proving the representability of the v-sheaf local models, it is not difficult to reduce to the case that $G$ is the Weil restriction of a split group (see \Cref{assumption_iwahori_res_split} and the last paragraph \eqref{closed-immersion-of-groups} of the proof of \Cref{thm_local_model_representable}). 
	In this case, it will be helpful to partially resolve the local model via convolution.
	We recall that the Beilinson--Drinfeld Grassmannian admits the following convolution variant
	\begin{equation}
	\Gr_{\calG} \widetilde{\times} \dots \widetilde{\times} \Gr_{\calG}:=L_O\calG\times^{L^+_O\calG}\cdots \times^{L^+_O\calG} \Gr_{\calG}\to \Spd O,
	\end{equation}
	which, in terms of torsors, parametrizes successive modifications of $\calG$-torsors together with a generic trivialization of the last. 
	It admits natural closed sub-v-sheaves
	\begin{equation}
	\calM_{\calG, \mu_\bullet}:= \calM_{\calG, \mu_1} \widetilde{\times} \dots \widetilde{\times} \calM_{\calG, \mu_n},
	\end{equation}
	for any sequence $\mu_\bullet=(\mu_1, \dots, \mu_n)$ of $B_C$-dominant coweights $\mu_i$ of $T_C$, after base change to $\Spd O_{E}$, where $E$ is the reflex field of $\mu_\bullet$. We will still call them (convolution) local models for simplicity. 
	More precisely, denote by $\widetilde{\calM_{\calG,O_C,\mu_i}}$ the preimage in $L_{O_C}\calG$ of $\calM_{\calG,O_C,\mu_i}\subset \Gr_{\calG,O_C}$, which is an $L_{O_C}^+\calG$-torsor over $\calM_{\calG,O_C,\mu_i}$. Then
	\begin{equation}
	\label{eq:4}
	\calM_{\calG, O_C, \mu_\bullet}=\widetilde{\calM_{\calG,O_C,\mu_1}}\times^{L^+_{O_C}\calG} \cdots \times^{L^+_{O_C}\calG}\widetilde{\calM_{\calG,O_C,\mu_{n-1}}}\times^{L^+_{O_C}\calG} \calM_{\calG,\mu_n}.
	\end{equation}
	This presentation is not ``minimal'' in the following sense: 
	Namely, given a contracted product $X\times^H Y$ in any topos and a normal subgroup $N\subset H$ acting trivially on $Y$, then the natural map
	\begin{equation}
	X\times^H Y\to X/N\times^{H/N}Y
	\end{equation}
	is an isomorphism. Hence, in \eqref{eq:4} we may replace $L^+_{O_C}\calG$ by some sufficiently large congruence quotient, and accordingly the torsors $\widetilde{\calM_{\calG,O_C,\mu_i}}$ by their pushforwards to these congruence quotients.
	Let us note that the multiplication
	\begin{equation}
	\calM_{\calG,O_C,\mu_\bullet}\to \Gr_{\calG,O_C}
	\end{equation}
	has image $\calM_{\calG,O_C,|\mu_\bullet|}$ with $|\mu_\bullet|:=\mu_1+\ldots +\mu_n$, and can therefore be regarded as a (partial) resolution of the latter. 
	Regarding the structure of the convolution local models, we can record the following. 
	
	\begin{lemma}
		\label{sec:convolution-1-convoluted-local-models-are-kimberlites}
		The convolution local model $\calM_{\calG,\mu_\bullet}$ is a proper, flat $\pi$-adic kimberlite over $\Spd O_{E}$ with topologically dense generic fiber. 
	\end{lemma}
	\begin{proof}
		In order to prove that $\calM_{\calG,\mu_\bullet}$ is a proper flat $\pi$-adic kimberlite, we must first show that this proper v-sheaf is v-formalizing. It is difficult however to work with the expression for the convolution product, because the appearing $L^+_{O_C}\calG$ is never v-formalizing. Indeed, already the quotient $\calG^\diamondsuit$ is the analytic v-sheaf attached to an {\it affine} scheme, and the geometric points that formalize factor through the inclusion $\calG^\diamond\subseteq \calG^\diamondsuit$.
		For this reason, we replace the universal $L^+_{O_C}\calG$-torsors by the corresponding $W^+_{O_C}\calG$-torsors, see \cite[Definition 5.7]{Gle20} and denote by $\calM_{\calG, O_C,\mu_\bullet}^{\on{for}}$ the corresponding convolution. 
		From the presentation 
\begin{equation}
	\label{eq:4extra}
	\calM^{\on{for}}_{\calG, O_C, \mu_\bullet}=\widetilde{\calM^{\on{for}}_{\calG,O_C,\mu_1}}\times^{W^+_{O_C}\calG} \cdots \times^{W^+_{O_C}\calG}\widetilde{\calM^{\on{for}}_{\calG,O_C,\mu_{n-1}}}\times^{W^+_{O_C}\calG} \calM_{\calG,\mu_n}.
	\end{equation}
	it is not hard to see that $\calM^{\on{for}}_{\calG, O_C, \mu_\bullet}$ is v-formalizing since each individual $\widetilde{\calM^{\on{for}}_{\calG,O_C,\mu_i}}$ is already v-formalizing (see \cite[Proposition 5.19]{Gle20} for a similar argument).  
		The natural map
		\begin{equation}
		\calM_{\calG, O_C,\mu_\bullet}^{\on{for}} \to \calM_{\calG, O_C,\mu_\bullet}
		\end{equation}
		is an isomorphism because both v-sheaves are qcqs over $\Spd O_C$ and have the same geometric points, compare with \cite[Proposition 5.19]{Gle20}. It is now easy to show that it is formally separated and formally adic, with representable special fiber, see \Cref{prop_conv_LM_fibers}.
		Via projection to the first factor we have a map
		\begin{equation}
		\calM_{\calG, O_C,\mu_\bullet}\to \calM_{\calG, O_C,\mu_1},
		\end{equation}
		which splits after pullback to $\widetilde{\calM_{\calG, O_C,\mu_1}}$ into the projection of the product $\widetilde{\calM_{\calG, O_C,\mu_1}}\times_{\Spd O_C} \calM_{\calG, O_C,(\mu_{2},\ldots, \mu_{n})}$. 
		Hence, after replacing $L^+_{O_C}\calG$ (and hence $\widetilde{\calM_{\calG, O_C,\mu_1}}$) by a sufficiently large congruence quotient, we can deduce that $\calM_{\calG, O_C,\mu_\bullet}$ has dense generic fiber by induction on $n$, \Cref{LM_kimberlite} and preservation of closures under open maps.
	\end{proof}

	From now on, we will always consider sequences of minuscule dominant coweights $\mu_\bullet=(\mu_1, \dots, \mu_n)$ whose sum 
	\begin{equation}\label{eq:coweight_assumption}
	\lvert \mu_\bullet \rvert=\mu_1 +\dots +\mu_n 
	\end{equation}
	is still minuscule. 
	Basically, this means that the support of each $\mu_i$ lies in disjoint irreducible components of the Dynkin diagram. 
	We say that a coweight is \textit{tiny} if it is minuscule and its support is contained in at most one irreducible component.
	
	\begin{proposition}\label{prop_conv_LM_fibers}
		Both fibers of $\calM_{\calG,O_C,\mu_\bullet}\to \Spd O_C$ are representable. More precisely, we have isomorphisms
		\begin{equation}
		\calM_{\calG,O_C,\mu_\bullet}|_{\Spd C} \cong \calF_{G, C, \mu_\bullet}^\diamondsuit \cong \calF_{G, C, \mu_1}^\diamondsuit \times \dots \times \calF_{G, C, \mu_n}^\diamondsuit,
		\end{equation}
		and also 
		\begin{equation}
		\calM_{\calG,O_C,\mu_\bullet}|_{\Spd \bar k} \cong \calA_{\calG, \mu_\bullet}^\diamondsuit, 
		\end{equation}
		where on the right $\calA_{\calG, \mu_\bullet}$ denotes the convolution $\calA_{\calG, \mu_1} \widetilde{\times} \dots \widetilde{\times} \calA_{\calG, \mu_n}$.
	\end{proposition}
	
	\begin{proof}
		The description of the generic fiber via convolution is formal, and it is formal (using \Cref{theorem_special_fiber_admissible}) that $\calM_{\calG,O_C,\mu_{\bullet}}|_{\Spd k}$ is the convolution of the $\calA_{\calG,\mu_i}^{\diamondsuit}$. Using \Cref{sec:comp-etale-cohom-2-diamondsuit-preserves-finitely-presented-v-covers} this convolution identifies with $\calA_{\calG,\mu_\bullet}^\diamondsuit$.
	\end{proof}
	
	We are aiming to carefully write down certain ``minimal'' $\calG_{\on{ad}, O_C}^{>i,\diamondsuit}$-torsors $\calM^{\on{tor}}_{\calG, O_C, \mu_i}\to \calM_{\calG, O_C, \mu_i}$ for some associated smooth connected groups $\calG_{\on{ad},O_C}^{>i}$, such that
	\begin{equation}
	\calM^{\on{tor}}_{\calG,O_C, \mu_1} \times^{\calG_{\on{ad}, O_C}^{>1,\diamondsuit}} \dots \times^{\calG_{\on{ad}, O_C}^{>i-1,\diamondsuit}} \calM^{\on{tor}}_{\calG, O_C, \mu_i} \times^{\calG_{\on{ad}, O_C}^{>i,\diamondsuit}} \dots \times^{\calG_{\on{ad}, O_C}^{>n-1,\diamondsuit}}\calM^{\on{tor}}_{\calG, O_C, \mu_n}
	\end{equation}
	recovers the convolution local model $\calM_{\calG, O_C, \mu_\bullet}$. 
	We begin by introducing the group schemes $\calG_{\on{ad}, O_C}^{>i}$.
	
	\begin{lemma}\label{lem_min_adjoint_quotient_over_integers}
		Let $\mu_{>i}=\mu_{i+1}+\dots + \mu_n$ and let $G_{\on{ad}, C}^{>i}$ be the quotient of $G_C$ by the intersection of all conjugates of $P_{\mu_{>i}}^{-}$. Then, $G_{\on{ad}, C}^{>i}$ acts faithfully on $\calF_{G, C, \mu_{>i}}$.
		
		Furthermore, if we let $\calG_{\on{ad}, O_C}^{>i}$ be the unique fppf quotient\footnote{Beware that the morphism of parahoric group schemes $\calG \to \calG_{\on{ad}}$ is not always an fppf surjection.} of $\calG_{\on{ad},O_C}$ with generic fiber $G^{>i}_{\mathrm{ad},C}$, then $\calG_{\on{ad}, O_C}^{>i}$ acts on $\calM_{\calG, O_C, \mu_{>i}}$, and its fibers are smooth, affine, connected with trivial center.
	\end{lemma} 
	
	\begin{proof}
		The claims on $G_{\on{ad},C}^{>i}$ follow from \Cref{sec:convolution-1-convoluted-local-models-are-kimberlites}.
		The smooth group scheme quotient with the asserted properties obviously exist, due to \cite[Proposition 1.7.6]{BT84}, and it clearly inherits connected fibers. The generic fiber is clearly adjoint: apply semi-simplicity of $G_{\on{ad}}$. However, it is more delicate to show that the special fiber is adjoint.
		
		Assume without loss of generality (\Cref{prop_LM_basic_facts}) that $G=\Res_{F'/F}G'$ is an adjoint $F$-simple group with $G'$ absolutely simple, and $F^\prime$ a finite field extension of $F$. 
		So, $\mathcal G=\Res_{O_{F'}/O_F}\calG'$ for a unique parahoric group scheme $\calG^\prime$ of $G'$, see \cite[Proposition~4.7]{HR20}. 
		Then, a calculation shows that
		\begin{equation}
		\calG^{>i}= \Res_{A_i/O_C} \calG_{A_i}',
		\end{equation}
		where $A_i$ is the finite $O_C$-algebra obtained as the image of $O_{F'} \otimes_{O} O_C$ in the product of those copies of $C$ indexed by the support of $\mu_{>i}$. Indeed, this smooth connected group scheme has the desired universal property and its special fiber is adjoint by \cite[Proposition A.5.15 (1)]{CGP15}. (Notice that the reducedness hypothesis is superfluous for calculating the center.)
	\end{proof}
	
	Now, we come to the definition of the torsors.	
	
	\begin{definition}\label{def_torsor_minuscule_local_model} 
		The v-sheaf $\calM^{\on{tor}}_{\calG, O_C, \mu_i}$ is defined as the pushforward to $(\calG_{\on{ad}, O_C}^{>i})^\diamondsuit$ of the natural $L^+_{O_C}\calG$-torsor $\widetilde{\calM_{\calG, O_C,\mu_i}}$ over $\calM_{\calG, O_C, \mu_i}$. 
	\end{definition}

	Just like \Cref{thm_local_model_representable} proposes that the v-sheaves $\calM_{\calG, O_C, \mu_i}$ have algebraic models, one might also expect that the torsors $\calM^{\on{tor}}_{\calG, O_C, \mu_i}$ also have natural algebraic models defined over the ring of integers of some reflex field and this will be an integral part of our strategy. 
	Since the fibers are relatively easy to understand by the Demazure resolution, we now focus on the behavior over a certain finite union of $\calG_{O_C}^\dia$-orbits.
	
	Recall that for any $\la\in W_0 \cdot \mu$, where $W_0$ is the Weyl group of $(G_{\breve{F}},S_{\breve{F}})$, the induced point $[\la]\co \Spd C\to \Gr_{G,C,\mu}$ uniquely extends to a point $[\la]\co \Spd O_C \to \calM_{\calG,O_C,\mu}$ by \Cref{LM_kimberlite}.

	\begin{definition}\label{def_open_semi_homog_fixers}
		Let $\calM_{\calG, \mu}^\circ \subset \calM_{\calG, \mu}$ be the unique sub-v-sheaf whose base change $\calM_{\calG, O_C, \mu}^\circ\subset \calM_{\calG, O_C, \mu}$ to $\Spd O_C$ is given by the finite (non-disjoint) union
		\begin{equation}
		\calM_{\calG, O_C, \mu}^\circ= \bigcup_{\lambda \in W_0\cdot \mu}\calG_{O_C}^\diamondsuit\cdot [\la].
		\end{equation}
	\end{definition}

	We recall that the elements $\la\in W_0\cdot \mu$ are the $\breve{F}$-rational conjugates of $\mu$ in $X_*(T)$ and correspond to the open Schubert orbits in the $\mu$-admissible locus, see the discussion after \Cref{admissible_locus_definition}.
	Also, it is easy to see and left to the reader that the definition of $\calM_{\calG, \mu}^\circ$ does not depend on the choice of the auxiliary maximal $\breve{F}$-split $F$-torus $S \subset G$ whose connected Néron model $\calS$ embeds in $\calG$. 
	As we will see in the following lemma, the stabilizer of rational conjugates $[\la]$ is actually representable and well behaved.
	
	\begin{lemma}\label{lem_fixer_lambda_scheme_smooth}
		Let $\calP_\la^-$ be the flat closure in $\calG_{O_C}$ of the repeller parabolic $P_\la^- \subset G_C$ defined by $\la$. Then
		\begin{enumerate}
			\item $\calP_\la^{-,\diamondsuit}$ is the $\calG^\diamondsuit_{O_C}$-fixer of $\la$ inside $\calM_{\calG, O_C, \mu}$.
			\item $\calP_\la^-\to \Spec(O_C)$ is smooth affine with connected fibers. 
		\end{enumerate}
	\end{lemma}
	
	\begin{proof}
		By topological flatness, it is clear that $\calP_\la^{-,\diamondsuit}$ fixes $\la$. 
		For dimension reasons, the special fiber of $\calP_\la^-$ is equal to the fixer in $\calG_{O_C}$ of $\la_I$ in the affine flag variety $\Fl_{\calG,\bar k}$, which in particular shows that the special fiber of $\calP_\la^-$ is connected. 
		Having described the special fiber of $\calP_\la^-$, we see that all $(K,K^+)$-valued points of the $\calG^\diamondsuit_{O_C}$-fixer of $\la$ inside $\calM_{\calG, O_C, \mu}$ actually belong to the closed subgroup $\calP_\la^{-,\diamondsuit}$, so it necessarily lies in that closed subgroup. 
		This shows the first statement.
		
	For the second statement, it suffices to verify smoothness of $\calP_\la^-$ since by construction $\calP_\la^-$ is closed in $\calG_{O_C}$ and in particular affine. 
		(Observe that the conjugation action of $\la$ on $G_C$ does not always extend to $\calG_{O_C}$, so $\calP_\la^-$ is not a repeller subgroup.)
		By \cite[Corollaire 2.2.5]{BT84}, we can do this by restricting to the flat closures of $a$-root groups of $P_\la^-$ with respect to the (non-maximal) split torus $S_C$. 
		Using the structure of $\calU_a \subset \calG$ defined over $O$, we see that this amounts to check that the morphism
		\begin{equation}
		\Res_{B/A} \bbA^1 \to \Res_{C/A} \bbA^1
		\end{equation}
		induced by a surjection $B\to C$ of finite free $A$-algebras is a smooth cover. 
		Indeed, the $a$-root group of $\calP_\la^-$ decomposes scheme-theoretically as a product of fibers of such morphisms.
	\end{proof}

	\begin{remark}
		\label{intermediate-rep}
		It follows from \Cref{lem_fixer_lambda_scheme_smooth} that the v-sheaf $\calM_{\calG, O_C, \mu}^\circ$ admits an open cover by v-sheaves of the form $(\calG_{O_C}/\calP_\la^-)^\Diamond$. 
		It is also clear that the glueing morphisms are schematic. 
		Overall, this shows that $\calM_{\calG, O_C, \mu}^\circ$ is represented by a scheme.
	\end{remark}
	
	We want to uniquely characterize the left $\calG^\diamondsuit_{O_C}$-equivariant right $\calG^{>i,\diamondsuit}_{O_C}$-torsor
	\begin{equation}
	\label{eq:5}
	\calM^{\on{tor},\circ}_{\calG, O_C, \mu_i}:=\calM^{\on{tor}}_{\calG, O_C, \mu_i}\times_{\calM_{\calG, O_C, \mu_i}}\calM^{\circ}_{\calG, O_C, \mu_i}.
	\end{equation}
	For this, we use the following abstract statement.

	\begin{lemma}
		\label{sec:convolution-1-abstract-lemma-on-left-right-torsors}
		Let $\mathfrak{X}$ be any topos, and let $J,A\in \mathfrak{X}$ be group objects. 
		Let $Z:=J/P$ be an orbit for $J$.
		\begin{enumerate}
			\item The groupoid of left $J$-equivariant right $A$-torsors $\calT$ over $Z$ is equivalent to the groupoid of right $A$-torsors $\calS$ over a terminal object equipped with a morphism of groups $\varphi_{\calT}\colon P\to \mathrm{Aut}_A(\calS)$, with equivalence given by sending a $\calT$ to the fiber $\calS:=\calT_{1\cdot P}$ of $1\cdot P\in Z$ with its action by $P$.
			\item If $P$ is self-normalizing, then for each left $J$-equivariant right $A$-torsor $f\co \calT{\to} Z$ each morphism $\sigma\colon \calT\to \calT$, which is equivariant for $J$ and $A$, is automatically a morphism of left $J$-equivariant right $A$-torsors, that is, $f\circ \sigma=f$. 
			If furthermore $\varphi_{\calT}$ has trivial centralizer, then $\sigma=\mathrm{id}$.
		\end{enumerate}
	\end{lemma}
	Note that if $\calS$ is trivial, then $\mathrm{Aut}_A(\calS)\cong A$.
	\begin{proof}
		For (1), it suffices to note that an inverse is given by sending a pair $(\calS, P\to \mathrm{Aut}_A(\calS))$ to the contracted product $J\times^P\calS$. 
		For (2), one notes that by $A$-equivariance $\sigma$ descents to an $J$-equivariant morphism of $Z$. 
		As $P$ is self-normalizing, this morphism must be the identity.
	\end{proof}
	
	We can therefore conclude that
	the left $\calG^\diamondsuit_{O_C}$-equivariant $\calG^{>i,\diamondsuit}_{O_C}$-right torsor $\calM^{\on{tor},\circ}_{\calG, O_C, \mu_i}$ is governed by certain morphisms of groups
	\begin{equation}
	\calP_{\la_i}^\diamondsuit \rightarrow (\calG^{>i}_{\on{ad}, O_C})^\diamondsuit
	\end{equation}	
	for all $\breve{F}$-rational conjugates $\lambda_i$ of $\mu_i$ by applying \Cref{sec:convolution-1-abstract-lemma-on-left-right-torsors}. 
	In our situation, the generic fiber of $\calM^{\on{tor},\circ}_{\calG, O_C, \mu_i}$ is quite well-understood, and thus we will apply the following general result describing extensions of a given left equivariant right torsor.
	
	\begin{lemma}
		\label{sec:convolution-2-extensions-of-given-generic-left-h-right-a-torsor}
		We use the notation of \Cref{sec:convolution-1-abstract-lemma-on-left-right-torsors}. 
		Furthermore, let $Y\in \mathfrak{X}$ be any object, and denote by a subscript $(\str)_Y$ the base change to $Y$. 
		Assume that $\widetilde{\calT}\to Z_Y$ is a left $J_Y$-equivariant right $A_Y$-torsor with associated tuple $(\widetilde{\calS},\varphi_{\widetilde{\calT}}\colon P_Y\to \mathrm{Aut}_{A_Y}(\widetilde{\calS}))$. 
		Then the groupoid of pairs of a left $J$-equivariant right $A$-torsors $\calT$ over $Z$ with an isomorphism $\calT_Y\cong \widetilde{\calT}$ identifies with the groupoid of following data: $\calS$ an $A$-torsor over a terminal object of $\mathfrak{X}$, an identification $\gamma\colon \calS_Y\cong \widetilde{\calS}$, and a morphism of groups $\varphi\colon P\to \mathrm{Aut}_A(\calS)$ such that $\varphi_Y$ agrees with $\varphi_{\widetilde{\calT}}$ under the identification $\mathrm{Aut}_{A_Y}(\calS_Y)\cong \mathrm{Aut}_{A_Y}(\widetilde{\calS})$ induced by $\gamma$.
	\end{lemma}
	\begin{proof}
		This follows from \Cref{sec:convolution-1-abstract-lemma-on-left-right-torsors}.
	\end{proof}
	
	In our case, the $A$-torsors $\widetilde{\calS}\cong A_{Y\times Z}, \calS\cong A_Z$ we are interested in are trivial, and the morphism $A(\mathfrak{X})\to A(Y)$ is injective. 
	Then $\varphi$ is determined by $\gamma$ (and $\varphi_{\widetilde{\calT}}$), and after fixing a section $z\in \widetilde{\calS}(Y)\subset \widetilde{\calT}(Y)$ we get thus an injection from isomorphism classes of pairs $(\calT, \calT_Y\cong \widetilde{\calT})$ to $A(Y)/A(\mathfrak{X})$ by sending $(\calT,\calT_Y\cong \widetilde{\calT})$ to the class of $z^{-1}y_Y\in A(Y)/A(\mathfrak{X})$, where $y\in \calS(\mathfrak{X})\subset \calT(Y)$ is any section.
	If $\varphi_{\widetilde{\calT}}$ has trivial centralizer, the groupoid of such pairs is equivalent to the groupoid of left $J$-equivariant right $A$-torsors $\calT$ over $Z$, which are isomorphic to $\widetilde{\calT}$ over $Y$ (but we do not fix such an isomorphism).
	
	Applying these considerations to the orbits in $\calM^{\circ}_{\calG,O_C,\mu_i}$ with its left $\calG_{O_C}^\diamondsuit$-equivariant right $\calG_{\mathrm{ad},O_C}^{>i}$-torsor $\calM^{\on{tor},\circ}_{\calG,\mu_i}$ having generic fiber $\calF^{\mathrm{tor}}_{G,C,\mu_i}$, we need therefore to
	\begin{itemize}
		\item fix base points over $O_C$ in the orbits in $\calM^{\circ}_{\calG, O_C,\mu_i}$,
		\item fix $C$-points $\tilde{\lambda}_j\in \calF^{\mathrm{tor}, \dia}_{G,C, \mu_i}$ lying over the generic fibers of the chosen base points in $\calM^{\circ}_{\calG, O_C,\mu_i}$,
		\item find sections $y_j\in \calM^{\circ, \mathrm{tor}}_{\calG, O_C,\mu_i}(O_C)$ lying over the chosen base points in $\calM^{\circ}_{\calG, O_C,\mu_i}$,
		\item calculate the difference of $y_j^{-1}\tilde{\lambda}_j \in G_{\mathrm{ad}, C}^{>i}(C)$, which yields the desired $n$ classes modulo $\calG_{\mathrm{ad}, O_C}^{>i}(O_C)$.
	\end{itemize}
	The first point is easy as we can take the $[\lambda]\colon \Spd O_C\to \calM_{\calG,O_C,\mu_i}^{\circ}$ with $\lambda$ running through the $\breve{F}$-rational Weyl group conjugates of $\mu_i$. For the second point, we can fix a (suitable) uniformizer $\xi\in B_{\dR}^+(C)$ and consider the images of the $\lambda(\xi)\in LG(C)$ in $\calF_{\calG,C,\mu_i}^{\mathrm{tor},\diamondsuit}$.   
	To state the outcome, we have to make the following definition.
	
	\begin{definition}
		\label{different-definition}
		Let $\nu\in X_*(T)$. 
		The different $\delta_G(\nu)$ is the class in $T(C)/\calT(O_C)$ of
		\begin{equation}
		\prod_{\sigma \neq 1}\nu^\sigma(\pi_E^\sigma-\pi_E),
		\end{equation} 
		where $F\subset E\subset C$ is the reflex field of $\nu$, $\pi_E\in O_E$ some uniformizer and $\sigma$ varies over the non-trivial cosets in the quotient $I_F/I_E$ of the absolute inertia groups.
	\end{definition}
	
	
	For a uniformizer $\pi_E\in O_E$ for a finite field extension $E/F$, contained in $C$, we denote by $\pi_E^\flat\in O_C^\flat$ a chosen sequence of compatible $p^n$-roots of $\pi_E$. 
	Recall that $\xi_E:=\pi_E-[\pi_E^\flat]\in W_{O_E}(O_C^\flat)$ maps to a uniformizer of $B_\dR^+(C)$. 
	We can now define the $\tilde{\lambda}$ as the images of $\lambda(\xi_F)\in \calF^{\mathrm{tor,\diamondsuit}}_{\calG,C, \mu_i}(C)$ for any $\lambda\in W_0\cdot \mu_i$. 
	
	\begin{proposition}\label{prop_torsor_local_model_different}
		The v-sheaf $\calM^{\on{tor},\circ}_{\calG, O_C, \mu_i}$ is the unique left $\calG_{O_C}$-equivariant right $\calG_{\on{ad}, O_C}^{>i}$-torsor over $\calM^{\circ}_{\calG,O_C, \mu_i}$ with generic fiber isomorphic to $\calF^{\mathrm{tor},\diamondsuit}_{\calG,C,\mu_i}$ determined by the images of $\delta_G(\lambda)$ in $G_{\on{ad}, C}^{>i}(C)/\calG_{\on{ad}, O_C}^{>i}(O_C)$ for $\lambda\in W_0\cdot \mu_i$ (and the above choices for the $\tilde{\lambda}$'s).
	\end{proposition}
	
	\begin{proof}
		Let us fix some $\lambda\in W_0\cdot \mu_i$. 
		Consider the morphism 
		\begin{equation}
		T^\prime:=\Res_{\breve{E}/\breve{F}} \bbG_m \rightarrow G_{\breve{F}}
		\end{equation}
		of algebraic groups induced by $\lambda$ as follows: compose the Weil restriction of its $\breve{F}$-realization $T^\prime:=\Res_{\breve{E}/\breve{F}}\bbG_{m} \to \Res_{\breve{E}/\breve{F}}T_{\breve{E}}$ with the norm map $\Res_{\breve{E}/\breve{F}}T_{\breve{E}} \to T$. 
		Note that $\lambda\colon \mathbb{G}_{m,{\breve{E}}}\to G_{\breve{E}}$ can be reconstructed from this composition by restricting its base change to $\breve{E}$ to the first factor. 
		Set $\mathcal{T}^\prime:=\mathrm{Res}_{\breve{O}_{E}/\breve{O}}\mathbb{G}_m$. 
		We will now construct a section $y\in L_{O}\mathcal{T}^\prime(O_C)$, whose image in $\Gr_{\calG}$ is the section $[\lambda]$, and then calculate $y^{-1}\tilde{\lambda}\in LT(C)$ as necessary.
		
		For this we claim that the element $\xi_E=\pi_E-[\pi_E^\flat]$ becomes a unit after inverting $\xi_F=\pi_F-[\pi_F^\flat]$, thus giving rise to an element $y\in \mathbb{G}_m(W_{O_E}(O_C^\flat)[\frac{1}{\xi_F}])\subset L_O\calT^\prime(O_C)$. 
		Indeed, let $P(X)=X^d+ a_1X^{d-1}+\ldots +a_d$ be the minimal polynomial of $\pi_E$ over $\breve{F}$, which is Eisenstein as $\breve{E}/\breve{F}$ is totally ramified. 
		Then the norm of $\xi_E$ in $W_{O_F}(O_C^\flat)$ equals
		\begin{equation}
		P([\pi_E^{\flat}])=[\pi_E^\flat]^d+a_1[\pi_E^\flat]^{d-1}+\dots +a_d. 
		\end{equation}
		Reducing modulo $\xi_F$, this element certainly vanishes because $[\pi_E^\flat]\equiv [\pi_E^\flat]^\sharp=\pi_E$ modulo $\xi_F$. 
		On the other hand, $P([\pi_E^{\flat}])$ is clearly a primitive element of degree $1$ inside $W_{O}(O_C^\flat)$, as $a_d \in \pi_F O^\times$. 
		Hence, $P([\pi_E^\flat])$ and $\xi_F$ generate the same principal ideal, see \cite[Lemma 2.24]{BS19}.
		
		Let us now consider the generic fiber of the point $y$.	
		For this we must pass to 
		\begin{equation} B_{\dR}(C) \otimes_{\breve{F}} \breve{E} = \prod_{\sigma \in I_F/I_E} B_{\dR}^\sigma(C).
		\end{equation}
		Here, notice that we are conjugating the natural $\breve{E}$-structure on the corresponding factor of the right side by $\sigma$. Then the coordinate of $\xi_E$ for $\sigma \neq 1$ is a unit in $B^+_\dR(C)$ which reduces modulo the uniformizer $\xi_F$ to $\pi_E^\sigma-\pi_E$. As for its coordinate on the $\sigma=1$ factor, it must be a prime element, because so is the norm in $B_\dR(C)$, as seen in the previous calculation.
		We can conclude that the generic fiber of the section $y$ maps to $\lambda$, and that $\tilde{\lambda}^{-1}y$ is $\delta_G(\lambda)$.
	\end{proof}
	
	\begin{remark}\label{rem_torsors_representable}
		Inspecting \Cref{prop_torsor_local_model_different}, one sees that $\calM_{\calG, O_C, \mu_i}^{\on{tor},\circ}$ is representable by a smooth $O_C$-scheme that naturally descends to $\breve{O}_{E_i}$ (and even to the ring of integers of a finite unramified extension of the reflex field $E_i$ of $\mu_i$). 
		Indeed, by \Cref{intermediate-rep}, $\calM_{\calG, O_C, \mu_i}^{\circ}$ is representable by a smooth $O_{C}$-scheme descending to $\breve{O}_{E_i}$. Similarly, the defining maps $\varphi_\lambda\colon \calP^{-,\diamondsuit}_\lambda\to \calG^{>i,\diamondsuit}_{\mathrm{ad}, O_C}$ of $\calM_{\calG, O_C, \mu_i}^{\on{tor},\circ}$ are algebraic in the generic fiber over $C$ and even descend to $\breve{E}_i$. 
		So, the same holds for their integral models by \cite[Proposition 1.7.6]{BT84}. 
		The arguments following \Cref{sec:convolution-2-extensions-of-given-generic-left-h-right-a-torsor} yield an algebraic space which, being a torsor under an affine group scheme over a scheme, is necessarily itself a scheme.
	\end{remark}

	\subsection{Specialization maps}
	
	The aim of this subsection is to characterize the specialization map. 
	In the following, we consider all pairs $(\calG,\mu_I)$ where $\calG$ is a parahoric $O$-model of some reductive $F$-group $G$, $I$ some finite index set and $\mu_I=(\mu_i)_{i\in I}$ is a sequence of minuscule coweights in $G_C$ such that $\sum_{i\in I}\mu_i$ is still minuscule (and so are all subsums).
	Morphisms of such pairs $(\calG,\mu_{I})\to (\calG',\mu'_J)$ are given by morphisms of $O$-group schemes $\calG\to \calG'$, surjections of sets $\on{pr}\co I\onto J$ such that, for all $j\in J$, the image of $\sum_{i\in \on{pr}^{-1}(j)}\mu_i$ in $G'_C$ lies in the conjugacy class of $\mu_j'$.
	This generalizes the functoriality considered in \Cref{prop_LM_basic_facts}.
	We denote $(\calG,\mu_I)$ also by $(\calG,\mu_\bullet)$ if the index set $I$ is understood.
	By \Cref{sec:convolution-1-convoluted-local-models-are-kimberlites} the convolution local model $\calM_{\calG,\mu_\bullet}$ admits a specialization map, see \Cref{sec:formal-theory-v-sheaves-p-adic-kimberlites}.
	
	\begin{theorem}\label{thm_specialization_local_models}
		The specialization maps for all pairs $(\calG,\mu_\bullet)$ as above
		\begin{equation}
		\on{sp}_{\calG,\mu_\bullet}\colon |\calF_{G,C, \mu_\bullet}| \rightarrow |\calA_{\calG,\bar{k},\mu_\bullet}|
		\end{equation}
		are the only functorial collection of continuous and spectral maps, whose restrictions to $\calM^\circ_{\calG, O_C, \mu_\bullet}(O_C)$ agree with the reduction maps induced by $O_C\to \bar k$.
	\end{theorem}

	\begin{proof}
		We are going to uniquely determine the values taken by $\on{sp}_{\calG, \mu_\bullet}$ on the subset $\calF_{G, \mu_\bullet}(K)$ for a cofinal set of finite extensions $K/F$ given by those Galois extensions that split $G$. This characterizes the map $\on{sp}_{\calG, \mu_\bullet}$ by continuity with respect to the constructible topology. Indeed, $\calF_{G, \mu_\bullet}$ is a smooth rigid space defined over $E$ and, by a theorem of Huber, these points are dense in the constructible topology see \cite[Theorem 4.1]{Hub93} (or \cite[Theorem 4.47]{Gle20}).
		
		We first observe that points of the form $\calM^\circ_{\calG, \mu_\bullet}(O_K)\subseteq \calF_{G, \mu_\bullet}(K)$ for a fixed $\mu_\bullet$ will usually not be enough to characterize $\on{sp}_{\calG, \mu_\bullet}$, even if $K$ varies over all Galois extensions $F\subseteq K\subseteq \bar{\bbQ}_p$.\footnote{We note that as $K$ increases, but $(G,\mu_\bullet)$ remain fixed, the complement $\calF_{G, \mu_\bullet}(K)\setminus \calM^\circ_{\calG, \mu_\bullet}(O_K)$ will also increase unless $\calG$ is hyperspecial. In the hyperspecial case, we have an equality $\calF_{G, \mu_\bullet}(K)= \calM^\circ_{\calG, \mu_\bullet}(O_K)$ for all $K$.} 
		The key is to exploit functoriality and to vary the group $\calG$.
		Indeed, if we find a morphism of groups $\iota_K:\calG\to \calH$ inducing a closed immersion 
		\begin{equation}\iota^\calM_{K}:\calM_{\calG, \mu_\bullet}\hookrightarrow \calM_{\calH, \mu_\bullet}\end{equation}
		then for $x\in \calF_{G, \mu_\bullet}(K)$ the value of $\on{sp}_{\calG, \mu_\bullet}(x)$ is determined by the value of $\on{sp}_{\calH, \mu_\bullet}(\iota^\calM_{K}(x))$.
		This exploits functoriality on $\calG$, but we have also allowed ourselves to have functoriality in $\mu_\bullet$.
		To wit, in order to control the specializations of $K$-rational points $x\in \calF_{G, \mu_\bullet}(K)$, it suffices to construct a diagram of the form 
		\begin{equation}
			\label{diagram-lifts}
		\begin{tikzcd}
		   & \calM_{\calH, \mu'_\bullet}  \ar{d}{r} \\
			\calM_{\calG, \mu_\bullet} \arrow{r}{\iota^\calM_{K}} & \calM_{\calH, \mu_\bullet} 
		\end{tikzcd}
		\end{equation}
		where $\iota^\calM_{K}$ is the closed immersion induced by a morphism $\iota_K\colon \calG\to \calH$, and $r$ is a convolution morphism obtained by writing the terms $\mu_i=\sum_{j\in J_i} \mu'_j$, such that the every point $\calF_{\calH, \mu'_\bullet}(K)$ lies in $\calM^\circ_{\calH, \mu'_\bullet}(O_K)$.  
		Functoriality, will then determine $\on{sp}_{\calG, \mu_\bullet}(x)$ for every $x \in \mathcal{F}_{G,\mu_\bullet}(K)$. The point of this construction is that, if we find ourselves unable to resolve the local model via convolution around a certain point, we can enlarge it first and then resolve it further.
		
		In what follows, we construct $\iota^\calM_K$ and $r$.
		Having fixed a Galois extension $K/F$ splitting $G$, we consider the map coming from restriction of scalars $\iota_K:G\to H:=\Res_{K/F} G_K$.  This comes equipped with a natural equivariant embedding of their Bruhat--Tits buildings $\iota^{\mathcal{B}}_{K}\colon \mathcal{B}(G,F)\to \mathcal{B}(H,F)\simeq\mathcal{B}(G,K)$. For any point $x \in \mathcal{B}(G,F)$ fixed under $\calG(O)$, we let $y \in \mathcal{B}(H,F)$ be its image and denote by $\calH$ the associated parahoric. Equivariance yields a $\varphi$-equivariant map $\calG(\breve{O})\to \calH(\breve{O})$ and by \cite[Proposition 1.7.6]{BT84}, we can extend this to a map between parahoric models $\iota_K\colon \calG\to \calH$.
This already yields our desired enlargement map $\iota^\calM_{K}$ for local models, provided we show it is also a closed immersion. For that, we are allowed to pass to adjoint groups as the map $\calG \to \calG_{\mathrm{ad}}$ identifies local models, see \Cref{prop_LM_basic_facts}. Then, under the assumption $G=G_{\mathrm{ad}}$, we conclude that $\iota_K$ is a locally closed immersion by \cite[Proposition 2.4.8]{KZ21}, exploiting the fact that $T$ is an induced torus, hence $R$-smooth in the sense of \cite[Definition 2.4.3]{KZ21}. In particular, $\iota^\calM_{K}$ has to be a closed immersion again by \Cref{prop_LM_basic_facts}.
	To define $r$ we refine the induced cocharacter $\mu_\bullet$ in $H$ to a cocharacter $\mu'_\bullet$ so that every element in its sequence is tiny. 
	With this setup, it is already true that for $x\in \calF_{H, \mu'_\bullet}(K)$ lies in $\calM^{\circ}_{\calH, \mu'_\bullet}(O_K)$.
	Indeed, this is the content of \Cref{lemma_integral_points_tiny_convolution_inside_circ} below. 
	\end{proof}

	\begin{lemma}\label{lemma_integral_points_tiny_convolution_inside_circ}
		Suppose $K/F$ is Galois, $G=\Res_{K/F}H$ and all $\mu_i \in \mu_\bullet$ are tiny. Then, we have an equality
		\begin{equation}
		\calM^{\circ}_{\calG, \mu_\bullet}(O_K)=	\calF_{G, \mu_\bullet}(K)
		\end{equation} of sets induced by the natural morphism.
	\end{lemma}
	
	\begin{proof}
		Let us assume first that $K=F$. Then $\calG=\calH$ is a parahoric model of a split reductive $F$-group and the right side can be given by the Iwasawa decomposition 
		\begin{equation}
			\label{iwasawaequation}
		\calF_{G,\mu}(K) = \bigcup_{\lambda} \calH(O_K) \cdot \la,
		\end{equation}
		see \cite[Proposition 4.4.3]{BT72} (or \cite[Theorem 5.3.3]{KP21}). 
		Now, obviously the points of the form $\lambda$ extend to integral points of $\calM^\circ_{\calH,\mu}$, due to the splitness assumption, and thus the same holds for its $\calH(O_K)$-orbits.
		
		Now, consider an arbitrary finite Galois extension $K/F$. We get the result immediately for tiny coweights, as the $\calG(O_K)$-action on $\calM_{\calG,O_K,\mu}$ is via the parahoric subgroup $ \calH(O_K) \subset H(K)$. This addressed the case in which $\mu_\bullet=\mu$ consists of only one tiny coweight.
		In general, we use this to show the claim by an inductive procedure. 
		Let $\mu_\bullet=(\mu_1,\dots,\mu_n)$ with each $\mu_i$ tiny.
		Fix an element $(x_1, \dots, x_n) \in \calF_{G, \mu_\bullet}(K)$. 
		We will show inductively that each of the tuples $(x_1,\dots,x_i)\in \calF_{G,K, \mu_{\bullet\leq i}}(K)$ define an integral point in $\calM_{\calG,O_K, \mu_{\bullet\leq i}}^{\circ}(O_K)$.
		Consider the projection map 
		\begin{equation}\calM_{\calG, O_K,\mu_{\bullet\leq i+1}}^{\circ}\to \calM_{\calG, O_K,\mu_{\bullet\leq i}}^{\circ}.\end{equation}
		We are going to show by induction on $i$ the stronger assertion that every $K$-point of $\calF_{G, K,\mu_{\bullet\leq i}}$ extends to an $O_K$-point of $\calM_{\calG,O_K, \mu_{\bullet\leq i}}^{\circ}$, and that one may even lift these $O_K$-points along the projection of the right $\calG^{>i}$-torsor 
		\begin{equation}\calM_{\calG,O_K, \mu_1}^{\on{tor},\circ} \times^{\calG_{O_K}^{>1}} \dots \times^{\calG_{O_K}^{>i-1}}  \calM_{\calG,O_K, \mu_{i}}^{\on{tor},\circ}\to \calM_{\calG,O_K, \mu_{\bullet\leq i}}^{\circ}.\end{equation}

%

		First observe that, by the $\mu_\bullet=\mu$ case, $\calF_{G,\mu_1}(K)=\calM^\circ_{\calG,\mu_1}(O_K)$.
		We claim that for all $x_1\in \calF_{G,K,\mu_1}(K)$ we may find a lift $\tilde{x}_1$ to $\calM^{\on{tor},\circ}_{\calG,O_K,\mu_1}(O_K)$.
		Indeed, the $O_C$-section, $y\in L_\calO\calT'(O_C)\subseteq L_\calO\calG(O_C)$ projecting to $\lambda$ that we constructed in \Cref{prop_torsor_local_model_different}, descends to $O_K$ after projecting to $\calM^{\on{tor},\circ}_{\calG,O_K,\mu_1}$, see also \Cref{rem_torsors_representable}.
		The left $\calG(O_K)$-orbit of $y$ in $\calM^{\on{tor},\circ}_{\calG,O_K,\mu_1}(O_K)$ provides the set of lifts for points in the $\calG(O_K)$-orbit of $\lambda$ in $\calM^{\circ}_{\calG,O_K,\mu_1}(O_K)$. 
		The claim follows from ranging over all the $\lambda$ as in \eqref{iwasawaequation}.

		More generally, suppose (by induction) that we have found representatives $\tilde{x}_j \in \calF_{G, \mu_j}^{\on{tor}}(K)$ contained in $\calM_{\calG, \mu_j}^{\on{tor},\circ}(O_K)$ for $j<i$. 
		Let $\calS$ be the fiber of the projection $\calM_{\calG, \mu_{\bullet \leq i}}^{\circ}\to \calM_{\calG, \mu_{\leq i-1}}^{\circ}$ along the $O_K$-point that $(x_1,\dots,x_{i-1})$ defines.
		Since by hypothesis  
		\begin{equation}\calM_{\calG,O_K, \mu_1}^{\on{tor},\circ} \times^{\calG_{O_K}^{>1}} \dots \times^{\calG_{O_K}^{>{i-2}}}  \calM_{\calG,O_K, \mu_{i-1}}^{\on{tor},\circ}(O_K)\to \calM_{\calG,O_K, \mu_{\bullet\leq i-1}}^{\circ}(O_K)\end{equation}
		is surjective, the natural $\calG_{O_K}^{>i-1}$-torsor is trivial. 
		This gives an identification $\calS\simeq \calM_{\calG,O_K, \mu_{i}}^{\circ}$, under which the natural $\calG_{O_K}^{>i}$-torsor over $\calS$ identifies with $\calM_{\calG, O_K,\mu_{i}}^{\on{tor},\circ}$. 
		By the $\mu=\mu_\bullet$ case, $\calS$ has an $O_K$-point which lifts to its $\calG_{O_K}^{>i}$-torsor as we wanted to show.
		This finishes the inductive step and the proof.
%
%
%
	\end{proof}
	
	\begin{remark}\label{rem_sp_LM}
		It follows by inspecting the proof of \Cref{thm_specialization_local_models} that in order to compute the specialization mapping, it is enough to consider Weil restrictions of split groups and their Iwahori models, see \Cref{assumption_iwahori_res_split}.
		Reduction to this type of group is an integral part of our strategy to show \Cref{thm_local_model_representable} (see the paragraph containing \eqref{closed-immersion-of-groups}).
	\end{remark}
	
	\begin{remark}
		He--Pappas--Rapoport \cite[Conjecture 2.12]{HPR20} conjecture that, for any fixed pair $(\calG,\mu)$, there is at most one flat projective $O_E$-scheme equipped with an $\calG_{O_E}$-action having the correct fibers, identified $\calG$-equivariantly. 
		This is much stronger than \Cref{thm_specialization_local_models} above, since it makes no reference to convolution or functoriality. 
		Our approach is inspired by their conjecture in applying equivariant methods to pin down the specialization map.
	\end{remark}

	\subsection{Comparison isomorphisms}\label{subsec_comparison}
	
	In this subsection, we use our work from the previous ones to establish certain comparison isomorphisms between (at least some of) our local models and those that have appeared elsewhere, see \cite{PZ13,Lev16,Lou19,FHLR22}. During this subsection, we shall work under the following:
	
	\begin{assumption}\label{assumption_iwahori_res_split}
		Given a pinned split simple adjoint group $(H,T_H,B_H,e_H)$, let $G=\Res_{K/F} H$ with $K/F$ an arbitrary finite extension.
		In this context, we let $K_0$ denote the maximal unramified subextension of $K/F$ with ring of integers $O_0$ and residue field $k_0$. We let $\calI$ be the standard Iwahori model of $G$ with respect to the chosen pinning.		
	\end{assumption}

	In order to prove \Cref{thm_local_model_representable}, we need to compare $\calM_{\calI,\mu}$ to certain candidates $\calN_{\underline{\calI},\mu}^{\on{sch}}$ constructed in \cite[Definition~5.11]{FHLR22} and denoted $\widetilde M_{\underline \calG,\mu}$ there, which are variations on the work of Levin \cite{Lev16}. 
	Here, $\underline{\calI}$ is a $O\pot{t}$-lift of $\calI$ along $t \mapsto \pi$, obtained by taking restriction of scalars along an ad hoc lift $O\pot{t}\to O_0\pot{u}$ of $O \to O_K$ of the dilatation of $H\otimes O_0\pot{u}$ along $B_H \otimes O_0$ concentrated in the $u$-divisor, see \cite[Theorem 4.1]{PZ13} and \cite[Definition 2.1, Example 3.3]{MRR20}. The various lifts $O\pot{t}\to O_0\pot{u}$ defined in \cite[Subsection 2.2]{FHLR22} are given by choosing uniformizers and lifting Eisenstein polynomials over $O_0$ in such a way that they remain separable Eisenstein over both $k_0\pot{t}$ and $K_0\pot{t}$.
	
	One has a schematic Beilinson--Drinfeld Grassmannian $\Gr_{\underline{\calI}}^{\mathrm{sch}}$ defined in terms of power series rings, classifying $\underline{\calI}$-torsors over $R\pot{t-\pi}$ trivialized over $R\rpot{t-\pi}$, and admitting uniformization via loop groups $L_O^{\on{sch}}\underline{\calI}/L^{\on{sch},+}_{O}\underline{\calI}$. 
	The generic fiber is equivariantly isomorphic to the schematic affine Grassmannian $\Gr_{G}^{\on{sch}}$ over $F$, see \cite{FHLR22}. 
	So we get an embedding $\calF_{G,\mu}\subset \Gr_{G}^{\on{sch}}|_{\Spec\, E}$ for a minuscule coweight $\mu$.

	\begin{definition}
		The $O_E$-scheme $\calN_{\underline{\calI},\mu}^{\on{sch}}$ is defined as the seminormalization of the flat closure of $\calF_{G,\mu}$ inside $\Gr_{\underline{\calI},O_E}^{\mathrm{sch}}$. 
		For a minuscule sequence $\mu_\bullet$ of dominant coweights, we set $\calN_{\underline{\calI},O_E,\mu_\bullet}^{\on{sch}}$ as the convolution product of the $\calN_{\underline{\calI},O_E,\mu_i}^{\on{sch}}$. 
		We define the $\calI^{>i}_{O_C}$-torsor $\calN_{\underline{\calI},O_C,\mu_i}^{\on{sch,tor}}$ by pushing forward the universal $L^{\on{sch},+}_{O_C}\underline{\calI}$-torsor under the natural projection.
	\end{definition}
	
	The $\calN_{\underline{\calI},\mu}^{\on{sch}}$ are normal flat $O_E$-schemes with reduced special fiber by \cite[Theorem~5.14]{FHLR22}, so the same remains true for their base change to $O_{\bar E}$.	
	Indeed, the base changed scheme $\calN_{\underline{\calI},O_{\bar E},\mu}^{\on{sch}}:=\calN_{\underline{\calI},\mu}^{\on{sch}}\otimes_{O_E}O_{\bar E}$ has smooth generic fiber and reduced special fiber, so this follows from Serre's criterion for normality, see \cite[Proposition 8.2]{PZ13}, and from writing $\bar E$ as a union of finite its finite subextensions over $E$. In particular, for any finite extension $E'/E$, we have that $\calN^{\mathrm{sch}}_{\underline{\calI}, O_{E'},\mu}$ is the unique normal flat $O_{E'}$-scheme representing its associated $v$-sheaf.

	The candidates, $\calN_{\underline{\calI},\mu}^{\on{sch}}$, also come with a weak form of functoriality. 
	More precisely, given extensions $\tilde{K}/K/F$ we have a morphism of groups $\Res_{K/F} H\to \Res_{\tilde{K}/F} H$ compatible with a morphism of parahoric group schemes $\calI\to \tilde{\calI}$. These induce finite maps
	\begin{equation}
		\label{res funcotriality}
	\calN_{\underline{\calI},O_C,\mu}^{\on{sch}} \rightarrow \calN_{\tilde{\underline{\calI}},O_C,\tilde{\mu}}^{\on{sch}},
      \end{equation}
      of schemes that are universal homeomorphisms towards their images, where $\tilde{\underline{\calI}}$ is a $O\pot{t}$-lift of $\tilde{\calI}$ defined with respect to the extension $\tilde{K}/K/F$ (see \cite[Section~5.3]{FHLR22}) and where $\tilde{\mu}$ is the image of $\mu$ in $\tilde{G}$. 
      In particular, they induce closed immersions upon passing to the associated $v$-sheaves.
	In what follows, we shall simply say the candidates are \emph{$R$-functorial} in $(\underline{\calI},\mu)$. Similarly, we may consider functoriality under convolution maps $\calN_{\underline{\calI},O_C, \mu_\bullet}^{\mathrm{sch}}\to \calN_{\underline{\calI},O_C, \mu}^{\mathrm{sch}}$. 
	We will refer to this by the term \textit{$R$-functorial in $(\underline{\calI},\mu_{\bullet}) $}.

	Our next goal is to compare the v-sheaves associated with $\calN_{\underline{\calI},O_C,\mu_\bullet}^{\on{sch}}$ to $\calM_{\calG,O_C,\mu_\bullet}$. We start by recording what happens in the generic fiber.
	
	\begin{lemma}\label{lem_comparison_generic_fiber}
		There are unique equivariant isomorphisms
		\begin{equation}
		\calN_{\underline{\calI},\mu_i}^{\on{sch,tor}}|_{\Spec\, C}\cong \calF_{G, C,\mu_i}^{\on{tor}}
		\end{equation}
		for each term $\mu_i$ of the sequence $\mu_\bullet$. 
		They yield canonical equivariant isomorphisms
		\begin{equation}
		\calN_{\underline{\calI},\mu_\bullet}^{\on{sch}}|_{\Spec\, C} \cong \calF_{G,C, \mu_\bullet}
		\end{equation}
		which are $R$-functorial in $(\underline{\calI},\mu_\bullet)$.
	\end{lemma}
	
	\begin{proof}
		This follows by definition and uniqueness is ensured by \Cref{sec:convolution-1-abstract-lemma-on-left-right-torsors}.
	\end{proof}
	
	Next, we need to take care of the special fiber:

	\begin{proposition}\label{prop_comparison_special_fiber}
		There are unique equivariant isomorphisms
		\begin{equation}
		\left (\calN_{\underline{\calI},\mu_i}^{\on{sch,tor}}|_{\Spec \,\bar k}\right)^{\on{perf}}	 \cong \calA_{\calI, \bar k, \mu_i}^{\on{tor}}
		\end{equation}
		for each term $\mu_i$ of the sequence $\mu_\bullet$. 
		They yield canonical equivariant isomorphisms
		\begin{equation}
		\left(\calN_{\underline{\calI},\mu_\bullet}^{\on{sch}}|_{\Spec\, \bar k }\right)^{\on{perf}} \cong \calA_{\calI, \bar k, \mu_\bullet}
		\end{equation}
		which are $R$-functorial in $(\underline{\calI},\mu_\bullet)$.
	\end{proposition}
	
	\begin{proof}
		Set $\calI'= \underline{\calI}\otimes k\pot{t}$, a standard Iwahori model of the connected reductive group $G'=\Res_{k_0\rpot{u}/k\rpot{t}}H$. 
		By \cite[Theorem~5.14]{FHLR22}, we have $\calN_{\underline{\calI},\mu, \bar k}^{\on{sch,perf}}=\calA_{\calI', \bar k, \mu'}$.
		Hence, the statement above is just a generalization of \Cref{lem_comparison_sch_vars_equi_and_mixed} to convolution products. 
		
		Let $w\in \tilde{W}$ be an element such that $\Fl_{\calI,\bar{k},w} \subset \calA_{\calI,\mu_i, \bar k,\mu_i}$ and choose a Demazure resolution $\pi_{\dot w}\co \calD_{\calI,\bar{k}, \dot{w}} \to \Fl_{\calI,\bar{k},w}$. 
		We have to compare the following pullback square
		\begin{equation}
		\begin{tikzcd}
		\calD_{\calI,\bar{k}, \dot{w}}^{\on{tor}} \arrow[r, "\pi_{\dot{w}}^{\mathrm{tor}}"] \arrow[d, "p_{\calD}"]  & \arrow[d, "p_{\Fl}"]\Fl_{\calI,\bar{k},w}^{\on{tor}}  \\
		\calD_{\calI,\bar{k}, \dot{w}}\arrow[r, "\pi_{\dot{w}}"] &  \Fl_{\calI,\bar{k},w}
		\end{tikzcd}
		\end{equation}
		with its counterpart in the equicharacteristic setting. 
		The bottow arrow was dealt with in \Cref{lem_comparison_sch_vars_equi_and_mixed} and the second paragraph there applies verbatim to comparing the left arrow. We claim that these suffice to recover the remainder of the diagram.
		
		Note that the vertical arrows are affine morphisms, so they can be written via relative spectra, that is, we have $\calD_{\calI,\bar{k}, \dot{w}}^{\on{tor}}=\Spec(p_{\calD,*}\calO_{\calD^{\on{tor}}})$ and $\Fl_{\calI,\bar{k}, w}^{\on{tor}}=\Spec(p_{\Fl,*}\calO_{\Fl^{\on{tor}}})$. 
		On the other hand, we know that $\pi_{\dot w,*}\calO_{\calD}=\calO_{\Fl}$, so the same equality holds for $\pi_{\dot w}^{\on{tor}}$ by flat base change, as the vertical arrows are perfectly smooth. 
		But this means $p_{\Fl,*}\calO_{\Fl^{\on{tor}}}=\pi_{\dot w,*}p_{\calD,*}\calO_{\calD^{\on{tor}}}$, just as asserted.
		
		Next, we show that the isomorphisms constructed above are unique. Without torsors, this has been verified in \Cref{prop_trivial_equiv_autos}, so any automorphism respects the orbit part $\Fl_{\calI,\bar k,w}^{\circ,\on{tor}}$. So we only have to verify the conditions of \Cref{sec:convolution-1-abstract-lemma-on-left-right-torsors} on centralizers of the transfer homomorphism $\varphi_w$. In this case, it is given by
		\begin{equation}
		\on{int}(w^{-1})\colon L^+\calI \cap wL^+\calI w^{-1} \to w^{-1}L^+\calI w \cap L^+\calI. 
		\end{equation} 
		We see that the image of the right side in $(\calI_{\bar{k}}^{>i})^{\on{perf}}$ contains the image of $\calB_{\bar{k}}^{\on{perf}}$, where $\calB \subset \calG$ is the flat closure of some Borel $B \subset G$. 
		In particular, the centralizer is contained in $\calT_{\bar{k}}^{>i}$. Now, we inspect the description of $\calI_{\bar{k}}^{>i}$ given in \Cref{lem_min_adjoint_quotient_over_integers} more closely to deduce triviality of the centralizer.
		First, observe that taking the product along all positive simple roots $a\colon T_H\to \mathbb{G}_m$ and composing with the diagonal map $T_H \to T_H^{\Delta_H}$, we obtain an isomorphism $T_H \simeq \mathbb{G}_{m,K}^{\Delta_H}$ by adjointness of $H$.
		We denote the corresponding quotient by $T_H^a$ for any $a\in \Delta_H$. Passing to restrictions of scalars and Néron models, we get also a decomposition $\calT \simeq \prod_{a \in \Delta_H}\calT^a$ where $\calT^a\simeq \Res_{O_K/O_F}\mathbb{G}_{m,O_K}$ is the corresponding quotient.
		Finally, we observe that $\calT^{>i}$ also decomposes as a product of the $\calT^a$ for $a \in \Delta_H^{>i}$, that is, all those roots that have non-trivial restriction to the split $O_F$-torus $\calS^{>i}$.
		The claim on the centralizer is now a consequence of the faithful action of $\calT_{\bar{k}}^a$ on the corresponding $a$-root group $\calU_{\bar{k}}^a$.
		
		Finally, we must show that the isomorphisms just constructed are $R$-functorial with respect to $(\underline{\calI},\mu_\bullet)$. First, we handle the case of $\underline{\calI}$ as follows. 
		Fix $\tilde{K}/K/F$ as in \eqref{res funcotriality}. 
		Notice that the image of the closed immersion $\calA_{\calI, \bar k,\mu}\to \calA_{\tilde{\calI}, \bar k,\mu}$ is characterized as closure of the $\calI_{\bar k}$-orbit of the $\lambda_I(\pi)$, as $\lambda$ varies through the $\breve{F}$-rational conjugates of $\mu$. The same observation holds for their equicharacteristic counterparts, so when writing down the obvious square with the comparison isomorphisms, we see that the only datum possibly preventing it from commuting is a $\calI_{\bar k}$-equivariant automorphism of the left side $\calA_{\calI,\bar k,\mu}$. But this is necessarily trivial, so we deduce the $R$-functoriality in $\underline{\calI}$ when $\mu_\bullet=\mu$. The case of torsors $\calA_{\calI,\bar k, \mu_i}^{\mathrm{tor}}$ can now be proved in the same vein, using uniqueness of equivariant automorphisms, so we deduce $R$-functoriality in $\underline{\calI}$ also for the convolution $\calA_{\calI,\bar k, \mu_\bullet}$.
		
		Finally, we handle the functoriality in $\mu_\bullet$ by appealing to Stein factorizations of semi-ample line bundles, compare with \Cref{prop_line_bundles_stein_fac}. Our calculation of Picard groups in \Cref{calculo do grupo de picard da grassmanniana dos vectores de witt} and \Cref{sec:affine-flag-vari-1-constructing-schubert-varieties-via-demazure} allows us to deduce that 
		\begin{equation}\mathrm{Pic}(\calA_{\calI,\bar k,\mu})\to \mathrm{Pic}(\calA_{\calI,\bar k,\mu_\bullet})
		\end{equation}
	is injective and identifies with the diagonal inclusion $\mathbb{Z}[p^{-1}]^{m+1} \to \mathbb{Z}[p^{-1}]^{n(m+1)}$ with $m=\mathrm{rk}(G)$, compare with \cite[Corollary 5.8]{FHLR22} to see that every simple affine reflection appears in the translated $\mu$-admissible locus. Again, this description also holds for the equicharacteristic counterparts, so we see that we can descend the equivariant isomorphism previously constructed for $\mu_\bullet$ to a diagram of the form 
	\begin{equation}
	\label{stein diagram orsomthmn}
	\begin{tikzcd}
		\calA_{\calI,\bar k,\mu_\bullet} \arrow{r} \arrow{d}{\simeq}  & \calA^{\on{st}}_{\calI,\bar k, \mu}  \arrow{d}{h,\simeq} \ar{r}{f} &  \calA_{\calI,\bar k, \mu} \arrow[dotted]{d}{m,\simeq} \\
		\calA_{\calI',\bar{k},\mu'_\bullet} \arrow{r} & \calA_{{\calI'}, \bar{k},\mu'}^{\on{st}} \ar{r}{g} &  \calA_{{\calI'}, \bar{k},\mu'}
	\end{tikzcd}
	\end{equation}
where the left square is commutative, but at the moment we might not know if the right square is or not commutative.
Here $\calA^{\on{st}}_{\calI,\bar k, \mu}$ (respectively $\calA_{{\calI'}, \bar{k},\mu'}^{\on{st}}$) denotes the Stein factorization of the proper surjection
		\begin{equation}
			\label{stein map}
			\calA_{\calI,\bar k, \mu_\bullet} \to \calA_{\calI,\bar k, \mu} \text{ (respectively } \calA_{\calI', \bar{k}, \mu'_\bullet} \to \calA_{\calI', \bar{k}, \mu'}\text{).}
		\end{equation}
	It turns out that the geometric fibers of \eqref{stein map} are already geometrically connected by an application of the unibranchness of $\calM_{\calI,\mu}$ proved in \cite[Theorem 1.3]{GL24} (which holds even for non-minuscule $\mu$, so it does not logically depend on the representability handled in this section) and an argument as in \Cref{prop_connected_tubes_implies_normality}. 
	This would show that $f$ and $g$ are isomorphisms. Then, $\calI_{\bar k}$-equivariance and the uniqueness proved in \Cref{prop_trivial_equiv_autos} shows that the right hand square of \eqref{stein diagram orsomthmn} is commutative. 

	Let us also explain a variant of this argument that does not rely on \cite{GL24}.
		In the equicharecteristic setting, we can lift the convolution map to a morphism of normal $O_C$-schemes 
		\begin{equation}\calN_{\underline{\calI},O_C, \mu_\bullet}^{\mathrm{sch}} \to \calN_{\underline{\calI},O_C, \mu}^{\mathrm{sch}}
		\end{equation} and apply Zariski's connectedness theorem to deduce that its fibers are geometrically connected. This implies the corresponding result on the special fibers, so $g:\calA^{\on{st}}_{\calI', \bar{k}, \mu'_\bullet}\to \calA_{\calI', \bar{k}, \mu'}$ is an isomorphism. 
		This gives rise to a new equivariant surjection $ \calA_{\calI', \bar k,\mu'}\to  \calA_{\calI, \bar k,\mu}$, namely $f\circ h^{-1}\circ g^{-1}$. We may compose it with $m$ in \eqref{stein diagram orsomthmn} (i.e. the equivariant isomorphism $\calA_{\calI, \bar k,\mu} \cong  \calA_{\calI',\bar k, \mu'}$ of \Cref{lem_comparison_sch_vars_equi_and_mixed}) and the result, $m\circ f\circ h^{-1}\circ g^{-1}$ is a finite equivariant cover of $\calA_{\calI',\bar k, \mu'}$ by itself. This has to be an isomorphism on the open $\calI_{\bar k}$-orbits since $f$ is easily seen to be birational, and hence the uniqueness proved in \Cref{prop_trivial_equiv_autos} shows that it is the identity on an dense open subset. Since $ \calA_{\calI',\bar k, \mu'}$ is separated, the diagram \eqref{stein diagram orsomthmn} is commutative, and the desired functoriality in $\mu_\bullet$ has been shown.
	\end{proof}
	
	The last comparison involves the sub-v-sheaf from \Cref{def_open_semi_homog_fixers}.
	
	\begin{proposition}\label{prop_comparison_semi_orbit}
		There are unique equivariant isomorphisms
		\begin{equation}
		(\calN_{\calI,O_C, \mu_i}^{\on{sch,\circ,tor}})^\diamondsuit \simeq \calM_{\calI, O_C, \mu_i}^{\circ,\on{tor}}
		\end{equation}
		for each term $\mu_i$ of the sequence $\mu_\bullet$. They yield canonical equivariant isomorphisms
		\begin{equation}
		(\calN_{\calI,O_C,\mu_\bullet}^{\on{sch},\circ})^\diamondsuit \simeq  \calM_{\calI,O_C,\mu_\bullet}^{\circ},
		\end{equation}
		compatibly with those of \Cref{lem_comparison_generic_fiber} and \Cref{prop_comparison_special_fiber} in the obvious sense.
	\end{proposition}
	
	\begin{proof}
		We have already identified the generic fibers of these v-sheaves, see \Cref{lem_comparison_generic_fiber}. By \Cref{sec:convolution-2-extensions-of-given-generic-left-h-right-a-torsor}, we reduce to calculating $\Spd O_C$-valued points of the left side torsor and compare their different (as in \Cref{different-definition}) to that of \Cref{prop_torsor_local_model_different}. The resulting isomorphism will then reduce to the expected isomorphisms over $\Spd k$ obtained in \Cref{prop_comparison_special_fiber}, by uniqueness of equivariant automorphisms.
		
		Now, we repeat the same calculation of \Cref{prop_torsor_local_model_different}, that goes back to Zhu \cite{Zhu14}, see \cite[Proposition 4.2.8]{Lev16}. 
		Here, we work with the power series loop group $L_{O_C}^{\on{sch}}\underline{\calT}$ in the setting of \cite[Section~2.2]{FHLR22}. 
		After refining $\mu_\bullet$, we may and do assume that each term $\mu_i\in \mu_\bullet$ is concentrated in a single component of the Dynkin diagram of $G$.
		There is a natural map
		\begin{equation}
		\Res_{O_0\pot{u}/O_0\pot{t}} \bbG_m \to \underline{\calT_{O_0}}
		\end{equation} induced by $\la_i$ via taking the norm of restriction of scalars. Hence, we only need to compute the corresponding point in $\Res_{O_0\pot{u}/O_0\pot{t}} \bbG_m$. Note that here $O_0\pot{u} $ is a finite $O_0\pot{t} $-algebra, where $u $ satisfies an Eisenstein--Teichmüller type polynomial
		\begin{equation}
		u^n+a_1(t)u^{n-1}+\dots +a_n(t)=0
		\end{equation} in $t$ based on some fixed choices of uniformizers $\pi_K$ for $K$ and $\pi$ for $F$, see \cite[beginning of Section 2.2]{FHLR22}. 
		Now, we claim for any $\sigma \in \Gal_F$, the element $\sig z_u=u-\sig\pi_K$ is a unit in $O_C\pot{u}[z_t^{-1}]$ with $z_t=t-\pi$. Notice that its norm in $O_C\pot{t}$ equals
		\begin{equation}
		P(t)=\sig\pi_K^n+a_1(t)\sig\pi_K^{n-1}+\dots +a_n(t)
		\end{equation}
		which is the product of $z_t$ with a unit of $O_C\pot{t}$. Indeed, we calculate the value $P(\pi)=0$, also of the first derivative $P'(\pi) \in O_C^\times$ and apply the Taylor series inside $C\pot{t}$. Finally, we notice that $\sig z_u$ reduces to the unit $\tau\pi_K-\sig\pi_K$ for all $\tau\neq \sig$ of $\Gal_F/\Gal_K$ in $C\pot{u}[z_t^{-1}]$, and to a prime element in the factor indexed by $\sigma$ due to norm considerations. 
		The desired claim that the pair of differents agree has been shown. 
		Consequently, the two torsors $(\calN_{\calI,O_C, \mu_i}^{\on{sch,\circ,tor}})^\diamondsuit$ and $\calM_{\calI, O_C, \mu_i}^{\circ,\on{tor}}$ agree under the identification $(\calN_{\calI,O_C, \mu_i}^{\on{sch,\circ}})^\diamondsuit\simeq \calM_{\calI, O_C, \mu_i}^{\circ}$ as we wanted to show.
	\end{proof}
	
	\subsection{The Scholze--Weinstein conjecture}
	
	In this subsection, we finally prove the Scholze--Weinstein conjecture, see \Cref{thm_local_model_representable} and \Cref{cor_special_fiber_local_model_reduced} below.

	We start by adressing the representability problem as in \cite[Conjecture IV.4.18]{Lou20}, which is one half of \cite[Conjecture 21.4.1]{SW20}. Recall that $F/\bbQ_p$ is a complete non-archimedean field with perfect residue field $k$, $G$ is an arbitrary (connected) reductive $F$-group, $\mu$ is a dominant coweight of $G_C$ and $\calG$ an arbitrary parahoric $O$-model of $G$.
	
	\begin{theorem}\label{thm_local_model_representable}
		Let $\mu$ be minuscule. Then, there is a unique (up to unique isomorphism) flat, projective and weakly normal $O_E$-model $\calM_{\calG,\mu}^{\on{sch}}$ of the $E$-scheme $\calF_{G,\mu}$ endowed with a $\calG_{O_E}$-action for which 
		\begin{equation} 
		\calM_{\calG,\mu}^{\on{sch},\diamondsuit}\cong \calM_{\calG,\mu},
		\end{equation}
		prolonging $\calF^{\diamondsuit}_{G,\mu} \cong \Gr_{G,\mu}$ equivariantly under $\calG^\diamondsuit_{O_E}$.
	\end{theorem}
	
	\begin{proof}
		First of all, let us work under \Cref{assumption_iwahori_res_split}.
		We know that the geometric fibers of $\calN_{\underline{\calI},O_C, \mu}^{\on{sch},\diamondsuit}$ and $\calM_{\calI,O_C,\mu}$ are uniquely equivariantly isomorphic by \Cref{lem_comparison_generic_fiber}, \Cref{prop_comparison_special_fiber}. By uniqueness, this commutes with the Galois action, so it descends to the fibers over $\Spd O_E$.
		
		Furthermore, thanks also to \Cref{prop_comparison_semi_orbit}, \Cref{thm_specialization_local_models}, and \Cref{rem_sp_LM}, we know that the specialization maps
		\begin{equation}
		\on{sp}\colon \calF_{G, \mu}(C) \rightarrow \calA_{\calI,\mu}(\bar k),
		\end{equation} arising respectively from the $\pi$-adic kimberlite $\calM_{\calI, O_C,\mu}$ and $\calN^{\on{sch},\diamondsuit}_{\underline{\calI}, O_C, \mu}$ must coincide. By continuity for the constructible topology, we obtain an equivariant isomorphism of specialization triples:
		\begin{equation}
		\big(\calN^{\on{sch},\diamondsuit}_{\underline{\calI},E, \mu}, \calN^{\on{sch},\diamondsuit}_{\underline{\calI},\mu, k_E}, \on{sp}_{\breve{\calN}^{\on{sch}}_{\underline{\calI}, \mu}}\big) \cong\big(\calM_{\calI, E,\mu}, \calM_{\calI, k_E, \mu}, \on{sp}_{\breve{\calM}_{\calI, \mu}}\big)
		\end{equation}
		associated with v-sheaves over $\Spd O_E$. Observing that both v-sheaves satisfy the hypothesis of \Cref{prop_fully_faith_triples_kimberlites}, we may directly appeal to it in order to get a necessarily equivariant isomorphism
		\begin{equation}
		\calN^{\on{sch},\diamondsuit}_{\underline{\calI}, \mu} \cong \calM_{\calI, \mu}.
		\end{equation}
		
		Now maintain the part of \Cref{assumption_iwahori_res_split} that refers to $G$, but suppose $\calG$ is now an arbitrary parahoric model admitting a map $\calI \to \calG$.
		We get a v-cover
		\begin{equation}
		\calM_{\calI, \mu} \to \calM_{\calG, \mu}
		\end{equation}
		and, parallelly, a scheme-theoretic projective cover
		\begin{equation}\label{eq_map_alg_loc_models_iwahori_to_parahoric}
		\calN_{\underline{\calI}, \mu}^{\on{sch}} \to \calN_{\underline{\calG}, \mu}^{\on{sch}}
		\end{equation}
		by virtue of \cite[Section~5.3]{FHLR22}.
		We observe that any identification $\calM_{\calG,\mu}\simeq  (\calN_{\underline{\calG}, \mu}^{\on{sch}})^\diamondsuit$ extending the $G_E^\diamondsuit$-equivariant identification of generic fibers is automatically $\calG^\diamondsuit_{O_E}$-equivariant, since $G_E^\diamondsuit$ is dense in $\calG_{O_E}^\diamondsuit$.
		Therefore, it is enough to verify that the v-sheaf-theoretic equivalence relations coincide along the left side identification.
		In other words, if by abuse of notation we let $X$ denote $\calM_{\calI, \mu}$ and $\calN_{\underline{\calI}, \mu}^{\on{sch}}$ at the same time, then the surjections $X\to (\calN_{\underline{\calG}, \mu}^{\on{sch}})^\diamondsuit$ and $X\to \calM_{\calG, \mu}$ determine closed subsheaves
		\[R_{(\calN_{\underline{\calG}, \mu}^{\on{sch}})^\diamondsuit}\subseteq X\times X\supseteq  R_{\calM_{\calG, \mu}}.\]
		Whether these subsheaves agree or not can be verified on geometric points, and since it clearly holds in the generic fiber, it suffices to show that they agree on the special fiber.
		This reduces us to compare the surjection

		\begin{equation}
		\calA_{\calI,\mu} \to \calA_{\calG, \mu}
		\end{equation}
		to the one obtained in the equicharacteristic situation. 
		Exploiting $\calI_k^{\on{perf}}$-equivariance we may proceed analogously to what was done in \Cref{lem_comparison_sch_vars_equi_and_mixed} and \Cref{prop_comparison_special_fiber}, so we omit it.
		
		Finally, suppose that $G$ is arbitrary. Thanks to \Cref{prop_LM_basic_facts}, $\calM_{\calG,\mu} $ is isomorphic to  $\calM_{\calG_\ad,\mu_\ad} $ after base change to $\Spd O_E$, and decomposes into products, hence we may assume $G$ is simple and adjoint. 
		We can find a locally closed immersion
		\begin{equation}
			\label{closed-immersion-of-groups}
		\calG \to \tilde{\calG},
		\end{equation}
		where $\tilde{\calG}$ is a parahoric model of a Weil-restricted split form of $G$, arguing as in the proof of \Cref{thm_specialization_local_models}. But the case of $\tilde{\calG}$ was treated in the previous paragraph.
		Since we have an inclusion $\calM_{\calG,\mu} \subset \calM_{\tilde{\calG},\tilde{\mu}}$, it now suffices to take the absolute weak normalization of the flat closure of $\calF_{G,\mu}$ inside the scheme-theoretic local model attached to $(\tilde{\calG},\tilde{\mu})$.
	\end{proof}
	\begin{remark}
		Let us explain how representability can be proved for classical groups without resorting to the characterization of the specialization map found in \Cref{thm_specialization_local_models}. Indeed, for those groups we can directly understand the v-sheaves $\calM^{\on{tor}}_{\calG,\mu_i}$ by embedding them in a similar torsor attached to Weil-restricted $\on{PGL}_n$. Those had been studied already by Pappas--Rapoport, see \cite[Proposition 5.2]{PR05}, and a careful analysis of the map in \cite[Proposition 21.6.9]{SW20} reveals that all proposed definitions coincide. The result follows by v-descent.
	\end{remark}

	We have found certain finite type $O_E$-schemes $\calM_{\calG, \mu}^{\mathrm{sch}}$ representing $\calM_{\calG,\mu}$, but we still do not know a lot about the geometry of its special fiber, see the discussion after \Cref{SW-conjecture}.
	We recall the canonical deperfection $\calA_{\calG,\mu}^{\on{can}}$ of the $\mu$-admissible locus introduced in \Cref{admissible_locus_definition} and \Cref{canonical_deperfection_definition}.
	
	\begin{theorem}\label{cor_special_fiber_local_model_reduced}
		Under \Cref{hyp_wild_odd_unitary} and \Cref{hyp_wild_triality}, the special fiber of the $O_E$-scheme $\calM_{\calG,\mu}^{\on{sch}}$ of \Cref{thm_local_model_representable} is uniquely $\calG_{k_E}$-equivariantly isomorphic to the canonical deperfection of the $\mu$-admissible locus:
		\begin{equation}
		\calM_{\calG,\mu}^{\on{sch}}|_{\Spec\, k_E} \cong \calA_{\calG,\mu}^{\on{can}}
		\end{equation} 
		In particular, $\calM_{\calG, \mu}^{\on{sch}}$ is normal, Cohen--Macaulay and has a reduced, weakly normal, Frobenius split special fiber.
	\end{theorem}
	
	\begin{proof}
		During the proof of \Cref{thm_local_model_representable}, we already saw that the algebraic local models $\calM_{\calG, \mu}^{\on{sch}}$ are actually the $\calN_{\underline{\calG}, \mu}^{\on{sch}}$ constructed in \cite[Definition~5.11]{FHLR22} by a variation on the techniques of Pappas--Zhu \cite{PZ13}, Levin \cite{Lev16} and also the third author \cite{Lou19, Lou20}. For this, we may pass to a finite unramified extension of $F$, so $G$ is quasi-split and residually split, so that $\calN_{\underline{\calG}, \mu}^{\on{sch}}$ is defined (under \Cref{hyp_wild_odd_unitary}). Then, it embeds in a local model associated with a Weil-restricted split group, confer \cite[Section~5.3.2]{FHLR22} (this is where \Cref{hyp_wild_triality} is used). 
		We conclude under the given hypothesis that the $\calM_{\calG, \mu}^{\on{sch}}$ are indeed normal, Cohen--Macaulay and have a Frobenius split special fiber by \cite[Theorem~5.14]{FHLR22}. Indeed, the special fiber of $\calN_{\underline{\calG}, \mu}^{\on{sch}}$ is reduced equal to an admissible locus $\calA_{\calG', \mu'}^{\on{can}}$ in the equicharacteristic setting, which equivariantly identifies with $\calA_{\calG,\mu}^{\on{can}}$ by \Cref{lem_comparison_sch_vars_equi_and_mixed}.
	\end{proof}
	
	\begin{remark}\label{last_cases_remark}
		More generally, \cite[Corollary~1.4]{GL24} proves \Cref{cor_special_fiber_local_model_reduced} without \Cref{hyp_wild_triality} but still assuming \Cref{hyp_wild_odd_unitary}.
		Invoking \cite{Lou19,Lou20} we get \Cref{cor_special_fiber_local_model_reduced} except if $p=2$ and $G_\ad$ has an odd unitary $\breve F$-factor defined by a ramified, quadratic root-of-unit extension.
		The remaining case is handled in \cite{CL24}, except for the assertion on Cohen--Macaulayness, compare with \Cref{remark_cass-lou}.
	\end{remark}
	
	To conclude, let us only use \Cref{thm_local_model_representable} --and not resort to the construction of local models in \cite{FHLR22}-- in order to study the geometry of the special fiber of $\calM_{\calG,\mu}^{\on{sch}}$.
	
	
	First of all, we know that the perfection of $\calM_{\calG, k_E, \mu}^{\on{sch}}$ equals $\calA_{\calG, \mu}$ by \Cref{theorem_special_fiber_admissible} and fully faithfulness of $\diamondsuit$ on perfect schemes, see \cite[Proposition 18.3.1]{SW20}. 
	By the weak normality property and \Cref{lem_fixer_lambda_scheme_smooth}, we conclude that $\calM_{\calG, \mu}^{\on{sch}}$ admits a smooth open subscheme $\calM_{\calG, \mu}^{\on{sch}, \circ}$ descending
	\begin{equation}
	\calM_{\calG, O_{\breve E}, \mu}^{\on{sch},\circ} =\bigcup_\la \calG_{O_{\breve E}}/\calP_\la^-, 
	\end{equation} 
	compare with the argument in \cite[Corollary 2.14]{Ric16}. 
	It follows that we have a natural morphism
	\begin{equation}\label{eq_mor_can_admissible_to_special_fiber}
	\calA_{\calG,\mu}^{\on{can}}\rightarrow\calM_{\calG, k_E, \mu}^{\on{sch}} .
	\end{equation}
	The following conjecture is then the full $p$-adic coherence conjecture:
	
	\begin{conjecture}\label{conjecture_special_fiber}
		The map \eqref{eq_mor_can_admissible_to_special_fiber} is always an isomorphism.
	\end{conjecture} 
	
	This conjecture is settled in \cite{CL24} in all cases, see \Cref{remark_cass-lou,last_cases_remark}.
	Indeed, the condition in \Cref{lemma-v-sheaf-reduced} below has been verified in \cite{GL24} after the first version of this paper was written. 
	So, we know by \cite[Corollary~1.4]{GL24} that $\calM^{\on{sch}}_{\calG, k_E,\mu}$ is always reduced.	
To show that \eqref{eq_mor_can_admissible_to_special_fiber} is an isomorphism, we are hence reduced to comparing Hilbert polynomials via the dimension formula in \Cref{theorem_coherence_allp}. This applies in all cases, since \Cref{hyp_wild_odd_unitary} was lifted in \cite{CL24}.

	\begin{lemma}\label{lemma-v-sheaf-reduced}
		Suppose $(\widehat{\calM_{\calG,O_C,\mu}}_{/\bar{x}})_\eta$ is connected for every $\bar{k}$-valued point $\bar{x}$ of $\calA_{\calG, \mu}$. Then $\calM^{\on{sch}}_{\calG, k_E,\mu}$ is geometrically reduced. Under \Cref{hyp_wild_odd_unitary}, \Cref{conjecture_special_fiber} holds.
	\end{lemma}
	
	\begin{proof}
		By \Cref{prop_connected_tubes_implies_normality}, we know that $\calM_{\calG,\mu}^{\mathrm{sch}}$ is normal. 
		In particular, $\calM^{\on{sch}}_{\calG, k_E,\mu}$ is S1, but it must also be R0, as it contains a smooth dense open $\calA^{\on{can},\circ}_{\calG,\mu}$. 
		So Serre's criterion for reducedness furnishes the claim.
		As for identifying the special fiber with $\calA_{\calG,\mu}^{\on{can}}$ as per \Cref{conjecture_special_fiber}, we appeal to \Cref{theorem_coherence_allp}, which computes the dimension of the vector spaces of global sections of ample line bundles. 
	\end{proof}
	
	Finally, let us also mention the following conjecture, arising from \cite{FHLR22}, on the singularities of $\calM^{\on{sch}}_{\calG,\mu}$:
	
	\begin{conjecture}\label{pseudo_rational_conjecture}
		The local model $\calM^{\on{sch}}_{\calG,\mu}$ has pseudo-rational singularities. 
	\end{conjecture}
	
	\section{The test function conjecture}\label{test_function_section}
	Throughout this section, we let $F/\bbQ_p$ be a finite field extension with ring of integers $O$ and finite residue field $k$ of cardinality $q$.
	We fix an algebraic closure $\bar \bbQ_p$, an embedding $F\hookto \bar\bbQ_p$ and denote by $\Gamma=\Gal(\bar\bbQ_p/F)$ the absolute Galois group of $F$ with inertia subgroup $I$.
	Let $G$ be a reductive $F$-group with parahoric $O$-model $\calG$.
	
	Furthermore, fix a square root $\sqrt q$, an auxiliary prime $\ell\nmid q$ and put $\Lambda=\bbQ_\ell(\sqrt q)$. 
	We let $^LG=\widehat G_\Lambda\rtimes \Gamma$ be the Langlands dual group viewed as a pro-algebraic $\Lambda$-group scheme. 	
	Each algebraic representation $V$ of $^LG$ furnishes, by choosing a quasi-inverse to the geometric Satake equivalence, a semi-simple perverse $\Lambda$-sheaf $\Sat(V)$ of ``weight zero'' on the $B_\dR^+$-affine Grassmannian $\Gr_G\to \Spd F$.
	Here $\sqrt q$ is needed to define a square root of the $\ell$-adic cyclotomic character used when Tate twisting irreducible perverse sheaves supported on components of $\Gr_G$ of odd parity to be of ``weight zero''. 
	More precisely, for a dominant coweight $\mu$ defined over $F$, we have 
	\begin{equation}\label{eq:normalized_IC}
	\Sat(V_\mu)=i_{\mu,*} j_{\mu,!*}\Lambda_{\Gr_\mu^\circ}\big(\textstyle{{\langle2\rho,\mu\rangle}\over 2}\big), 
	\end{equation}
	where $\Gr_{G,\mu}^\circ\overset{j_\mu\,}{\to}\Gr_{G,\mu}\overset{i_\mu\,}{\to}\Gr_G$ and $V_\mu$ is the irreducible representation of $^LG$ of highest weight $\mu$. 
	Every simple object is of this form, up to taking a finite Galois orbit of $\mu$'s and tensoring with simple $\Lambda$-local systems on $\Spd F$ of weight zero (corresponding to irreducible representations of $\Gamma$ factoring through a finite quotient).
	
	As in \Cref{nearby.cycles.section}, we consider the functor of nearby cycles
	\begin{equation}\label{eq:adic_nearby_cycles}
	\Psi_{\calG}:=i^*Rj_*(\str)|_{\Spd \bbC_p}\colon \D(\Hk_{G}, \Lambda) \rightarrow \D(\Hk_{\calG, \bar k},\Lambda),
	\end{equation}
	where $\Hk_{G, \bbC_p} \xrightarrow{j} \Hk_{\calG, O_{\bbC_p}} \xleftarrow{i} \Hk_{\calG, \bar k}$ are the inclusions of the geometric fibers.
	
	\begin{lemma} \label{adic_nearby_cycles}
		For every finite dimensional algebraic $^LG$-representation $V$, the sheaf of nearby cycles $\Psi_\calG(\Sat(V))$ naturally defines an object in the category 
		\begin{equation}\label{Galois_action_category}
		\D_{\on{cons}}\big([\underline{\Gamma} \backslash \Hk_{\calG,\bar{k}}^{\on{sch}}],\Lambda\big)^{\on{bd}} 
		\end{equation}
		of constructible $\Lambda$-sheaves with bounded support on the v-stack $[\underline{\Gamma} \backslash \Hk_{\calG,\bar{k}}^{\on{sch}}]$. 
		Here, $\underline{\Gamma}$ denotes the associated group v-sheaf and the action on the schematic Hecke stack $\Hk_{\calG,\bar{k}}^{\on{sch}}$ is induced by the quotient map $\Gamma\to \Gal(\bar k/k)$. 
		In particular, the cohomology sheaves $\on{R}^n\!\Psi_\calG(\Sat(V))$, $n\in \bbZ$, define $L^+_{\bar k}\calG$-equivariant, constructible $\Lambda$-sheaves with bounded support on $\Fl_{\calG,\bar k}$ equipped with an equivariant continuous $\Gamma$-action as defined in \cite[Exposé XIII]{SGA7.2}, compatibly with the $L^+_{\bar k}\calG$-action.
	\end{lemma}
	\begin{proof} 
		The group $\Gamma$ is identified with the group of continuous automorphisms of $\bbC_p$ over $F$.
		Since the geometric fiber inclusions $i$ and $j$ are $\underline \Gamma$-equivariant, we obtain maps of v-stacks
		 \begin{equation}
		\Hk_G=[\underline{\Gamma}\backslash\Hk_{G,\bbC_p}] \xrightarrow{[\underline{\Gamma}\backslash j]} [\underline{\Gamma}\backslash\Hk_{\calG,O_{\bbC_p}}] \xleftarrow{[\underline{\Gamma}\backslash i]} [\underline{\Gamma}\backslash\Hk_{\calG, \bar k}],
		 \end{equation} 
		 and define the Galois equivariant nearby cycles functor $[\underline{\Gamma}\backslash\Psi_{\calG}]:=[\underline \Gamma\backslash i]^*R[\underline \Gamma \backslash j]_*(\str)$ in analogy to \eqref{eq:adic_nearby_cycles}.
		 Consider the quotient map $v\co \Hk_{\calG, \bar k}\to [\underline{\Gamma}\backslash\Hk_{\calG, \bar k}]$. 
		 We claim that the map $v^*[\underline{\Gamma}\backslash\Psi_{\calG}](\mathrm{Sat}(V))\to \Psi_{\calG}(\mathrm{Sat}(V))$ induced by base change is an isomorphism.
		 
		 Note that we cannot apply the base change theorem directly because $j$ is not quasi-compact and $\Gamma$ is only profinite.
		Instead, we apply the constant term functor, which commutes with arbitrary base change: 
		By finite étale descent, we may replace $F$ by an unramified, finite Galois extension (hence, also $\Gamma$ by the corresponding normal subgroup of finite index) and assume that $G$ is quasi-split and residually split. 
		So, every $\breve{F}$-Borel descends to $F$. 
		Using the conservativity of the constant term functor, see \Cref{prop_conservative_constant_term}, we see again as in \Cref{ULA_nearby_cycles} that the equivariant integral extension $R[\underline{\Gamma}\backslash j]_*\on{Sat}(V)$ is ULA over $[\underline{\Gamma}\backslash\Spd O_{\bbC_p}]$. 
		In particular, so is its pullback to $\Spd O_{\bbC_p}$, which implies by \Cref{ULA_nearby_cycles} that it equals $Rj_*(\on{Sat}(V)|_{\Spd \bbC_p})$. 
		Restricting to geometric special fibers implies the claim. 
		 
		 It formally follows from \Cref{sec:comp-etale-cohom-1-derived-categories-of-hecke-stacks-are-isomorphic} and the construction of derived categories of $\Lambda$-sheaves that the comparison functor \eqref{comparison_functor:eq} induces an equivalence
		 \begin{equation}
		 \D([\underline \Gamma\backslash \mathrm{Hk}_{\calG,\bar k}^{\mathrm{sch}}],\Lambda)^{\mathrm{bd}}\cong \D([\underline \Gamma\backslash \mathrm{Hk}_{\calG,\bar k}],\Lambda)^{\mathrm{bd}}
		 \end{equation}
		 under which constructible sheaves correspond to ULA sheaves, see \Cref{prop_ula_special_fiber}.
		 Note that both properties are preserved and detected under the functor $v^*$, respectively its schematic counterpart. 
		 Therefore, $[\underline{\Gamma}\backslash\Psi_{\calG}](\mathrm{Sat}(V))$ naturally defines an object in the category \eqref{Galois_action_category} and its underlying sheaf is $\Psi_{\calG}(\mathrm{Sat}(V))$.
		 
		 For the final statement on the comparison with \cite[Exposé XIII]{SGA7.2}, we reduce to the case where $\Lambda$ is a finite ring by the construction of categories of $\ell$-adic sheaves, see also \eqref{eq:localization_categories}. 
		 Then, for any qcqs $k$-scheme $X$, the category of abelian $\Lambda$-sheaves on $(X_{\bar k})_\et$ equipped with a continuous $\Gamma$-action as in \cite[Exposé XIII, Définition 1.1.2]{SGA7.2} embeds fully faithfully into the category of abelian $\Lambda$-sheaves on $[\underline\Gamma\backslash X_{\bar k}]$ inducing an equivalence on full subcategories of constructible sheaves.
		 Applying this to closed subschemes $X\subset \Fl_{\calG}$ implies the lemma.
	\end{proof}
	
	For every $\Phi\in \Gamma$, we define a function $\Hk_\calG^{\on{sch}}(k)\to \Lambda$ by the formula
	\begin{equation}\label{test_function}
	\tau_{\calG,V}^\Phi(x):=(-1)^{d_V}\sum_{n\in \bbZ}(-1)^n\on{trace}\big(\Phi\;|\;\on{R}^n\!\Psi_\calG\Sat(V)_{\bar x}\big)
	\end{equation}
	whenever $V$ is irreducible and extend the definition to general $V$ by linearity.
	Here, $d_V=\langle2\rho,\mu\rangle$ with $\mu$ being the highest weight of $V$. 
	So the sign $(-1)^{d_V}$ in \eqref{test_function} only depends on the parity of the connected component of $\Gr_G$ that supports $\Sat(V)$.
	
	\begin{lemma}\label{}
		For every finite dimensional algebraic $^LG$-representation $V$ and every $\Phi\in \Gamma$, the function $\tau_{\calG,V}^\Phi$ naturally lies in the center of the parahoric Hecke algebra $\calH(G(F),\calG(O))_\Lambda$.
	\end{lemma}
	\begin{proof}
		Lang's lemma together with an approximation argument \cite[Lemma A.3]{RS20} implies that $\on{H}^1_\et(k, L^+\calG)$ vanishes, so $\Hk_\calG^{\on{sch}}(k)=\calG(O)\backslash G(F)/ \calG(O)$.
		As the function $\tau_{\calG,V}^\Phi$ is supported on finitely many double cosets, it lies in $\calH(G(F),\calG(O))_\Lambda$. 
		Centrality follows from \Cref{nearby_cycles_central} and the usual sheaf function dictionary.
	\end{proof}
	
	On the other hand, the theory of Bernstein centers defines another function: 
	Namely, for every choice of lift $\Phi\in\Gamma$ of geometric Frobenius, we let $z_{\calG,V}^\Phi$ be the unique function in the center of $\calH(G(F),\calG(O))_\Lambda$ that acts on every smooth irreducible $\calG(O)$-spherical representation $\pi$ over $\Lambda$ by the scalar
	\begin{equation}
	\on{trace}\Bigl(s^\Phi(\pi)\;\big|\;  V\Bigr),
	\end{equation}
	where $s^\Phi(\pi)\in [\widehat{G}^{I}\rtimes \Phi]_{\on{ss}}/\widehat{G}^{I}$ is the Satake parameter for $\pi$ with respect to $\Phi$ constructed in \cite{Hai15}. 
	
	\begin{theorem}\label{tfc_reformulation}
		For every finite dimensional algebraic $^LG$-representation $V$ and every choice of lift $\Phi$ of geometric Frobenius, there is an equality 
		\begin{equation}\label{eq:tfc_reformulation}
		\tau_{\calG,V}^\Phi=z_{\calG,V}^\Phi
		\end{equation}
		of functions in the parahoric Hecke algebra.
	\end{theorem}
	\begin{proof}
		As both sides of \eqref{eq:tfc_reformulation} are additive in $V$, we may freely assume that $V$ is irreducible, and even further that $V|_{\widehat G\rtimes I}$ is irreducible: otherwise both sides in \eqref{eq:tfc_reformulation} are zero (hence, equal) by elementary considerations, see \cite[Lemma 7.7]{HR21}.
		
		Fix a maximal $F$-split torus $A\subset G$ whose Néron model embeds in $\calG$ and a regular cocharacter $\la\co \bbG_m\to A$.
		Then $\la$ induces a minimal $F$-Levi $M$, respectively $F$-parabolic $P$ in $G$.
		Denote by $\calM\subset \calP$ their flat closures in $\calG$.
		Then, the constant term morphism \cite[Section 11.11]{Hai14} induces an injective morphism on the centers of the parahoric Hecke algebras
		\begin{equation}
		\on{ct}_\calP\co \calZ(G(F),\calG(O))_\Lambda \hookto \calZ(M(F),\calM(O))_\Lambda.
		\end{equation}
		As in \cite[Lemma 7.8, Equation (7.15)]{HR21}, one checks the formulas
		\begin{equation}\label{eq:constant_terms_function}
		\on{ct}_\calP\big(\tau_{\calG,V}^\Phi\big) = \tau_{\calM,V|_{^LM}}^\Phi , \;\;\; \on{ct}_\calP\big(z_{\calG,V}^\Phi\big) = z_{\calM,V|_{^LM}}^\Phi,
		\end{equation}
		where $^LM=\widehat M\rtimes \Gamma$ is viewed as a closed subgroup of $^LG$. 
		The second formula in \eqref{eq:constant_terms_function} is straightforward. 
		The first formula in \eqref{eq:constant_terms_function} is based on the isomorphism
		\begin{equation}
		\on{CT}_\calP[\deg_\calP]\circ \Psi_\calG\cong \Psi_\calM\circ \on{CT}_P[\deg_P] \colon \Sat(\Hk_G,\Lambda)\to \D_{\on{cons}}([\underline{\Gamma}\backslash\Fl_{\calM,\bar k}],\Lambda)^{\on{bd}},
		\end{equation}
		see \Cref{prop_commutativity_nearby_cycles_constant_term}, using that $\on{CT}_P[\deg_P]$ corresponds to the restriction of representations $V\mapsto V|_{^LM}$ under the geometric Satake equivalence \cite[Section VI]{FS21}.
		(We note that the sign $(-1)^{d_V}$ in \eqref{test_function} appears when comparing $\on{CT}_\calP[\deg_\calP]$ and $\on{ct}_\calP$ under the sheaf function dictionary, see also \cite[Lemma 7.2]{HR21}.)
		
		Hence, we reduce to the case where $G=M$ is a minimal $F$-Levi, so anisotropic modulo center, and $V|_{\widehat G\rtimes I}$ is irreducible. 
		Let $\calM_{\calG,V}$ be the v-sheaf theoretic closure of the support of $\Sat(V)$ in $\Gr_\calG$, a finite union of $\calM_{\calG,\mu}$ for $\mu$ ranging over the highest weights of $V$.
		The proof of \cite[Lemma 7.13]{HR21} is based on Iwahori-Weyl group combinatorics, hence applies to show that $\calM_{\calG,V}$ has only a single $\Spd k$-valued point $x_V$.
		As $\Phi$ lifts the geometric Frobenius, we can apply the Grothendieck-Lefschetz trace formula to $\Psi_\calG\Sat(V)$ viewed as an object in $\D_{\on{cons}}([\underline{\Gamma}\backslash \Fl_{\calG,\bar k}],\Lambda)^{\on{bd}}$ to compute
		\begin{equation}\label{eq:GL_trace_formula}
		\on{trace}\big(\Phi\;|\;\Psi_\calG\Sat(V)_{\overline{x_V}}\big)=\on{trace}\big(\Phi\;|\;\on{H}^*(\Fl_{\calG,\bar k},\Psi_\calG\Sat(V))\big)
		\end{equation}
		Since $Rj_*(\Sat(V)|_{\Spd\, \bbC_p})$ is ULA by \Cref{ULA_nearby_cycles}, the latter cohomology group is $\Gamma$-equivariantly isomorphic to 
		\begin{equation}\label{eq:quasisplit_inner}
		\on{H}^*(\Gr_{G, \bbC_p},\Sat(V))=\on{H}^*(\Gr_{G^*, \bbC_p},\Sat(V)),
		\end{equation}
		where $G^*$ is the unique quasi-split inner form of $G$.
		We note that there is a canonical identification $^LG={^L}G^*$ so that on Satake categories $\on{Sat}(\Hk_G,\Lambda)\cong \on{Sat}(\Hk_{G^*},\Lambda)$ by \cite[Section VI]{FS21}.	
		Let $\calG^*$ denote the parahoric corresponding to $\calG$ (necessarily, an Iwahori) and $\calM_{\calG^*,V}$ the associated v-sheaf local model. 
		On the other hand, we know \cite[Proposition 11.12.6]{Hai14} that $z_{\calG,V}^\Phi$ is supported at $x_V$ with value
		\begin{equation}\label{eq:tfc_anisotropic_sum}
		z_{\calG,V}^\Phi(x_V)=\sum_{x\in \calM_{\calG^*,V}(\Spd\, k)}z_{\calG^*,V}^\Phi(x).
		\end{equation}
		Now assuming the test function conjecture for the pair $(\calG^*,V)$, that is, assuming $z_{\calG^*,V}^\Phi=\tau_{\calG^*,V}^\Phi$, we can apply the Grothendieck-Lefschetz trace formula again to see that \eqref{eq:tfc_anisotropic_sum} equals the trace of $\Phi$ on \eqref{eq:quasisplit_inner}, up to the sign $(-1)^{d_V}$.
		So $z_{\calG^*,V}^\Phi=\tau_{\calG^*,V}^\Phi$ implies $z_{\calG,V}^\Phi=\tau_{\calG,V}^\Phi$. 
		
		Hence, we reduce to the case where $G=G^*$ is quasi-split. 
		Now, the minimal Levi $M=T$ is a maximal torus, so $\eqref{eq:constant_terms_function}$ reduces us to the case where $G=T$ is a torus and $\calG=\calT$ its connected locally finite type N\'eron model. 
		Without loss of generality, we assume that $V|_{\widehat T\rtimes I}$ is irreducible.
		Evidently, $T$ is anisotropic modulo center so that both functions $\tau_{\calT,V}^\Phi, z_{\calT,V}^\Phi$ are supported at $x_V$. 
		Using \eqref{eq:GL_trace_formula}, the ULA property of $Rj_*(\Sat(V)|_{\Spd\, \bbC_p})$ and $\on{H}^0(\Gr_{T, \bbC_p},\Sat(V))=V$, we see
		\begin{equation}
		\tau_{\calT,V}^\Phi(x_V)=(-1)^{d_V}\on{trace}\big(\Phi\;|\; V\big)
		\end{equation}
		which equals $z_{\calT,V}^\Phi(x_V)$ because $d_V=0$. 
		This finishes the proof.
	\end{proof}

	\begin{lemma}
		\label{sec:test-funct-conj-reformulation-implies-test-function-conjecture}
		\Cref{tfc_reformulation} implies \Cref{HK_conjecture_intro}.
	\end{lemma}
	\begin{proof}
		Let $\mu$ be a conjugacy class of geometric cocharacters in $G$.
		Denote by $E \subset\bar\bbQ_p$ its reflex field with maximal unramified subextension $E_0/F$. 
		Their rings of integers are denoted by $O_E\supset O_{E_0}$ with residue fields $k_E=k_{E_0}$ and absolute Galois groups $\Gamma_E\subset \Gamma_{E_0}$.
		For every $\Phi\in \Gamma_E$ and $x\in \Gr_\calG(k_E)$, there is an equality
		\begin{equation}\label{eq:induction_tfc}
		\on{trace}\big(\Phi\;|\;\Psi_{\calG,O_E}\Sat(V_\mu)_{\bar x}\big) = \on{trace}\big(\Phi\;|\;\Psi_{\calG,O_{E_0}}\Sat(I_E^{E_0}(V_\mu))_{\bar x}\big),
		\end{equation}
		where $I_E^{E_0}(V_\mu)$ is the induction to $\widehat G\rtimes \Gamma_{E_0}$ of the $\widehat G\rtimes \Gamma_E$-representation $V_\mu$ and $\Sat(V_\mu)$, $\Sat(I_E^{E_0}(V_\mu))$ the corresponding Satake sheaves on $\Gr_{G}|_{\Spec\, E}$, respectively $\Gr_G|_{\Spec\, E_0}$.
		Indeed, \eqref{eq:induction_tfc} follows from the commutation of nearby cycles with proper pushforward applied to the finite morphism $\Gr_\calG|_{\Spd O_E}\to \Gr_\calG|_{\Spd O_{E_0}}$, noting that it induces the induction of representations on Satake categories.
		
		Now, we apply \eqref{eq:induction_tfc} to the pair $\calG_0:=\calG_{O_{E_0}}, V_{\mu,0}:=I_E^{E_0}(V_\mu)$ and any choice of lift $\Phi\in \Gamma_E\subset \Gamma_{E_0}$ of geometric Frobenius to obtain 
		\begin{equation}\label{eq:induction_tfc_app}
		\tau_{\calG_0,V_{\mu,0}}^\Phi=z_{\calG_0,V_{\mu,0}}^\Phi.
		\end{equation}
		The left hand side of \eqref{eq:induction_tfc_app} is equal to the function from \Cref{HK_conjecture_intro} by \eqref{eq:induction_tfc} and so is the right hand side of \eqref{eq:induction_tfc_app} by a similar equality \cite[Lemma 8.1]{Hai18} for inductions of representations along the totally ramified extension $E/E_0$.
		That the function \eqref{eq:induction_tfc_app} takes, after multiplying by $(\sqrt {q_E})^{\langle2\rho, \mu\rangle}$, values in $\bbZ$ independently of the choice of $\ell\neq p$, $\sqrt{q_E}$ and $E\hookto \bar\bbQ_p$ follows from \cite[Theorem 7.15]{HR21} where the statement is verified for the semi-simplified version of the right hand side of \eqref{eq:induction_tfc_app} without any assumptions on $(\calG,\mu)$.
		The same arguments apply here.
	\end{proof}

	\begin{appendices}

		\section{\'Etale sheaves over v-stacks on perfect schemes}
		\label{sec:comp-etale-cohom}	
		In this section, we extend some parts of \cite[Section 27]{Sch17} to v-stacks on perfect schemes, see also \cite[Appendix A]{Wu21}.
		Let $k$ be a perfect field of characteristic $p>0$. 
		Let $X$ be a perfect scheme over $k$, and let $X^\diamondsuit$ be the associated v-sheaf from \Cref{sec:perfect-schemes} (under the tilting equivalence), that is, if $\Spa(R,R^+)$ is an affinoid perfectoid space over $\Spa(k,k)$ and $X$ affine, then
		\begin{equation}
		X^\diamondsuit(\Spa(R,R^+))=X(\Spec(R)),
		\end{equation}
		while for $X$ general $X^\diamondsuit$ is the analytic sheafification of the presheaf $\Spa(R,R^+)\mapsto X(\Spec(R))$. 
		We note that the functor $(\str)^\diamondsuit$ does depend on $k$.

		Fix a torsion ring $\Lambda$ with $p\in \Lambda^\times$.                      
		We let $\D(X, \Lambda):=\widehat{\D}(X_\et,\Lambda)$ be the left-completed \'etale derived $\infty$-category of $X$, see \cite[Section 27]{Sch17}.
		By \cite[Section 27]{Sch17} there is a morphism
		\begin{equation}
			\label{c-functor}
		c_X\colon (X^\diamondsuit)_v\to X_\et
		\end{equation}
		of sites (even to the proétale site of $X$), and the induced functor of $\infty$-categories
		\begin{equation}\label{comparison_functor:eq}
		c_X^\ast\colon \D(X,\Lambda)\to \D(X^\diamondsuit, \Lambda)
		\end{equation}
		is fully faithful, \cite[Proposition 27.2.]{Sch17}.
		In general the functor $c_X^\ast$ is not essentially surjective, for example, on topological spaces $|X^\diamondsuit|\to |X|$ is surjective, but very often not injective. 
		
		The functor $c_X^\ast$ enjoys many compatibilities. 
		If $f\co Y\to X$ is a map of schemes, then $c_X^\ast\circ f^*\cong (f^\diamondsuit)^*\circ c_X^\ast$ and $c_X^\ast(A)\otimes^\bbL_\Lambda c_X^\ast(B)\cong c_X^\ast(A\otimes_\Lambda^\bbL B)$, see \cite[Proposition 27.1.]{Sch17}. 
		If $f\co Y\to X$ is separated perfectly of finite type, then $c_X^\ast\circ Rf_!\cong Rf^\diamondsuit_!\circ c_X^\ast$, see \cite[Proposition 27.4.]{Sch17}. 
		As we now justify $c_X^\ast$ also preserves ULA-objects.
		
		\begin{proposition}
			\label{algebraic-is-ula}
			Let $S$ be a qcqs perfect scheme in characteristic $p$, and let $f\co X\to S$ be a separated perfect scheme perfectly of finite presentation over $S$. 
			Let $A\in \D(X, \Lambda)$ such that $A$ is ULA with respect to $f$. 
			Then $c_X^\ast A$ is ULA with respect to $f^\diamondsuit$, and $c_X^\ast \mathbb{D}_{X/S}(A)\cong \mathbb{D}_{X^\diamondsuit/S^\diamondsuit}(c_X^\ast A)$ for the Verdier duals.   
		\end{proposition}       
		\begin{proof}
			As in \cite[Section 3]{HS21}, we let $\mathcal{C}_S$ denote the $2$-category whose objects are schemes over $S$ as in the hypothesis, and where morphisms from $X$ to $Y$ are given by objects in $\D(X\times_S Y, \Lambda)$.
			Given two maps $A\in \Hom_{\mathcal{C}_S}(X_1,X_2)$ and $B\in \Hom_{\mathcal{C}_S}(X_2,X_3)$, we define their composition $A\ast B\in \Hom_{\mathcal{C}_S}(X_1,X_3)$ by the formula 	
			\begin{equation}
			A\ast B := R\pi_{1,3}{}_!(\pi_{1,2}^* A\otimes_{\Lambda}^{\bbL}\pi_{2,3}^* B).
			\end{equation}			
			By \cite[Proposition 3.4, Definition 3.2]{HS21}, the object $A\in \Hom_{\mathcal{C}_S}(X,S)$ is ULA with respect to $f$ if and only if $A$ is a left adjoint in $\mathcal{C}_S$. 
			Analogously, let $\mathcal{C}_{S^\diamondsuit}$ denote the category considered in \cite[Section IV.2.3.3.]{FS21}. 
			By \cite[Theorem IV.2.23.]{FS21}, the object $c_X^\ast A\in \Hom_{\mathcal{C}_S}(X^\diamondsuit,S^\diamondsuit)$ is ULA with respect to $f^\diamondsuit$ if it is a left adjoint in $\mathcal{C}_{S^\diamondsuit}$. 
			Now, we observe that the functors $c_X^\ast$ can be promoted to a functor of $2$-categories $c^*\co {\mathcal{C}_S}\to {\mathcal{C}_{S^\diamondsuit}}$ by the rule $c^*X=X^\diamondsuit$ and $c^*(A)=c_{X\times_S Y}^*(A)$ for $A\in \Hom_{\mathcal{C}_S}(X,Y)$. 
			Here, we use that $c^*$ commutes with the required operations by \cite[Propositions 27.1, 27.4]{Sch17}.
			But functors between 2-categories preserve the adjunctions between 1-morphisms which finishes the proof as the right adjoints are given by Verdier duals. 
		\end{proof}
		
		We move on to study stacks. 
		Let $\mathrm{Ani}$ be the category of anima (also called spaces, $\infty$-groups or Kan complexes).
		By left Kan extension along the Yoneda embedding
		\begin{equation}
		\mathrm{SchPerf}_{k}\to \mathrm{Fun}(\mathrm{SchPerf}_{k}^{\mathrm{op}},\mathrm{Ani}),
		\end{equation}
		we can extend\footnote{We take care of the set-theoretic issues by fixing a suitable cut-off cardinal, large enough to allow all the examples that we are interested in.} the functors $\D(\str,\Lambda), \D((\str)^\diamondsuit,\Lambda)$ using $\ast$-pullbacks and the natural transformation $c_{(\str)}^\ast$ to contravariant functors $\D(\str,\Lambda), \D((\str)^\diamondsuit,\Lambda)^{\mathrm{Kan}}$ on $\mathrm{Fun}(\mathrm{SchPerf}_{k}^{\mathrm{op}},\mathrm{Ani})$ with values in symmetric monoidal stable $\infty$-categories, sending colimits to limits. 
		More concretely, if a functor (also known as, higher prestack)
		\begin{equation}
		\mathfrak{X}\cong \underset{i}{\colim} X_i\in \mathrm{Fun}(\mathrm{SchPerf}_{k}^{\mathrm{op}},\mathrm{Ani})
		\end{equation}
		is written as a colimit of representables, then
		\begin{equation}
		\D(\mathfrak{X},\Lambda)\cong \lim_{i} \D(X_i,\Lambda)
		\end{equation}
		and similarly for $\D(\mathfrak{X}^\diamondsuit,\Lambda)^{\mathrm{Kan}}$. 
		By \cite[Proposition 27.2]{Sch17} and \cite[Lemma B.6]{BN19} (more precisely, its proof of $1.$), the natural transformation
		\begin{equation}
		c_{\mathfrak{X}}^\ast\colon \D(\mathfrak{X},\Lambda)\to \D(\mathfrak{X}^\diamondsuit,\Lambda)^{\mathrm{Kan}}
		\end{equation}
		is fully faithful. 
		Note that at this moment the right hand side is not the left completed derived \'etale category of some (higher) v-stack ``$\mathfrak{X}^\diamondsuit$'' on $\mathrm{Perf}_{k}$, but depends on $\mathfrak{X}$ and is defined abstractly (therefore, we have added the superscript $\mathrm{Kan}$).
		
		Assume now that $\mathfrak{X}$ is a stack (in $1$-groupoids) with representable diagonal such that there exists a v-cover by a perfect scheme $X\to \mathfrak{X}$. Then
		\begin{equation}
		\mathfrak{X}\cong \colim_{\Delta^{\mathrm{op}}} X^{\bullet/\mathfrak{X}}
		\end{equation}
		with the colimit of the \v{C}ech nerve of $X\to \mathfrak{X}$ taken in $\mathrm{Ani}$-valued v-sheaves on $\mathrm{SchPerf}_{k}$. 
		Using \cite[Theorem 5.7]{HS21}, we get
		\begin{equation}
		\D(\mathfrak{X},\Lambda)\cong \lim_{\Delta }\D(X^{\bullet/\mathfrak{X}},\Lambda).
		\end{equation}
		Indeed, by definition
		\begin{equation}
		\D(\mathfrak{X},\Lambda)\cong \lim_{U\to \mathfrak{X}} \D(U,\Lambda)
		\end{equation}
		where the limit is taken over all perfect schemes with a morphism to $\mathfrak{X}$, and thus
		\begin{equation}
		\begin{matrix}
		& \D(\mathfrak{X},\Lambda) \\
		= & \lim_{U\to \mathfrak{X}} \D(U,\Lambda)\\
		\cong & \lim_{U\to \mathfrak{X}} \lim_{\Delta} \D(U\times_{\mathfrak{X}} X^{\bullet/\mathfrak{X}},\Lambda) \\
		\cong & \lim_{\Delta} \lim_{U\to \mathfrak{X}} \D(U\times_{\mathfrak{X}} X^{\bullet/\mathfrak{X}},\Lambda)\\
		\cong & \lim_{\Delta} \D(\mathfrak{X}\times_{\mathfrak{X}} X^{\bullet/\mathfrak{X}},\Lambda) \\
		\cong & \lim_{\Delta} \D(X^{\bullet/\mathfrak{X}},\Lambda),
		\end{matrix}
		\end{equation}
		where the first isomorphism comes from \cite[Theorem 5.7]{HS21} applied to the covering $X \times_{\mathfrak{X}} U\to U$, the second isomorphism just commutes two inverse limits, and the last two isomorphisms use that $\D(\str,\Lambda)$ sends (by definition) colimits to limits and that $\colim_{U\to \mathfrak{X}} U\times_{\mathfrak{X}} X^{n/\mathfrak{X}}\cong X^{n/\mathfrak{X}}$.
		
		Let $X^{\diamondsuit,\bullet/\mathfrak{X}}$ be the simplicial v-sheaf obtained by applying the functor $(\str)^\diamondsuit$ to $X^{\bullet/\mathfrak{X}}$. 
		Now assume additionally that the projection maps in $X^{\diamondsuit,\bullet/\mathfrak{X}}$ are v-covers, and let $\mathfrak{X}^\diamondsuit$ be the colimit of $X^{\diamondsuit,\bullet/\mathfrak{X}}$ in v-stacks on $\mathrm{Perf}_{k}$. 
		Then $\mathfrak{X}^\diamondsuit$ is a small v-stack with well-defined $\D(\mathfrak{X}^\diamondsuit,\Lambda)$, and actually is the v-stackification of $\Spa(R,R^+)\mapsto \mathfrak{X}(\Spec(R))$. 
		By \cite[Proposition 17.3]{Sch17}, we can conclude that
		\begin{equation}
		\D(\mathfrak{X}^\diamondsuit,\Lambda)\cong \lim_{\Delta^{\mathrm{op}}} \D(X^{\diamondsuit,\bullet/\mathfrak{X}},\Lambda).
		\end{equation}
		Moreover, note that there exists a canonical morphism
		\begin{equation}
		\D(\mathfrak{X}^\diamondsuit,\Lambda)^{\mathrm{Kan}}\to \D(\mathfrak{X}^\diamondsuit,\Lambda),
		\end{equation}
		which probably need not be an equivalence in general, but whose composite with $c^\ast_{\mathfrak{X}}$ is still fully faithful.
		
		\begin{lemma}
			\label{sec:comp-etale-cohom-2-diamondsuit-preserves-finitely-presented-v-covers} 
			If $f\colon Y\to X$ is a v-cover of perfect schemes which is a map of (perfectly) finite presentation, then $f^\diamondsuit\colon Y^{\diamondsuit}\to X^{\diamondsuit}$ is a v-cover.                                              
		\end{lemma}
		\begin{proof}
			The functor $(\str)^{\diamondsuit}$ preserves open covers, so we may assume that $Y\to X$ are affine. 
			In this case, a (perfectly) finite presented v-cover is a cofiltered limit of v-covers between (perfect) affine schemes of (perfect) finite presentation over $\Spec(k)$, see \cite[Lemma 2.12]{BS17}. 
			As $f$ is of (perfectly) finite presentation, we may assume that $f$ is the base change of a v-cover between (perfect) affine schemes of (perfect) finite presentation over $k$. 
			The functor $(\str)^\diamondsuit$ preserves fiber products, and base changes of v-cover of v-sheaves on $\Perf_{k}$ are again v-covers. 
			Thus, we may reduce to the case that $Y,X$ are of (perfect) finite presentation over $k$. Then the statement follows from \cite[Proposition 5.4]{Gle20}.
		\end{proof}
		
		In general, it need however not be true that for a v-cover $Y\to X$ of (perfect) schemes the map $Y^\diamondsuit\to X^\diamondsuit$ is a v-cover of small v-sheaves.
		We include the following example which shows that one needs to be careful when applying the above formalism to, say schematic Hecke stacks as in \eqref{equation-schematic-hecke} as they arise as quotients of schematic loop groups which are of infinite type over the base field. 
		
		\begin{example}
			\label{sec:comp-etale-cohom-1-diamondsuit-does-not-preserve-v-covers}
			Let $C$ be a perfect, non-archimedean field over $\mathbb{F}_p$, and fix a pseudo-uniformizer $\pi_C\in C$. 
			Let
			\begin{equation}
			Z:=\{\textstyle{\frac{1}{n}}\ |\ n\in \mathbb{N}\}\cup \{0\}\subset [0,1]
			\end{equation}
			with its subspace topology, which is profinite. 
			We consider the space of continuous functions
			\begin{equation}
			X:=\Spec(C^0(Z,C_{\mathrm{disc}})),
			\end{equation}
			where $C_{\mathrm{disc}}$ is the field $C$ equipped with the discrete topology. 
			Note that $|X|\cong Z$.
			For $n\in \mathbb{N}$, let $g_n,h_n\colon Z\to C_{\mathrm{disc}}$ be the locally constant functions with
			\begin{equation}
			g_n(1/m)=1/\pi_C^m, h_n(1/m)=1
			\end{equation}
			for $m\leq n$ and $g_n(z)=h_n(z)=0$ otherwise. 
			Let
			\begin{equation}
			Y_n\subset \mathbb{A}^1_X=\Spec(C^0(Z,C_{\mathrm{disc}})[T])
			\end{equation}
			be the vanishing locus of $h_nT-g_n$. 
			Then $Y_{n+1}\subset Y_n$, and we can set $Y:=\lim_{n} Y_n$ which is a v-cover of $X$. Indeed, each $Y_n$ is a v-cover, and inverse limits of v-covers between affine schemes are v-covers. 
			More concretely, each map $\Spec(V)\to X$ with $V$ a valuation ring must factor through a (closed) point of $X$, and then it suffices to see that $Y\to X$ is surjective (it is bijective over $X\setminus\{0\}$, and $\mathbb{A}^1_{C_{\mathrm{disc}}}$ over $0$).
			
			We claim that $Y^\diamondsuit\to X^\diamondsuit$ is not a v-cover. Set $R:=C^0(Z, C)$ (now $C$ given its valuation topology), and $R^+=C^0(Z,\mathcal{O}_C)$. 
			The canonical map $C^0(Z,C_{\mathrm{disc}})\to R$ defines a map
			\begin{equation}
			S:=\Spa(R,R^+)\to X^\diamondsuit,
			\end{equation}
			which does not v-locally factor through $Y^\diamondsuit\subset (\mathbb{A}^1_{X})^{\diamondsuit}$. 
			Indeed, assume that $S^\prime\to S$ is a v-cover with $S^\prime$ affinoid and $S^\prime\to Y^\diamondsuit$ a lift of $S\to X^\diamondsuit$. 
			Then the image of $S^\prime\to Y^{\diamondsuit}\times_{X^\diamondsuit} S\subset \mathbb{A}^{1,\mathrm{ad}}_S$ must factor through some quasi-compact subset. 
			But over each point $z_n:=1/n\in Z\cong |S|$ with $n\in \mathbb{N}$, we have that $Y^{\diamondsuit}\times_{X^\diamondsuit} \{z_n\}$ is the point $1/\pi_C^n\in \mathbb{A}^{1,\mathrm{ad}}_{z_n}$, and in $\mathbb{A}^{1,\mathrm{ad}}_S$ this set of points does not lie in a quasi-compact open.
		\end{example}
		
		We get the following consequence of \Cref{sec:comp-etale-cohom-2-diamondsuit-preserves-finitely-presented-v-covers}.
		
		\begin{lemma}
			\label{sec:comp-etale-cohom-1-c-x-fully-faithful-for-stacks}
			Assume $\mathfrak{X},\mathfrak{Y}$ are v-stacks on $\mathrm{SchPerf}_k$ with representable diagonal of perfectly finite presentation, and that $\mathfrak{X},\mathfrak{Y}$ admit a (perfectly) finitely presented v-covers by a perfect schemes, which are of (perfect) finite presentation over $\Spec(k)$. Let $\mathfrak{X}^\diamondsuit$ be the v-stackification of the functor $\Spa(R,R^+)\mapsto \mathfrak{X}(\Spec(R))$ on $\mathrm{Perf}_k$, and similarly for $\mathfrak{Y}$. Let $f\colon \mathfrak{X}\to \mathfrak{Y}$ be a morphism of v-stacks.
			\begin{enumerate}
				\item Then the canonical morphism
				\begin{equation}
				\D(\mathfrak{X}^\diamondsuit,\Lambda)^{\mathrm{Kan}}\to \D(\mathfrak{X}^\diamondsuit,\Lambda)
				\end{equation}
				is an equivalence, and the composite (still denoted $c_{\mathfrak{X}}^\ast$)
				\begin{equation}
				\D(\mathfrak{X},\Lambda)\overset{c_{\mathfrak{X}}^\ast}{\to} \D(\mathfrak{X}^\diamondsuit,\Lambda)^{\mathrm{Kan}}\cong \D(\mathfrak{X}^\diamondsuit,\Lambda)
				\end{equation}
				is fully faithful.
				\item The functors $c^\ast_{\mathfrak{X}}\circ f^\ast, (f^\diamondsuit)^\ast\circ c^{\ast}_{\mathfrak{Y}}$ are naturally isomorphic.
				\item If $\mathfrak{X}\cong [\Spec(k)/H]$ for some perfectly finitely presented group scheme $H$ over $k$, then the functor
				\begin{equation}
				\D(\mathfrak{X},\Lambda)\to \D(\mathfrak{X}^\diamondsuit,\Lambda)
				\end{equation}
				is an equivalence.
			\end{enumerate}
		\end{lemma}
		\begin{proof}
			We prove the first point. 
			From the assumptions on $\mathfrak{X}^\diamondsuit$ we can conclude that the morphisms $U\to \mathfrak{X}$ of perfectly finite presentation such that $U$ is of perfectly finite presentation over $k$ are cofinal among all maps $V\to \mathfrak{X}$ with $V$ a perfect scheme. 
			In particular, in the definition of $\D(\mathfrak{X}^\diamondsuit,\Lambda)^{\mathrm{Kan}}$ one can replace the limit over all $V's$ with a morphism to $\mathfrak{X}$ by the limit over all $U's$ with morphism to $\mathfrak{X}$. Using \Cref{sec:comp-etale-cohom-2-diamondsuit-preserves-finitely-presented-v-covers} and the same argument as for $\D(\mathfrak{X},\Lambda)$, we can then conclude that $\D(\mathfrak{X}^\diamondsuit,\Lambda)^{\mathrm{Kan}}\cong \D(\mathfrak{X}^\diamondsuit,\Lambda)$. 
			Fully faithfulness follows from fully faithfulness of $c_{\mathfrak{X}}^\ast$.
			The second point follows by expressing the categories as limits over $\Delta$ by choosing v-covers $X\to \mathfrak{X}, Y\to \mathfrak{Y}$ of perfectly finite presentation with $X,Y$ of perfectly finite presentation over $k$, such that $X\to \mathfrak{X}\overset{f}{\to} \mathfrak{Y}$ factors over $Y\to \mathfrak{Y}$.
			For the last point, we note that $\D(\Spec(k),\Lambda)\cong \D(\Spa(k,k)^\diamond,\Lambda)$ as both identify naturally with the derived category of $\Lambda$-modules, see \cite[Theorem V.1.1]{FS21}. 
			Computing both sides via the \v{C}ech nerve of the covering $\Spec(k)\to [\Spec(k)/H]$, the statement follows from \cite[Lemma B.6]{BN19}.
		\end{proof}
		
		The results apply to the schematic Hecke stack as follows. 
		Let $O$ be a complete discrete valuation ring with residue field $k$, $\calG$ a parahoric group scheme over $O$ with generic fiber $G$ and $\Fl_{\calG}=L_kG/L^+_k\calG$ the partial affine flag variety for $\calG$ as in \Cref{sec:affine-flag-vari}. 
		Let
		\begin{equation}\label{equation-schematic-hecke}
		\mathrm{Hk}_{\calG,k}^{\mathrm{sch}}:=[L^+_k\calG\backslash L_kG/L^+_k\calG]
		\end{equation}
		be the (schematic) Hecke stack where the quotient is taken in v-stacks on $\mathrm{SchPerf}_k$. 
		Then $\mathrm{Hk}_{\calG,k}:=(\mathrm{Hk}^{\mathrm{sch}}_{\calG,k})^\diamondsuit$ is the small v-stack on $\Perf_k$ considered in \Cref{section_nearby_cycles}.
		We denote by
		\begin{equation}
		\D(\mathrm{Hk}_{\calG,k}^{\mathrm{sch}},\Lambda)^{\mathrm{bd}},\ \D(\mathrm{Hk}_{\calG,k},\Lambda)^{\mathrm{bd}} 
		\end{equation}
		the categories of objects with bounded support.
		
		\begin{proposition}
			\label{sec:comp-etale-cohom-1-derived-categories-of-hecke-stacks-are-isomorphic}
			The categories $\D(\mathrm{Hk}_{\calG,k}^{\mathrm{sch}},\Lambda)^{\mathrm{bd}},\ \D(\mathrm{Hk}_{\calG,k},\Lambda)^{\mathrm{bd}}$ are equivalent.
		\end{proposition}
		\begin{proof}
			Consider a closed substack
			\begin{equation}
			[L^+_k\calG \backslash X]\subset \mathrm{Hk}^{\mathrm{sch}}_{\calG,k}
			\end{equation}
			with $X\subset \Fl_{\calG}$ a closed $L^+_k\calG$-stable subscheme, that is, a union of Schubert varieties.
			By the argument of \cite[Proposition VI.4.1]{FS21}, the vanishing of the cohomology of the affine line over $k$ implies that
			\begin{equation}
			\D([L^+_k\calG \backslash X],\Lambda)\cong \D([H\backslash X],\Lambda)
			\end{equation}
			for any perfectly finitely presented quotient $H$ of $L^+_k\calG$ by some congruence subgroup whose action on $X$ is trivial.
			By \Cref{sec:comp-etale-cohom-1-c-x-fully-faithful-for-stacks}, we have, abusing notation, a natural fully faithful functor
			\begin{equation}
			c^\ast_{[H\backslash X]}\colon \D([H\backslash X],\Lambda)\to \D([H\backslash X]^\diamondsuit,\Lambda)\cong \D([H^\diamondsuit\backslash X^\diamondsuit],\Lambda),
			\end{equation}
			and we claim that this functor is an equivalence.
			We prove this by induction on the number of Schubert strata contained in $X$.
			Let $i\colon Y\subset X$ be a closed $H$-stable (perfect) subscheme with non-empty open complement $j\colon U\to X$, for which $[H\backslash U]$ is a disjoint union of classifying stacks for closed subgroups of $H$. 
			By \Cref{sec:comp-etale-cohom-1-c-x-fully-faithful-for-stacks}, we have
			\begin{equation}
			\D([H\backslash U],\Lambda)\cong \D([H^\diamondsuit\backslash U^\diamondsuit],\Lambda),
			\end{equation}
			and by induction induction hypothesis
			
			\begin{equation}
			\D([H\backslash Y],\Lambda)\cong \D([H^\diamondsuit\backslash {Y}^\diamondsuit],\Lambda).
			\end{equation}
			Let us note that as in \cite[Section 27]{Sch17} the functors $c^\ast_{[H\backslash X]}, c^\ast_{[H\backslash U]}, c^\ast_{[H\backslash Y]}$ admit right adjoints $Rc_{[H\backslash X],\ast}$, $Rc_{[H\backslash U],\ast}$, $Rc_{[H\backslash Y],\ast}$, and it suffices to see that $Rc_{[H\backslash X],\ast}$ is conservative.  
			It is formal that
			\begin{equation}
			\label{eq:6}
			[H\backslash j]^\ast\circ Rc_{[H\backslash X],\ast}\cong Rc_{[H\backslash U],\ast}\circ [H\backslash j]^{\diamondsuit,\ast},
			\end{equation}
			where $[H\backslash j]\colon [H\backslash U]\to [H\backslash X]$ denotes the morphism induced by $j$. More precisely, there exists a natural morphism from the left hand side to the right hand side, and it suffices to see that the morphims induced on the left adjoints is an isomorphism. If $T\to [H\backslash X]$ is a $v$-cover with $T\to \Spec(k)$ of morphism of schemes of finite type, then it suffices (by \Cref{sec:comp-etale-cohom-2-diamondsuit-preserves-finitely-presented-v-covers}) to prove the statement on the isomorphism of left adjoints after pullback to $T^\diamondsuit$. Here, the functor $c_{T}^\ast$ on the \'etale derived categories is induced by a morphism of topoi, and then (\ref{eq:6}) follows by general base change to open subtopoi.

			Let $A\in \D([H\backslash X]^\diamondsuit,\Lambda)$ such that $Rc_{[H\backslash X],\ast}(A)=0$. 
			Then we deduce $[H\backslash j]^{\diamondsuit,\ast}(A)=0$ because $Rc_{[H\backslash U],\ast}$ is an equivalence. 
			In particular, $A\cong [H\backslash i]_\ast^\diamondsuit[H\backslash i]^{\diamondsuit,\ast}(A)$. Now note that
			\begin{equation}
			[H\backslash i]_\ast\circ Rc_{[H\backslash Y],\ast}\cong Rc_{[H\backslash X],\ast}\circ [H\backslash i]^\diamondsuit
			\end{equation}
			as follows by adjunction from $[H\backslash i]^\ast\circ c_{[H\backslash X]}^\ast\cong c_{[H\backslash Y]}^\ast \circ [H\backslash i]^{\diamondsuit,\ast}$. We can conclude that
			\begin{equation}
			[H\backslash i]_{\ast}Rc_{[H\backslash Y],\ast}([H\backslash i]^{\diamondsuit,\ast}A)=0,
			\end{equation}
			which implies $[H\backslash i]^{\diamondsuit,\ast}(A)=0$ because $[H\backslash i]_{\ast}$ is conservative and $Rc_{[H\backslash Y],\ast}$ an equivalence. This implies that $A=0$ as desired.
			
			The equivalence $c^\ast_{[H\backslash X]}$ is natural with respect to inclusions $X\to X^\prime$ of $H$-stable subsets, and morphism $H^\prime\to H$ of quotients of $L^+_k\calG$ (which act trivially on $X$). 
			More precisely, from $c^\ast_{[H\backslash X]}$ we get an equivalence
			\begin{equation}
			\D([L^+_k\calG\backslash X],\Lambda)\cong \D([(L_k^+\calG)^{\diamondsuit}\backslash X^\diamondsuit],\Lambda)
			\end{equation}
			using \cite[VI.4.1]{FS21}, and then we can pass to the colimits of both sides along closed $L^+_k\calG$-stable subschemes of $\Fl_{\calG}$. 
			Then the left hand side is $\D(\mathrm{Hk}_{\calG,k}^{\mathrm{sch}},\Lambda)^{\mathrm{bd}}$ while the right hand side is $\D(\mathrm{Hk}_{\calG,k},\Lambda)^{\mathrm{bd}}$. 
		\end{proof}

	\end{appendices}

	\bibliographystyle{alpha}

\end{document}